\documentclass[a4paper,11pt]{article}
\usepackage{amsmath,amsthm,amssymb}
\usepackage{graphicx}
\usepackage[all]{xy}
\makeatletter
\@addtoreset{equation}{section}
\makeatother

\def\labelenumi{(\theenumi)}
\newtheorem{theorem}{Theorem}[section]
\newtheorem*{theorem*}{Theorem}

\newtheorem{proposition}[theorem]{Proposition}
\newtheorem{corollary}[theorem]{Corollary}
\newtheorem*{corollary*}{Corollary}
\newtheorem{lemma}[theorem]{Lemma}
\newtheorem{example}[theorem]{Example}
\newtheorem{remark}[theorem]{Remark}

\newcommand{\BZ}{\mathbb{Z}}
\newcommand{\BR}{\mathbb{R}}
\newcommand{\BC}{\mathbb{C}}
\newcommand{\BH}{\mathbb{H}}
\newcommand{\BO}{\mathbb{O}}
\newcommand{\BF}{\mathbb{F}}
\newcommand{\bk}{\mathbf{k}}
\newcommand{\bl}{\mathbf{l}}
\newcommand{\bm}{\mathbf{m}}

\newcommand{\bs}{\mathbf{s}}
\newcommand{\cB}{\mathcal{B}}
\newcommand{\cF}{\mathcal{F}}
\newcommand{\cH}{\mathcal{H}}
\newcommand{\cL}{\mathcal{L}}
\newcommand{\cO}{\mathcal{O}}
\newcommand{\cP}{\mathcal{P}}

\newcommand{\fg}{\mathfrak{g}}
\newcommand{\fh}{\mathfrak{h}}
\newcommand{\fk}{\mathfrak{k}}
\newcommand{\fl}{\mathfrak{l}}

\newcommand{\fn}{\mathfrak{n}}
\newcommand{\fp}{\mathfrak{p}}
\newcommand{\rj}{\mathrm{j}}
\newcommand{\rK}{\mathrm{K}}
\newcommand{\tr}{\operatorname{tr}}
\newcommand{\Tr}{\operatorname{Tr}}
\newcommand{\Det}{\operatorname{Det}}
\newcommand{\eqspace}{\mathrel{\phantom{=}}}
\newcommand{\ds}{\displaystyle}
\newcommand{\itinv}{{\mathit{-1}}}
\renewcommand{\Re}{\operatorname{Re}}

\renewcommand{\det}{\operatorname{det}}
\newcommand{\Proj}{\operatorname{Proj}}

\newcommand{\Sym}{\operatorname{Sym}}
\newcommand{\Herm}{\operatorname{Herm}}
\newcommand{\Alt}{\operatorname{Alt}}
\newcommand{\Sol}{\operatorname{Sol}}
\newcommand{\Hom}{\operatorname{Hom}}
\newcommand{\End}{\operatorname{End}}
\newcommand{\rank}{\operatorname{rank}}
\newcommand{\hsum}{\sideset{}{^{\smash{\oplus}}}{\sum}}

\newcommand{\hotimes}{\mathbin{\hat{\otimes}}}
\allowdisplaybreaks[3]

\title{Computation of weighted Bergman inner products on bounded symmetric domains and restriction to subgroups II}
\author{Ryosuke Nakahama\thanks{This work was supported by JST CREST Grant Number JPMJCR2113, Japan.}
\thanks{Email: ryosuke.nakahama@ntt.com} \\
\textit{NTT Institute for Fundamental Mathematics,} \\ \textit{Communication Science Laboratories, NTT, Inc.} \\
\textit{3-9-11 Midori-cho, Musashino-shi, Tokyo 180-8585, Japan}}
\date{\today}

\begin{document}

\maketitle

\begin{abstract}
Let $(G,G_1)$ be a symmetric pair of holomorphic type, and we consider a pair of Hermitian symmetric spaces $D_1=G_1/K_1\subset D=G/K$, 
realized as bounded symmetric domains in complex vector spaces $\fp^+_1\subset\fp^+$ respectively. 
Then the universal covering group $\widetilde{G}$ of $G$ acts unitarily on the weighted Bergman space $\cH_\lambda(D)\subset\cO(D)$ on $D$. 
Its restriction to the subgroup $\widetilde{G}_1$ decomposes discretely and multiplicity-freely, 
and its branching law is given explicitly by Hua--Kostant--Schmid--Kobayashi's formula 
in terms of the $K_1$-decomposition of the space $\cP(\fp^+_2)$ 
of polynomials on the orthogonal complement $\fp^+_2$ of $\fp^+_1$ in $\fp^+$. 
The object of this article is to construct explicitly $\widetilde{G}_1$-intertwining operators (symmetry breaking operators) 
$\cH_\lambda(D)|_{\widetilde{G}_1}\to\cH_{\varepsilon_1\lambda}(D_1,\cP_\bk(\fp^+_2))$ from holomorphic discrete series representations of $\widetilde{G}$ 
to those of $\widetilde{G}_1$, which are unique up to constant multiple for sufficiently large $\lambda$. 
These operators are given by differential operators whose symbols are computed as the inner products of polynomials on $\fp^+_2$. 
In this article, we treat the case $\fp^+,\fp^+_2$ are both simple of tube type and $\rank\fp^+=\rank\fp^+_2$. 
When $\rank\fp^+=3$, we treat all partitions $\bk$, and when $\rank\fp^+$ is general, we treat partitions of the form $\bk=(k,\ldots,k,k-l)$. 
This article is a continuation of the author's previous article \cite{N2}. 
\bigskip

\end{abstract}

\tableofcontents

\section{Introduction}

The purpose of this article is to compute explicitly the weighted Bergman inner product of polynomials on some subspace 
of a bounded symmetric domain, and study the decomposition of the restriction of a holomorphic discrete series representation to some subgroup in detail. 
This article is a direct continuation of the author's previous article \cite{N2}, and also related to the works \cite{N1, N3}. 
In this article, we construct symmetry breaking operators for pairs of representations that were missing in \cite{N2}. 

We consider a Hermitian symmetric space $D\simeq G/K$, realized as a bounded symmetric domain $D\subset\fp^+$ in a complex vector space $\fp^+$ centered at the origin. 
We assume that $G$ is connected, and $K$ is the isotropy subgroup at $0\in\fp^+$. 
Let $\widetilde{G},\widetilde{K}$ denote the universal covering groups of $G,K$, let $(\tau,V)$ be a finite-dimensional representation of $\widetilde{K}$, 
and we consider the homogeneous vector bundle $\widetilde{G}\times_{\widetilde{K}} V\to \widetilde{G}/\widetilde{K}\simeq D$. 
Then we can trivialize this bundle, and the space of holomorphic sections is identified with the space $\cO(D,V)=\cO_\tau(D,V)$ of $V$-valued holomorphic functions on $D$, 
on which $\widetilde{G}$ acts by 
\[ (\hat{\tau}(g)f)(x):=\tau(\kappa(g^{-1},x))^{-1}f(g^{-1}x) \qquad (g\in\widetilde{G},\; x\in D,\; f\in\cO(D,V)), \]
by using a function $\kappa\colon \widetilde{G}\times D\to\widetilde{K}^\BC$ satisfying the cocycle condition. 
According to the choice of the representation $(\tau,V)$ of $\widetilde{K}$, there may or may not exist a Hilbert subspace $\cH_\tau(D,V)\subset\cO_\tau(D,V)$ 
on which $\widetilde{G}$ acts unitarily. The classification of $(\tau,V)$ such that $\cH_\tau(D,V)$ exists is given by \cite{EHW, J}. 
Especially, if the $\widetilde{G}$-invariant inner product of $\cH_\tau(D,V)$ is given by an integral on $D$, 
then this is called a \textit{weighted Bergman inner product} and $(\hat{\tau},\cH_\tau(D,V))$ is called a \textit{holomorphic discrete series representation}. 
When $(\tau,V)$ is of the form $(\tau,V)=(\chi^{-\lambda}\otimes\tau_0,V)$ with a suitably normalized character $\chi$ of $\widetilde{K}$ and a fixed representation $(\tau_0,V)$ of $K$, 
we write $\cH_\tau(D,V)=:\cH_\lambda(D,V)\subset\cO_\tau(D,V)=:\cO_\lambda(D,V)$. In addition, if $(\tau,V)=(\chi^{-\lambda},\BC)$, then we write 
$\cH_\tau(D,V)=:\cH_\lambda(D)\subset\cO_\tau(D,V)=:\cO_\lambda(D)$. 

Let us give a concrete example. We consider 
\begin{align*}
&G=SO^*(2s):=\biggl\{ g\in GL(2s,\BC) \biggm| 
{}^t\hspace{-1pt}g \begin{pmatrix} 0&I\\ I&0 \end{pmatrix}g=\begin{pmatrix} 0&I\\ I&0 \end{pmatrix}, \;
g\begin{pmatrix} 0&I\\ -I&0 \end{pmatrix}=\begin{pmatrix} 0&I\\ -I&0 \end{pmatrix}\overline{g} \biggr\}. 
\end{align*}
Then $G$ acts transitively on 
\[ D:=\{x\in \Alt(s,\BC)\mid I-xx^*\text{ is positive definite}\}\subset\fp^+:=\Alt(s,\BC) \]
by the linear fractional transform 
\[ \begin{pmatrix}a&b\\c&d\end{pmatrix}.x:=(ax+b)(cx+d)^{-1}. \]
Next let $\lambda\in\BC$, and let $(\tau_0,V)$ be an irreducible representation of $K^\BC=GL(s,\BC)$. Then the universal covering group $\widetilde{G}$ acts on $\cO(D,V)$ by 
\[ \biggl(\tau_{\lambda}\biggl(\begin{pmatrix}a&b\\c&d\end{pmatrix}^{-1}\biggr)f\biggr)(x)
:=\det(cx+d)^{-\lambda/2}\tau_0\bigl({}^t\hspace{-1pt}(cx+d)\bigr)f\bigl((ax+b)(cx+d)^{-1}\bigr). \]
We note that $\det(cx+d)^{-\lambda/2}$ is not well-defined on $G\times D$ unless $\lambda\in 2\BZ$, 
but is well-defined on the universal covering space $\widetilde{G}\times D$. 
If $\lambda\in\BR$ is sufficiently large, then $\tau_\lambda$ preserves the weighted Bergman inner product 
\[ \langle f,g\rangle_{\lambda}:=C_\lambda\int_D (f(x),g(x))_{\tau_0}\det(I-xx^*)^{(\lambda-2(s-1))/2}\,dx, \]
and the corresponding Hilbert space $\cH_\lambda(D,V)\subset\cO(D,V)$ gives a holomorphic discrete series representation, 
with the lowest $\widetilde{K}$-type $\chi^{-\lambda}\otimes V={\det}^{-\lambda/2}\otimes V$. 

Next we consider an involution $\sigma$ on $G$, and let $G_1:=(G^\sigma)_0$ be the identity component of the group of fixed points by $\sigma$. 
Without loss of generality, we may assume $\sigma$ stabilizes $K$, and let $K_1:=G_1\cap K$. 
Then $D_1:=G_1.0\simeq G_1/K_1$ is either a complex submanifold or a totally real submanifold of $D\simeq G/K$. 
The pair $(G,G_1)$ is called a symmetric pair of \textit{holomorphic type} for the former case, and of \textit{anti-holomorphic type} for the latter case (see \cite[Section 3.4]{Kmf1}).  
In the following, we assume $(G,G_1)$ is of holomorphic type. Then $\sigma$ induces a holomorphic action on $\fp^+\simeq T_0(G/K)$, 
and let $(\fp^+)^\sigma=:\fp^+_1$, $(\fp^+)^{-\sigma}=:\fp^+_2$, so that $\fp^+=\fp^+_1\oplus\fp^+_2$ holds. 
Now we consider the restriction of a holomorphic discrete series representation $\cH_\tau(D,V)$ of $\widetilde{G}$ to the subgroup $\widetilde{G}_1$. 
Then the restriction $\cH_\tau(D,V)|_{\widetilde{G}_1}$ decomposes into a Hilbert direct sum of irreducible representations of $\widetilde{G}_1$, 
and this decomposition is given in terms of the decomposition of $\cP(\fp^+_2)\otimes (V|_{\widetilde{K}_1})$ under $\widetilde{K}_1$, 
where $\cP(\fp^+_2)$ denotes the space of polynomials on $\fp^+_2$. That is, if $\cP(\fp^+_2)\otimes (V|_{\widetilde{K}_1})$ is decomposed under $\widetilde{K}_1$ as 
\[ \cP(\fp^+_2)\otimes (V|_{\widetilde{K}_1})\simeq\bigoplus_{j} m(\tau,\rho_j)(\rho_j,W_j) \qquad (m(\tau,\rho_j)\in\BZ_{\ge 0}), \]
then $\cH_\tau(D,V)|_{\widetilde{G}_1}$ is decomposed abstractly into the Hilbert direct sum 
\[ \cH_\tau(D,V)|_{\widetilde{G}_1}\simeq\hsum_j m(\tau,\rho_j)\cH_{\rho_j}(D_1,W_j) \]
(see Kobayashi \cite[Lemma 8.8]{Kmf1}, \cite[Section 8]{Kmf0-1}. For earlier results, see also \cite{JV}, \cite{Ma}). 
We note that if the unitary subrepresentation $\cH_\tau(D,V)$ is not a holomorphic discrete series representation, 
then its $\widetilde{K}$-finite part $\cH_\tau(D,V)_{\widetilde{K}}$ may be strictly smaller than $\cO_\tau(D,V)_{\widetilde{K}}=\cP(\fp^+)\otimes V$, 
and the above decomposition may not hold in general. 

One of main problems for restriction of representations is to determine the $\widetilde{G}_1$-\hspace{0pt}intertwining operators 
\begin{alignat*}{4}
&\cF_{\tau\rho_j}^\downarrow\colon &\cH_\tau(D,V)|_{\widetilde{G}_1}&\longrightarrow \cH_{\rho_j}(D_1,W_j) \quad &&\text{or} \quad
& \cO_\tau(D,V)|_{\widetilde{G}_1}&\longrightarrow \cO_{\rho_j}(D_1,W_j), \\
&\cF_{\tau\rho_j}^\uparrow\colon &\cH_{\rho_j}(D_1,W_j)&\longrightarrow \cH_\rho(D,V)|_{\widetilde{G}_1} \quad &&\text{or} \quad 
& \cO_{\rho_j}(D_1,W_j)&\longrightarrow \cO_\rho(D,V)|_{\widetilde{G}_1}. 
\end{alignat*}
Such $\cF_{\tau\rho_j}^\downarrow$ is called a symmetry breaking operator (SBO) and $\cF_{\tau\rho_j}^\uparrow$ is called a holographic operator, 
according to the terminology introduced in \cite{KP1, KP2} and \cite{KP3} respectively.  
Such a problem is proposed by Kobayashi from the viewpoint of the representation theory (see \cite{K1}), 
and studied from various viewpoints for holomorphic discrete series, principal series, and complementary series representations. 
For example, differential SBOs for tensor products of holomorphic discrete series representations of $SL(2,\BR)$ are Rankin--Cohen brackets. 
The study of such operators was initiated by \cite{C, R} from the viewpoint of automorphic forms, and their generalizations for larger groups are studied by e.g., \cite{BCK, Cl, OR, P, PZ, vDP}. 
Similarly, the study of differential SBOs for spherical principal series representations of $(G,G_1)=(SO_0(1,n+1),SO_0(1,n))$ was initiated by Juhl \cite{Ju} 
from the viewpoint of conformal geometry on the sphere $S^n$, and is generalized to the indefinite case \cite{KOSS}. 
Also, differential SBOs for symmetric pairs $(G,G_1)$ of higher split rank are studied by e.g., Ibukiyama--Kuzumaki--Ochiai \cite{IKO} and the author \cite{N1, N2, N3}. 
In addition, the complete classification of general (possibly non-differential) SBOs for spherical principal series representations of $(G,G_1)=(O(1,n+1),O(1,n))$ is given by Kobayashi--Speh \cite{KS1}, 
and general SBOs for various symmetric pairs are studied by e.g., \cite{FO, FW1, FW2, KL, KS2, MO}. 
On the other hand, the study of holographic operators is recently initiated by Kobayashi--Pevzner \cite{KP3} and is studied by, e.g., \cite{La, N1}. 

Especially, it is proved by \cite{Kfmeth, KP1} that in the holomorphic setting, symmetry breaking operators $\cF_{\tau\rho_j}^\downarrow\colon\cO_\tau(D,V)|_{\widetilde{G}_1}\to\cO_{\rho_j}(D_1,W_j)$ 
are always given by differential operators, that is, $\cF_{\tau\rho_j}^\downarrow$ are of the form 
\begin{align*}
\cF_{\tau\rho_j}^\downarrow\colon \cO_\tau(D,V)|_{\widetilde{G}_1}&\longrightarrow\cO_{\rho_j}(D_1,W_j), \\
f(x)=f(x_1,x_2)&\longmapsto F_{\tau\rho_j}^\downarrow\biggl(\frac{\partial}{\partial x}\biggr)f(x)\biggr|_{x_2=0}
\end{align*}
for some operator-valued polynomials $F_{\tau\rho_j}^\downarrow(z)\in\cP(\fp^-)\otimes\Hom_\BC(V,W_j)$ on the dual space $\fp^-$ of $\fp^+$ (localness theorem), 
and $F_{\tau\rho_j}^\downarrow(z)$ are characterized as polynomial solutions of certain systems of differential equations (F-method). 
This F-method gives a new insight on the study of differential symmetry breaking operators, which unifies Rankin--Cohen brackets and Juhl's operators, and gives new results \cite{KKP, KOSS, KP2}. 
Also, in \cite{N1}, the author proved that when $\cH_\tau(D,V)\subset\cO_\tau(D,V)$ is a holomorphic discrete series representation, $F_{\tau\rho_j}^\downarrow(z)$ is given by the inner product 
\[ F_{\tau\rho_j}^\downarrow(z)=\Bigl\langle e^{(x|z)_{\fp^+}},\rK(x_2)\Bigr\rangle_{\cH_{\tau}(D,V),x}, \quad 
\rK(x_2)\in (\cP(\fp^+_2)\otimes\Hom_\BC(\overline{V},\overline{W_j}))^{\widetilde{K}_1} \]
up to constant multiple. Here $(\cdot|\cdot)_{\fp^+}$ is a suitable non-degenerate pairing on $\fp^+\times\fp^-$. 
Especially, if $\cH_\tau(D,V)=\cH_\lambda(D)$ is a unitary highest weight representation of scalar type, 
then its restriction $\cH_\lambda(D)|_{\widetilde{G}_1}$ decomposes multiplicity-freely by \cite[Theorem A]{Kmf1}, 
and symmetry breaking operators from $\cH_\lambda(D)|_{\widetilde{G}_1}$ are unique up to constant multiple. 
Moreover, if $\cH_\lambda(D)$ is a holomorphic discrete series representation, then symmetry breaking operators from $\cO_\lambda(D)|_{\widetilde{G}_1}$ are also unique, 
since $\cH_\lambda(D)\subset\cO_\lambda(D)$ is dense. 
On the other hand, if $\cO_\lambda(D)$ does not contain a holomorphic discrete series representation, 
we do not know a priori whether symmetry breaking operators from $\cO_\lambda(D)|_{\widetilde{G}_1}$ are unique or not. 
Indeed, for tensor product case, the dimension of the space of symmetry breaking operators from $\cO_\lambda(D)\hotimes\cO_\mu(D)$ may become greater than 1 for some $(\lambda,\mu)\in\BC$ 
(see \cite[Section 9]{KP2}, \cite[Section 8]{N3}). 

In the following, we assume $G$ is simple, and consider the branching of the holomorphic discrete series representations of scalar type $\cH_\tau(D,V)\equiv\cH_\lambda(D)$. 
We recall that $G_1=(G^\sigma)_0$, and set $G_2:=(G^{\sigma\vartheta})_0$, where $\vartheta$ is the Cartan involution of $G$ corresponding to $K$. 
We additionally assume that both $G$ and the non-compact part of $G_2$ are of tube type, or equivalently, both $\fp^+$ and $\fp^+_2$ have unital Jordan algebra structures. 
Then one of the following holds. 
\begin{enumerate}
\setlength{\parskip}{2pt}
\setlength{\itemsep}{2pt}
\item $\fp^+_2$ is a direct sum of two simple Jordan algebras, and $\rank\fp^+_2=\rank\fp^+$. 
\item $\fp^+_2$ is simple, and $\rank\fp^+_2=\rank\fp^+/2$. 
\item $\fp^+_2$ is simple, and $\rank\fp^+_2=\rank\fp^+$. 
\end{enumerate}
Case (1) corresponds to the triples $(G,G_1,G_2)$ one of 
\[ \begin{array}{ccccl}
(&SO_0(2,d+2), &SO_0(2,d)\times SO(2), &SO_0(2,2)\times SO(d)&), \\ (&Sp(r,\BR), &U(r',r''), &Sp(r',\BR)\times Sp(r'',\BR)&), \\
(&U(r,r), &U(r',r'')\times U(r'',r'), &U(r',r')\times U(r'',r'')&), \\ (&SO^*(4r), &U(2r',2r''), &SO^*(4r')\times SO^*(4r'')&), \\ 
(&E_{7(-25)}, &U(1)\times E_{6(-14)}, &SL(2,\BR)\times Spin_0(2,10)&) 
\end{array} \]
($r=r'+r''$). Let $\fp^+_2=\fp^+_{11}\oplus\fp^+_{22}$. Then the decomposition of $\cH_\lambda(D)|_{\widetilde{G}_1}$ is given by 
\[ \cH_\lambda(D)|_{\widetilde{G}_1}\simeq\hsum_{(\bk,\bl)\in\BZ_{++}^{r'}\times\BZ_{++}^{r''}}\cH_{\varepsilon_1\lambda}(D_1,\cP_\bk(\fp^+_{11})\boxtimes\cP_\bl(\fp^+_{22})), \]
where $\BZ_{++}^r:=\{\bk=(k_1,\ldots,k_r)\in\BZ^r\mid k_1\ge\cdots\ge k_r\ge 0\}$, and $\varepsilon_1\in\{1,2\}$. 
If $\bk=(k,\ldots,k)$ and $\bl=(l,\ldots,l)$, then $\cH_{\varepsilon_1\lambda}(D_1,\cP_{(k,\ldots,k)}(\fp^+_{11})\boxtimes\cP_{(l,\ldots,l)}(\fp^+_{22}))$ is of scalar type. 
In \cite{N1, N2}, the author constructed the intertwining operators 
\begin{align*}
&\cF^\uparrow_{\lambda,\bk,\bl}\colon\cH_{\varepsilon_1\lambda}(D_1,\cP_\bk(\fp^+_{11})\boxtimes\cP_\bl(\fp^+_{22}))_{\widetilde{K}_1}\longrightarrow
\cH_\lambda(D)_{\widetilde{K}}|_{\fg_1}, \\ 
&\cF^\downarrow_{\lambda,\bk,\bl}\colon\cH_\lambda(D)|_{\widetilde{G}_1}\longrightarrow\cH_{\varepsilon_1\lambda}(D_1,\cP_\bk(\fp^+_{11})\boxtimes\cP_\bl(\fp^+_{22})), 
\end{align*}
when $\bk$ is of the form $\bk=(k,\ldots,k)$ $(k\in\BZ_{\ge 0})$ and $\bl\in\BZ_{++}^{r''}$ is arbitrary. 
Next, Case (2) corresponds to the triples $(G,G_1,G_2)$ one of 
\[ \begin{array}{ccccl}
(&SO_0(2,n), &SO_0(2,n-1), &SO_0(2,1)\times SO(n-1)&), \\ (&Sp(2r_2,\BR), &Sp(r_2,\BR)\times Sp(r_2,\BR), &U(r_2,r_2)&), \\
(&SU(2r_2,2r_2), &Sp(2r_2,\BR), &SO^*(4r_2)&). 
\end{array} \]
Then the decomposition of $\cH_\lambda(D)|_{\widetilde{G}_1}$ is given by 
\[ \cH_\lambda(D)|_{\widetilde{G}_1}\simeq\hsum_{\bk\in\BZ_{++}^{r_2}}\cH_{\varepsilon_1\lambda}(D_1,\cP_\bk(\fp^+_2)). \]
If $\bk=(k,\ldots,k)$, then $\cH_{\varepsilon_1\lambda}(D_1,\cP_{(k,\ldots,k)}(\fp^+_2))\simeq\cH_{\varepsilon_1(\lambda+k)}(D_1)$ is of scalar type. 
In \cite{N1, N2}, the author constructed the intertwining operators $\cF^\uparrow_{\lambda,\bk}$, $\cF^\downarrow_{\lambda,\bk}$ when $\bk$ is of the form $\bk=(k+1,\ldots,k+1,k,\ldots,k)$, 
so that $\cH_{\varepsilon_1\lambda}(D_1,\cP_\bk(\fp^+_2))$ is multiplicity-free under the maximal compact subgroup $\widetilde{K}_1\subset \widetilde{G}_1$. 
Next, Case (3) corresponds to the triples $(G,G_1,G_2)$ one of 
\[ \begin{array}{ccccl}
(&SO_0(2,n), &SO_0(2,n')\times SO(n''), &SO_0(2,n'')\times SO(n')&), \\ (&SU(r,r), &SO^*(2r), &Sp(r,\BR)&), \\ (&SO^*(4r), &SO^*(2r)\times SO^*(2r), &U(r,r)&), \\
(&E_{7(-25)}, &SU(2,6), &SU(2)\times SO^*(12)&) 
\end{array} \]
($n''\ge 3$). Then the decomposition of $\cH_\lambda(D)|_{\widetilde{G}_1}$ is given by 
\[ \cH_\lambda(D)|_{\widetilde{G}_1}\simeq\hsum_{\bk\in\BZ_{++}^r}\cH_{\varepsilon_1\lambda}(D_1,\cP_\bk(\fp^+_2)). \]
If $\bk=(k,\ldots,k)$, then $\cH_{\varepsilon_1\lambda}(D_1,\cP_{(k,\ldots,k)}(\fp^+_2))\simeq\cH_{\varepsilon_1(\lambda+2k)}(D_1)$ is of scalar type. 
In \cite{N1}, the author constructed the holographic operators $\cF^\uparrow_{\lambda,\bk}$ when $\bk$ is of the form one of 
{\def\labelenumi{(3-\alph{enumi})}
\begin{enumerate}
\setlength{\parskip}{2pt}
\setlength{\itemsep}{2pt}
\item $\bk=(k+l,k,\ldots,k)$ \; $(k,l\in\BZ_{\ge 0})$, 
\item $\bk=(k,\ldots,k,k-l)$ \; $(k,l\in\BZ_{\ge 0},\; k\ge l)$, 
\item $\bk=(k_1,k_2,k_3)\in\BZ_{++}^3$ \; (when $r=\rank\fp^+=3$), 
\end{enumerate}}
\noindent so that $\cH_{\varepsilon_1\lambda}(D_1,\cP_\bk(\fp^+_2))$ is multiplicity-free under $\widetilde{K}_1\subset \widetilde{G}_1$ for classical $G$, 
and in \cite{N2}, the author constructed the symmetry breaking operators $\cF^\downarrow_{\lambda,\bk}$ when $\bk$ is of the form (3-a). 
In any case, if $\bk=(k,\ldots,k)$, $\bl=(l,\ldots,l)$, then $\cF^\uparrow_{\lambda,\bk,\bl}$, $\cF^\uparrow_{\lambda,\bk}$ are given by infinite-order differential operators 
whose symbols are given by multivariate ${}_0F_1$ functions, and $\cF^\downarrow_{\lambda,\bk,\bl}$, $\cF^\downarrow_{\lambda,\bk}$ are given by finite-order differential operators 
whose symbols are given by multivariate ${}_2F_1$ polynomials (or equivalently, Heckman--Opdam's hypergeometric polynomials of type $BC$, see \cite{BO}). 
In this article, we consider $(G,G_1,G_2)$ of Case (3), and we construct the symmetry breaking operators $\cF^\downarrow_{\lambda,\bk}$ when $\bk$ is of the form (3-b) and (3-c). 
The result for (3-c) is conjectured in \cite[Section 7]{N3}, and this is verified in this paper. 

In addition, in \cite{N3}, the author determined the operator norms of $\cF^\downarrow_{\lambda,\bk,\bl}$, $\cF^\downarrow_{\lambda,\bk}$ for all $(\bk,\bl)$, $\bk$ 
under a suitable normalization, and determined the Parseval--Plancherel type formula. 
Also, in the same paper, the author found polynomials $b_{\bk,\bl}(\lambda),b_\bk(\lambda)\in\BC[\lambda]$ such that 
\begin{align*}
b_{\bk,\bl}(\lambda)\Bigl\langle f(x_2),e^{(x|\overline{z})_{\fp^+}}\Bigr\rangle_{\lambda,x}& &&(f(x_2)\in\cP_\bk(\fp^+_{11})\boxtimes\cP_\bl(\fp^+_{22})), \\
b_\bk(\lambda)\Bigl\langle f(x_2),e^{(x|\overline{z})_{\fp^+}}\Bigr\rangle_{\lambda,x}& &&(f(x_2)\in\cP_\bk(\fp^+_2))
\end{align*}
are holomorphically continued for all $\lambda\in\BC$ for all $(\bk,\bl)$, $\bk$. 
This result gives some information on the branching of $\cO_\lambda(D)_{\widetilde{K}}|_{\fg_1}$ for $\lambda$ such that $\cO_\lambda(D)_{\widetilde{K}}$ is reducible. 
By \cite{N2}, when $f(x_2)\ne 0$, the above give non-zero polynomials for all $\lambda\in\BC$ if, for example, 
$(G,G_1,G_2)$ is of Case (1) and $\bk=(k,\ldots,k)$, $\bl\in\BZ_{++}^{\min\{r',r''\}}\subset\BZ_{++}^{r''}$, 
$(G,G_1,G_2)$ is of Case (2) and $\bk=(k+1,\ldots,k+1,k,\ldots,k)$, 
or $(G,G_1,G_2)$ is of Case (3) with $r$ even and $\bk=(k+l,k,\ldots,k)$ ((3-a)). 
However, for many other cases, the non-vanishing of the above formula is not proved. 
In this paper, we prove the non-vanishing for $(G,G_1,G_2)$ of Case (3) (with $r$ both even and odd) and $\bk$ is of the form either (3-a), (3-b) or (3-c). 

This paper is organized as follows. In Section \ref{section_prelim}, we review Jordan triple systems, Jordan algebras and holomorphic discrete series representations. 
In Section \ref{section_rank3}, we consider the case $\rank\fp^+=3$ and general $\bk\in\BZ_{++}^3$, i.e., consider (3-c) case. 
In Section \ref{section_general}, we consider $\fp^+$ of general rank, and $\bk$ of the form (3-b). 
For the convenience of the readers, we also rewrite the results for $\bk$ of the form (3-a) given in \cite{N2}.

\section{Preliminaries}\label{section_prelim}

In this section, we review Jordan triple systems, Jordan algebras and holomorphic discrete series representations. 
We use almost the same notations as in \cite{N2, N3}, and readers who read \cite{N2, N3} can skip this section, except for Section \ref{subsection_proj}.  
For detail see, e.g., \cite[Parts III, V]{FKKLR}, \cite{FK}, \cite{L0}, \cite{L}, \cite{Sat}. 

\subsection{Hermitian positive Jordan triple systems}\label{subsection_HPJTS}

We consider a Hermitian positive Jordan triple system $(\fp^+,\fp^-,\{\cdot,\cdot,\cdot\},\overline{\cdot})$, 
where $\fp^\pm$ are finite-dimensional vector spaces over $\BC$, with a non-degenerate bilinear form $(\cdot|\cdot)_{\fp^\pm}\colon \fp^\pm\times\fp^\mp\to\BC$, 
$\{\cdot,\cdot,\cdot\}\colon\fp^\pm\times\fp^\mp\times\fp^\pm\to\fp^\pm$ is a $\BC$-trilinear map satisfying 
\begin{align*}
\{x,y,z\}&=\{z,y,x\}, \\
\{u,v,\{x,y,z\}\}&=\{\{u,v,x\},y,z\}-\{x,\{v,u,y\},z\}+\{x,y,\{u,v,z\}\}, \\
(\{u,v,x\}|y)_{\fp^\pm}&=(x|\{v,u,y\})_{\fp^\pm} 
\end{align*}
for any $u,x,z\in\fp^\pm$, $v,y\in\fp^\mp$, and $\overline{\cdot}\colon\fp^\pm\to\fp^\mp$ is a $\BC$-antilinear involutive isomorphism 
satisfying $(x|\overline{x})_{\fp^\pm}\ge 0$ for any $x\in\fp^\pm$. 
Let $D,B\colon\fp^\pm\times\fp^\mp\to\End_\BC(\fp^\pm)$, $Q\colon\fp^\pm\times\fp^\pm\to\Hom_\BC(\fp^\mp,\fp^\pm)$, $Q\colon\fp^\pm\to\Hom_\BC(\fp^\mp,\fp^\pm)$ be the maps given by 
\begin{align*}
D(x,y)z=Q(x,z)y&:=\{x,y,z\}, \\
Q(x)&:=\frac{1}{2}Q(x,x), \\
B(x,y)&:=I_{\fp^\pm}-D(x,y)+Q(x)Q(y)
\end{align*}
for $x,z\in\fp^\pm$, $y\in\fp^\mp$, let $h=h_{\fp^\pm}\colon\fp^\pm\times\fp^\mp\to\BC$ be the \textit{generic norm}, which is a polynomial on $\fp^+\times\fp^-$ 
irreducible on each simple Jordan triple subsystem, and write $B(x):=B(x,\overline{x})$, $h(x):=h(x,\overline{x})$. 
If $\fp^+$ is simple, then $B(x,y)$ and $h(x,y)$ are related as 
\begin{equation}\label{formula_hB}
h(x,y)^p=\Det_{\fp^+}B(x,y) \qquad (x\in\fp^+,\; y\in\fp^-) 
\end{equation}
for some $p\in\BZ_{>0}$. 

If $e\in\fp^+$ is a tripotent, i.e., $\{e,\overline{e},e\}=2e$, then $D(e,\overline{e})\in\End_\BC(\fp^+)$ has the eigenvalues $0,1,2$. For $j=0,1,2$, let 
\begin{align}
\fp^+(e)_j=\fp^+(\overline{e})_j:=\{x\in\fp^+\mid D(e,\overline{e})x=jx\}\subset\fp^+, \label{Peirce}\\
\fp^-(\overline{e})_j=\fp^-(e)_j:=\{x\in\fp^-\mid D(\overline{e},e)x=jx\}\subset\fp^-, \notag
\end{align}
so that $\fp^\pm=\fp^\pm(e)_2\oplus\fp^\pm(e)_1\oplus\fp^\pm(e)_0$ holds (\textit{Peirce decomposition}). 
A non-zero tripotent $e\in\fp^+$ is called \textit{primitive} if $\fp^+(e)_2=\BC e$, and \textit{maximal} if $\fp^+(e)_0=\{0\}$. 
If $\fp^+(e)_2=\fp^+$ holds for some (or equivalently any) maximal tripotent $e\in\fp^+$, then we say that $\fp^+$ is of \textit{tube type}. 
Throughout the paper we assume that the bilinear form $(\cdot|\cdot)_{\fp^\pm}\colon\fp^\pm\times\fp^\mp\to\BC$ is normalized such that 
$(e|\overline{e})_{\fp^+}=(\overline{e}|e)_{\fp^-}=1$ holds for any primitive tripotent $e\in\fp^+$. 
If $\fp^+$ is simple, then under this normalization, $(x|y)_{\fp^+}$ and $D(x,y)$ are related as 
\begin{equation}\label{formula_pairingD}
p(x|y)_{\fp^+}=\Tr_{\fp^+}D(x,y) \qquad (x\in\fp^+,\; y\in\fp^-)
\end{equation}
with the same $p\in\BZ_{>0}$ as in (\ref{formula_hB}).

\subsection{Jordan algebras}

Next we consider a complex Jordan algebra $\fn^{+\BC}$ with the Euclidean real form $\fn^+\subset\fn^{+\BC}$ and the unit element $e\in\fn^+$, that is, the product on $\fn^{+\BC}$ satisfies 
\[ x\circ y=y\circ x, \qquad x^{\mathit{2}}\circ(x\circ y)=x\circ(x^{\mathit{2}}\circ y) \qquad (x,y\in\fn^{+\BC}), \]
where $x^{\mathit{2}}:=x\circ x$, and there exists a symmetric bilinear form $(\cdot|\cdot)_{\fn^+}\colon \fn^{+\BC}\times\fn^{+\BC}\to\BC$, which is positive definite on $\fn^+$, satisfying 
\[ (x\circ y|z)_{\fn^+}=(x|y\circ z)_{\fn^+} \qquad (x,y,z\in\fn^{+\BC}). \]
Let $L\colon\fn^{+\BC}\to\End_{\BC}(\fn^{+\BC})$, $D_{\fn^+},P\colon\fn^{+\BC}\times\fn^{+\BC}\to\End_\BC(\fn^{+\BC})$, $P\colon\fn^{+\BC}\to\End_\BC(\fn^{+\BC})$ be the maps given by 
\begin{align*}
L(x)y&:=x\circ y, \\
D_{\fn^+}(x,y)z=P(x,z)y&:=2(x\circ(y\circ z)+z\circ(y\circ x)-(x\circ z)\circ y), \\
P(x)y&:=2x\circ(x\circ y)-x^{\mathit{2}}\circ y=\frac{1}{2}P(x,x)y, 
\end{align*}
let $\tr_{\fn^+}$, $\det_{\fn^+}$ be the trace form and the determinant polynomial of $\fn^{+\BC}$, and we normalize $(\cdot|\cdot)_{\fn^+}$ by $(x|e)_{\fn^+}=\tr_{\fn^+}(x)$. 
If $\fn^{+\BC}$ is simple of rank $r$, dimension $n$, then we have 
\begin{align*}
(x|y)_{\fn^+}&=\tr_{\fn^+}(x\circ y)=\frac{r}{n}\Tr(L(x\circ y))=\frac{r}{2n}\Tr(D_{\fn^+}(x,y)), \\ \det_{\fn^+}(x)^{2n/r}&=\Det(P(x)). 
\end{align*}
For $x\in\fn^{+\BC}$, if $\det_{\fn^+}(x)\ne 0$, then the inverse $x^\itinv$ and the adjugate element $x^\sharp$ are given by 
\[ x^\itinv:=P(x)^{-1}x, \qquad x^\sharp:=\det_{\fn^+}(x)x^\itinv. \]
Then the directional derivatives of $x^\itinv$ and $\det_{\fn^+}(x)$ are given by 
\begin{align*}
\frac{d}{dt}\biggr|_{t=0}(x+ty)^\itinv&=-P(x)^{-1}y, \quad \frac{d}{dt}\biggr|_{t=0}\det_{\fn^+}(x+ty)=\det_{\fn^+}(x)(x^\itinv|y)_{\fn^+}=(x^\sharp|y)_{\fn^+}. 
\end{align*}
In addition, let 
\begin{align*}
\Omega:\hspace{-3pt}&=(\text{connected component of }\{x\in\fn^+\mid P(x)\text{ is positive definite}\} \text{ which contains }e) \\
&=(\text{connected component of }\{x\in\fn^+\mid \det_{\fn^+}(x)>0\}\text{ which contains }e)
\end{align*}
be the \textit{symmetric cone}. 

For a given Hermitian positive Jordan triple system $\fp^\pm$, if it is of tube type, then for a fixed maximal tripotent $e\in\fp^+$, 
\[ Q(\overline{e})\colon \fp^+\longrightarrow\fp^-, \qquad Q(e)\colon \fp^-\longrightarrow\fp^+ \]
are mutually inverse, and $\fp^+$ becomes a Jordan algebra by the product 
\[ x\circ y:=\frac{1}{2}\{x,\overline{e},y\} \qquad (x,y\in\fp^+), \]
with the unit element $e\in\fp^+$, and the Euclidean real form 
\[ \fn^+=\{x\in\fp^+\mid Q(e)\overline{x}=x\}. \]
Then the corresponding operations and the generic norm are given by 
\begin{gather*}
D_{\fn^+}(x,y)=D(x,Q(\overline{e})y), \quad P(x,y)=Q(x,y)Q(\overline{e}), \quad P(x)=Q(x)Q(\overline{e}), \\ 
(x|y)_{\fn^+}=(x|Q(\overline{e})y)_{\fp^+}=(y|Q(\overline{e})x)_{\fp^+}, \quad h_{\fn^+}(x,y)=h_{\fp^+}(x,Q(\overline{e})y).  
\end{gather*}
Throughout the paper, for a Hermitian positive Jordan triple system $\fp^\pm$ of tube type, when a maximal tripotent $e\in\fp^+$ is fixed, 
we omit $Q(e),Q(\overline{e})$, and identify $\fp^+=\fp^-$ by the identity map, so that $\overline{\cdot}\colon\fp^\pm\to\fp^\mp=\fp^\pm$ is the complex conjugate with respect to the 
Euclidean real form $\fn^+\subset\fp^+$, and the corresponding operations are 
\begin{gather*}
D_{\fn^+}(x,y)=D(x,y), \qquad P(x,y)=Q(x,y), \qquad P(x)=Q(x), \\ (x|y)_{\fn^+}=(x|y)_{\fp^+}, \qquad h_{\fn^+}(x,y)=h_{\fp^+}(x,y). 
\end{gather*}
Conversely, a complex unital Jordan algebra $\fn^{+\BC}$ with the Euclidean real form $\fn^+\subset\fn^{+\BC}$ becomes a Hermitian positive Jordan triple system by 
$\{x,y,z\}=D_{\fn^+}(x,y)z$ and by the complex conjugate $\bar{\cdot}\colon\fn^{+\BC}\to\fn^{+\BC}$ with respect to $\fn^+$.

\subsection{Structure groups and the Kantor--Koecher--Tits construction}\label{subsection_KKT}

In this subsection, we consider some Lie algebras corresponding to Jordan triple systems $\fp^\pm$. 
For $l\in\End_\BC(\fp^+)$, let $\overline{l},{}^t\hspace{-1pt}l\in\End_\BC(\fp^-)$, $l^*\in\End_\BC(\fp^+)$ be the elements given by 
$\overline{l}\overline{x}=\overline{lx}$, $(lx|\overline{y})_{\fp^+}=(x|{}^t\hspace{-1pt}l\overline{y})_{\fp^+}=(x|\overline{l^*y})_{\fp^+}$ for $x,y\in\fp^+$. 
Then the \textit{Structure group} $\operatorname{Str}(\fp^+)$ and the \textit{automorphism group} $\operatorname{Aut}(\fp^+)$ are defined as 
\begin{align*}
\operatorname{Str}(\fp^+)&:=\{l\in GL_\BC(\fp^+)\mid \{lx,{}^t\hspace{-1pt}l^{-1}y,lz\}=l\{x,y,z\} \quad (x,z\in\fp^+,\; y\in\fp^-)\}, \\
\operatorname{Aut}(\fp^+)&:=\{k\in \operatorname{Str}(\fp^+)\mid k^{-1}=k^*\}. 
\end{align*}
Let $\mathfrak{str}(\fp^+)=\fk^\BC$ and $\mathfrak{der}(\fp^+)=\fk$ denote the Lie algebras of $\operatorname{Str}(\fp^+)$ and $\operatorname{Aut}(\fp^+)$ respectively. 
Then $D(\fp^+,\fp^-)=\mathfrak{str}(\fp^+)$ holds. 

Next we recall the \textit{Kantor--Koecher--Tits construction}. As vector spaces let 
\begin{align*}
\fg^\BC&:=\fp^+\oplus\fk^\BC\oplus\fp^-, \\
\fg&:=\{(x,k,\overline{x})\mid x\in\fp^+,\; k\in\fk\}\subset\fg^\BC, 
\end{align*}
and give the Lie algebra structure on $\fg^\BC$ by 
\[ [(x,k,y),(z,l,w)]:=(kz-lx,[k,l]+D(x,w)-D(z,y),-{}^t\hspace{-1pt}kw+{}^t\hspace{-1pt}ly). \]
Then $\fg^\BC$ becomes a Lie algebra and $\fg$ becomes a real form of $\fg^\BC$. 
Let $(\cdot|\cdot)_{\fg^\BC}\colon \fg^\BC\times\fg^\BC\to\BC$ be the $\fg^\BC$-invariant bilinear form normalized such that $(x|y)_{\fg^\BC}=(x|y)_{\fp^+}$ holds for any $x\in\fp^+$, $y\in\fp^-$. 
We fix a connected complex Lie group $G^\BC$ with the Lie algebra $\fg^\BC$, 
and let $G,K^\BC,K,P^+,P^-\subset G^\BC$ be the connected closed subgroups corresponding to the Lie algebras $\fg,\fk^\BC,\fk,\fp^+,\fp^-\subset\fg^\BC$ respectively. 
Then $Ad|_{\fp^+}\colon K^\BC\to\operatorname{Str}(\fp^+)_0$ gives a covering map, and for $l\in K^\BC$, $x\in\fp^+$, we abbreviate $Ad(l)x=:lx$. 

Next suppose $\fp^+$ is of tube type. We fix a tripotent $e\in\fp^+$, regard $\fp^+$ as a Jordan algebra, and let $\fn^+\subset\fp^+$ be its Euclidean real form. Also, let 
\begin{alignat*}{2}
\fn^-&:=\overline{\fn^+}=Q(\overline{e})\fn^+&&\subset\fp^-, \\
\fl&:=D(\fn^+,\fn^-)=[\fn^+,\fn^-]&&\subset\fk^\BC, \\
{}^c\hspace{-1pt}\fg&:=\fn^+\oplus\fl\oplus\fn^-&&\subset\fg^\BC. 
\end{alignat*}
These become real forms of the right hand sides. Let ${}^c\hspace{-1pt}G\subset G^\BC$ be the connected closed subgroup corresponding to the Lie algebra ${}^c\hspace{-1pt}\fg\subset\fg^\BC$, 
and let $L:=K^\BC\cap {}^c\hspace{-1pt}G$, $K_L:=L\cap K$, $\fk_\fl:=\fl\cap\fk$. 
Then $G$ and ${}^c\hspace{-1pt}G$ are isomorphic via the Cayley transform in $G^\BC$, 
$L$ acts transitively on the symmetric cone $\Omega\subset\fn^+$, and $K_L$ acts on $\fn^+$ as Jordan algebra automorphisms. 
Also, for $l\in\fk^\BC\subset\End_\BC(\fp^+)=\End_\BC(\fn^{+\BC})$, we abbreviate $Q(e){}^t\hspace{-1pt}lQ(\overline{e})=:{}^t\hspace{-1pt}l\in\fk^\BC$, 
and extend it to the anti-automorphism on $K^\BC$.

\subsection{Simultaneous Peirce decomposition}\label{subsection_simul_Peirce}

In this subsection, we assume $\fp^+$ is simple, or equivalently, the corresponding Lie algebra $\fg$ is simple. 
We fix a Jordan frame $\{e_1,\ldots,e_r\}\subset\fp^+$, i.e., a maximal set of primitive tripotents in $\fp^+$ satisfying $D(e_i,\overline{e_j})=0$ for $i\ne j$, 
where $r=\rank\fp^+=\rank_\BR\fg$. Using this, let 
\begin{align*}
\fp^+_{ij}&:=\{x\in\fp^+\mid D(e_l,\overline{e_l})x=(\delta_{il}+\delta_{jl})x \quad (l=1,\ldots,r)\} && (1\le i\le j\le r), \\
\fp^-_{ij}&:=\overline{\fp^+_{ij}} && (1\le i\le j\le r), \\
\fp^+_{0j}&:=\{x\in\fp^+\mid D(e_l,\overline{e_l})x=\delta_{jl}x \quad (l=1,\ldots,r)\} && (1\le i\le r), \\
\fp^-_{0j}&:=\overline{\fp^+_{0j}} && (1\le i\le r), 
\end{align*}
so that 
\[ \fp^\pm=\bigoplus_{\substack{0\le i\le j\le r \\ (i,j)\ne(0,0)}}\fp^\pm_{ij} \]
holds. This is called the \textit{simultaneous Peirce decomposition}. Using this decomposition, we define integers $(d,b,p,n)$ by 
\begin{align}
d&:=\dim\fp^+_{ij} \quad (1\le i<j\le r), & b&:=\dim\fp^+_{0j} \quad (1\le j\le r), \notag \\
p&:=2+d(r-1)+b, & n&:=\dim\fp^+=r+\frac{d}{2}r(r-1)+br. \label{str_const}
\end{align}
We note that if $r=1$, then $d$ is not determined uniquely, and any number is allowed. 
Then $p$ coincides with the one in (\ref{formula_hB}), (\ref{formula_pairingD}).

\subsection{Space of polynomials on Jordan triple systems}\label{subsection_polynomials}

In this subsection, we consider the space $\cP(\fp^\pm)$ of polynomials on the Jordan triple system $\fp^\pm$, on which $K^\BC$ acts by 
\begin{align*}
(Ad|_{\fp^+}(l))^\vee f(x)&=f(l^{-1}x) && (l\in K^\BC,\; f\in\cP(\fp^+),\; x\in\fp^+), \\
(Ad|_{\fp^-}(l))^\vee f(y)&=f({}^t\hspace{-1pt}ly) && (l\in K^\BC,\; f\in\cP(\fp^-),\; y\in\fp^-). 
\end{align*}
We assume $\fp^+$ is simple, fix a Jordan frame $\{e_1,\ldots,e_r\}\subset\fp^+$, and consider the tripotents $e^k:=\sum_{j=1}^k e_j$. 
Then the Peirce subspaces $\fp^+(e^k)_2$ corresponding to $e^k$ (see (\ref{Peirce})) are given by 
\[ \fp^+(e^k)_2=\bigoplus_{1\le i\le j\le k}\fp^+_{ij}. \]
These are of tube type, and have Jordan algebra structures. Let $\fn^+(e^k)\subset\fp^+(e^k)_2$ be their Euclidean real forms, 
and we extend the determinant polynomials $\det_{\fn^+(e^k)}$ on $\fp^+(e^k)_2$ to polynomials on $\fp^+$. We also write $e^r=:e$, $\fn^+(e^r)=:\fn^+$. Using these, for 
\[ \bm\in\BZ_{++}^r:=\{\bm=(m_1,\ldots,m_r)\in\BZ^r\mid m_1\ge m_2\ge \cdots\ge m_r\ge 0\}, \]
we define the polynomial $\Delta^{\fn^+}_\bm(x)$ on $\fp^+$ by 
\[ \Delta^{\fn^+}_\bm(x):=\prod_{k=1}^r \det_{\fn^+(e^k)}(x)^{m_k-m_{k+1}}, \]
where we set $m_{r+1}:=0$, and let 
\begin{alignat*}{2}
\cP_\bm(\fp^+)&:=\operatorname{span}_\BC \{\Delta^{\fn^+}_\bm(l^{-1}x)\mid l\in K^\BC\}&&\subset\cP(\fp^+), \\
\cP_\bm(\fp^-)&:=\{\overline{f(\overline{y})}\mid f(x)\in\cP_\bm(\fp^+)\} &&\subset\cP(\fp^-). 
\end{alignat*}
Then the following holds. 

\begin{theorem}[{Hua--Kostant--Schmid, \cite[Part III, Theorem V.2.1]{FKKLR}, \cite[Theorem XI.2.4]{FK}, \cite{H, Jo, Sch}}]\label{thm_HKS}
Under the $K^\BC$-action, $\cP(\fp^\pm)$ is decomposed into the sum of irreducible submodules as 
\[ \cP(\fp^\pm)=\bigoplus_{\bm\in\BZ_{++}^r}\cP_\bm(\fp^\pm). \]
\end{theorem}

Next suppose $\fp^+$ is of tube type. For $\mu\in\BC$, let 
\[ \underline{\mu}_r:=(\underbrace{\mu,\ldots,\mu}_r)\in\BC^r, \]
and for $\bm\in\BZ_{++}^r+\underline{\mu}_r$, let 
\[ \cP_\bm(\fp^\pm):=\cP_{\bm-\underline{\mu}_r}(\fp^\pm)\det_{\fn^\pm}(x)^\mu\subset\cP(\fp^\pm)\det_{\fn^\pm}(x)^\mu, \]
where $\det_{\fn^-}(x)$ is the determinant polynomial on $\fp^-=\fn^{-\BC}$. Then $\cP_\bm(\fp^+)\simeq\cP_{-\bm^\vee}(\fp^-)$ holds as a $\widetilde{K}^\BC$-module, 
where $\bm^\vee:=(m_r,\ldots,m_1)$ and $\widetilde{K}^\BC$ is the universal covering group of $K^\BC$. In addition, for $\lambda\in\BC$, $\bs\in\BC^r$, $\bm\in(\BZ_{\ge 0})^r$, let 
\begin{align}
\Gamma_r^d(\lambda+\bs)&:=(2\pi)^{dr(r-1)/4}\prod_{j=1}^r \Gamma\biggl(\lambda+s_j-\frac{d}{2}(j-1)\biggr), \label{Gamma} \\
(\lambda+\bs)_{\bm,d}&:=\prod_{j=1}^r\biggl(\lambda+s_j-\frac{d}{2}(j-1)\biggr)_{m_j}, \label{Pochhammer}
\end{align}
where $(\lambda)_m:=\lambda(\lambda+1)(\lambda+2)\cdots(\lambda+m-1)$, and we abbreviate $\Gamma_r^d(\lambda+\underline{0}_r)=:\Gamma_r^d(\lambda)$, 
$\Gamma_r^d(0+\bs)=:\Gamma_r^d(\bs)$, $(\lambda+\underline{0}_r)_{\bm,d}=:(\lambda)_{\bm,d}$, $(0+\bs)_{\bm,d}=:(\bs)_{\bm,d}$. 
Then the following hold. 

\begin{lemma}[Gindikin, {\cite[Lemma XI.2.3, Section IX.3]{FK}}]\label{lem_Laplace}
Let $\bm\in\BZ_{++}^r$, $f\in\cP_\bm(\fp^+)$. 
\begin{enumerate}
\item For $\Re\lambda>\frac{n}{r}-1$, $w\in\Omega+\sqrt{-1}\fn^+$, we have 
\[ \int_\Omega e^{-(z|w)_{\fn^+}}f(z)\det_{\fn^+}(z)^{\lambda-\frac{n}{r}}dz=\Gamma_r^d(\lambda+\bm)f(w^\itinv)\det_{\fn^+}(w)^{-\lambda}. \]
\item For $\Re\lambda>\frac{2n}{r}-1$, $z,a\in\Omega$, we have 
\[ \frac{\Gamma_r^d(\lambda+\bm)}{(2\pi\sqrt{-1})^n}\int_{a+\sqrt{-1}\fn^+}e^{(z|w)_{\fn^+}}f(w^\itinv)\det_{\fn^+}(w)^{-\lambda}dw
=f(z)\det_{\fn^+}(z)^{\lambda-\frac{n}{r}}, \]
and this does not depend on the choice of $a\in\Omega$. 
\end{enumerate}
\end{lemma}

\begin{lemma}[{\cite[Proposition VII.1.6]{FK}}]\label{lem_diff}
For $k\in\BZ_{\ge 0}$, $\bm\in\BZ_{++}^r+\BC\underline{1}_r:=\bigcup_{\mu\in\BC}(\BZ_{++}^r+\underline{\mu}_r)$, $f(x)\in\cP_\bm(\fp^+)$, we have 
\begin{align*}
\det_{\fn^+}\biggl(\frac{\partial}{\partial x}\biggr)^k f(x) 
&=\biggl(\frac{d}{2}(r-1)+1-k+\bm\biggr)_{\underline{k}_r,d} f(x)\det_{\fn^+}(x)^{-k} \\
&=(-1)^{kr}(-\bm^\vee)_{\underline{k}_r,d} f(x)\det_{\fn^+}(x)^{-k}. 
\end{align*}
\end{lemma}

\subsection{Projection of polynomials}\label{subsection_proj}

In this subsection, we assume $\fp^+=\fn^{+\BC}$ is of tube type. For $\bm\in\BZ_{++}^r$, let 
\[ \Proj_\bm^{\fp^+}\colon\cP(\fp^+)\longrightarrow \cP_\bm(\fp^+) \]
be the orthogonal projection. Also, let $y\in\fp^+$ be an element of rank 1. Then for $\bk\in\BZ_{++}^r$, $l\in\BZ_{\ge 0}$, $f(x)\in\cP_\bk(\fp^+)$, as a function of $x$ we have 
\begin{equation}\label{formula_Proj_Pieri}
(x|y)_{\fn^+}^l f(x)=\sum_{\substack{\bm\in\BZ_{++}^r,\,|\bm|=|\bk|+l \\ k_{j-1}\ge m_j\ge k_j}} \Proj_\bm^{\fp^+}\bigl((x|y)_{\fn^+}^l f(x)\bigr), 
\end{equation}
where we set $k_0:=+\infty$. More generally, for $\mu\in\BC$, $\bm\in\BZ_{++}^r+\underline{\mu}_r$, let 
\[ \Proj_\bm^{\fp^+}\colon\cP(\fp^+)\det_{\fn^+}(x)^\mu\longrightarrow \cP_\bm(\fp^+)=\cP_{\bm-\underline{\mu}_r}(\fp^+)\det_{\fn^+}(x)^\mu \]
be the orthogonal projection, with the same symbol. 

\begin{lemma}\label{lem_Proj_Pieri}
Suppose that $\fp^+=\fn^{+\BC}$ is of tube type. 
We fix a primitive idempotent $e'\in\fn^+$, let $\fp^+(e')_2=\BC e'=:\fp^{+\prime}$, $\fp^+(e')_0=:\fp^{+\prime\prime}$ be as in (\ref{Peirce}), 
and for $x\in\fp^+$, let $x'\in\fp^{+\prime}$, $x''\in\fp^{+\prime\prime}$ be the orthogonal projections. 
Also, let $y\in\fp^+$ be an element of rank 1. 
\begin{enumerate}
\item For $\mu\in\BC$, $\bm\in\BZ_{++}^r+\underline{\mu}_r$, $k,l\in\BZ_{\ge 0}$, $f(x)\in\cP(\fp^+)\det_{\fn^+}(x)^\mu$, we have 
\[ \Proj_{\bm+\underline{k}_r}^{\fp^+}\bigl((x|y)_{\fn^+}^l\det_{\fn^+}(x)^kf(x)\bigr)=\det_{\fn^+}(x)^k\Proj_\bm^{\fp^+}\bigl((x|y)_{\fn^+}^lf(x)\bigr). \]
\item For $\bk\in\BZ_{++}^{r-1}$, $\bm\in\BZ_{++}^r$, $k,l\in\BZ_{\ge 0}$ with $k\le \min\{l,m_r\}$ and for $f(x'')\in\cP_\bk(\fp^{+\prime\prime})$, we have 
\begin{align*}
&\Proj_\bm^{\fp^+}\bigl((x|y)_{\fn^+}^l f(x'')\bigr) \\
&=\frac{(-l)_k(-\bk^\vee)_{\underline{k}_{r-1},d}}{(-\bm^\vee)_{\underline{k}_r,d}}\det_{\fn^+}(x)^k(e'|y)_{\fn^+}^k
\Proj_{\bm-\underline{k}_r}^{\fp^+}\bigl((x|y)_{\fn^+}^{l-k}\det_{\fn^{+\prime\prime}}(x'')^{-k}f(x'')\bigr). 
\end{align*}
\item For $\bk\in\BZ_{++}^{r-1}$, $l\in\BZ_{\ge 0}$ with $l\le k_{r-1}$ and for $f(x'')\in\cP_\bk(\fp^{+\prime\prime})$, we have 
\begin{align*}
\Proj_{(\bk,l)}^{\fp^+}\bigl((x|y)_{\fn^+}^l f(x'')\bigr)
&=\frac{(-\bk^\vee)_{\underline{l}_{r-1},d}}{\bigl(-\frac{d}{2}-\bk^\vee\bigr)_{\underline{l}_{r-1},d}}\det_{\fn^+}(x)^l(e'|y)_{\fn^+}^l\det_{\fn^{+\prime\prime}}(x'')^{-l}f(x''). 
\end{align*}
\item For $\bk\in\BZ_{++}^{r-1}$, $l\in\BZ_{\ge 0}$ with $l\ge k_1$ and for $f(x'')\in\cP_\bk(\fp^{+\prime\prime})$, we have 
\[ \Proj_{(l,\bk)}^{\fp^+}\bigl((x|e')_{\fn^+}^l f(x'')\bigr)\bigr|_{x=x'+x''\in\fp^{+\prime}\oplus\fp^{+\prime\prime}}
=\frac{(-l)_{\bk,d}}{\bigl(-l-\frac{d}{2}\bigr)_{\bk,d}}(x'|e')_{\fn^+}^lf(x''). \]
\item There exists an element $y\in\fp^+$ of rank 1 such that for $\mu\in\BC$, $\bk,\bm\in\BZ_{++}^r+\underline{\mu}_r$, $l\in\BZ_{\ge 0}$ with $k_{j-1}\ge m_j\ge k_j$, $|\bm|=|\bk|+l$, we have 
\[ \Proj_\bm^{\fp^+}\bigl((x|y)_{\fn^+}^l\Delta_\bk^{\fn^+}(x)\bigr)\ne 0. \]
\end{enumerate}
\end{lemma}

\begin{proof}
(1) Clear from $\cP_{\bm+\underline{k}_r}(\fp^+)=\det_{\fn^+}(x)^k\cP_\bm(\fp^+)$. 

(2) Suppose $k_r=0$, $f(x)=f(x'')\in\cP_\bk(\fp^{+\prime\prime})$, and we apply $\det_{\fn^+}\bigl(\frac{\partial}{\partial x}\bigr)^k$ to the both sides of (\ref{formula_Proj_Pieri}). 
Then the right hand side is computed as 
\begin{align*}
&\det_{\fn^+}\biggl(\frac{\partial}{\partial x}\biggr)^k\sum_{\substack{\bm\in\BZ_{++}^r,\,|\bm|=|\bk|+l \\ k_{j-1}\ge m_j\ge k_j}} \Proj_\bm^{\fp^+}\bigl((x|y)_{\fn^+}^l f(x'')\bigr) \\
&=\sum_{\substack{\bm\in\BZ_{++}^r,\,|\bm|=|\bk|+l \\ k_{j-1}\ge m_j\ge k_j \\ k_{r-1}\ge m_r\ge k}} (-1)^{kr}(-\bm^\vee)_{\underline{k}_r,d}
\det_{\fn^+}(x)^{-k}\Proj_\bm^{\fp^+}\bigl((x|y)_{\fn^+}^l f(x'')\bigr). 
\end{align*}
Similarly, for the left hand side, since $y$ is of rank 1, the operator $\det_{\fn^+}\bigl(\frac{\partial}{\partial x}\bigr)^ke^{(x|y)_{\fn^+}}$ is characterized by 
\begin{align*}
&\det_{\fn^+}\biggl(\frac{\partial}{\partial x}\biggr)^ke^{(x|y)_{\fn^+}}e^{(x|z)_{\fn^+}}=\det_{\fn^+}(y+z)^ke^{(x|y)_{\fn^+}}e^{(x|z)_{\fn^+}} \\
&=\bigl(\det_{\fn^+}(z)+(z^\sharp|y)_{\fn^+}\bigr)^ke^{(x|y)_{\fn^+}}e^{(x|z)_{\fn^+}} \\
&=\sum_{l'=0}^k\sum_{l''=0}^\infty \binom{k}{l'}\frac{1}{l''!}\det_{\fn^+}(z)^{k-l'}(z^\sharp|y)_{\fn^+}^{l'}(x|y)_{\fn^+}^{l''}e^{(x|z)_{\fn^+}} \qquad (z\in\fp^+), 
\end{align*}
and by taking the homogeneous terms with respect to $y$, we have 
\begin{align*}
\det_{\fn^+}\biggl(\frac{\partial}{\partial x}\biggr)^k(x|y)_{\fn^+}^l
=l!\sum_{l'+l''=l}\binom{k}{l'}\frac{1}{l''!}(x|y)_{\fn^+}^{l''}\det_{\fn^+}\biggl(\frac{\partial}{\partial x}\biggr)^{k-l'}
\biggl(\biggl(\frac{\partial}{\partial x}\biggr)^\sharp\biggm|y\biggr)_{\fn^+}^{l'}. 
\end{align*}
Then, since $\det_{\fn^+}\bigl(\frac{\partial}{\partial x}\bigr)^{k-l'}f(x'')$ vanishes for $l'<k$, we get 
\begin{align*}
&\det_{\fn^+}\biggl(\frac{\partial}{\partial x}\biggr)^k(x|y)_{\fn^+}^l f(x'')
=\frac{l!}{(l-k)!}(x|y)_{\fn^+}^{l-k}\biggl(\biggl(\frac{\partial}{\partial x}\biggr)^\sharp\biggm|y\biggr)_{\fn^+}^k f(x'') \\
&=(-1)^k(-l)_k(x|y)_{\fn^+}^{l-k}(e'|y)_{\fn^+}^k\det_{\fn^{+\prime\prime}}\biggl(\frac{\partial}{\partial x''}\biggr)^k f(x'') \\
&=(-1)^{kr}(-l)_k(-\bk^\vee)_{\underline{k}_{r-1},d}(e'|y)_{\fn^+}^k(x|y)_{\fn^+}^{l-k}\det_{\fn^{+\prime\prime}}(x'')^{-k}f(x'') \\
&=(-1)^{kr}(-l)_k(-\bk^\vee)_{\underline{k}_{r-1},d}(e'|y)_{\fn^+}^k \\* &\eqspace{}\times\sum_{\substack{\bm\in\BZ_{++}^r,\,|\bm|=|\bk|+l \\ k_{j-1}\ge m_j\ge k_j \\ k_{r-1}\ge m_r\ge k}}
\Proj_{\bm-\underline{k}_r}^{\fp^+}\bigl((x|y)_{\fn^+}^{l-k}\det_{\fn^{+\prime\prime}}(x'')^{-k}f(x'')\bigr). 
\end{align*}
Comparing the both sides, for $m_r\ge k$ we have 
\begin{align*}
&(-\bm^\vee)_{\underline{k}_r,d}\det_{\fn^+}(x)^{-k}\Proj_\bm^{\fp^+}\bigl((x|y)_{\fn^+}^l f(x'')\bigr) \\ 
&=(-l)_k(-\bk^\vee)_{\underline{k}_{r-1},d}(e'|y)_{\fn^+}^k \Proj_{\bm-\underline{k}_r}^{\fp^+}\bigl((x|y)_{\fn^+}^{l-k}\det_{\fn^{+\prime\prime}}(x'')^{-k}f(x'')\bigr), 
\end{align*}
and we get the desired formula. 

(3) In (2), we set $\bm=(\bk,l)$, $k=l$. Then we have 
\begin{align*}
&\Proj_{(\bk,l)}^{\fp^+}\bigl((x|y)_{\fn^+}^l f(x'')\bigr) \\
&=\frac{(-l)_l(-\bk^\vee)_{\underline{l}_{r-1},d}}{(-(l,\bk^\vee))_{\underline{l}_r,d}}
\det_{\fn^+}(x)^l(e'|y)_{\fn^+}^l\Proj_{(\bk,l)-\underline{l}_r}^{\fp^+}\bigl(\det_{\fn^{+\prime\prime}}(x'')^{-l}f(x'')\bigr) \\
&=\frac{(-l)_l(-\bk^\vee)_{\underline{l}_{r-1},d}}{(-l)_l\bigl(-\frac{d}{2}-\bk^\vee\bigr)_{\underline{l}_{r-1},d}}
\det_{\fn^+}(x)^l(e'|y)_{\fn^+}^l\det_{\fn^{+\prime\prime}}(x'')^{-l}f(x'') \\
&=\frac{(-\bk^\vee)_{\underline{l}_{r-1},d}}{\bigl(-\frac{d}{2}-\bk^\vee\bigr)_{\underline{l}_{r-1},d}}\det_{\fn^+}(x)^l(e'|y)_{\fn^+}^l\det_{\fn^{+\prime\prime}}(x'')^{-l}f(x''). 
\end{align*}

(4) Since the map 
\begin{align*}
\cP_\bk(\fp^{+\prime\prime})&\longrightarrow\cP(\fp^{+\prime})\otimes\cP(\fp^{+\prime\prime}), \\
f(x'')&\longmapsto \Proj_{(l,\bk)}^{\fp^+}\bigl((x|e')_{\fn^+}^l f(x'')\bigr)\bigr|_{x=x'+x''\in\fp^{+\prime}\oplus\fp^{+\prime\prime}}
\end{align*}
is $\fk^{\BC\prime\prime}:=[\fp^{+\prime\prime},\fp^{+\prime\prime}]$-equivariant, it is enough to prove when $f(x'')\in\cP(\fp^{+\prime\prime})$ is a lowest weight vector. 
We fix a Jordan frame $\{e_2,\ldots,e_r\}$ of $\fn^{+\prime\prime}$ so that $\{e_1=e',e_2,\ldots,e_r\}$ is a Jordan frame of $\fn^+$, 
and define $\Delta_\bk^{\fn^{+\prime\prime}}(x'')\in\cP_\bk(\fp^{+\prime\prime})$, $\Delta_\bm^{\fn^+}(x)\in\cP_\bm(\fp^+)$ by using these Jordan frames. 
We prove 
\begin{equation}\label{formula_Proj_Pieri1}
\Proj_{(l,\bk)}\bigl((x|e')_{\fn^+}^l\Delta_\bk^{\fn^{+\prime\prime}}(x'')\bigr)=\frac{(-l)_{\bk,d}}{\bigl(-l-\frac{d}{2}\bigr)_{\bk,d}}\Delta_{(l,\bk)}^{\fn^+}(x)
\end{equation}
by induction on $r=\rank\fp^+$. When $\bk=\underline{0}_{r-1}$ (``$r=1$ case''), this is trivial. Next we assume (\ref{formula_Proj_Pieri1}) for $r-1$, and prove it for $r$. 
Then by (2) with $k=k_{r-1}$, $y=e'$, we have 
\begin{align*}
&\Proj_{(l,\bk)}\bigl((x|e')_{\fn^+}^l \Delta_\bk^{\fn^{+\prime\prime}}(x'')\bigr) \\
&=\frac{(-l)_{k_{r-1}}(-\bk^\vee)_{\underline{k_{r-1}}_{r-1},d}}{(-(\bk^\vee,l))_{\underline{k_{r-1}}_r,d}} \\*
&\eqspace{}\times \det_{\fn^+}(x)^{k_{r-1}}\Proj_{(l,\bk)-\underline{k_{r-1}}_r}^{\fp^+}\bigl((x|e')_{\fn^+}^{l-k_{r-1}}
\det_{\fn^{+\prime\prime}}(x'')^{-k_{r-1}}\Delta_\bk^{\fn^{+\prime\prime}}(x'')\bigr) \\
&=\frac{(-l)_{k_{r-1}}(-\bk^\vee)_{\underline{k_{r-1}}_{r-1},d}}{(-\bk^\vee)_{\underline{k_{r-1}}_{r-1},d}\bigl(-l-\frac{d}{2}(r-1)\bigr)_{k_{r-1},d}} \\*
&\eqspace{}\times \det_{\fn^+}(x)^{k_{r-1}}\Proj_{(l,\bk)-\underline{k_{r-1}}_r}^{\fp^+}\bigl((x|e')_{\fn^+}^{l-k_{r-1}}\Delta_{\bk-\underline{k_{r-1}}_{r-1}}^{\fn^{+\prime\prime}}(x'')\bigr) \\
&=\frac{(-l)_{k_{r-1}}}{\bigl(-l-\frac{d}{2}(r-1)\bigr)_{k_{r-1}}}\det_{\fn^+}(x)^{k_{r-1}} 
\frac{(-l+k_{r-1})_{\bk-\underline{k_{r-1}}_{r-1},d}}{\bigl(-l+k_{r-1}-\frac{d}{2}\bigr)_{\bk-\underline{k_{r-1}}_{r-1},d}}
\Delta_{(l,\bk)-\underline{k_{r-1}}_r}^{\fn^+}(x) \\
&=\frac{(-l)_{k_{r-1}}}{\bigl(-l-\frac{d}{2}(r-1)\bigr)_{k_{r-1}}} 
\frac{(-l)_{\bk,d}}{\bigl(-l-\frac{d}{2}\bigr)_{\bk,d}}\frac{\bigl(-l-\frac{d}{2}\bigr)_{\underline{k_{r-1}}_{r-1},d}}{(-l)_{\underline{k_{r-1}}_{r-1},d}}\Delta_{(l,\bk)}^{\fn^+}(x) \\
&=\frac{(-l)_{\bk,d}}{\bigl(-l-\frac{d}{2}\bigr)_{\bk,d}}\frac{(-l)_{\underline{k_{r-1}}_r,d}}{(-l)_{\underline{k_{r-1}}_r,d}}\Delta_{(l,\bk)}^{\fn^+}(x)
=\frac{(-l)_{\bk,d}}{\bigl(-l-\frac{d}{2}\bigr)_{\bk,d}}\Delta_{(l,\bk)}^{\fn^+}(x), 
\end{align*}
where we have applied the induction hypothesis (\ref{formula_Proj_Pieri1}) for $\fp^+(e_r)_0=\fp^+(\sum_{j=1}^{r-1}e_j)_2\subset\fp^+$ at the 3rd equality. 
This proves (\ref{formula_Proj_Pieri1}) for all $r$. Then since 
\[ \Delta_{(l,\bk)}^{\fn^+}(x)\bigr|_{x=x'+x''\in\fp^{+\prime}\oplus\fp^{+\prime\prime}}=(x'|e')_{\fn^+}^l\Delta_\bk^{\fn^{+\prime\prime}}(x'') \]
holds, we get the desired formula. 

(5) Let $\{e_1,\ldots,e_r\}\subset\fn^+$ be the Jordan frame defining $\Delta_\bk^{\fn^+}(x)$, and we take an element $y\in\fp^+$ of rank 1 
such that $(e_j|y)_{\fn^+}\ne 0$ holds for all $j=1,\ldots,r$. We prove 
\begin{equation}\label{formula_Proj_Pieri2}
\Proj_\bm^{\fp^+}\bigl((x|y)_{\fn^+}^l\Delta_\bk^{\fn^+}(x)\bigr)\ne 0
\end{equation}
by induction on $r$. When $r=1$ this is clear. Next we assume (\ref{formula_Proj_Pieri2}) for $r-1$, and prove it for $r$. 
Then by (1) we have 
\[ \Proj_\bm^{\fp^+}\bigl((x|y)_{\fn^+}^l\Delta_\bk^{\fn^+}(x)\bigr)=\det_{\fn^+}(x)^{k_r}\Proj_{\bm-\underline{k_r}_r}^{\fp^+}\bigl((x|y)_{\fn^+}^l\Delta_{\bk-\underline{k}_r}^{\fn^+}(x)\bigr), \]
and it suffices to prove (\ref{formula_Proj_Pieri2}) when $k_r=0$. Then, by applying (2) for $e'=e_r$, $k=m_r$ so that $\fp^{+\prime\prime}=\fp^+(e_r)_0=\fp^+(\sum_{j=1}^{r-1}e_j)_2$, we have 
\begin{align*}
&\Proj_\bm^{\fp^+}\bigl((x|y)_{\fn^+}^l\Delta_\bk^{\fn^+}(x)\bigr)
=\Proj_\bm^{\fp^+}\bigl((x|y)_{\fn^+}^l\Delta_\bk^{\fn^{+\prime\prime}}(x'')\bigr) \\
&=\frac{(-l)_{m_r}(-\bk^\vee)_{\underline{m_r}_{r-1},d}}{(-\bm^\vee)_{\underline{m_r}_r,d}} 
\det_{\fn^+}(x)^{m_r}(e_r|y)_{\fn^+}^{m_r}\Proj_{\bm-\underline{m_r}_r}^{\fp^+}\bigl((x|y)_{\fn^+}^{l-m_r}
\Delta_{\bk-\underline{m_r}_{r-1}}^{\fn^{+\prime\prime}}(x'')\bigr), 
\end{align*}
and it suffices to prove (\ref{formula_Proj_Pieri2}) when $k_r=m_r=0$. Then for $k_r=0$, by restricting 
\[ (x|y)_{\fn^+}^l \Delta_\bk^{\fn^{+\prime\prime}}(x'')
=\sum_{\substack{\bm\in\BZ_{++}^r,\,|\bm|=|\bk|+l \\ k_{j-1}\ge m_j\ge k_j}} \Proj_\bm^{\fp^+}\bigl((x|y)_{\fn^+}^l \Delta_\bk^{\fn^{+\prime\prime}}(x'')\bigr), \]
to the subalgebra $\fp^{+\prime\prime}=\fp^+(e_r)_0\subset\fp^+$, we get 
\begin{align*}
&\sum_{\substack{\bm\in\BZ_{++}^r,\,|\bm|=|\bk|+l \\ k_{j-1}\ge m_j\ge k_j \\ m_r=0}} \Proj_\bm^{\fp^+}\bigl((x|y)_{\fn^+}^l \Delta_\bk^{\fn^{+\prime\prime}}(x'')\bigr)\bigr|_{\fp^{+\prime\prime}}
=(x|y)_{\fn^+}^l \Delta_\bk^{\fn^{+\prime\prime}}(x'')\bigr|_{\fp^{+\prime\prime}} \\
&=(x''|y'')_{\fn^{+\prime\prime}}^l \Delta_\bk^{\fn^{+\prime\prime}}(x'')
=\sum_{\substack{\bm\in\BZ_{++}^{r-1},\,|\bm|=|\bk|+l \\ k_{j-1}\ge m_j\ge k_j}} \Proj_\bm^{\fp^{+\prime\prime}}\bigl((x''|y'')_{\fn^{+\prime\prime}}^l 
\Delta_\bk^{\fn^{+\prime\prime}}(x'')\bigr), 
\end{align*}
and hence for $k_r=m_r=0$, we have 
\begin{align*}
&\Proj_\bm^{\fp^+}\bigl((x|y)_{\fn^+}^l \Delta_\bk^{\fn^{+\prime\prime}}(x'')\bigr)\bigr|_{\fp^{+\prime\prime}}
=\Proj_\bm^{\fp^{+\prime\prime}}\bigl((x''|y'')_{\fn^{+\prime\prime}}^l \Delta_\bk^{\fn^{+\prime\prime}}(x'')\bigr). 
\end{align*}
By the induction hypothesis, this is non-zero. Hence (\ref{formula_Proj_Pieri2}) holds for all $r$. 
\end{proof}

Also, the following holds. 

\begin{proposition}\label{prop_Proj}
Suppose that $\fp^+$ is of tube type and $\rank\fp^+\ge 2$. 
Let $k,l,m\in\BZ_{\ge 0}$, $m\le k,l$, and let $y,z\in\fp^+$ be elements of rank 1. Then we have 
\begin{align*}
&\Proj_{(l-m,\underline{0}_{r-2},m-k)}^{\fp^+}\bigl((x|y)^l_{\fn^+}(x^\itinv|z)^k_{\fn^+}\bigr) \\
&=\frac{(-1)^m}{\bigl(-k-l+m-\frac{d}{2}(r-1)\bigr)_mm!}
\sum_{j=m}^{\min\{k,l\}}\frac{(-k)_j(-l)_j}{\bigl(-k-l+2m-\frac{d}{2}(r-1)+1\bigr)_{j-m}(j-m)!} \\*
&\hspace{255pt}\times(x|y)^{l-j}_{\fn^+}(x^\itinv|z)^{k-j}_{\fn^+}(y|z)^j_{\fn^+}, \\
&\Proj_{(k+l-m,m,\underline{0}_{r-2})}^{\fp^+}\bigl((x|y)^l_{\fn^+}(x|z)^k_{\fn^+}\bigr) \\
&=\frac{(-1)^m}{\bigl(-k-l+m-\frac{d}{2}\bigr)_mm!}
\sum_{j=m}^{\min\{k,l\}}\frac{(-k)_j(-l)_j}{\bigl(-k-l+2m-\frac{d}{2}+1\bigr)_{j-m}(j-m)!} \\*
&\hspace{150pt}\times(x|y)^{l-j}_{\fn^+}(x|z)^{k-j}_{\fn^+}\bigl((x|y)_{\fn^+}(x|z)_{\fn^+}-(P(x)y|z)_{\fn^+}\bigr)^j. 
\end{align*}
\end{proposition}

\begin{proof}
(1) First we prove 
\begin{align}
&\bigoplus_{j=0}^{\min\{k,l\}}\BR(x|y)^{l-j}_{\fn^+}(x^\itinv|z)^{k-j}_{\fn^+}(y|z)_{\fn^+}^j
=\bigoplus_{m=0}^{\min\{k,l\}}\BR\Proj_{(l-m,\underline{0}_{r-2},m-k)}^{\fp^+}\bigl((x|y)^l_{\fn^+}(x^\itinv|z)^k_{\fn^+}\bigr). \label{formula_proof_span}
\end{align}
Indeed, by the definition of $\Proj_\bm^{\fp^+}$ we have 
\[ (x|y)^l_{\fn^+}(x^\itinv|z)^k_{\fn^+}
=\sum_{m=0}^{\min\{k,l\}}\Proj_{(l-m,\underline{0}_{r-2},m-k)}^{\fp^+}\bigl((x|y)^l_{\fn^+}(x^\itinv|z)^k_{\fn^+}\bigr). \]
We fix a primitive idempotent $e'\in\fn^+$, and let $\fn^+(e')_0=:\fn^{+\prime\prime}$. Then, since $z$ is of rank 1, $z=ge'$ holds for some $g\in K^\BC$. 
Then, as in the proof of Lemma \ref{lem_Proj_Pieri}\,(2), for $j=0,1,\ldots,\min\{k,l\}$ we have 
\begin{align}
&\det_{\fn^+}(x)^{j-k}\det_{\fn^+}\biggl(\frac{\partial}{\partial x}\biggr)^j\det_{\fn^+}(x)^k(x|y)^l_{\fn^+}(x^\itinv|z)^k_{\fn^+} \notag \\
&=\det_{\fn^+}(g^{-1}x)^{j-k}\det_{\fn^+}\biggl({}^t\hspace{-1pt}g\frac{\partial}{\partial x}\biggr)^j\det_{\fn^+}(g^{-1}x)^k
(g^{-1}x|{}^t\hspace{-1pt}gy)^l_{\fn^+}((g^{-1}x)^\itinv|e')^k_{\fn^+} \notag \\
&=\det_{\fn^+}(w)^{j-k}\det_{\fn^+}\biggl(\frac{\partial}{\partial w}\biggr)^j
(w|{}^t\hspace{-1pt}gy)^l_{\fn^+}\det_{\fn^{+\prime\prime}}(w'')^k\biggr|_{w=g^{-1}x} \notag \\
&=(-1)^{jr}(-l)_j(-k)_{\underline{j}_{r-1},d} \det_{\fn^+}(w)^{j-k}(w|{}^t\hspace{-1pt}gy)^{l-j}_{\fn^+}
\det_{\fn^{+\prime\prime}}(w'')^{k-j}(e'|{}^t\hspace{-1pt}gy)_{\fn^+}^j\Bigr|_{w=g^{-1}x} \notag \\
&=(-1)^{jr}(-l)_j(-k)_{\underline{j}_{r-1},d}(g^{-1}x|{}^t\hspace{-1pt}gy)^{l-j}_{\fn^+}((g^{-1}x)^\itinv|e')_{\fn^+}^{k-j}({}^t\hspace{-1pt}gy|e')_{\fn^+}^j \notag \\
&=(-1)^{jr}(-l)_j(-k)_{\underline{j}_{r-1},d}(x|y)^{l-j}_{\fn^+}(x^\itinv|z)^{k-j}_{\fn^+}(y|z)_{\fn^+}^j. \label{formula_diffproj_left}
\end{align}
On the other hand, for $j,m=0,1,\ldots,\min\{k,l\}$ we have 
\begin{align}
&\det_{\fn^+}(x)^{j-k}\det_{\fn^+}\biggl(\frac{\partial}{\partial x}\biggr)^j\det_{\fn^+}(x)^k
\Proj_{(l-m,\underline{0}_{r-2},m-k)}^{\fp^+}\bigl((x|y)^l_{\fn^+}(x^\itinv|z)^k_{\fn^+}\bigr) \notag \\
&=(-1)^{jr}(-(m,\underline{k}_{r-2},k+l-m))_{\underline{j}_r,d}\Proj_{(l-m,\underline{0}_{r-2},m-k)}^{\fp^+}\bigl((x|y)^l_{\fn^+}(x^\itinv|z)^k_{\fn^+}\bigr), \label{formula_diffproj_right}
\end{align}
and hence 
\begin{align*}
&(x|y)^{l-j}_{\fn^+}(x^\itinv|z)^{k-j}_{\fn^+}(y|z)_{\fn^+}^j
\in \BR\det_{\fn^+}(x)^{j-k}\det_{\fn^+}\biggl(\frac{\partial}{\partial x}\biggr)^j\det_{\fn^+}(x)^k(x|y)^l_{\fn^+}(x^\itinv|z)^k_{\fn^+}\\
&=\BR\det_{\fn^+}(x)^{j-k}\det_{\fn^+}\biggl(\frac{\partial}{\partial x}\biggr)^j\det_{\fn^+}(x)^k
\sum_{m=0}^{\min\{k,l\}}\Proj_{(l-m,\underline{0}_{r-2},m-k)}^{\fp^+}\bigl((x|y)^l_{\fn^+}(x^\itinv|z)^k_{\fn^+}\bigr) \\
&\subset \bigoplus_{m=j}^{\min\{k,l\}}\BR\Proj_{(l-m,\underline{0}_{r-2},m-k)}^{\fp^+}\bigl((x|y)^l_{\fn^+}(x^\itinv|z)^k_{\fn^+}\bigr). 
\end{align*}
Therefore, the left hand side of (\ref{formula_proof_span}) is contained in the right hand side. 
Since both sides have the dimension $\min\{k,l\}+1$, we get (\ref{formula_proof_span}). 

Next we determine $\Proj_{(l,\underline{0}_{r-2},-k)}^{\fp^+}\bigl((x|y)^l_{\fn^+}(x^\itinv|z)^k_{\fn^+}\bigr)$. By (\ref{formula_proof_span}) this is of the form 
\[ \Proj_{(l,\underline{0}_{r-2},-k)}^{\fp^+}\bigl((x|y)^l_{\fn^+}(x^\itinv|z)^k_{\fn^+}\bigr)
=\sum_{j=0}^{\min\{k,l\}} a_j(x|y)^{l-j}_{\fn^+}(x^\itinv|z)^{k-j}_{\fn^+}(y|z)_{\fn^+}^j \]
for some $a_j\in\BC$. Then since 
\begin{align*}
\det_{\fn^+}(x)^k\Proj_{(l-m,\underline{0}_{r-2},m-k)}^{\fp^+}\bigl((x|y)^l_{\fn^+}(x^\itinv|z)^k_{\fn^+}\bigr)&\in\cP_{(k+l-m,\underline{k}_{r-2},m)}(\fp^+), \\
\det_{\fn^+}(x)^k(x^\itinv|z)^{k-j}_{\fn^+}&\in\cP_{(\underline{k}_{r-1},j)}(\fp^+)
\end{align*}
vanish for $m,j\ge 1$ if $x$ is not invertible, we have 
\begin{align*}
\det_{\fn^+}(x)^k\Proj_{(l,\underline{0}_{r-2},-k)}^{\fp^+}\bigl((x|y)^l_{\fn^+}(x^\itinv|z)^k_{\fn^+}\bigr)&=(x|y)^l_{\fn^+}(x^\sharp|z)^k_{\fn^+}, \\
\det_{\fn^+}(x)^k\sum_{j=0}^{\min\{k,l\}} a_j(x|y)^{l-j}_{\fn^+}(x^\itinv|z)^{k-j}_{\fn^+}(y|z)_{\fn^+}^j&=a_0(x|y)^l_{\fn^+}(x^\sharp|z)^k_{\fn^+}
\end{align*}
for non-invertible $x$, and hence we get $a_0=1$. To determine other $a_j$, we consider the differential operator 
\[ \cL:=\sum_{\alpha\beta}(P(x)e_\alpha^\vee|e_\beta^\vee)_{\fn^+}\frac{\partial^2}{\partial x_\alpha\partial x_\beta}, \]
where $\{e_\alpha\},\{e_\alpha^\vee\}\subset\fn^+$ are mutually dual bases of $\fn^+$ and $\frac{\partial}{\partial x_\alpha}$ is the directional derivative along $e_\alpha$. 
Then by \cite[Proposition VI.4.4]{FK}, $\Proj_{(l-m,\underline{0}_{r-2},m-k)}^{\fp^+}\allowbreak\bigl((x|y)^l_{\fn^+}(x^\itinv|z)^k_{\fn^+}\bigr)$ is the eigenfunction of $\cL$ with the eigenvalue 
\begin{align*}
&(l-m)^2-(l-m)+(m-k)^2-(m-k)(1+d(r-1)) \\
&=2m^2-m(2k+2l+d(r-1))+k^2+l^2+k(1+d(r-1))-l, 
\end{align*}
which are distinct for all $m=0,1,\ldots,\min\{k,l\}$. Also, we have 
\begin{align*}
&\cL (x|y)^{l-j}_{\fn^+}(x^\itinv|z)^{k-j}_{\fn^+}(y|z)_{\fn^+}^j \\
&=(\cL (x|y)^{l-j}_{\fn^+})(x^\itinv|z)^{k-j}_{\fn^+}(y|z)_{\fn^+}^j+(x|y)^{l-j}_{\fn^+}(\cL (x^\itinv|z)^{k-j}_{\fn^+})(y|z)_{\fn^+}^j \\*
&\eqspace{}+2\sum_{\alpha\beta}(P(x)e_\alpha^\vee|e_\beta^\vee)_{\fn^+}\biggl(\frac{\partial}{\partial x_\alpha}(x|y)^{l-j}_{\fn^+}\biggr)
\biggl(\frac{\partial}{\partial x_\beta}(x^\itinv|z)^{k-j}_{\fn^+}\biggr)(y|z)_{\fn^+}^j \\
&=\bigl((l-j)^2-(l-j)\bigr)(x|y)^{l-j}_{\fn^+}(x^\itinv|z)^{k-j}_{\fn^+}(y|z)_{\fn^+}^j \\*
&\eqspace{}+\bigl((j-k)^2-(j-k)(1+d(r-1))\bigr)(x|y)^{l-j}_{\fn^+}(x^\itinv|z)^{k-j}_{\fn^+}(y|z)_{\fn^+}^j \\*
&\eqspace{}-2(l-j)(k-j)(P(x)y|P(x)^{-1}z)_{\fn^+}(x|y)^{l-j-1}_{\fn^+}(x^\itinv|z)^{k-j-1}_{\fn^+}(y|z)_{\fn^+}^j \\*
&=\bigl(2j^2-j(2k+2l+d(r-1))+k^2+l^2+k(1+d(r-1))-l\bigr)(x|y)^{l-j}_{\fn^+}(x^\itinv|z)^{k-j}_{\fn^+}(y|z)_{\fn^+}^j \\*
&\eqspace{}-2(l-j)(k-j)(x|y)^{l-j-1}_{\fn^+}(x^\itinv|z)^{k-j-1}_{\fn^+}(y|z)_{\fn^+}^{j+1}. 
\end{align*}
Hence we have 
\begin{align*}
0&=\bigl(\cL-(k^2+l^2+k(1+d(r-1))-l)\bigr)\Proj_{(l,\underline{0}_{r-2},-k)}^{\fp^+}\bigl((x|y)^l_{\fn^+}(x^\itinv|z)^k_{\fn^+}\bigr) \\
&=\sum_{j=0}^{\min\{k,l\}}a_j\bigl(\bigl(2j^2-j(2k+2l+d(r-1))\bigr)(x|y)^{l-j}_{\fn^+}(x^\itinv|z)^{k-j}_{\fn^+}(y|z)_{\fn^+}^j \\*
&\eqspace{}-2(l-j)(k-j)(x|y)^{l-j-1}_{\fn^+}(x^\itinv|z)^{k-j-1}_{\fn^+}(y|z)_{\fn^+}^{j+1}\bigr) \\
&=2\sum_{j=0}^{\min\{k,l\}-1}\biggl(a_{j+1}\biggl(j-k-l-\frac{d}{2}(r-1)+1\biggr)(j+1)-a_j(j-k)(j-l)\biggr) \\*
&\eqspace{}\times (x|y)^{l-j-1}_{\fn^+}(x^\itinv|z)^{k-j-1}_{\fn^+}(y|z)_{\fn^+}^{j+1}, 
\end{align*}
and since $a_0=1$, we have 
\[ a_j=\frac{(-k)_j(-l)_j}{\bigl(-k-l-\frac{d}{2}(r-1)+1\bigr)_j j!}, \]
that is, we get 
\begin{align*}
&\Proj_{(l,\underline{0}_{r-2},-k)}^{\fp^+}\bigl((x|y)^l_{\fn^+}(x^\itinv|z)^k_{\fn^+}\bigr) \\
&=\sum_{j=0}^{\min\{k,l\}}\frac{(-k)_j(-l)_j}{\bigl(-k-l-\frac{d}{2}(r-1)+1\bigr)_j j!}(x|y)^{l-j}_{\fn^+}(x^\itinv|z)^{k-j}_{\fn^+}(y|z)^j_{\fn^+}. 
\end{align*}

Finally we determine $\Proj_{(l-m,\underline{0}_{r-2},m-k)}^{\fp^+}\bigl((x|y)^l_{\fn^+}(x^\itinv|z)^k_{\fn^+}\bigr)$ for $m\ge 1$. 
If $z=e'\in\fn^+$ is a primitive idempotent, then by Lemma \ref{lem_Proj_Pieri}\,(2), we have 
\begin{align*}
&\Proj_{(l-m,\underline{0}_{r-2},m-k)}^{\fp^+}\bigl((x|y)^l_{\fn^+}(x^\itinv|z)^k_{\fn^+}\bigr) \\
&=\det_{\fn^+}(x)^{-k}\Proj_{(k+l-m,\underline{k}_{r-2},m)}^{\fp^+}\bigl((x|y)^l_{\fn^+}\det_{\fn^{+\prime\prime}}(x'')^k\bigr) \\
&=\frac{(-l)_m(-k)_{\underline{m}_{r-1},d}}{(-(m,\underline{k}_{r-2},k+l-m))_{\underline{m}_r,d}} \\*
&\eqspace{}\times \det_{\fn^+}(x)^{m-k}\Proj_{(k+l-2m,\underline{k-m}_{r-2},0)}^{\fp^+}\bigl((x|y)^{l-m}_{\fn^+}\det_{\fn^{+\prime\prime}}(x'')^{k-m}\bigr)(e'|y)_{\fn^+}^m \\
&=\frac{(-l)_m(-k)_m\bigl(-k-\frac{d}{2}\bigr)_{\underline{m}_{r-2},d}}{(-m)_m\bigl(-k-\frac{d}{2}\bigr)_{\underline{m}_{r-2},d}\bigl(-k-l+m-\frac{d}{2}(r-1)\bigr)_m} \\*
&\eqspace{}\times \Proj_{(l-m,\underline{0}_{r-2},m-k)}^{\fp^+}\bigl((x|y)^{l-m}_{\fn^+}(x^\itinv|z)_{\fn^+}^{k-m}\bigr)(y|z)_{\fn^+}^m \\
&=\frac{(-1)^m(-l)_m(-k)_m}{\bigl(-k-l+m-\frac{d}{2}(r-1)\bigr)_m m!} \\*
&\eqspace{}\times \sum_{j=0}^{\min\{k-m,l-m\}}\frac{(-k+m)_j(-l+m)_j}{\bigl(-k-l+2m-\frac{d}{2}(r-1)+1\bigr)_j j!}(x|y)^{l-m-j}_{\fn^+}(x^\itinv|z)^{k-m-j}_{\fn^+}(y|z)^{j+m}_{\fn^+} \\
&=\frac{(-1)^m}{\bigl(-k-l+m-\frac{d}{2}(r-1)\bigr)_mm!}
\sum_{j=m}^{\min\{k,l\}}\frac{(-k)_j(-l)_j(x|y)^{l-j}_{\fn^+}(x^\itinv|z)^{k-j}_{\fn^+}(y|z)^j_{\fn^+}}{\bigl(-k-l+2m-\frac{d}{2}(r-1)+1\bigr)_{j-m}(j-m)!}. 
\end{align*}
For general $z\in\fp^+$ of rank 1, when $z=ge'$ with $g\in K^\BC$, by replacing $x,y$ with $g^{-1}x, {}^t\hspace{-1pt}gy$ and applying the above, we get the desired formula. 

(2) First we consider the case $\rank\fp^+=2$. Then for $x,z\in\fp^+$, we have 
\[ (x|z)_{\fn^+}=(x^\sharp|z^\sharp)_{\fn^+}=\det_{\fn^+}(x)(x^\itinv|z^\sharp)_{\fn^+}, \]
and $z^\sharp$ is of rank 1 if $z$ is of rank 1. Hence, by (1) we have 
\begin{align*}
&\Proj_{(k+l-m,m)}^{\fp^+}\bigl((x|y)^l_{\fn^+}(x|z)^k_{\fn^+}\bigr)=\det_{\fn^+}(x)^k\Proj_{(l-m,m-k)}^{\fp^+}\bigl((x|y)^l_{\fn^+}(x^\itinv|z^\sharp)_{\fn^+}^k\bigr) \\
&=\frac{(-1)^m}{\bigl(-k-l+m-\frac{d}{2}\bigr)_mm!}\det_{\fn^+}(x)^k
\sum_{j=m}^{\min\{k,l\}}\frac{(-k)_j(-l)_j(x|y)^{l-j}_{\fn^+}(x^\itinv|z^\sharp)_{\fn^+}^{k-j}(y|z^\sharp)^j_{\fn^+}}{\bigl(-k-l+2m-\frac{d}{2}+1\bigr)_{j-m}(j-m)!} \\
&=\frac{(-1)^m}{\bigl(-k-l+m-\frac{d}{2}\bigr)_mm!}
\sum_{j=m}^{\min\{k,l\}}\frac{(-k)_j(-l)_j(x|y)^{l-j}_{\fn^+}(x|z)^{k-j}_{\fn^+}(y|z^\sharp)^j_{\fn^+}\det_{\fn^+}(x)^j}{\bigl(-k-l+2m-\frac{d}{2}+1\bigr)_{j-m}(j-m)!}. 
\end{align*}
Hence it suffices to show 
\begin{equation}\label{formula_proof_rank2det}
(x|y)_{\fn^+}(x|z)_{\fn^+}-(P(x)y|z)_{\fn^+}=(y|z^\sharp)_{\fn^+}\det_{\fn^+}(x). 
\end{equation}
Indeed, when $y=z$ we have 
\begin{align*}
(x|y)_{\fn^+}^2-(P(x)y|y)_{\fn^+}&=\bigl(e\big|P(x^{\mathit{1/2}})y\bigr)_{\fn^+}^2-\bigl(P(x^{\mathit{1/2}})y\big|P(x^{\mathit{1/2}})y\bigr)_{\fn^+} \\
&=2\det_{\fn^+}\bigl(P(x^{\mathit{1/2}})y\bigr)=2\det_{\fn^+}(x)\det_{\fn^+}(y)=\det_{\fn^+}(x)(y|y^\sharp)_{\fn^+}, 
\end{align*}
and by polarization, we get (\ref{formula_proof_rank2det}). Thus we get the desired formula when $\fp^+$ is of rank 2. 
For $\fp^+$ of general rank, since the both sides of the desired formula depend only on the rank 2 subalgebra containing $y$ and $z$, 
this is reduced to the rank 2 case. 
\end{proof}

\begin{remark}
By (\ref{formula_diffproj_left}), (\ref{formula_diffproj_right}) and Proposition \ref{prop_Proj}, for $i=0,1,\ldots,\min\{k,l\}$ we have 
\begin{align*}
&(-l)_i(-k)_i(x|y)^{l-i}_{\fn^+}(x^\itinv|z)^{k-i}_{\fn^+}(y|z)_{\fn^+}^i \\
&=\sum_{m=i}^{\min\{k,l\}}(-m)_i\biggl(-k-l+m-\frac{d}{2}(r-1)\biggr)_i\Proj_{(l-m,\underline{0}_{r-2},m-k)}^{\fp^+}\bigl((x|y)^l_{\fn^+}(x^\itinv|z)^k_{\fn^+}\bigr) \\
&=\sum_{m=i}^{\min\{k,l\}}\frac{(-1)^{m-i}}{\bigl(-k-l+i+m-\frac{d}{2}(r-1)\bigr)_{m-i}(m-i)!} \\*
&\eqspace\times\sum_{j=m}^{\min\{k,l\}}\frac{(-k)_j(-l)_j}{\bigl(-k-l+2m-\frac{d}{2}(r-1)+1\bigr)_{j-m}(j-m)!}(x|y)^{l-j}_{\fn^+}(x^\itinv|z)^{k-j}_{\fn^+}(y|z)^j_{\fn^+}, 
\end{align*}
that is, for $j=i+1,i+2,\ldots,\min\{k,l\}$ we have 
\begin{align*}
&\sum_{m=i}^j \frac{(-1)^{m-i}}{\bigl(-k-l+i+m-\frac{d}{2}(r-1)\bigr)_{m-i}(m-i)!}\frac{(-k)_j(-l)_j}{\bigl(-k-l+2m-\frac{d}{2}(r-1)+1\bigr)_{j-m}(j-m)!} \\
&=\sum_{m=0}^{j-i} \frac{(-1)^m(-k)_j(-l)_j}
{\bigl(-k\hspace{-1pt}-\hspace{-1pt}l\hspace{-1pt}+\hspace{-1pt}2i\hspace{-1pt}+\hspace{-1pt}m\hspace{-1pt}-\hspace{-1pt}\frac{d}{2}(r\hspace{-1pt}-\hspace{-1pt}1)\bigr)_m m!
\bigl(-k\hspace{-1pt}-\hspace{-1pt}l\hspace{-1pt}+\hspace{-1pt}2i\hspace{-1pt}+\hspace{-1pt}2m\hspace{-1pt}-\hspace{-1pt}\frac{d}{2}(r\hspace{-1pt}-\hspace{-1pt}1)
\hspace{-1pt}+\hspace{-1pt}1\bigr)_{j-i-m}(j-i-m)!} \\
&=0. 
\end{align*}
In general, for $\mu\in\BC$, $j\in\BZ_{>0}$, we can directly verify 
\begin{align*}
&\sum_{m=0}^j \frac{(-1)^m}{(\mu\hspace{-1pt}+\hspace{-1pt}m)_mm!\bigl(\mu\hspace{-1pt}+\hspace{-1pt}2m\hspace{-1pt}+\hspace{-1pt}1\bigr)_{j-m}(j\hspace{-1pt}-\hspace{-1pt}m)!} 
=\sum_{m=0}^j\frac{(-1)^m(\mu)_m(\mu\hspace{-1pt}+\hspace{-1pt}2m)(\mu\hspace{-1pt}+\hspace{-1pt}j\hspace{-1pt}+\hspace{-1pt}m\hspace{-1pt}+\hspace{-1pt}1)_{j-m}}{(\mu)_{2j+1}m!(j-m)!} \\
&=\frac{(-1)^j}{(\mu)_{2j+1}}\sum_{m=0}^j\frac{(\mu)_m(\mu+2m)}{m!}\frac{(-\mu-2j)_{j-m}}{(j-m)!}=0, 
\end{align*}
by comparing the coefficients for $t^j$ of 
\begin{align*}
&\sum_{q=0}^\infty\sum_{m=0}^q\frac{(\mu)_m(\mu+2m)}{m!}\frac{(-\mu-2j)_{q-m}}{(q-m)!}t^q
=\sum_{m=0}^\infty\frac{(\mu)_m(\mu+2m)}{m!}t^m\sum_{n=0}^\infty\frac{(-\mu-2j)_n}{n!}t^n \\
&=\mu(1+t)(1-t)^{-\mu-1}(1-t)^{\mu+2j}=\mu(1+t)(1-t)^{2j-1}=-\frac{\mu}{j}\sum_{q=0}^\infty\frac{(-2j)_q(-j+q)}{q!}t^q. 
\end{align*}
\end{remark}

\subsection{Holomorphic discrete series representations}\label{subsection_HDS}

In this subsection, we review holomorphic discrete series representations. 
First we recall from, e.g., \cite[Section VII.9]{Kn}, \cite[Part III, Section III]{FKKLR}, 
the bounded symmetric domain realization (Harish-Chandra realization) of the Hermitian symmetric space $G/K$ via the Borel embedding, 
\[ \xymatrix{ G/K \ar[r] \ar@{-->}[d]^{\mbox{\rotatebox{90}{$\sim$}}} & G^\BC/K^\BC P^- \\
 D \ar@{^{(}->}[r] & \fp^+ \ar[u]_{\exp} } \]
where 
\begin{align*}
D&=(\text{connected component of }\{x\in\fp^+\mid B(x)\text{ is positive definite}\} \text{ which contains }0) \\
&=(\text{connected component of }\{x\in\fp^+\mid h(x)>0\}\text{ which contains }0). 
\end{align*}
Here we write $B(x):=B(x,\overline{x})$, $h(x):=h(x,\overline{x})$. 
Then its Bergman--Shilov boundary $\Sigma\subset\partial D$ consists of all maximal tripotents in $\fp^+$. 
For $g\in G^\BC$, $x\in D$, if $g\exp(x)\in P^+K^\BC P^-$ holds, then we write 
\[ g\exp(x)=\exp(\pi^+(g,x))\kappa(g,x)\exp(\pi^-(g,x)), \]
where $\pi^\pm(g,x)\in \fp^\pm$ and $\kappa(g,x)\in K^\BC$. Then the map $\pi^+\colon G\times D\to D$ gives an action of $G$ on $D$, and we abbreviate $\pi^+(g,x)=:gx$. 

Next let $(\tau,V)$ be an irreducible holomorphic representation of the universal covering group $\widetilde{K}^\BC$ of $K^\BC$, 
with the $\widetilde{K}$-invariant inner product $(\cdot,\cdot)_\tau$. 
Then the universal covering group $\widetilde{G}$ of $G$ acts on the space $\cO(D,V)=\cO_\tau(D,V)$ of $V$-valued holomorphic functions on $D$ by 
\[ (\hat{\tau}(g)f)(x):=\tau(\kappa(g^{-1},x))^{-1}f(g^{-1}x) \qquad (g\in\widetilde{G},\; x\in D,\; f\in\cO(D,V)), \]
where we lift the map $\kappa\colon G\times D\to K^\BC$ to the universal covering spaces, and represent by the same symbol $\kappa\colon \widetilde{G}\times D\to \widetilde{K}^\BC$. 
Its differential action is given by 
\begin{equation}\label{HDS_diff_action}
(d\hat{\tau}(z,k,w)f)(x)=d\tau(k-D(x,w))f(x)+\frac{d}{dt}\biggr|_{t=0}f(x-t(z+kx-Q(x)w))
\end{equation}
for $z\in\fp^+$, $k\in\fk^\BC$, $w\in\fp^-$. This becomes a highest weight representation with the minimal $\widetilde{K}$-type $(\tau,V)$. 
If this contains a unitary subrepresentation $\cH_\tau(D,V)\subset\cO_\tau(D,V)$, 
then its reproducing kernel is proportional to $\tau(B(x,\overline{y}))$, and such unitary subrepresentation is unique. 
Here we lift the map $B\colon D\times \overline{D}\to \operatorname{Str}(\fp^+)_0\subset\End_\BC(\fp^+)$ to the universal covering space, 
and represent by the same symbol $B\colon D\times \overline{D}\to\widetilde{K}^\BC$. Especially, if the $\widetilde{G}$-invariant inner product is given by the converging integral 
\[ \langle f,g\rangle_{\hat{\tau}}:=C_\tau\int_D (\tau(B(x)^{-1})f(x),g(x))_\tau \Det_{\fp^+}(B(x))^{-1}\,dx \]
(a \textit{weighted Bergman inner product}), then $(\hat{\tau},\cH_\tau(D,V))$ is called a \textit{holomorphic discrete series representation}. 
Here we normalize the Lebesgue measure $dx$ on $D\subset\fp^+$ with respect to the inner product $(\cdot|\overline{\cdot})_{\fp^+}$, 
and determine the constant $C_\tau$ such that $\Vert v\Vert_{\hat{\tau}}=|v|_\tau$ holds for all constant functions $v\in V$. 

Next, let $\chi\colon \widetilde{K}^\BC\to\BC^\times$ be the character of $\widetilde{K}^\BC$ normalized such that 
\begin{equation}\label{chi_normalize}
d\chi([x,y])=(x|y)_{\fp^+} \qquad (x\in\fp^+,\; y\in\fp^-), 
\end{equation}
so that $h(x,y)=\chi(B(x,y))$ holds, and we fix a representation $(\tau_0,V)$ of $K^\BC$. 
When $(\tau,V)$ is of the form $(\tau,V)=(\chi^{-\lambda}\otimes\tau_0,V)$, we write $\cH_\tau(D,V)=\cH_\lambda(D,V)\subset\cO_\tau(D,V)=\cO_\lambda(D,V)$. 
In addition, if $(\tau,V)=(\chi^{-\lambda},\BC)$, then we write $\cH_\tau(D,V)=\cH_\lambda(D)\subset\cO_\tau(D,V)=\cO_\lambda(D)$ and $\hat{\tau}=\tau_\lambda$. 

In the rest of this subsection, we assume $\fp^+$ is simple and $(\tau,V)=(\chi^{-\lambda},\BC)$. 
Then $\cH_\lambda(D)$ is a holomorphic discrete series representation if $\lambda>p-1$, and then the inner product is given by 
\begin{gather}
\langle f,g\rangle_\lambda=\langle f,g\rangle_{\lambda,\fp^+}:=C_\lambda\int_D f(x)\overline{g(x)}h(x)^{\lambda-p}\,dx, \label{Bergman_inner_prod}\\
C_\lambda:=\frac{1}{\pi^n}\frac{\Gamma_r^d(\lambda)}{\Gamma_r^d\bigl(\lambda-\frac{n}{r}\bigr)}, \notag
\end{gather}
where $\Gamma_r^d(\lambda)$ is as in (\ref{Gamma}). 
We consider another inner product on $\cP(\fp^+)$, called the \textit{Fischer inner product} (see e.g. \cite[Section XI.1]{FK}), defined by 
\begin{equation}\label{Fischer_inner_prod}
\langle f,g\rangle_{F}=\langle f,g\rangle_{F,\fp^+}:=\frac{1}{\pi^n}\int_{\fp^+}f(x)\overline{g(x)}e^{-(x|\overline{x})_{\fp^+}}\,dx
=\overline{g\biggl(\overline{\frac{\partial}{\partial x}}\biggr)}f(x)\biggr|_{x=0}. 
\end{equation}
Here $\overline{g(\overline{\cdot})}$ is a holomorphic polynomial on $\fp^-$, and we normalize $\frac{\partial}{\partial x}$ with respect to the bilinear form 
$(\cdot|\cdot)_{\fp^+}\colon\fp^+\times\fp^-\to\BC$. Then the following holds. 

\begin{theorem}[{Faraut--Kor\'anyi, \cite{FK0}, \cite[Part III, Corollary V.3.9]{FKKLR}}]\label{thm_FK}
For $\lambda>p-1$, $\bm\in\BZ_{++}^r$, $f\in\cP_\bm(\fp^+)$, $g\in\cP(\fp^+)$, we have 
\[ \langle f,g\rangle_\lambda=\frac{1}{(\lambda)_{\bm,d}}\langle f,g\rangle_F, \]
where $(\lambda)_{\bm,d}$ is as in (\ref{Pochhammer}). 
\end{theorem}

Since the reproducing kernel on $\cP(\fp^+)$ with respect to $\langle\cdot,\cdot\rangle_F$ is given by $e^{(x|\overline{y})_{\fp^+}}$, the following holds. 

\begin{corollary}\label{cor_FK}
For $\lambda>p-1$, $\bm\in\BZ_{++}^r$, $f(x)\in\cP_\bm(\fp^+)$, we have 
\[ \left\langle f(x),e^{(x|\overline{y})_{\fp^+}}\right\rangle_{\lambda,x}=\frac{1}{(\lambda)_{\bm,d}}f(y). \]
\end{corollary}

Here the subscript $x$ stands for the variable of integration. 
By this theorem, $\langle\cdot,\cdot\rangle_{\lambda}$ is meromorphically continued for all $\lambda\in\BC$, 
and the $\widetilde{K}$-finite part $\cO_\lambda(D)_{\widetilde{K}}=\chi^{-\lambda}\otimes\cP(\fp^+)$ is reducible as a $(\fg,\widetilde{K})$-module 
if and only if $\lambda$ is a pole of $\langle\cdot,\cdot\rangle_{\lambda}$. 
Especially, for $j=1,2,\ldots,r$, $\lambda\in\frac{d}{2}(j-1)-\BZ_{\ge 0}$, 
\begin{equation}\label{submodule}
M_j(\lambda)=M_j^\fg(\lambda):=\bigoplus_{\substack{\bm\in\BZ_{++}^r \\ m_j\le \frac{d}{2}(j-1)-\lambda}}\chi^{-\lambda}\otimes\cP_\bm(\fp^+)
\subset\chi^{-\lambda}\otimes\cP(\fp^+)=\cO_\lambda(D)_{\widetilde{K}}
\end{equation}
gives a $(\fg,\widetilde{K})$-submodule. Moreover, $\langle\cdot,\cdot\rangle_\lambda$ is positive definite on $\cP(\fp^+)$ for $\lambda>\frac{d}{2}(r-1)$, 
and on $M_j^\fg\bigl(\frac{d}{2}(j-1)\bigr)$ for $\lambda=\frac{d}{2}(j-1)$, $j=1,\ldots,r$, that is, $\cO_\lambda(D)_{\widetilde{K}}$ contains a unitary submodule $\cH_\lambda(D)$ 
if $\lambda$ sits in the \textit{Wallach set} 
\begin{equation}\label{Wallach_set}
\lambda\in\biggl\{0,\frac{d}{2},d,\ldots,\frac{d}{2}(r-1)\biggr\}\cup\biggl(\frac{d}{2}(r-1),\infty\biggr), 
\end{equation}
and its $\widetilde{K}$-finite part is given by 
\[ \cH_\lambda(D)_{\widetilde{K}}=\begin{cases} \chi^{-\lambda}\otimes\cP(\fp^+) & (\lambda>\frac{d}{2}(r-1)), \\
M_j^\fg\bigl(\frac{d}{2}(j-1)\bigr) & (\lambda=\frac{d}{2}(j-1),\; j=1,2,\ldots,r). \end{cases} \]
In addition, the quotient module $\cO_\lambda(D)_{\widetilde{K}}/M_r(\lambda)$ $(\lambda\in\frac{d}{2}(r-1)-\BZ_{\ge 0})$ also gives an infinitesimally unitary module (see \cite{FK0}). 
Also by the corollary, for $\bk\in\BZ_{++}^r$, $f(x)\in\cP(\fp^+)$, $(\lambda)_{\bk,d}\bigl\langle f(x),e^{(x|\overline{y})_{\fp^+}}\bigr\rangle_{\lambda,x}$ is 
holomorphically continued for all $\lambda\in\BC$ if and only if 
\[ f(x)\in\bigoplus_{\substack{\bm\in\BZ_{++}^r \\ m_j\le k_j \; (j=1,\ldots,r)}}\cP_\bm(\fp^+) \]
holds, and if this is satisfied, then for $j=1,\ldots,r$, 
\[ f(x)\in M_j(\lambda)\quad \text{holds if}\quad \lambda\in\frac{d}{2}(j-1)-k_j-\BZ_{\ge 0}. \]

\subsection{Classification}\label{subsection_classification}

A simple Hermitian positive Jordan triple system $\fp^\pm$ is isomorphic to one of the following (see, e.g., \cite[\S 7]{L0}). 
\begin{align*}
\fp^\pm&=\BC^n \quad (n\ge 3), &&\Sym(r,\BC), &&M(q,s;\BC), \\ &\eqspace \Alt(s,\BC), && \Herm(3,\BO)^\BC, && M(1,2;\BO)^\BC. 
\end{align*}
Here $\Sym(r,\BC)$ and $\Alt(s,\BC)$ denote the spaces of symmetric and alternating matrices over $\BC$ respectively, 
and $\Herm(3,\BO)^\BC$ denotes the complexification of the space of $3\times 3$ Hermitian matrices over the octonions $\BO$. 
Then the corresponding Lie groups $G$ and their maximal compact subgroups $K$ are given by 
\[ (G,K)=\begin{cases} (SO_0(2,n),SO(2)\times SO(n)) & (\fp^\pm=\BC^n), \\ (Sp(r,\BR),U(r)) & (\fp^\pm=\Sym(r,\BC)), \\
(SU(q,s),S(U(q)\times U(s))) & (\fp^\pm=M(q,s;\BC)), \\ (SO^*(2s),U(s)) & (\fp^\pm=\Alt(s,\BC)), \\
(E_{7(-25)},U(1)\times E_6) & (\fp^\pm=\Herm(3,\BO)^\BC), \\ (E_{6(-14)},U(1)\times Spin(10)) & (\fp^\pm=M(1,2;\BO)^\BC) \end{cases} \]
(up to covering), and the numbers $(n,r,d,b,p)$ (see (\ref{str_const})) are given by 
\[ (n,r,d,b,p)=\begin{cases} (n,2,n-2,0,n) & (\fp^\pm=\BC^n,\; n\ge 3), \\
\left(\frac{1}{2}r(r+1), r, 1, 0, r+1\right) & (\fp^\pm=\Sym(r,\BC)), \\
(qs, \min\{q,s\}, 2, |q-s|, q+s) & (\fp^\pm=M(q,s;\BC)), \\
\left(\frac{1}{2}s(s-1), \frac{s}{2}, 4, 0, 2(s-1)\right) & (\fp^\pm=\Alt(s,\BC),\; s\colon\text{even}), \\
\left(\frac{1}{2}s(s-1), \left\lfloor\frac{s}{2}\right\rfloor, 4, 2, 2(s-1)\right) & (\fp^\pm=\Alt(s,\BC),\; s\colon\text{odd}), \\
(27,3,8,0,18) & (\fp^\pm=\Herm(3,\BO)^\BC), \\ (16,2,6,4,12) & (\fp^\pm=M(1,2;\BO)^\BC). \end{cases} \]
Here, if $r=1$, then $d$ is not determined uniquely, and any number is allowed. 
When $b=0$, $\fp^\pm$ is of tube type, and has a Jordan algebra structure. That is, $\fp^\pm=\BC^n$, $\Sym(r,\BC)$, $M(r,\BC)$, $\Alt(2r,\BC)$ and $\Herm(3,\BO)^\BC$ are of tube type. 
Their corresponding Euclidean real forms $\fn^+\subset\fp^+$ are given by 
\[ \fn^+=\begin{cases} \BR^{1,n-1} \\ \Sym(r,\BR) \\ \Herm(r,\BC) \\ \Herm(r,\BH) \\ \Herm(3,\BO) \end{cases}
\subset \fp^+=\begin{cases} \BC^n \\ \Sym(r,\BC) \\ M(r,\BC) \\ \Alt(2r,\BC) \\ \Herm(3,\BO)^\BC, \end{cases} \]
where $\BR^{1,n-1}\subset\BC^n$ is a real form on which the quadratic form $\det_{\fn^+}(x)$ has the signature $(1,n-1)$. 
For these cases, $p=\frac{2n}{r}=d(r-1)+2$ holds. 

Next we fix the parametrization of finite-dimensional irreducible representations of $\widetilde{GL}(s,\BC)$ and $Spin(n,\BC)$ $(n\ge 3)$. 
We take a basis $\{\epsilon_j\}_{j=1}^s\subset\fh^{\BC\vee}$ of the dual space of a Cartan subalgebra $\fh^\BC\subset\mathfrak{gl}(s,\BC)$ 
such that the positive root system is given by $\{\epsilon_i-\epsilon_j\mid 1\le i<j\le s\}$. 
For $\bm\in\BC^s$ with $m_j-m_{j+1}\in\BZ_{\ge 0}$, let $V_\bm^{(s)}$ be the irreducible representation of $\widetilde{GL}(s,\BC)$ with the highest weight $\sum_j m_j\epsilon_j$, 
and let $V_\bm^{(s)\vee}$ be that with the lowest weight $-\sum_j m_j\epsilon_j$. 
If $\bm\in\BZ^r$, then $V_\bm^{(s)}$, $V_\bm^{(s)\vee}$ define genuine representations of $GL(s,\BC)$. 
Similarly, we take a basis $\{\epsilon_j\}_{j=1}^{\lfloor n/2\rfloor}\subset\fh^{\BC\vee}$ of the dual space of a Cartan subalgebra $\fh^\BC\subset\mathfrak{so}(n,\BC)$ 
such that the positive root system is given by 
\begin{align*}
&\{\epsilon_i\pm\epsilon_j\mid 1\le i<j\le n/2\} && (n\colon\text{even}), \\
&\{\epsilon_i\pm\epsilon_j\mid 1\le i<j\le \lfloor n/2\rfloor\}\cup\{\epsilon_j\mid 1\le j\le \lfloor n/2\rfloor\} && (n\colon\text{odd}). 
\end{align*}
For $\bm\in\BZ^{\lfloor n/2\rfloor}\cup\left(\BZ+\frac{1}{2}\right)^{\lfloor n/2\rfloor}$ with $m_1\ge\cdots\ge m_{n/2-1}\ge |m_{n/2}|$ for even $n$, 
$m_1\ge\cdots\ge m_{\lfloor n/2\rfloor}\ge 0$ for odd $n$, let $V_\bm^{[n]}$ be the irreducible representation of $Spin(n,\BC)$ with the highest weight $\sum_j m_j\epsilon_j$, 
and let $V_\bm^{[n]\vee}$ be that with the lowest weight $-\sum_j m_j\epsilon_j$. 
If $\bm\in\BZ^{\lfloor n/2\rfloor}$, then $V_\bm^{[n]}$, $V_\bm^{[n]\vee}$ define genuine representations of $SO(n,\BC)$. 
Under this notation, the $\widetilde{K}$-type decompositions of the holomorphic discrete series representations of scalar type 
\[ \cO_\lambda(D)_{\widetilde{K}}=\chi^{-\lambda}\otimes\cP(\fp^+)=\chi^{-\lambda}\otimes\bigoplus_{\bm\in\BZ_{++}^r}\cP_\bm(\fp^+) \]
are given as 
\begin{align}
\chi^{-\lambda}&\simeq\begin{cases} 
\BC_{-\lambda}\boxtimes V_{(0,\ldots,0)}^{[n]\vee} & (\fp^+=\BC^n,\; n\ge 3), \\
V_{(\lambda,\ldots,\lambda)}^{(r)\vee} & (\fp^+=\Sym(r,\BC)), \\
V_{\left(\frac{s\lambda}{q+s},\ldots,\frac{s\lambda}{q+s}\right)}^{(q)\vee}\boxtimes V_{\left(\frac{q\lambda}{q+s},\ldots,\frac{q\lambda}{q+s}\right)}^{(s)} & (\fp^+=M(q,s;\BC)), \\
V_{\left(\frac{\lambda}{2},\ldots,\frac{\lambda}{2}\right)}^{(s)\vee} & (\fp^+=\Alt(s,\BC)), \\
\BC_{-\lambda}\boxtimes V_{(0,\ldots,0)}^{[10]\vee} & (\fp^+=M(1,2;\BO)^\BO), \end{cases} \notag \\ 
\cP_\bm(\fp^+)&\simeq\begin{cases}
\BC_{-m_1-m_2}\boxtimes V_{(m_1-m_2,0,\ldots,0)}^{[n]\vee} & (\fp^+=\BC^n,\; n\ge 3), \\
V_{(2m_1,\ldots,2m_r)}^{(r)\vee}=:V_{2\bm}^{(r)\vee} & (\fp^+=\Sym(r,\BC)), \\
V_{\bm}^{(q)\vee}\boxtimes V_{(m_1,\ldots,m_q,0,\ldots,0)}^{(s)}=:V_\bm^{(q)\vee}\boxtimes V_\bm^{(s)} & (\fp^+=M(q,s;\BC),\; q\le s), \\
V_{(m_1,\ldots,m_s,0,\ldots,0)}^{(q)\vee}\boxtimes V_{\bm}^{(s)}=:V_\bm^{(q)\vee}\boxtimes V_\bm^{(s)} & (\fp^+=M(q,s;\BC),\; q\ge s), \\
V_{(m_1,m_1,m_2,m_2,\ldots,m_{\lfloor s/2\rfloor},m_{\lfloor s/2\rfloor}(,0))}^{(s)\vee}=:V_{\bm^2}^{(s)\vee} & (\fp^+=\Alt(s,\BC)), \\
\BC_{-\frac{3}{4}|\bm|}\boxtimes V_{\left(\frac{m_1+m_2}{2},\frac{m_1-m_2}{2},\frac{m_1-m_2}{2},\frac{m_1-m_2}{2},\frac{m_1-m_2}{2}\right)}^{[10]\vee} & (\fp^+=M(1,2;\BO)^\BC) 
\end{cases}\label{explicit_Ktype}
\end{align}
when $K^\BC$ is classical, if we normalize the representations $\BC_{-\lambda}$ of $\widetilde{SO}(2)\simeq \widetilde{U}(1)$ for the first and the last cases suitably.

\subsection{Restriction to symmetric subgroups}\label{subsection_sym_subalg}

In this subsection, we consider a $\BC$-linear involution $\sigma$ on a Hermitian positive Jordan triple system $\fp^\pm$, i.e., 
a Jordan triple system automorphism $\sigma\colon\fp^\pm\to\fp^\pm$ of order 2 which commutes with the $\BC$-antilinear map $\overline{\cdot}\colon \fp^\pm\to\fp^\mp$, 
and extend to the involution of the Lie algebra $\fg^\BC=\fp^+\oplus\fk^\BC\oplus\fp^-$ by letting $\sigma$ act on $\fk^\BC=\mathfrak{str}(\fp^+)\subset\End_\BC(\fp^+)$ by 
$\sigma(l):=\sigma l\sigma$. Also let $\vartheta:=-I_{\fp^+}+I_{\fk^\BC}-I_{\fp^-}$. Using these, we set 
\begin{align*}
\fp^\pm_1&:=(\fp^\pm)^\sigma=\{x\in\fp^\pm\mid \sigma(x)=x\}, \\
\fp^\pm_2&:=(\fp^\pm)^{-\sigma}=\{x\in\fp^\pm\mid \sigma(x)=-x\}, \\
\fk^\BC_1&:=(\fk^\BC)^\sigma=\{l\in\fk^\BC\mid \sigma(l)=l\}, \\
\fk_1&:=\fk^\sigma=\fk^\BC_1\cap\fk, \\
\fg_1^\BC&:=(\fg^\BC)^\sigma=\fp^+_1\oplus\fk^\BC_1\oplus\fp^-_1, \\
\fg_2^\BC&:=(\fg^\BC)^{\sigma\vartheta}=\fp^+_2\oplus\fk^\BC_1\oplus\fp^-_2, \\
\fg_1&:=\fg^\sigma=\fg^\BC_1\cap\fg, \\
\fg_2&:=\fg^{\sigma\vartheta}=\fg^\BC_2\cap\fg, 
\end{align*}
and let $G_1,G_1^\BC,G_2,G_2^\BC,K_1,K_1^\BC\subset G^\BC$ be the connected closed subgroups corresponding to $\fg_1,\fg_1^\BC,\fg_2,\fg_2^\BC,\fk_1,\fk_1^\BC$ respectively. 
Such $(G,G_1)$ is called a symmetric pair of holomorphic type (see \cite[Section 3.4]{Kmf1}). 
Also let $\widetilde{G}_1\subset\widetilde{G}$ and $\widetilde{K}_1^\BC\subset\widetilde{K}^\BC$ be the connected closed subgroups of the universal covering groups of 
$G$ and $K^\BC$ corresponding to $\fg_1$ and $\fk_1^\BC$ respectively, and let $D\subset\fp^+$, $D_1\subset\fp^+_1$ be the corresponding bounded symmetric domains, 
so that $D\simeq G/K$, $D_1\simeq G_1/K_1$ hold. 
Let $\chi$ be the character of $K^\BC$ given in (\ref{chi_normalize}). Similarly, we normalize the bilinear form 
$(\cdot|\cdot)_{\fp^+_j}\colon \fp^+_j\times\fp^-_j\to\BC$ $(j=1,2)$ such that $(e|\overline{e})_{\fp^+_j}=1$ holds for any primitive tripotent $e\in\fp^+_j$, 
and define the characters $\chi_j$ $(j=1,2)$ of $K_1^\BC$ by 
\begin{align*}
d\chi_j([x,y])&=(x|y)_{\fp^+_j} && (x\in\fp^+_j,\; y\in\fp^-_j), \\
d\chi_j(l)&=0 && (l\in [\fp^+_j,\fp^-_j]^\bot\subset \fk^\BC_1). 
\end{align*}
As in Section \ref{subsection_HDS}, for $\lambda\in\BR$ and for an irreducible representation $V$ of $K_1^\BC$, 
let $\cH_\lambda(D)\subset\cO_\lambda(D)$ and $\cH_{\lambda}(D_1,V)\subset\cO_\lambda(D_1,V)$ be the unitary representations of $\widetilde{G}$ and $\widetilde{G}_1$ 
with the minimal $\widetilde{K}$-type $\chi^{-\lambda}$ and with the minimal $\widetilde{K}_1$-type $\chi_1^{-\lambda}\otimes V$ respectively, if they exist. 

In the following, we assume $G$ is simple. Then $\fp^+_2$ is a direct sum of at most two simple Jordan triple systems. 
Let $\rank\fp^+=:r$, $\rank\fp^+_2=:r_2$, and define $\varepsilon_1,\varepsilon_2\in\{1,2\}$ by $d\chi|_{\fk_1^\BC}=\varepsilon_j d\chi_j$, or equivalently, by 
\begin{equation}\label{epsilon_j}
(x|y)_{\fp^+}=\varepsilon_j(x|y)_{\fp^+_j} \qquad (j=1,2,\; x\in\fp^+_j,\; y\in\fp^-_j). 
\end{equation}
When $\fp^+_2$ is not simple, we write $\fp^+_2=:\fp^+_{11}\oplus\fp^+_{22}$, $\fp^+_1=:\fp_{12}^+$, 
and let $\rank\fp^+_{11}=:r'$, $\rank\fp^+_{22}=:r''$. Now we consider the restriction of the representation $\cH_\lambda(D)$ of $\widetilde{G}$ to the subgroup $\widetilde{G}_1$. 
If $\lambda$ satisfies (\ref{Wallach_set}), then the unitary representation $\cH_\lambda(D)$ exists, $\cH_\lambda(D)|_{\widetilde{G}_1}$ is discretely decomposable, 
and every $\widetilde{G}_1$-submodule in $\cH_\lambda(D)|_{\widetilde{G}_1}$ contains a $\fp^+_1$-null vector, 
that is, has an intersection with $(\cH_\lambda(D)_{\widetilde{K}})^{\fp^+_1}=(\chi_1^{-\varepsilon_1\lambda}\otimes\cP(\fp^+_2))\cap\cH_\lambda(D)_{\widetilde{K}}$. 
Especially, $\cH_\lambda(D)_{\widetilde{K}}=\chi^{-\lambda}\otimes\cP(\fp^+)$ holds if $\lambda>\frac{d}{2}(r-1)$, 
and $\cH_\lambda(D)$ is a holomorphic discrete series representation if $\lambda>p-1$. For such $\lambda$, according to the decomposition 
\begin{align*}
\cP(\fp^+_2)&=\bigoplus_{\bk\in\BZ_{++}^{r_2}}\cP_\bk(\fp^+_2) && (\fp^+_2\colon\text{simple}), \\
\cP(\fp^+_2)&=\bigoplus_{\bk\in\BZ_{++}^{r'}}\bigoplus_{\bl\in\BZ_{++}^{r''}}\cP_\bk(\fp^+_{11})\boxtimes\cP_\bl(\fp^+_{22}) && (\fp^+_2\colon\text{non-simple})
\end{align*}
(Theorem \ref{thm_HKS}), the following holds. We note that a similar formula holds also for holomorphic discrete series representations of non-scalar type. 
See \cite[Lemma 8.8]{Kmf1}, \cite[Section 8]{Kmf0-1}. 

\begin{theorem}[{Kobayashi, \cite[Theorems 8.3]{Kmf1}}]\label{thm_HKKS}
For $\lambda>\frac{d}{2}(r-1)$, the restriction of $\cH_\lambda(D)$ to the subgroup $\widetilde{G}_1$ is 
decomposed into the Hilbert direct sum of irreducible representations as 
\begin{align*}
\cH_\lambda(D)|_{\widetilde{G}_1}&\simeq\hsum_{\bk\in\BZ_{++}^{r_2}}\cH_{\varepsilon_1\lambda}(D_1,\cP_\bk(\fp^+_2)) && (\fp^+_2\colon\text{simple}), \\
\cH_\lambda(D)|_{\widetilde{G}_1}&\simeq\hsum_{\bk\in\BZ_{++}^{r'}}\,\hsum_{\bl\in\BZ_{++}^{r''}}\cH_{\varepsilon_1\lambda}(D_1,\cP_\bk(\fp^+_{11})\boxtimes\cP_\bl(\fp^+_{22})) 
&& (\fp^+_2\colon\text{non-simple}). 
\end{align*}
\end{theorem}

In the following, when $\fp^+_2$ is not simple, we write $\cP_{(\bk,\bl)}(\fp^+_2):=\cP_\bk(\fp^+_{11})\boxtimes\cP_\bl(\fp^+_{22})$. 
Then by the above decomposition, for each $\lambda>\frac{d}{2}(r-1)$ and $\bk\in\BZ_{++}^{r_2}$ or $\bk\in\BZ_{++}^{r'}\times\BZ_{++}^{r''}$, 
there exist $\widetilde{G}_1$-intertwining operators 
\begin{align*}
\cF_{\lambda,\bk}^\downarrow\colon \cH_\lambda(D)|_{\widetilde{G}_1}\longrightarrow\cH_{\varepsilon_1\lambda}(D_1,\cP_\bk(\fp^+_2)), \\
\cF_{\lambda,\bk}^\uparrow\colon \cH_{\varepsilon_1\lambda}(D_1,\cP_\bk(\fp^+_2))\longrightarrow\cH_\lambda(D)|_{\widetilde{G}_1}, 
\end{align*}
which are unique up to constant multiple. $\cF_{\lambda,\bk}^\downarrow$ is called a \textit{symmetry breaking operator}, and $\cF_{\lambda,\bk}^\uparrow$ is called a \textit{holographic operator}, 
according to the terminology introduced in \cite{KP1, KP2} and \cite{KP3} respectively. 
Then the construction of $\cF_{\lambda,\bk}^\downarrow$ is reduced to the computation of some inner product. 
In the following, let $V_\bk$ be an abstract $K_1$-module isomorphic to $\cP_\bk(\fp^+_2)$, 
let $\Vert\cdot\Vert_{\varepsilon_1\lambda,\bk,\fp^+_1}$ be the $\widetilde{G}_1$-invariant norm on $\cH_{\varepsilon_1\lambda}(D_1,V_\bk)$ 
normalized such that $\Vert v\Vert_{\varepsilon_1\lambda,\bk,\fp^+_1}=|v|_{V_\bk}$ holds for all constant functions $v\in V_\bk$, 
and let $\Vert\cdot\Vert_{F,\fp^+}$ be the Fischer norm on $\fp^+$ given in (\ref{Fischer_inner_prod}). 

\begin{theorem}[{\cite[Theorem 3.10\,(1)]{N1}, \cite[Proposition 2.9]{N3}}]
For $\lambda>p-1$ and for $\bk\in\BZ_{++}^{r_2}$ (when $\fp^+_2$ is simple) or $\bk\in\BZ_{++}^{r'}\times\BZ_{++}^{r''}$ (when $\fp^+_2$ is not simple), 
let $C_{\fp^+,\fp^+_2}(\lambda,\bk)\in\BC$ be the constant satisfying 
\begin{equation}\label{SBO0}
\Bigl\langle f(x_2),e^{(x|\overline{z})_{\fp^+}}\Bigr\rangle_{\lambda,x}\Bigr|_{z_1=0}=C_{\fp^+,\fp^+_2}(\lambda,\bk)f(z_2) \qquad (f(x_2)\in\cP_\bk(\fp^+_2)). 
\end{equation}
We take a vector-valued polynomial $\rK_\bk(x_2)\in\cP(\fp^+_2,\overline{V_\bk})^{K_1}$ normalized such that 
\[ \bigl| \langle f(x_2),\rK_\bk(x_2)\rangle_{F,\fp^+} \bigr|_{V_\bk}^2=\Vert f\Vert_{F,\fp^+}^2 \qquad (f(x_2)\in\cP_\bk(\fp^+_2)), \]
and define the vector-valued polynomial $F_{\lambda,\bk}^\downarrow(z)\in \cP(\fp^-,V_\bk)$ by 
\begin{equation}\label{SBO1}
F_{\lambda,\bk}^\downarrow(z):=\frac{1}{C_{\fp^+,\fp^+_2}(\lambda,\bk)}\Bigl\langle e^{(x|z)_{\fp^+}},\rK_\bk(x_2)\Bigr\rangle_{\lambda,x}. 
\end{equation}
Then the differential operator 
\begin{equation}\label{SBO2}
\cF_{\lambda,\bk}^\downarrow\colon \cH_\lambda(D)|_{\widetilde{G}_1}\longrightarrow\cH_{\varepsilon_1\lambda}(D_1,V_\bk), \qquad
(\cF_{\lambda,\bk}^\downarrow f)(x_1):=F_{\lambda,\bk}^\downarrow \biggl(\frac{\partial}{\partial x}\biggr)f(x)\biggr|_{x_2=0}
\end{equation}
becomes a symmetry breaking operator satisfying 
\[ \bigl\Vert \cF_{\lambda,\bk}^\downarrow f\bigr\Vert_{\varepsilon_1\lambda,\bk,\fp^+_1}^2=\Vert f\Vert_{F,\fp^+}^2 \qquad (f(x_2)\in\cP_\bk(\fp^+_2)). \]
Its operator norm is given by 
\[ \bigl\Vert \cF_{\lambda,\bk}^\downarrow \bigr\Vert_{\mathrm{op}}^2=C_{\fp^+,\fp^+_2}(\lambda,\bk)^{-1}. \]
\end{theorem}

We note that \cite[Theorem 3.10\,(1)]{N1} holds also for holomorphic discrete series representations of non-scalar type, namely, 
we can construct symmetry breaking operators for $\cH_\lambda(D,V)|_{\widetilde{G}_1}$ of non-scalar type by using a polynomial analogous to (\ref{SBO1}), 
although we cannot determine the normalizing constant $C_{\fp^+,\fp^+_2}(\lambda,\bk)$ by (\ref{SBO0}) unless $\cP(\fp^+_2)\otimes V|_{K_1}$ is multiplicity-free under $K_1$, 
or equivalently, $\cH_\lambda(D,V)|_{\widetilde{G}_1}$ is multiplicity-free under $\widetilde{G}_1$. 

The notion of symmetry breaking operator is extended to the map 
\[ \cF_{\lambda,\bk}^\downarrow\colon \cO_\lambda(D)|_{\widetilde{G}_1}\longrightarrow\cO_{\varepsilon_1\lambda}(D_1,\cP_\bk(\fp^+_2)) \]
for all $\lambda\in\BC$. Especially, the family of operators constructed in the above theorem is meromorphically continued for all $\lambda\in\BC$, 
and gives symmetry breaking operators on $\cO_\lambda(D)$ for all $\lambda\in\BC$ except at its poles. 
Moreover, symmetry breaking operators between $\cO_\lambda(D)|_{\widetilde{G}_1}$ and $\cO_{\varepsilon_1\lambda}(D_1,\cP_\bk(\fp^+_2))$ are always given by differential operators 
(localness theorem, \cite[Theorem 5.3]{KP1}, which holds more generally for vector-valued case). 
However, the uniqueness of $\cF_{\lambda,\bk}^\downarrow$ is not trivial if $\cO_\lambda(D)$ does not contain a holomorphic discrete series representation. 
Indeed, the uniqueness of symmetry breaking operators for tensor product case sometimes fails. See \cite[Section 9]{KP2}, \cite[Section 8]{N3}. 
The symbols of differential symmetry breaking operators are characterized as polynomial solutions of some differential equations as follows. 
For $\lambda\in\BC$, let $\cB_\lambda^{\fp^\pm}$ be the vector-valued differential operator 
\begin{gather}
\cB_\lambda^{\fp^\pm}\colon \cP(\fp^\pm)\longrightarrow \cP(\fp^\pm)\otimes\fp^\mp, \notag \\
\cB_\lambda^{\fp^\pm}f(z):=\sum_{\alpha,\beta}\frac{1}{2}Q(e^\mp_\alpha,e^\mp_\beta)z\frac{\partial^2f}{\partial z_\alpha^\pm\partial z_\beta^\pm}(z)
+\lambda\sum_{\alpha}e_\alpha^\mp \frac{\partial f}{\partial z_\alpha^\pm}(z), \label{formula_Bessel_diff}
\end{gather}
where $\{e_\alpha^\mp\}\subset\fp^\mp$ is a basis of $\fp^\mp$, and $\{z_\alpha^\pm\}$ is the coordinate of $\fp^\pm$ dual to $\{e_\alpha^\mp\}$ 
(\textit{Bessel operator} \cite{D}, \cite[Section XV.2]{FK}), let 
\[ (\cB_\lambda^{\fp^\pm})_1\colon \cP(\fp^\pm)\longrightarrow\cP(\fp^\pm)\otimes\fp^\mp_1 \]
be the orthogonal projection of $\cB_\lambda^{\fp^\pm}$ onto $\fp^\mp_1$, and let 
\[ \operatorname{Sol}_{\cP(\fp^\pm)}\bigl((\cB_\lambda^{\fp^\pm})_1\bigr):=\bigl\{ f(z)\in\cP(\fp^\pm) \bigm| (\cB_\lambda^{\fp^\pm})_1 f=0 \bigr\}. \]
\begin{theorem}[{F-method, \cite[Theorem 3.1]{KP2}}]
Let $V$ be a $K_1^\BC$-module, and $V^\vee$ be its contragredient. Then we have 
\begin{align*}
\Hom_{\widetilde{G}_1}(\cO_\lambda(D),\cO_{\varepsilon_1\lambda}(D_1,V))
&\simeq \Hom_{\fk_1^\BC\oplus\fp_1^-}\Bigl(\chi_1^{\varepsilon_1\lambda}\otimes V^\vee,\operatorname{ind}_{\fk^\BC\oplus\fp^-}^{\fg^\BC}(\chi^{\lambda})\Bigr) \\
&\simeq\Bigl(\operatorname{Sol}_{\cP(\fp^-)}\bigl((\cB_\lambda^{\fp^-})_1\bigr)\otimes V\Bigr)^{K_1}. 
\end{align*}
\end{theorem}

We note that \cite[Theorem 3.1]{KP2} holds also for vector-valued case. See also \cite[Theorem 2.13]{N2}. 

Especially, if $\lambda>\frac{d}{2}(r-1)$, then since $\cH_\lambda(D)|_{\widetilde{G}_1}$ decomposes multiplicity-freely, the space of symmetry breaking operators is 1-dimensional 
for each $V\simeq\cP_\bk(\fp^+_2)$, and combining the above two theorems, for $\lambda>\frac{d}{2}(r-1)$ we get
\[ \Hom_{K_1}\Bigl(\cP_\bk(\fp^+_2),\operatorname{Sol}_{\cP(\fp^+)}\bigl((\cB_\lambda^{\fp^+})_1\bigr)\Bigr)
=\BC\Bigl(f(x_2)\mapsto \Bigl\langle f(x_2),e^{(x|\overline{z})_{\fp^+}}\Bigr\rangle_{\lambda,x}\Bigr). \]

In this article, we consider the symmetry breaking operators for the case $\fp^+,\fp^+_2$ are simple of tube type and $\varepsilon_2=1$, so that $r=r_2$ holds. 
In \cite[Section 6]{N2}, we treated $\bk=(k+l,k,\ldots,k)$ case. In this article, we treat general $\bk\in\BZ_{++}^3$ for $r=3$ case in Section \ref{section_rank3}, 
and treat $\bk=(k,\ldots,k,k-l)$ case for general $r$ in Section \ref{section_general}.

\section{Rank 3 case}\label{section_rank3}

In this section, we treat $\fp^\pm=(\fp^\pm)^\sigma\oplus(\fp^\pm)^{-\sigma}=\fp^\pm_1\oplus\fp^\pm_2$ such that $\fp^\pm$, $\fp^\pm_2$ are both simple, of tube type, and of rank 3, 
that is, we treat 
\begin{align*}
(\fp^\pm,\fp^\pm_1,\fp^\pm_2)
&=(\Herm(3,\BF)^\BC,\Alt(3,\BF')^\BC,\Herm(3,\BF')^\BC) \quad ((\BF,\BF')=(\BC,\BR),(\BH,\BC),(\BO,\BH)) \\
&=\begin{cases} (M(3,\BC),\Alt(3,\BC),\Sym(3,\BC)) & ((\BF,\BF')=(\BC,\BR)), \\ (\Alt(6,\BC),\Alt(3,\BC)\oplus\Alt(3,\BC),M(3,\BC)) & ((\BF,\BF')=(\BH,\BC)), \\
(\Herm(3,\BO)^\BC,M(2,6;\BC),\Alt(6,\BC)) & ((\BF,\BF')=(\BO,\BH)). \end{cases}
\end{align*}
Then the corresponding symmetric pairs $(G,(G^\sigma)_0,(G^{\sigma\vartheta})_0)=(G,G_1,G_2)$ are given by 
\begin{align*}
(G,G_1,G_2)=\begin{cases} (SU(3,3),SO^*(6),Sp(3,\BR)) & ((\BF,\BF')=(\BC,\BR)), \\ (SO^*(12),SO^*(6)\times SO^*(6),U(3,3)) & ((\BF,\BF')=(\BH,\BC)), \\ 
(E_{7(-25)},SU(2,6),SU(2)\times SO^*(12)) & ((\BF,\BF')=(\BO,\BH)) \end{cases}
\end{align*}
(up to covering). We have $(d,d_2)=(\dim_\BR\BF,\dim_\BR\BF')\in\{(2,1),(4,2),(8,4)\}$, and $\varepsilon_1=2$ for $(\BF,\BF')=(\BC,\BR)$, $\varepsilon_1=1$ for $(\BF,\BF')=(\BH,\BC)$, $(\BO,\BH)$. 
$\fp^+$ and $\fp^+_2$ have Jordan algebra structures with the Euclidean real forms $\fn^+=\Herm(3,\BF)$ and $\fn^+_2=\Herm(3,\BF')$ respectively.

\subsection{Jordan algebra structure on $\Herm(3,\BF)^\BC$}

First we look at the Jordan algebra structure on $\fp^+=\fn^{+\BC}=\Herm(3,\BF)^\BC$ for $\BF=\BR,\BC,\BH,\BO$ given by 
\[ x\circ y:=\frac{1}{2}(xy+yx), \]
with the unit element $e=I_3$, the Euclidean real form $\fn^+=\Herm(3,\BF)$, the symmetric cone $\Omega=\Herm_+(3,\BF)\subset\fn^+$ (the set of positive definite matrices), 
and the associative bilinear form $(\cdot|\cdot)=(\cdot|\cdot)_{\fn^+}\colon\fn^{+\BC}\times\fn^{+\BC}\to\BC$, 
\[ (x|y)=(x|y)_{\fn^+}:=\Re_\BF\tr(xy)=\Re_\BF\tr(yx), \]
where $\Re_\BF\colon \BF^\BC=\BF\otimes_\BR\BC\to\BC$ is the $\BC$-linear extension of the $\BF$-real part. 
Next, for $x\in\fn^{+\BC}$, let $x^\sharp$ be the adjugate matrix given by 
\begin{align*}
\begin{pmatrix} a_1&x_3&\hat{x}_2 \\ \hat{x}_3&a_2&x_1 \\ x_2&\hat{x}_1&a_3 \end{pmatrix}^\sharp
:=\begin{pmatrix} a_2a_3-x_1\hat{x}_1&\hat{x}_2\hat{x}_1-a_3x_3&x_3x_1-a_2\hat{x}_2 \\ x_1x_2-a_3\hat{x}_3&a_3a_1-x_2\hat{x}_2&\hat{x}_3\hat{x}_2-a_1x_1 \\
\hat{x}_1\hat{x}_3-a_2x_2&x_2x_3-a_1\hat{x}_1&a_1a_2-x_3\hat{x}_3 \end{pmatrix} \qquad (a_i\in\BC,\; x_i\in\BF^\BC), 
\end{align*}
where $x_i\mapsto\hat{x}_i$ is the $\BC$-linear extension of the $\BF$-conjugate, and for $x,y\in\Herm(3,\BF)^\BC$, let $x\times y$ be the \textit{Freudenthal product} given by 
\[ x\times y:=(x+y)^\sharp-x^\sharp-y^\sharp, \]
so that $x\times x=2x^\sharp$. Then the determinant polynomial, the inverse element, the generic norm, and the map $P(x)\in\End_\BC(\fn^{+\BC})$ are given by 
\begin{align*}
\det(x)&=\det_{\fn^+}(x):=\frac{1}{3}(x|x^\sharp), & x^\itinv&:=\frac{1}{\det(x)}x^\sharp, \\
h(x,y):\hspace{-3pt}&=1-(x|y)+(x^\sharp|y^\sharp)-\det(x)\det(y), & P(x)y&:=(x|y)x-x^\sharp\times y. 
\end{align*}
$xx^\sharp=x^\sharp x=\det(x)I$ holds for $\BF=\BR,\BC,\BH$, and $\frac{1}{2}(xx^\sharp+x^\sharp x)=\det(x)I$ holds for $\BF=\BO$. 
Also, for $x,y,z\in\fn^{+\BC}$ we have 
\[ (x\times y|z)=(y\times z|x)=(z\times x|y), \qquad (x^\sharp)^\sharp=\det(x)x. \]
If $\BF=\BR,\BC,\BH$, then by taking the directional derivative of $zz^\sharp=\det(z)I$ along $y$, we get 
\begin{equation}\label{formula_Freudenthal}
yz^\sharp+z(z\times y)=(z^\sharp|y)I, 
\end{equation}
and by substituting $z=x^\sharp$ and multiplying $\det(x)^{-1}x$ from the left, we get 
\[ P(x)y=(x|y)x-x^\sharp\times y=xyx. \]
We note that these do not hold for $\BF=\BO$. 
Let $\overline{\cdot}\colon\fn^{+\BC}\to\fn^{+\BC}$ be the complex conjugate with respect to the real form $\fn^+\subset\fn^{+\BC}$. 
Then the bounded symmetric domain $D\subset\fp^+=\fn^{+\BC}$ is the connected component of $\{x\in\fn^{+\BC}\mid h(x,\overline{x})>0\}$ which contains the origin. 

Next we consider the Cayley--Dickson extension $\BF=\BF'\oplus\BF'\rj$ for $(\BF,\BF')=(\BC,\BR)$, $(\BH,\BC)$, $(\BO,\BH)$, 
and as in \cite{Y1,Y2} we consider the involution $\sigma$ on $\fp^\pm=\Herm(3,\BF)$, 
\[ \sigma\begin{pmatrix} a_1&x_3+y_3\rj&\hat{x}_2-y_2\rj \\ \hat{x}_3-y_3\rj&a_2&x_1+y_1\rj \\ x_2+y_2\rj&\hat{x}_1-y_1\rj&a_3 \end{pmatrix}
:=-\begin{pmatrix} a_1&x_3-y_3\rj&\hat{x}_2+y_2\rj \\ \hat{x}_3+y_3\rj&a_2&x_1-y_1\rj \\ x_2-y_2\rj&\hat{x}_1+y_1\rj&a_3 \end{pmatrix} \]
($a_i\in\BC$, $x_i,y_i\in\BF^{\prime\BC}$), so that we have 
\begin{align*}
(\fp^\pm_1,\fp^\pm_2)&:=((\fp^\pm)^\sigma,(\fp^\pm)^{-\sigma})=(\Alt(3,\BF')^\BC\rj,\Herm(3,\BF')^\BC), \\ 
(\fn^+_1,\fn^+_2)&:=((\fn^+)^\sigma,(\fn^+)^{-\sigma})=(\Alt(3,\BF')\rj,\Herm(3,\BF')). 
\end{align*}
$\fp^+_2=\fn^{+\BC}_2$ becomes a Jordan subalgebra with the Euclidean real form $\fn^+_2$ and the symmetric cone $\Omega_2:=\Omega\cap\fn^+_2\subset\fn^+_2$, but $\fp^+_1=\fn^{+\BC}_1$ does not. 
We identify $\fp^+_1=\fn^{+\BC}_1=\Alt(3,\BF')^\BC\rj$ with $(\BF^{\prime\BC})^3$ (the space of row vectors) by 
\[ (\BF^{\prime\BC})^3\overset{\sim}{\longrightarrow} \Alt(3,\BF')^\BC\rj, \qquad 
(y_1,y_2,y_3)\mapsto \begin{pmatrix} 0&y_3\rj&-y_2\rj \\ -y_3\rj&0&y_1\rj \\ y_2\rj&-y_1\rj&0 \end{pmatrix}. \]
In the following, we use lowercase Greek letters and capital Latin letters to express elements in $\fn^{+\BC}_1\simeq (\BF^{\prime\BC})^3$ and $\fn^{+\BC}_2=\Herm(3,\BF')^\BC$ respectively. 
For $\xi,\eta\in(\BF^{\prime\BC})^3$, let 
\[ (\xi|\eta):=\Re_{\BF'}(\xi{}^t\hspace{-1pt}\hat{\eta}). \]
Then for $x=(x_1,x_2)=(\xi,X)$, $y=(y_1,y_2)=(\eta,Y)\in\fn^{+\BC}=(\BF^{\prime\BC})^3\oplus\Herm(3,\BF')^\BC$, we have 
\begin{align*}
(x|y)_{\fn^+}&=(x_1|y_1)_{\fn^+}+(x_2|y_2)_{\fn^+}=2(\xi|\eta)+(X|Y), \\
(\xi,X)^\sharp&=(-\xi X,X^\sharp-{}^t\hspace{-1pt}\hat{\xi}\xi), \\
\det((\xi,X))&=\det(X)-\xi X{}^t\hspace{-1pt}\hat{\xi}, \\
P((\xi,X))(\eta,Y)&=(\xi{}^t\hspace{-1pt}\hat{\eta}\xi+\eta X^\sharp+(X|Y)\xi-\xi XY, \\*
&\hspace{17pt}XYX+({}^t\hspace{-1pt}\hat{\xi}\xi)\times Y+2(\xi|\eta)X-X{}^t\hspace{-1pt}\hat{\xi}\eta-{}^t\hspace{-1pt}\hat{\eta}\xi X). 
\end{align*}
$\fp^\pm_1$, $\fp^\pm_2$ are preserved by the action of $K_1:=G_1\cap G_2$. Especially, the action of the real form $L_2=K_1^\BC\cap L\subset K_1^\BC$, 
\[ L_2:=\begin{cases} \{1\}\times GL(3,\BR) \\ \{1\}\times GL(3,\BC) \\ Sp(1)\times GL(3,\BH) \end{cases} \subset 
L:=\begin{cases} \{g\in GL(3,\BC)\mid \det(g)\in\BR\} & ((\BF,\BF')=(\BC,\BR)), \\ GL(3,\BH) & ((\BF,\BF')=(\BH,\BC)), \\
\BR_{>0}\times E_{6(-26)} & ((\BF,\BF')=(\BO,\BH)), \end{cases} \]
on $\fn^+_1\oplus\fn^+_2=\BF^{\prime 3}\oplus\Herm(3,\BF')$ is given by, for $g=(g_1,g_2)\in L_2$, 
\[ gx=(gx_1,gx_2)=(\chi(g_2)g_1 \xi g_2^{-1},g_2 X{}^t\hspace{-1pt}\hat{g}_2), \]
and this preserves the symmetric cone $\Omega_2\subset \fn^+_2$.

\subsection{Definition of functions $F^\uparrow_{\bk,\bl}[f](y_1,x_2)$, $F^\downarrow_{\bk,\bl}[f](x_1,x_2)$, $F^\downarrow(\lambda,\bk;f;x)$}

In this subsection, we define some functions on $\fp^\pm_1\oplus\fp^\pm_2$. 
First we note that, for $m\in\BZ_{\ge 0}$, $(y_1,x_2)=(\eta,X)\in\fp^\mp_1\oplus\fp^\pm_2$, $(x_1,x_2)=(\xi,X)\in\fp^\pm_1\oplus\fp^\pm_2$, we have 
\begin{align*}
\bigl(-((y_1)^\sharp|(x_2)^\sharp)\bigr)^m&=(\eta X^\sharp{}^t\hspace{-1pt}\hat{\eta})^m\in \cP_{2m}(\fp^\mp_1)_{y_1}\otimes\cP_{(m,m,0)}(\fp^\pm_2)_{x_2}, \\
\biggl(-\frac{((x_1)^\sharp|x_2)}{\det_{\fn^\pm_2}(x_2)}\biggr)^m&=\biggl(\frac{\xi X{}^t\hspace{-1pt}\hat{\xi}}{\det(X)}\biggr)^m\in \cP_{2m}(\fp^\pm_1)_{x_1}\otimes\cP_{(0,-m,-m)}(\fp^\pm_2)_{x_2}, 
\end{align*}
where $\cP_m(\fp^\mp_1)$ is the space of homogeneous polynomials on $\fp^\mp_1$ of degree $m$. 
Next let $\bk\in\BZ_{++}^3+\BC\underline{1}_3=\BZ_{++}^3+\BC(1,1,1)$, and take $f(x_2)\in\cP_\bk(\fp^\pm_2)=\cP_{(k_1-k_3,k_2-k_3,0)}(\fp^\pm_2)\det_{\fn^\pm_2}(x_2)^{k_3}$. 
Then we have 
\begin{align*}
\bigl(-((y_1)^\sharp|(x_2)^\sharp)\bigr)^m f(x_2)&\in \bigoplus_{\substack{\bl\in\BZ^3,\,|\bl|=m \\ 0\le l_1\le k_1-k_2 \\ 0\le l_2\le k_2-k_3 \\ 0\le l_3}}
\cP_{2m}(\fp^\mp_1)_{y_1}\otimes\cP_{\substack{(k_1+l_2+l_3,k_2+l_1+l_3,\,\\ \hspace{40pt}k_3+l_1+l_2)}}(\fp^\pm_2)_{x_2}, \\
\biggl(-\frac{((x_1)^\sharp|x_2)}{\det_{\fn^\pm_2}(x_2)}\biggr)^m f(x_2)&\in \bigoplus_{\substack{\bl\in\BZ^3,\,|\bl|=m \\ 0\le l_1 \\ 0\le l_2\le k_1-k_2 \\ 0\le l_3\le k_2-k_3}}
\cP_{2m}(\fp^\pm_1)_{x_1}\otimes\cP_{\substack{(k_1-l_2-l_3,k_2-l_1-l_3,\,\\ \hspace{40pt}k_3-l_1-l_2)}}(\fp^\pm_2)_{x_2}. 
\end{align*}
According to these decompositions, for $\bl\in(\BZ_{\ge 0})^3$ we define the polynomials $F_{\bk,\bl}^\uparrow[f](y_1,x_2)\allowbreak=F_{\bk,\bl}^\uparrow[f](\eta,X)$ and 
$F_{\bk,\bl}^\downarrow[f](x_1,x_2)=F_{\bk,\bl}^\downarrow[f](\xi,X)$ by 
\begin{align*}
F_{\bk,\bl}^\uparrow[f](y_1,x_2)&:=\Proj^{\fp^\pm_2}_{(k_1+l_2+l_3,k_2+l_1+l_3,k_3+l_1+l_2),x_2}\biggl(\frac{1}{|\bl|!}\bigl(-((y_1)^\sharp|(x_2)^\sharp)\bigr)^{|\bl|} f(x_2)\biggr) \\*
&\hspace{15pt}\in \cP_{2|\bl|}(\fp^\mp_1)_{y_1}\otimes\cP_{(k_1+l_2+l_3,k_2+l_1+l_3,k_3+l_1+l_2)}(\fp^\pm_2)_{x_2}, \\
F_{\bk,\bl}^\downarrow[f](x_1,x_2)&:=\Proj^{\fp^\pm_2}_{(k_1-l_2-l_3,k_2-l_1-l_3,k_3-l_1-l_2),x_2}
\biggl(\frac{1}{|\bl|!}\biggl(-\frac{((x_1)^\sharp|x_2)}{\det_{\fn^\pm_2}(x_2)}\biggr)^{|\bl|} f(x_2)\biggr) \\*
&\hspace{15pt}\in \cP_{2|\bl|}(\fp^\pm_1)_{x_1}\otimes\cP_{(k_1-l_2-l_3,k_2-l_1-l_3,k_3-l_1-l_2)}(\fp^\pm_2)_{x_2}. 
\end{align*}
Under the identification $\cP_\bk(\fp^\pm_2)\simeq \cP_{-\bk^\vee}(\fp^\mp_2)$, $f\mapsto f((\cdot)^\itinv)$, where $\bk^\vee:=(k_3,k_2,k_1)$, 
these are related as 
\[ F_{\bk,\bl}^\downarrow[f](x_1,x_2)=F_{-\bk^\vee,\bl^\vee}^\uparrow[f((\cdot)^\itinv)](x_1,x_2^\itinv). \]
We also set 
\begin{align*}
\tilde{F}_{\bk,\bl}[f](x_1,x_2)&:=F_{\bk,\bl}^\downarrow[f](x_1,x_2)\det_{\fn^\pm_2}(x_2)^{|\bl|}
=\Proj^{\fp^\pm_2}_{\bk+\bl,x_2}
\biggl(\frac{1}{|\bl|!}\bigl(-((x_1)^\sharp|x_2)\bigr)^{|\bl|} f(x_2)\biggr) \\*
&\hspace{15pt}\in \cP_{2|\bl|}(\fp^\pm_1)_{x_1}\otimes\cP_{\bk+\bl}(\fp^\pm_2)_{x_2}. 
\end{align*}

Since $F_{\bk,\bl}^\uparrow[f](y_1,x_2)$ is recovered from $F_{\bk,\bl}^\downarrow[f](x_1,x_2)$, 
in the following, we only consider $F_{\bk,\bl}^\downarrow[f](x_1,x_2)$. 
For $\bk\in\BZ_{++}^3+\BC\underline{1}_3$, let $d_{\bk}^{\fp^+_2}:=\dim\cP_\bk(\fp^+_2)$, and let $\Sigma_2\subset\fp^+_2$ be the Bergman--Shilov boundary of $D_2:=D\cap\fp^+_2$. 

\begin{proposition}\label{prop_estimate_Fkl_rank3}
Let $\bk\in\BZ_{++}^3+\BC\underline{1}_3$, $f(x_2)\in\cP_\bk(\fp^+_2)$, $\bl\in(\BZ_{\ge 0})^3$ with $0\le l_2\le k_1-k_2$, $0\le l_3\le k_2-k_3$. 
\begin{enumerate}
\item For $x_1\in\fp^+_1$, $x_2=ge\in\fp^+_2$ with $g\in \widetilde{K}_1^\BC$, we have 
\[ \bigl|F_{\bk,\bl}^\downarrow[f](x_1,x_2)\bigr|
\le \frac{\sqrt{d_{\bk}^{\fp^+_2}d_{(|\bl|,0,0)}^{\fp^+_2}}}{|\bl|!}\frac{(g^{-1}x_1|\overline{g^{-1}x_1})_{\fp^+}^{|\bl|}}{2^{|\bl|}}\Vert f(g(\cdot))\Vert_{L^2(\Sigma_2)}. \]
\item If $f(x_2)\ne 0$, then $F_{\bk,\bl}^\downarrow[f](x_1,x_2)\ne 0$. 
\end{enumerate}
\end{proposition}
\begin{proof}
(1) Since $f(x_2)\mapsto F_{\bk,\bl}^\downarrow[f](x_1,x_2)$ is $\widetilde{K}_1^\BC$-equivariant, we have 
\[ F_{\bk,\bl}^\downarrow[f](x_1,x_2)=F_{\bk,\bl}^\downarrow[f](x_1,ge)=F_{\bk,\bl}^\downarrow[f(g(\cdot))](g^{-1}x_1,e), \]
and hence it suffices to prove the inequality when $x_2=e$, so that $F_{\bk,\bl}^\downarrow[f](x_1,e)\allowbreak=\tilde{F}_{\bk,\bl}[f](\xi,e)$. 
Next, since the reproducing kernel of $\cP_{\bk+\bl}(\fp^+_2)$ with respect to the $L^2(\Sigma_2)$-inner product is given by 
$K_{\bk+\bl}'(X,Z):=d_{\bk+\bl}^{\fp^+_2}\Phi_{\bk+\bl}^{\fn^+_2}(P(\overline{Z}^{\mathit{1/2}})X)=d_{\bk+\bl}^{\fp^+_2}\Phi_{\bk+\bl}^{\fn^+_2}(P(X^{\mathit{1/2}})\overline{Z})$, 
where $\Phi_{\bk}^{\fn^+_2}(X):=\int_{K_1\cap K_L}\Delta_\bk^{\fn^+_2}(kX)\,dk$ (see \cite[Proposition XII.2.4]{FK}), we have 
\begin{align*}
\bigl|\tilde{F}_{\bk,\bl}[f](\xi,e)\bigr|&=\bigl|\bigl\langle \tilde{F}_{\bk,\bl}[f](\xi,\cdot),K_{\bk+\bl}'(\cdot,e)\bigr\rangle_{L^2(\Sigma_2)}\bigr| \\
&\le \bigl\Vert \tilde{F}_{\bk,\bl}[f](\xi,\cdot)\bigr\Vert_{L^2(\Sigma_2)}\Vert K_{\bk+\bl}'(\cdot,e)\Vert_{L^2(\Sigma_2)} \\
&=\bigl\Vert \tilde{F}_{\bk,\bl}[f](\xi,\cdot)\bigr\Vert_{L^2(\Sigma_2)}\sqrt{K_{\bk+\bl}'(e,e)}
=\sqrt{d_{\bk+\bl}^{\fp^+_2}}\bigl\Vert \tilde{F}_{\bk,\bl}[f](\xi,\cdot)\bigr\Vert_{L^2(\Sigma_2)}. 
\end{align*}
Now since $\tilde{F}_{\bk,\bl}[f](\xi,X)$ is defined by the orthogonal projection of $(\xi X{}^t\hspace{-1pt}\hat{\xi})^{|\bl|}f(X)/|\bl|!$, we have 
\begin{align*}
\bigl\Vert \tilde{F}_{\bk,\bl}[f](\xi,\cdot)\bigr\Vert_{L^2(\Sigma_2)}&\le \biggl\Vert\frac{1}{|\bl|!}(\xi X{}^t\hspace{-1pt}\hat{\xi})^{|\bl|}f(X)\biggr\Vert_{L^2(\Sigma_2),X}
\le \frac{1}{|\bl|!}\bigl\Vert \xi X{}^t\hspace{-1pt}\hat{\xi}\bigr\Vert_{L^\infty(\Sigma_2),X}^{|\bl|}\Vert f\Vert_{L^2(\Sigma_2)}, 
\end{align*}
and for $X\in\Sigma_2$ we have 
\[ |\xi X{}^t\hspace{-1pt}\hat{\xi}|=|(\xi|\xi X)|\le \sqrt{(\xi|\overline{\xi})(\xi X|\overline{\xi X})}=(\xi|\overline{\xi})=\frac{1}{2}(x_1|\overline{x_1})_{\fp^+}. \]
We also have 
\[ d_{\bk+\bl}^{\fp^+_2}=\dim\cP_{\bk+\bl}(\fp^+_2)\le \dim\bigl(\cP_\bk(\fp^+_2)\otimes\cP_{(|\bl|,0,0)}(\fp^+_2)\bigr)=d_{\bk}^{\fp^+_2}d_{(|\bl|,0,0)}^{\fp^+_2}. \]
Hence the proposition follows. 

(2) Again since $f(x_2)\mapsto F_{\bk,\bl}^\downarrow[f](x_1,x_2)$ is $\widetilde{K}_1^\BC$-equivariant, it is enough to prove when $f(x_2)=\Delta_\bk^{\fn^+_2}(x_2)=\Delta_\bk^{\fn^+_2}(X)$. 
Then since ${}^t\hspace{-1pt}\hat{\xi}\xi\in\fp^+_2$ is at most of rank 1 for $x_1=\xi\in\fp^+_1$ and every rank 1 element is written in this form, by Lemma \ref{lem_Proj_Pieri}\,(5) we have 
\[ F_{\bk,\bl}^\downarrow\bigl[\Delta_\bk^{\fn^+_2}\bigr](x_1,x_2)
=\frac{1}{|\bl|!}\det(X)^{-|\bl|}\Proj_{\bk+\bl,X}^{\fp^+_2}\Bigl((X|{}^t\hspace{-1pt}\hat{\xi}\xi)^{|\bl|}\Delta_\bk^{\fn^+_2}(X)\Bigr)\ne 0 \]
for some $x_1=\xi\in\fp^+_1$. Hence we get the desired formula. 
\end{proof}

The following proposition gives some examples of $F_{\bk,\bl}^\downarrow[f](x_1,x_2)$. 
\begin{proposition}\label{prop_example_Fkl_rank3}
Let $x=(x_1,x_2)=(\xi,X)\in\fp^+_1\oplus\fp^+_2$. 
\begin{enumerate}
\item For $k\in\BC$, $\bl\in(\BZ_{\ge 0})^3$, $\bk\in\BZ_{++}^3$, $f(x_2)\in\cP_\bk(\fp^+_2)$, we have 
\[ F_{\bk+(k,k,k),\bl}^\downarrow\bigl[f\det_{\fn^+_2}^k\bigr](x_1,x_2)=\det_{\fn^+_2}(x_2)^kF_{\bk,\bl}^\downarrow[f](x_1,x_2). \]
\item For $k\in\BC$, $l\in\BZ_{\ge 0}$, we have 
\begin{align*}
&F_{(k,k,k),(l,0,0)}^\downarrow\bigl[\det_{\fn^+_2}^k\bigr](x_1,x_2) \\*
&=\frac{1}{l!}\biggl(-\frac{((x_1)^\sharp|x_2)}{\det_{\fn^+_2}(x_2)}\biggr)^l\det_{\fn^+_2}(x_2)^k
=\frac{1}{l!}\biggl(\frac{\xi X{}^t\hspace{-1pt}\hat{\xi}}{\det(X)}\biggr)^l\det(X)^k, \\
&F_{(k,k,k),(0,0,l)}^\uparrow\bigl[\det_{\fn^+_2}^k\bigr](x_1,x_2) \\*
&=\frac{1}{l!}\bigl(-((x_1)^\sharp|(x_2)^\sharp)\bigr)^l\det_{\fn^+_2}(x_2)^k
=\frac{1}{l!}(\xi X^\sharp{}^t\hspace{-1pt}\hat{\xi})^l\det(X)^k. 
\end{align*}
\item Let $e'\in\fn^+_2$ be a primitive idempotent, and let $\fp^+(e')_0=:\fp^{+\prime\prime}$, $\fp^+_2(e')_0=:\fp^{+\prime\prime}_2$ be as in (\ref{Peirce}). 
Then for $f(x_2'')\in\cP_{(k_1,k_2)}(\fp^{+\prime\prime}_2)$, $l=0,1,\ldots,k_2$, we have 
\begin{align*}
F_{(k_1,k_2,0),(0,0,l)}^\downarrow[f](x_1,x_2)
&=\frac{(-k_2)_l\bigl(-k_1-\frac{d_2}{2}\bigr)_l}{\bigl(-k_2-\frac{d_2}{2}\bigr)_l(-k_1-d_2)_l l!}
\biggl(-\frac{\det_{\fn^{+\prime\prime}}(x_1'')}{\det_{\fn^{+\prime\prime}}(x_2'')}\biggr)^l f(x_2'') \\
&=\frac{(-k_2)_l\bigl(-k_1-\frac{d_2}{2}\bigr)_l}{\bigl(-k_2-\frac{d_2}{2}\bigr)_l(-k_1-d_2)_l l!}
\biggl(\frac{({}^t\hspace{-1pt}\hat{\xi}\xi|e')}{\det_{\fn^{+\prime\prime}_2}(X'')}\biggr)^l f(X''). 
\end{align*}
\item Let $w_2=W\in\fp^+_2$ be of rank 1. Then for $k,l_1,l_2\in\BZ_{\ge 0}$ with $l_2\le k$ we have 
\begin{align*}
&F_{(k,0,0),(l_1,l_2,0)}^\downarrow\bigl[(\cdot|w_2)_{\fn^+_2}^k\bigr](x_1,x_2)
=F_{(0,0,-k),(0,l_2,l_1)}^\uparrow\bigl[((\cdot)^\itinv|w_2)_{\fn^+_2}^k\bigr](x_1,x_2^\itinv) \\
&=\frac{(-1)^{l_1}\det_{\fn^+_2}(x_2)^{-l_1-l_2}}{\bigl(-k-l_1-\frac{d_2}{2}\bigr)_{l_2}l_2!(l_1+l_2)!}
\sum_{j=l_2}^{\min\{k,l_1+l_2\}}\frac{(-k)_j(-l_1-l_2)_j}{\bigl(-k-l_1+l_2-\frac{d_2}{2}+1\bigr)_{j-l_2}(j-l_2)!} \\
&\eqspace{}\times(x_2|(x_1)^\sharp)^{l_1+l_2-j}(x_2|w_2)^{k-j}\bigl((x_2|(x_1)^\sharp)(x_2|w_2)-(x_2(x_1)^\sharp x_2|w_2)\bigr)^j \\
&=\frac{(-1)^{l_2}}{\bigl(-k-l_1-\frac{d_2}{2}\bigr)_{l_2}l_2!(l_1+l_2)!}
\sum_{j=l_2}^{\min\{k,l_1+l_2\}}\frac{(-k)_j(-l_1-l_2)_j}{\bigl(-k-l_1+l_2-\frac{d_2}{2}+1\bigr)_{j-l_2}(j-l_2)!} \\
&\eqspace{}\times(X|{}^t\hspace{-1pt}\hat{\xi}\xi)^{l_1+l_2-j}(X|W)^{k-j}\bigl((X|{}^t\hspace{-1pt}\hat{\xi}\xi)(X|W)-(X{}^t\hspace{-1pt}\hat{\xi}\xi X|W)\bigr)^j\det(X)^{-l_1-l_2}. 
\end{align*}
\item Let $w_2=W\in\fp^+_2$ be of rank 1. Then for $k,l_1,l_3\in\BZ_{\ge 0}$ with $l_3\le k$ we have 
\begin{align*}
&F_{(0,0,-k),(l_1,0,l_3)}^\downarrow\bigl[((\cdot)^\itinv|w_2)_{\fn^+_2}^k\bigr](x_1,x_2)
=F_{(k,0,0),(l_3,0,l_1)}^\uparrow\bigl[(\cdot|w_2)_{\fn^+_2}^k\bigr](x_1,x_2^\itinv) \\
&=\frac{(-1)^{l_1}}{(-k-l_1-d_2)_{l_3}l_3!(l_1+l_3)!}
\sum_{j=l_3}^{\min\{k,l_1+l_3\}}\frac{(-k)_j(-l_1-l_3)_j}{(-k-l_1+l_3-d_2+1)_{j-l_3}(j-l_3)!} \\
&\eqspace{}\times(x_2|(x_1)^\sharp)^{l_1+l_3-j}(x_2^\itinv|w_2)^{k-j}((x_1)^\sharp|w_2)^j\det_{\fn^+_2}(x_2)^{-l_1-l_3} \\
&=\frac{(-1)^{l_3}}{(-k-l_1-d_2)_{l_3}l_3!(l_1+l_3)!}
\sum_{j=l_3}^{\min\{k,l_1+l_3\}}\frac{(-k)_j(-l_1-l_3)_j}{(-k-l_1+l_3-d_2+1)_{j-l_3}(j-l_3)!} \\
&\eqspace{}\times(X|{}^t\hspace{-1pt}\hat{\xi}\xi)^{l_1+l_3-j}(X^\itinv|W)^{k-j}({}^t\hspace{-1pt}\hat{\xi}\xi|W)^j\det(X)^{-l_1-l_3}. 
\end{align*}\end{enumerate}
\end{proposition}
\begin{proof}
(1) Clear from $\cP_{\bk+(k,k,k)}(\fp^+_2)=\det_{\fn^+_2}(x_2)^k\cP_\bk(\fp^+_2)$. 

(2) By the definition of $F_{\bk,\bl}^\downarrow[f](x_1,x_2)$, $F_{\bk,\bl}^\uparrow[f](x_1,x_2)$, 
\begin{align*}
&F_{(k,k,k),(l,0,0)}^\downarrow\bigl[\det_{\fn^+_2}^k\bigr](x_1,x_2)
=\Proj^{\fp^+_2}_{(k,k-l,k-l),x_2}\biggl(\frac{1}{l!}\biggl(-\frac{((x_1)^\sharp|x_2)}{\det_{\fn^+_2}(x_2)}\biggr)^l\det_{\fn^+_2}(x_2)^k\biggr) \\
&=\frac{1}{l!}\biggl(-\frac{((x_1)^\sharp|x_2)}{\det_{\fn^+_2}(x_2)}\biggr)^l\det_{\fn^+_2}(x_2)^k=\frac{1}{l!}\biggl(\frac{\xi X{}^t\hspace{-1pt}\hat{\xi}}{\det(X)}\biggr)^l\det(X)^k, \\
&F_{(k,k,k),(0,0,l)}^\uparrow\bigl[\det_{\fn^+_2}^k\bigr](x_1,x_2)
=\Proj^{\fp^+_2}_{(k+l,k+l,k),x_2}\biggl(\frac{1}{l!}\bigl(-((x_1)^\sharp|(x_2)^\sharp)\bigr)^l\det_{\fn^+_2}(x_2)^k\biggr) \\
&=\frac{1}{l!}\bigl(-((x_1)^\sharp|(x_2)^\sharp)\bigr)^l\det_{\fn^+_2}(x_2)^k=\frac{1}{l!}(\xi X^\sharp{}^t\hspace{-1pt}\hat{\xi})^l\det(X)^k. 
\end{align*}

(3) By Lemma \ref{lem_Proj_Pieri}\,(3) we have 
\begin{align*}
&F_{(k_1,k_2,0),(0,0,l)}^\downarrow[f](x_1,x_2)=\tilde{F}_{(k_1,k_2,0),(0,0,l)}[f](x_1,x_2)\det_{\fn^+_2}(x_2)^{-l} \\
&=\frac{(-1)^l}{l!}\Proj_{(k_1,k_2,l),x_2}^{\fp^+_2}\bigl((x_2|(x_1)^\sharp)^lf(x_2'')\bigr)\det_{\fn^+_2}(x_2)^{-l} \\
&=\frac{(-1)^l}{l!}\frac{(-k_2)_l\bigl(-k_1-\frac{d_2}{2}\bigr)_l}{\bigl(-k_2-\frac{d_2}{2}\bigr)_l(-k_1-d_2)_l}
\frac{((x_1)^\sharp|e')^l}{\det_{\fn^{+\prime\prime}_2}(x_2'')^l}f(x_2''), 
\end{align*}
and we have 
\[ ((x_1)^\sharp|e')=\det_{\fn^{+\prime\prime}}(x_1''), \qquad ((x_1)^\sharp|e')=-({}^t\hspace{-1pt}\hat{\xi}\xi|e'). \]

(4) By Proposition \ref{prop_Proj}\,(2) we have 
\begin{align*}
&F_{(k,0,0),(l_1,l_2,0)}^\downarrow \bigl[(\cdot|w_2)_{\fn^+_2}^k\bigr](x_1,x_2)
=\tilde{F}_{(k,0,0),(l_1,l_2,0)}\bigl[(\cdot|w_2)_{\fn^+_2}^k\bigr](x_1,x_2)\det_{\fn^+_2}(x_2)^{-l_1-l_2} \\
&=\frac{(-1)^{l_1+l_2}}{(l_1+l_2)!}\Proj_{(k+l_1,l_2,0),x_2}^{\fp^+_2}\bigl((x_2|w_2)^k(x_2|(x_1)^\sharp)^{l_1+l_2}\bigr)\det_{\fn^+_2}(x_2)^{-l_1-l_2} \\
&=\frac{(-1)^{l_1+l_2}}{(l_1+l_2)!}\frac{(-1)^{l_2}}{\bigl(-k-l_1-\frac{d_2}{2}\bigr)_{l_2}l_2!}
\sum_{j=l_2}^{\min\{k,l_1+l_2\}}\frac{(-k)_j(-l_1-l_2)_j}{\bigl(-k-l_1+l_2-\frac{d_2}{2}+1\bigr)_{j-l_2}(j-l_2)!} \\*
&\eqspace{}\times(x_2|(x_1)^\sharp)^{l_1+l_2-j}(x_2|w_2)^{k-j}\bigl((x_2|(x_1)^\sharp)(x_2|w_2)-(x_2(x_1)^\sharp x_2|w_2)\bigr)^j\det_{\fn^+_2}(x_2)^{-l_1-l_2} \\
&=\frac{1}{(l_1+l_2)!}\frac{(-1)^{l_2}}{\bigl(-k-l_1-\frac{d_2}{2}\bigr)_{l_2}l_2!}
\sum_{j=l_2}^{\min\{k,l_1+l_2\}}\frac{(-k)_j(-l_1-l_2)_j}{\bigl(-k-l_1+l_2-\frac{d_2}{2}+1\bigr)_{j-l_2}(j-l_2)!} \\*
&\eqspace{}\times(X|{}^t\hspace{-1pt}\hat{\xi}\xi)^{l_1+l_2-j}(X|W)^{k-j}\bigl((X|{}^t\hspace{-1pt}\hat{\xi}\xi)(X|W)-(X{}^t\hspace{-1pt}\hat{\xi}\xi X|W)\bigr)^j\det(X)^{-l_1-l_2}. 
\end{align*}

(5) By Proposition \ref{prop_Proj}\,(1) we have 
\begin{align*}
&F_{\substack{(0,0,-k),\\ \;\;(l_1,0,l_3)}}^\downarrow\bigl[((\cdot)^\itinv|w_2)_{\fn^+_2}^k\bigr](x_1,x_2)
=\tilde{F}_{\substack{(0,0,-k),\\ \;\;(l_1,0,l_3)}}\bigl[((\cdot)^\itinv|w_2)_{\fn^+_2}^k\bigr](x_1,x_2)\det_{\fn^+_2}(x_2)^{-l_1-l_3} \\
&=\frac{(-1)^{l_1+l_3}}{(l_1+l_3)!}\Proj_{(l_1,0,l_3-k),x_2}^{\fp^+_2}\bigl((x_2^\itinv|w_2)^k(x_2|(x_1)^\sharp)^{l_1+l_3}\bigr)\det_{\fn^+_2}(x_2)^{-l_1-l_3} \\
&=\frac{(-1)^{l_1+l_3}}{(l_1+l_3)!}\frac{(-1)^{l_3}}{(-k-l_1-d_2)_{l_3}l_3!}
\sum_{j=l_3}^{\min\{k,l_1+l_3\}}\frac{(-k)_j(-l_1-l_3)_j}{(-k-l_1+l_3-d_2+1)_{j-l_3}(j-l_3)!} \\*
&\eqspace{}\times(x_2|(x_1)^\sharp)^{l_1+l_3-j}(x_2^\itinv|w_2)^{k-j}((x_1)^\sharp|w_2)^j\det_{\fn^+_2}(x_2)^{-l_1-l_3} \\
&=\frac{1}{(l_1+l_3)!}\frac{(-1)^{l_3}}{(-k-l_1-d_2)_{l_3}l_3!}
\sum_{j=l_3}^{\min\{k,l_1+l_3\}}\frac{(-k)_j(-l_1-l_3)_j}{(-k-l_1+l_3-d_2+1)_{j-l_3}(j-l_3)!} \\*
&\eqspace{}\times(X|{}^t\hspace{-1pt}\hat{\xi}\xi)^{l_1+l_3-j}(X^\itinv|W)^{k-j}({}^t\hspace{-1pt}\hat{\xi}\xi|W)^j\det(X)^{-l_1-l_3}. \qedhere
\end{align*}
\end{proof}

Using $F_{\bk,\bl}^\downarrow[f](x_1,x_2)$, for $\lambda\in\BC$, $\bk\in\BZ_{++}^3+\BC\underline{1}_3$ and for $f(x)\in\cP_\bk(\fp^\pm_2)$, we define a function of $x=(x_1,x_2)\in\Omega$ by 
\begin{align}
&F^\downarrow(\lambda,\bk;f;x) \notag\\
:\hspace{-3pt}&=\sum_{\substack{0\le l_1 \\ 0\le l_2\le k_1-k_2 \\ 0\le l_3\le k_2-k_3}}
\frac{((-k_3,-k_2,-k_1))_{(l_1+l_2,l_1+l_3,l_2+l_3),d_2}F_{\bk,\bl}^\downarrow[f](x_1,x_2)}{\bigl(-\lambda+\frac{3}{2}d_2+1-(k_2+k_3,k_1+k_3,k_1+k_2)\bigr)_{(l_1,l_2,l_3),d_2}} \notag\\
&=\sum_{\substack{0\le l_1 \\ 0\le l_2\le k_1-k_2 \\ 0\le l_3\le k_2-k_3}}
\frac{\bigl(-k_1-\frac{d}{2}\bigr)_{l_2+l_3}}{\bigl(-\lambda-k_2-k_3+\frac{3}{4}d+1\bigr)_{l_1}} 
\frac{\bigl(-k_2-\frac{d}{4}\bigr)_{l_1+l_3}}{\bigl(-\lambda-k_1-k_3+\frac{1}{2}d+1\bigr)_{l_2}} \notag \\*
&\hspace{117pt}\times\frac{(-k_3)_{l_1+l_2}}{\bigl(-\lambda-k_1-k_2+\frac{1}{4}d+1\bigr)_{l_3}}F_{\bk,\bl}^\downarrow[f](x_1,x_2), \label{def_Fdown_rank3}
\end{align}
where $(\lambda+\bs)_{\bm,d_2}$ is as in (\ref{Pochhammer}), and recall that $d_2=d/2$. 
This has continuous parameters $(\lambda,k_3)\in\BC^2$ and discrete parameters $(k_1',k_2'):=(k_1-k_3,k_2-k_3)\allowbreak\in\BZ_{++}^2$. 
When we put $f'(x_2):=f(x_2)\det_{\fn^\pm_2}(x_2)^{-k_3}\in\cP_{(k_1',k_2',0)}(\fp^\pm_2)$, we have 
\begin{align*}
F^\downarrow(\lambda,\bk;f;x)=\sum_{\substack{0\le l_1 \\ 0\le l_2\le k_1'-k_2' \\ 0\le l_3\le k_2'}}
\frac{(-k_3-(0,k_2',k_1'))_{(l_1+l_2,l_1+l_3,l_2+l_3),d_2}}{\bigl(-\lambda-2k_3+\frac{3}{2}d_2+1-(k_2',k_1',k_1'+k_2')\bigr)_{(l_1,l_2,l_3),d_2}} \\*
{}\times\det_{\fn^\pm_2}(x_2)^{k_3}F_{(k_1',k_2',0),\bl}^\downarrow[f'](x_1,x_2), 
\end{align*}
and when we fix $(k_1',k_2')\in\BZ_{++}^2$, as a function of $(\lambda,k_3)\in\BC^2$ this has poles at 
\begin{align*}
\lambda+2k_3&\in\biggl\{\frac{3}{2}d_2+1-k_2'+j \biggm| j\in\BZ_{\ge 0}\biggr\} \\*
&\eqspace{}\cup\{d_2+1-k_1'+j\mid j\in\BZ,\, 0\le j< k_1'-k_2'\} \\*
&\eqspace{}\cup\biggl\{\frac{1}{2}d_2+1-(k_1'+k_2')+j \biggm| j\in\BZ,\, 0\le j<k_2'\biggr\}. 
\end{align*}
If we additionally fix $k_3\in\BZ_{\ge 0}$, then $F^\downarrow(\lambda,\bk;f;x)$ becomes a polynomial in $x$, 
and the singularities $\lambda\in\bigl\{\frac{3}{2}d_2+1-k_2'-2k_3+j \bigm| j\in\BZ_{\ge k_3}\bigr\}$ are removable. 
We use the same notation $F^\downarrow(\lambda,\bk;f;x)$ to express this polynomial also for these $\lambda$. 

\begin{proposition}
Suppose $\lambda\in\BC$ is not a pole of $F^\downarrow(\lambda,\bk;f;x)$. Then $F^\downarrow(\lambda,\bk;f;\allowbreak x)$ converges if $x_2=ge$ for some $g\in K_1^\BC$ and 
$(g^{-1}x_1|\overline{g^{-1}x_1})_{\fp^+}<2$. Especially, this converges for $x=(x_1,x_2)\in\Omega$. 
\end{proposition}
\begin{proof}
By Proposition \ref{prop_estimate_Fkl_rank3}, for $x_1\in\fp^+_1$, $x_2\in ge\in\fp^+_2$ with $g\in \widetilde{K}_1^\BC$, we have 
\begin{align*}
&|F^\downarrow(\lambda,\bk;f;x)| \\
&\le \sum_{\substack{0\le l_1 \\ 0\le l_2\le k_1-k_2 \\ 0\le l_3\le k_2-k_3}}
\frac{\bigl|((-k_3,-k_2,-k_1))_{(l_1+l_2,l_1+l_3,l_2+l_3),d_2}\bigr|\sqrt{d_{\bk}^{\fp^+_2}d_{(|\bl|,0,0)}^{\fp^+_2}}}
{\bigl|\bigl(-\lambda+\frac{3}{2}d_2+1-(k_2+k_3,k_1+k_3,k_1+k_2)\bigr)_{(l_1,l_2,l_3),d_2}\bigr||\bl|!} \\*
&\hspace{185pt}\times\frac{(g^{-1}x_1|\overline{g^{-1}x_1})_{\fp^+}^{|\bl|}}{2^{|\bl|}}\Vert f(g(\cdot))\Vert_{L^2(\Sigma_2)}. 
\end{align*}
Then 
\[ d_{(|\bl|,0,0)}^{\fp^+_2}=\dim\cP_{(|\bl|,0,0)}(\fp^+_2)\le \dim\cP_{|\bl|}(\fp^+_2)=\binom{|\bl|+n_2-1}{n_2-1} \]
is of polynomial growth, where $n_2:=\dim\fp^+_2$. Also, since $l_2,l_3$ run over the finite sets, by Stirling's formula, 
there exist constants $C_{\lambda,\bk},C_{\lambda,\bk}'\in\BR_{>0}$ such that 
\begin{align*}
&\frac{\bigl|((-k_3,-k_2,-k_1))_{(l_1+l_2,l_1+l_3,l_2+l_3),d_2}\bigr|}
{\bigl|\bigl(-\lambda+\frac{3}{2}d_2+1-(k_2+k_3,k_1+k_3,k_1+k_2)\bigr)_{(l_1,l_2,l_3),d_2}\bigr||\bl|!} \\
&\le C_{\lambda,\bk}\Biggl|\frac{(l_1+l_2)^{-k_3+l_1+l_2-\frac{1}{2}}e^{-l_1-l_2}(l_1+l_3)^{-k_2-\frac{1}{2}d_2+l_1+l_3-\frac{1}{2}}e^{-l_1-l_3}}
{l_1^{-\lambda+\frac{3}{2}d_2-k_2-k_3+l_1+\frac{1}{2}}e^{-l_1}(l_1+l_2+l_3)^{l_1+l_2+l_3+\frac{1}{2}}e^{-l_1-l_2-l_3}}\Biggr| \\
&\le C_{\lambda,\bk}'|\bl|^{\Re\lambda-2d_2-2}
\end{align*}
holds. Hence the above converges if $(g^{-1}x_1|\overline{g^{-1}x_1})_{\fp^+}<2$. 
Especially, if $x=(x_1,x_2)=(x_1,ge)\in\Omega$ with $g\in L_2=K_1^\BC\cap L$, then by 
\begin{align*}
0&<\det_{\fn^+}(x)=\det_{\fn^+}(x_2)\det_{\fn^+}(e+g^{-1}x_1) \\
&=\det_{\fn^+_2}(x_2)\biggl(1-\frac{(g^{-1}x_1|g^{-1}x_1)_{\fn^+}}{2}\biggr)=\det_{\fn^+_2}(x_2)\biggl(1-\frac{(g^{-1}x_1|\overline{g^{-1}x_1})_{\fp^+}}{2}\biggr), 
\end{align*}
we have $(g^{-1}x_1|\overline{g^{-1}x_1})_{\fp^+}<2$. Hence the above converges if $x=(x_1,x_2)\in\Omega$. 
\end{proof}

This $F^\downarrow(\lambda,\bk;f;x)$ has the following property. Recall that $\underline{k}_3:=(k,k,k)$. 
\begin{proposition}\label{prop_shift_rank3}
For $\lambda\in\BC$, $\bk\in\BZ_{++}^3+\BC\underline{1}_3$, $f(x_2)\in\cP_\bk(\fp^+_2)$, $k\in\BZ_{\ge 0}$, we have 
\begin{align*}
&\det_{\fn^+_2}\biggl(\frac{\partial}{\partial x_2}\biggr)^kF^\downarrow(\lambda,\bk;f;x) \\
&=(-1)^{k}((-k_3,-k_2,-k_1))_{\underline{k}_3,d_2}F^\downarrow\Bigl(\lambda+2k,\bk-\underline{k}_3;f\det_{\fn^+_2}^{-k};x\Bigr). 
\end{align*}
\end{proposition}
\begin{proof}
Since $F_{\bk,\bl}^\downarrow[f](x_1,x_2)\in\cP_{2|\bk|}(\fp^+_1)\otimes\cP_{(k_1-l_2-l_3,k_2-l_1-l_3,k_3-l_1-l_2)}(\fp^+_2)$ holds, by Lemma \ref{lem_diff} we have 
\begin{align*}
&\det_{\fn^+_2}\biggl(\frac{\partial}{\partial x_2}\biggr)^k((-k_3,-k_2,-k_1))_{(l_1+l_2,l_1+l_3,l_2+l_3),d_2}F_{\bk,\bl}^\downarrow[f](x_1,x_2) \\
&=(-1)^{3k}((-k_3+l_1+l_2,-k_2+l_1+l_3,-k_1+l_2+l_3))_{(k,k,k),d_2} \\*
&\eqspace{}\times ((-k_3,-k_2,-k_1))_{(l_1+l_2,l_1+l_3,l_2+l_3),d_2}F_{\bk,\bl}^\downarrow[f](x_1,x_2)\det_{\fn^+_2}(x_2)^{-k} \\
&=(-1)^{k}((-k_3,-k_2,-k_1))_{(k+l_1+l_2,k+l_1+l_3,k+l_2+l_3),d_2}F_{\bk-\underline{k}_3,\bl}^\downarrow\Bigl[f\det_{\fn^+_2}^{-k}\Bigr](x_1,x_2) \\
&=(-1)^{k}((-k_3,-k_2,-k_1))_{(k,k,k),d_2} \\*
&\eqspace{}\times ((-k_3+k,-k_2+k,-k_1+k))_{(l_1+l_2,l_1+l_3,l_2+l_3),d_2}F_{\bk-\underline{k}_3,\bl}^\downarrow\Bigl[f\det_{\fn^+_2}^{-k}\Bigr](x_1,x_2), 
\end{align*}
and easily we have 
\begin{align*}
&\biggl(-\lambda+\frac{3}{2}d_2+1-(k_2+k_3,k_1+k_3,k_1+k_2)\biggr)_{\bl,d_2} \\
&=\biggl(-(\lambda+2k)+\frac{3}{2}d_2+1 \\*
&\hspace{22pt}-((k_2-k)+(k_3-k),(k_1-k)+(k_3-k),(k_1-k)+(k_2-k))\biggr)_{\bl,d_2}. 
\end{align*}
Hence the desired formula follows from the definition of $F^\downarrow(\lambda,\bk;f;x)$. 
\end{proof}

\subsection{Results on weighted Bergman inner products}

For $\lambda\in\BC$, $\bk\in\BZ_{++}^3$, let 
\begin{align}
&C^{d,d_2}(\lambda,\bk):=\frac{\prod_{1\le i< j\le 3}\left(\lambda-\frac{d}{4}(i+j-2)\right)_{k_i+k_j}}{\prod_{1\le i< j\le 4}\left(\lambda-\frac{d}{4}(i+j-3)\right)_{k_i+k_j}} \notag \\
&=\frac{\left(\lambda+k_1+k_3-\frac{d}{4}\right)_{k_2-k_3}\left(\lambda+\max\{k_1,k_2+k_3\}-\frac{d}{2}\right)_{\min\{k_3,k_1-k_2\}}\left(\lambda+k_2-\frac{3}{4}d\right)_{k_3}}
{(\lambda)_{k_1+k_2}\bigl(\lambda-\frac{d}{2}\bigr)_{\min\{k_1,k_2+k_3\}}(\lambda-d)_{k_3}} \notag \\
&=\frac{\bigl(\lambda-\frac{d_2}{2}+(k_1+k_3,\max\{k_1,k_2+k_3\},k_2)\bigr)_{(k_2-k_3,\min\{k_3,k_1-k_2\},k_3),d_2}}{(\lambda)_{(k_1+k_2,\min\{k_1,k_2+k_3\},k_3),d}}, \label{const_rank3}
\end{align}
where $k_4:=0$, $d_2=d/2$. Then the following holds. 

\begin{theorem}\label{thm_rank3}
Suppose $\bk\in\BZ_{++}^3+\BC\underline{1}_3=\BZ_{++}^3+\BC(1,1,1)$, $f(x_2)\in\cP_\bk(\fp^+_2)$. 
\begin{enumerate}
\item When $\bk\in\BZ_{++}^3$, for $\Re\lambda>2d+1$ we have 
\[ \Bigl\langle f(x_2),e^{(x|\overline{z})_{\fp^+}}\Bigr\rangle_{\lambda,x}=C^{d,d_2}(\lambda,\bk)F^\downarrow(\lambda,\bk;f;z). \]
\item When $\bk\in\BZ_{++}^3$, as a function of $\lambda$, 
\[ (\lambda)_{(k_1+k_2,\min\{k_1,k_2+k_3\},k_3),d}\Bigl\langle f(x_2),e^{(x|\overline{z})_{\fp^+}}\Bigr\rangle_{\lambda,x} \]
is holomorphically continued to the entire $\BC$, and give non-zero polynomials in $z$ for all $\lambda\in\BC$ if $f(x_2)\ne 0$. 
\item For $\bk\in\BZ_{++}^3+\BC\underline{1}_3$, $\lambda\in\BC$, we have 
\begin{align*}
&F^\downarrow(\lambda,\bk;f;z) \\
&=\det_{\fn^+}(z)^{-\lambda+d+1}F^\downarrow\Bigl(-\lambda+2(d+1),\bk+\underline{\lambda-d-1}_3;f\det_{\fn^+_2}^{\lambda-d-1};z\Bigr). 
\end{align*}
Especially, if $\bk\in\BZ_{++}^3$ and $\lambda=d+1-a$ with $a=1,2,\ldots,k_3$, then $F^\downarrow(\lambda,\bk;f;z)$ is factorized as 
\[ F^\downarrow(d+1-a,\bk;f;z)=\det_{\fn^+}(z)^{a}F^\downarrow\Bigl(d+1+a,\bk-\underline{a}_3;f\det_{\fn^+_2}^{-a};z\Bigr). \]
\end{enumerate}
\end{theorem}

Theorem \ref{thm_rank3}\,(1), (3) solves \cite[Conjecture 7.2]{N3}, and coincides with \cite[Theorem 6.3\,(1)]{N2} for $k_2=k_3$ case. 
Also, in (2) the holomorphy is already proved in \cite[Theorem 6.2]{N3}, and the non-vanishing is new except for $k_1=k_2=k_3$ or $k_3=0$ (\cite[Corollary 6.6\,(2)]{N2}). 

\begin{remark}
By \cite[Theorem 3.3\,(1)]{N3}, for $\bk\in\BZ_{++}^3$, $f(x_2)\in\cP_\bk(\fp^+_2)$, $\Re\lambda>2d+1$, we have 
\begin{align}
\Bigl\langle f(x_2),e^{(x|\overline{z})_{\fp^+}}\Bigr\rangle_{\lambda,x}
=\frac{1}{(\lambda)_{\underline{2k_3}_3,d}}\det_{\fn^+}(z)^{-\lambda+d+1}\det_{\fn^+_2}\biggl(\frac{\partial}{\partial z_2}\biggr)^{k_3}\det_{\fn^+}(z)^{\lambda+2k_3-d-1} \notag \\*
{}\times\Bigl\langle \det_{\fn^+_2}(x_2)^{-k_3}f(x_2),e^{(x|\overline{z})_{\fp^+}}\Bigr\rangle_{\lambda+2k_3,x}. \label{formula_diff_expr_rank3}
\end{align}
Combining this with Proposition \ref{prop_shift_rank3} and Theorem \ref{thm_rank3}\,(1), (3), we have 
\begin{align*}
&\Bigl\langle f(x_2),e^{(x|\overline{z})_{\fp^+}}\Bigr\rangle_{\lambda,x} \\
&=\frac{C^{d,d_2}(\lambda+2k_3,\bk-\underline{k_3}_3)}{(\lambda)_{\underline{2k_3}_3,d}}
\det_{\fn^+}(z)^{-\lambda+d+1}\det_{\fn^+_2}\biggl(\frac{\partial}{\partial z_2}\biggr)^{k_3}\det_{\fn^+}(z)^{\lambda+2k_3-d-1} \\*
&\eqspace{}\times F^\downarrow\Bigl(\lambda+2k_3,\bk-\underline{k_3}_3; f\det_{\fn^+_2}^{-k_3};z\Bigr) \\
&=\frac{\bigl(\lambda+k_1+k_3-\frac{d_2}{2}\bigr)_{k_2-k_3}}{(\lambda)_{(k_1+k_2,k_2+k_3,2k_3),d}}
\det_{\fn^+}(z)^{-\lambda+d+1}\det_{\fn^+_2}\biggl(\frac{\partial}{\partial z_2}\biggr)^{k_3} \\*
&\eqspace{}\times F^\downarrow\Bigl(-\lambda-2k_3+2(d+1),\bk+\underline{\lambda+k_3-d-1}_3; f\det_{\fn^+_2}^{\lambda+k_3-d-1};z\Bigr) \\
&=C^{d,d_2}(\lambda,\bk)\det_{\fn^+}(z)^{-\lambda+d+1}F^\downarrow\Bigl(-\lambda+2(d+1),\bk+\underline{\lambda-d-1}_3;f\det_{\fn^+_2}^{\lambda-d-1};z\Bigr) \\
&=C^{d,d_2}(\lambda,\bk)F^\downarrow(\lambda,\bk;f;z). 
\end{align*}
\end{remark}

First we prove (1). By \cite[Theorem 6.1]{N3}, we have 
\begin{align*}
&\frac{\bigl\langle f(x_2),e^{(x|\overline{z})_{\fp^+}}\bigr\rangle_{\lambda,x}\bigr|_{z_1=0}}{f(z_2)}
=\frac{\prod_{1\le i< j\le 3}\left(\lambda-\frac{d}{4}(i+j-2)\right)_{k_i+k_j}}{\prod_{1\le i< j\le 4}\left(\lambda-\frac{d}{4}(i+j-3)\right)_{k_i+k_j}}=C^{d,d_2}(\lambda,\bk). 
\end{align*}
Therefore to prove (1), it is enough to prove that $\bigl\langle f(x_2),e^{(x|\overline{z})_{\fp^+}}\bigr\rangle_{\lambda,x}$ is proportional to $F^\downarrow(\lambda,\bk;f;z)$, 
and by the F-method, enough to prove 
\begin{equation}\label{formula_Fmethod_rank3}
(f(x_2)\mapsto F^\downarrow(\lambda,\bk;f;z))\in\Hom_{K_1}\bigl(\cP_\bk(\fp^+_2),\Sol_{\cP(\fp^+)}((\cB_\lambda)_1)\bigr), 
\end{equation}
where $(\cB_\lambda)_1$ is the orthogonal projection of $\cB_\lambda=\cB_\lambda^{\fp^+}\colon\cP(\fp^+)\to\cP(\fp^+)\otimes\fp^+$ given in (\ref{formula_Bessel_diff}) 
(with the identification $\fp^-\simeq\fp^+=\fn^{+\BC}$) onto $\fp^+_1$. 
This follows from the next proposition. 
Recall that we write $z=(z_1,z_2)=(\zeta,Z)\in\fp^+_1\oplus\fp^+_2\simeq(\BF^{\prime\BC})^3\oplus\Herm(3,\BF')^\BC$, 
and use lowercase Greek letters and capital Latin letters to express elements in $(\BF^{\prime\BC})^3$ and in $\Herm(3,\BF')^\BC$ respectively. 
Also we abbreviate $\det_{\fn^+_2}(X)=:\det(X)$, and let $(X|Y):=\Re_{\BF'}\tr(XY)$, $(\xi|\eta):=\Re_{\BF'}(\xi{}^t\hspace{-1pt}\hat{\eta})$, 
so that $((\xi,X)|(\eta,Y))_{\fn^+}=2(\xi|\eta)+(X|Y)$ holds. 
Especially we have $-((z_1)^\sharp|z_2)_{\fn^+_2}=\zeta Z{}^t\hspace{-1pt}\hat{\zeta}$. Let $\Gamma_3^{d_2}(\lambda+\bs)$ be as in $(\ref{Gamma})$. 
\begin{proposition}
Let $\lambda\in\BC$, $\bk\in\BZ_{++}^3+\BC\underline{1}_3$, $f(x_2)\in\cP_\bk(\fp^+_2)$. 
\begin{enumerate}
\item If $\Re k_1<-d_2$, $\Re\bigl(\lambda+k_1+k_2\bigr)<-\frac{1}{2}d_2$, then for $z=(\zeta,Z)\in\Omega$, $A\in\Omega_2$, we have 
\begin{align*}
&F^\downarrow(\lambda,\bk;f;\zeta,Z) \\
&=\frac{\Gamma_3^{d_2}\bigl(-\lambda+\frac{3}{2}d_2+1-(k_2+k_3,k_1+k_3,k_1+k_2)\bigr)}{\Gamma_3^{d_2}((-k_3,-k_2,-k_1))}\int_{\Omega_2}e^{-(Y|Z)}\det(Y)^{2\lambda+|\bk|-2d_2-1} \\*
&\eqspace{}\times\biggl(\frac{1}{(2\pi\sqrt{-1})^{n_2}}\int_{A+\sqrt{-1}\fn^+_2}e^{(X|Y^\sharp)}e^{\zeta X^\itinv{}^t\hspace{-1pt}\hat{\zeta}}
\det(X)^{\lambda+|\bk|-\frac{3}{2}d_2-1}f(X^\itinv)\,dX\biggr)dY. 
\end{align*}
\item Suppose $\lambda\in\BC$ is not a pole of $F^\downarrow(\lambda,\bk;f;\zeta,Z)$. Then for $g\in \widetilde{K}_1^\BC$, we have 
\[ F^\downarrow(\lambda,\bk;f(g(\cdot));\zeta,Z)=F^\downarrow(\lambda,\bk;f;g(\zeta,Z)). \]
\item $F^\downarrow(\lambda,\bk;f;\zeta,Z)$ satisfies the differential equation 
\[ (\cB_\lambda)_1 F^\downarrow(\lambda,\bk;f;\zeta,Z)=0. \]
\end{enumerate}
\end{proposition}

\begin{proof}
(1) By the definition of $\tilde{F}_{\bk,\bl}[f]$ and Lemma \ref{lem_Laplace}, we have 
\begin{align*}
&\frac{\Gamma_3^{d_2}\bigl(-\lambda+\frac{3}{2}d_2+1-(k_2+k_3,k_1+k_3,k_1+k_2)\bigr)}{\Gamma_3^{d_2}((-k_3,-k_2,-k_1))}\int_{\Omega_2}e^{-(Y|Z)}\det(Y)^{2\lambda+|\bk|-2d_2-1} \\*
&\eqspace{}\times\biggl(\frac{1}{(2\pi\sqrt{-1})^{n_2}}\int_{A+\sqrt{-1}\fn^+_2}e^{(X|Y^\sharp)}e^{\zeta X^\itinv{}^t\hspace{-1pt}\hat{\zeta}}
\det(X)^{\lambda+|\bk|-\frac{3}{2}d_2-1}f(X^\itinv)\,dX\biggr)dY \\
&=\frac{\Gamma_3^{d_2}\bigl(-\lambda+\frac{3}{2}d_2+1-(k_2+k_3,k_1+k_3,k_1+k_2)\bigr)}{\Gamma_3^{d_2}((-k_3,-k_2,-k_1))}\int_{\Omega_2}e^{-(Y|Z)}\det(Y)^{2\lambda+|\bk|-2d_2-1} \\*
&\eqspace{}\times\biggl(\frac{1}{(2\pi\sqrt{-1})^{n_2}}\int_{A+\sqrt{-1}\fn^+_2}e^{(X|Y^\sharp)}\det(X)^{\lambda+|\bk|-\frac{3}{2}d_2-1}
\!\sum_{\substack{0\le l_1 \\ 0\le l_2\le k_1-k_2 \\ 0\le l_3\le k_2-k_3}}\!\tilde{F}_{\bk,\bl}[f](\zeta,X^\itinv)\,dX\biggr)dY \\
&=\frac{1}{\Gamma_3^{d_2}((-k_3,-k_2,-k_1))}\int_{\Omega_2}e^{-(Y|Z)} \det(Y)^{2\lambda+|\bk|-2d_2-1}\det(Y^\sharp)^{-\lambda-|\bk|+\frac{1}{2}d_2} \\*
&\eqspace{}\times \sum_\bl\frac{\Gamma_3^{d_2}\bigl(-\lambda+\frac{3}{2}d_2+1-(k_2+k_3,k_1+k_3,k_1+k_2)\bigr)}{\Gamma_3^{d_2}\bigl(-\lambda-|\bk|+\frac{3}{2}d_2+1+(k_1+l_1,k_2+l_2,k_3+l_3)\bigr)}
\tilde{F}_{\bk,\bl}[f](\zeta,Y^\sharp)\,dY \\
&=\frac{1}{\Gamma_3^{d_2}((-k_3,-k_2,-k_1))}\int_{\Omega_2} 
\sum_\bl\frac{e^{-(Y|Z)} \det(Y)^{-|\bk|-d_2-1}\tilde{F}_{\bk,\bl}[f]\bigl(\zeta,\det(Y)Y^\itinv\bigr)}{\bigl(-\lambda+\frac{3}{2}d_2+1-(k_2+k_3,k_1+k_3,k_1+k_2)\bigr)_{(l_1,l_2,l_3),d_2}}\,dY \\
&=\frac{1}{\Gamma_3^{d_2}((-k_3,-k_2,-k_1))}\int_{\Omega_2}
\sum_\bl\frac{e^{-(Y|Z)}\det(Y)^{|\bl|-d_2-1}\tilde{F}_{\bk,\bl}[f](\zeta,Y^\itinv)}{\bigl(-\lambda+\frac{3}{2}d_2+1-(k_2+k_3,k_1+k_3,k_1+k_2)\bigr)_{(l_1,l_2,l_3),d_2}}\,dY \\
&=\sum_\bl\frac{\Gamma_3^{d_2}(|\bl|+(-k_3-l_3,-k_2-l_2,-k_1-l_1))\det(Z)^{-|\bl|}\tilde{F}_{\bk,\bl}[f](\zeta,Z)}
{\Gamma_3^{d_2}((-k_3,-k_2,-k_1))\bigl(-\lambda+\frac{3}{2}d_2+1-(k_2+k_3,k_1+k_3,k_1+k_2)\bigr)_{(l_1,l_2,l_3),d_2}} \\
&=\sum_\bl\frac{((-k_3,-k_2,-k_1))_{(l_1+l_2,l_1+l_3,l_2+l_3),d_2}F_{\bk,\bl}^\downarrow[f](\zeta,Z)}{\bigl(-\lambda+\frac{3}{2}d_2+1-(k_2+k_3,k_1+k_3,k_1+k_2)\bigr)_{(l_1,l_2,l_3),d_2}} 
=F^\downarrow(\lambda,\bk;f;\zeta,Z). 
\end{align*}
Hence we get the desired formula. 

(2) Clear from the $\widetilde{K}_1^\BC$-equivariance of $f(X)\mapsto F_{\bk,\bl}^\downarrow[f](\zeta,Z)$. 

(3) We take bases $\{E_\alpha\}\subset\fn^+_2$, $\{\epsilon_\alpha\}\subset\fn^+_1$, 
let $\{E_\alpha^\vee\}\subset\fn^+_2$ be the dual basis of $\{E_\alpha\}$ with respect to $(\cdot|\cdot)_{\fn^+_2}$, 
and let $\{\epsilon_\alpha^\vee\}\subset\fn^+_1$ be the dual basis of $\{\epsilon_\alpha\}$ with respect to the bilinear form $(\xi|\eta)=\Re_{\BF'}(\xi{}^t\hspace{-1pt}\hat{\eta})$, 
so that $\{\epsilon_\alpha,E_\alpha\}$ and $\bigl\{\frac{1}{2}\epsilon_\alpha^\vee,E_\alpha^\vee\bigr\}$ are mutually dual with respect to $(\cdot|\cdot)_{\fn^+}$. 
Also let $\partial/\partial Z_\alpha$ and $\partial/\partial \zeta_\alpha$ denote the directional derivatives along $E_\alpha$ and $\epsilon_\alpha$ respectively. 
Then $(\cB_\lambda)_1$ is given by 
\begin{align*}
(\cB_\lambda)_1&=\frac{1}{8}\sum_{\alpha\beta}P(\epsilon_\alpha^\vee,\epsilon_\beta^\vee)\zeta\frac{\partial^2}{\partial \zeta_\alpha\partial \zeta_\beta}
+\frac{1}{2}\sum_{\alpha\beta}P(E_\alpha^\vee,E_\beta^\vee)\zeta\frac{\partial^2}{\partial Z_\alpha\partial Z_\beta} \\*
&\eqspace{}+\frac{1}{2}\sum_{\alpha\beta}P(\epsilon_\alpha^\vee,E_\beta^\vee)Z\frac{\partial^2}{\partial \zeta_\alpha \partial Z_\beta}
+\frac{\lambda}{2}\sum_\alpha \epsilon_\alpha^\vee\frac{\partial}{\partial \zeta_\alpha}. 
\end{align*}
We put $\det(X)^{\lambda+|\bk|-\frac{3}{2}d_2-1}f(X^\itinv)=:F(X)$. By (1) it is enough to prove, for $\Re k_1<-d_2$, $\Re\bigl(\lambda+k_1+k_2\bigr)<-\frac{1}{2}d_2$, 
\[ (\cB_\lambda)_1\int_{\Omega_2}e^{-(Y|Z)}\det(Y)^{2\lambda+|\bk|-2d_2-1}\biggl(\int_{A+\sqrt{-1}\fn^+_2}e^{(X|Y^\sharp)}e^{\zeta X^\itinv{}^t\hspace{-1pt}\hat{\zeta}}F(X)\,dX\biggr)dY=0. \]
First we consider $(\cB_\lambda)_1e^{-(Y|Z)}e^{\zeta X^\itinv{}^t\hspace{-1pt}\hat{\zeta}}$. 
In the following, we use the notation $\bigl(\bigl[\frac{\partial}{\partial w}f\bigr]g\bigr)h:=\frac{\partial f}{\partial w}gh+f\frac{\partial g}{\partial w}h$. Then we have 
\begin{align*}
&(\cB_\lambda)_1e^{-(Y|Z)}e^{\zeta X^\itinv{}^t\hspace{-1pt}\hat{\zeta}} \\
&=\biggl[\frac{1}{8}\sum_{\alpha\beta}P(\epsilon_\alpha^\vee,\epsilon_\beta^\vee)\zeta\frac{\partial^2}{\partial \zeta_\alpha\partial \zeta_\beta}
+\frac{1}{2}\sum_{\alpha\beta}P(E_\alpha^\vee,E_\beta^\vee)\zeta\frac{\partial^2}{\partial Z_\alpha\partial Z_\beta} \\*
&\eqspace{}+\frac{1}{2}\sum_{\alpha\beta}P(\epsilon_\alpha^\vee,E_\beta^\vee)Z\frac{\partial^2}{\partial \zeta_\alpha \partial Z_\beta}
+\frac{\lambda}{2}\sum_\alpha \epsilon_\alpha^\vee\frac{\partial}{\partial \zeta_\alpha}\biggr]e^{-(Y|Z)}e^{\zeta X^\itinv{}^t\hspace{-1pt}\hat{\zeta}} \\
&=\biggl(\frac{1}{8}\sum_{\alpha\beta}P(\epsilon_\alpha^\vee,\epsilon_\beta^\vee)\zeta(4(\zeta X^\itinv|\epsilon_\alpha)(\zeta X^\itinv|\epsilon_\beta)+2(\epsilon_\alpha X^\itinv|\epsilon_\beta))
+\frac{1}{2}P(Y,Y)\zeta \\*
&\eqspace{}-\sum_{\alpha}P(\epsilon_\alpha^\vee,Y)Z(\zeta X^\itinv|\epsilon_\alpha)
+\lambda\sum_\alpha \epsilon_\alpha^\vee(\zeta X^\itinv|\epsilon_\alpha)\biggr)e^{-(Y|Z)}e^{\zeta X^\itinv{}^t\hspace{-1pt}\hat{\zeta}} \\
&=\biggl(\hspace{-1pt}P(\zeta X^\itinv)\zeta\hspace{-1pt}+\hspace{-1pt}\frac{1}{4}\sum_{\alpha}P(\epsilon_\alpha^\vee,\epsilon_\alpha X^\itinv)\zeta\hspace{-1pt}+\hspace{-1pt}P(Y)\zeta
\hspace{-1pt}-\hspace{-1pt}P(\zeta X^\itinv,Y)Z\hspace{-1pt}+\hspace{-1pt}\lambda\zeta X^\itinv\hspace{-1pt}\biggr)e^{-(Y|Z)}e^{\zeta X^\itinv{}^t\hspace{-1pt}\hat{\zeta}} \\
&=\biggl[P(\zeta X^\itinv)\zeta+\zeta Y^\sharp+\biggl(\lambda-\frac{d_2}{2}+1\biggr)\zeta X^\itinv
+\sum_\alpha P(\zeta X^\itinv,Y)E_\alpha^\vee\frac{\partial}{\partial Y_\alpha}\biggr]e^{-(Y|Z)}e^{\zeta X^\itinv{}^t\hspace{-1pt}\hat{\zeta}}, 
\end{align*}
where at the last equality we have used 
\[ \sum_{\alpha}P(\epsilon_\alpha^\vee,\epsilon_\alpha X^\itinv)\zeta
=\sum_{\alpha}(\epsilon_\alpha^\vee{}^t\hspace{-1pt}\hat{\zeta}\epsilon_\alpha X^\itinv+\epsilon_\alpha X^\itinv{}^t\hspace{-1pt}\hat{\zeta}\epsilon_\alpha^\vee)=2(2-d_2)\zeta X^\itinv, \]
which follows from $\sum_\alpha\epsilon_\alpha^\vee {}^t\hspace{-1pt}\hat{\xi}\epsilon_\alpha=\sum_\alpha\epsilon_\alpha {}^t\hspace{-1pt}\hat{\xi}\epsilon_\alpha^\vee=(2-d_2)\xi$. 
Thus we get 
\begin{align*}
&(\cB_\lambda)_1\int_{\Omega_2}e^{-(Y|Z)}\det(Y)^{2\lambda+|\bk|-2d_2-1}\biggl(\int_{A+\sqrt{-1}\fn^+_2}e^{(X|Y^\sharp)}e^{\zeta X^\itinv{}^t\hspace{-1pt}\hat{\zeta}}F(X)\,dX\biggr)dY \\
&=\int_{\Omega_2}\biggl(\int_{A+\sqrt{-1}\fn^+_2}\det(Y)^{2\lambda+|\bk|-2d_2-1}e^{(X|Y^\sharp)}e^{\zeta X^\itinv{}^t\hspace{-1pt}\hat{\zeta}}F(X)\biggl(\biggl[P(\zeta X^\itinv)\zeta+\zeta Y^\sharp \\*
&\eqspace{}+\biggl(\lambda-\frac{d_2}{2}+1\biggr)\zeta X^\itinv+\sum_\alpha P(\zeta X^\itinv,Y)E_\alpha^\vee\frac{\partial}{\partial Y_\alpha}\biggr]e^{-(Y|Z)}\biggr)dX\biggr)dY \\
&=\int_{\Omega_2}\biggl(\int_{A+\sqrt{-1}\fn^+_2}e^{-(Y|Z)}e^{\zeta X^\itinv{}^t\hspace{-1pt}\hat{\zeta}}F(X)
\biggl(\biggl[P(\zeta X^\itinv)\zeta+\biggl(\lambda-\frac{d_2}{2}+1\biggr)\zeta X^\itinv \\*
&\eqspace{}+\zeta Y^\sharp-\sum_\alpha \frac{\partial}{\partial Y_\alpha}P(\zeta X^\itinv,Y)E_\alpha^\vee\biggr]\det(Y)^{2\lambda+|\bk|-2d_2-1}e^{(X|Y^\sharp)}\biggr)dX\biggr)dY. 
\end{align*}
Next we consider 
\begin{align*}
&\sum_\alpha \frac{\partial}{\partial Y_\alpha}P(\zeta X^\itinv,Y)E_\alpha^\vee\det(Y)^{2\lambda+|\bk|-2d_2-1}e^{(X|Y^\sharp)} \\
&=\biggl(\sum_\alpha P(\zeta X^\itinv,E_\alpha)E_\alpha^\vee
+(2\lambda+|\bk|-2d_2-1)\sum_\alpha P(\zeta X^\itinv,Y)E_\alpha^\vee(Y^\itinv|E_\alpha) \\*
&\eqspace{}+\sum_\alpha P(\zeta X^\itinv,Y)E_\alpha^\vee(X|Y\times E_\alpha)\biggr)\det(Y)^{2\lambda+|\bk|-2d_2-1}e^{(X|Y^\sharp)}. 
\end{align*}
Then, for the 1st term, we have 
\begin{align*}
&\sum_\alpha P(\zeta X^\itinv,E_\alpha)E_\alpha^\vee=\sum_\alpha((E_\alpha|E_\alpha^\vee)\zeta X^\itinv-\zeta X^\itinv E_\alpha E_\alpha^\vee) \\
&=3(1+d_2)\zeta X^\itinv-(1+d_2)\zeta X^\itinv=2(1+d_2)\zeta X^\itinv, 
\end{align*}
for the 2nd term, we have 
\begin{align*}
&\sum_\alpha P(\zeta X^\itinv,Y)E_\alpha^\vee(Y^\itinv|E_\alpha)=P(\zeta X^\itinv,Y)Y^\itinv \\
&=(Y|Y^\itinv)\zeta X^\itinv-\zeta X^\itinv YY^\itinv=3\zeta X^\itinv-\zeta X^\itinv=2\zeta X^\itinv, 
\end{align*}
and for the 3rd term, we have 
\begin{align*}
&\sum_\alpha P(\zeta X^\itinv,Y)E_\alpha^\vee(X|Y\times E_\alpha)=\sum_\alpha P(\zeta X^\itinv,Y)E_\alpha^\vee(X\times Y|E_\alpha) \\
&=P(\zeta X^\itinv,Y)(X\times Y)=(Y|X\times Y)\zeta X^\itinv-\zeta X^\itinv Y(X\times Y) \\
&=\zeta X^{-1}(2(X|Y^\sharp)I-Y(X\times Y))=\zeta X^{-1}((X|Y^\sharp)I+XY^\sharp) \\
&=(X|Y^\sharp)\zeta X^\itinv+\zeta Y^\sharp, 
\end{align*}
where we have used (\ref{formula_Freudenthal}) at the 5th equality. 
Hence we have 
\begin{align*}
&\biggl[P(\zeta X^\itinv)\zeta+\zeta Y^\sharp+\biggl(\lambda-\frac{d_2}{2}+1\biggr)\zeta X^\itinv-\sum_\alpha \frac{\partial}{\partial Y_\alpha}P(\zeta X^\itinv,Y)E_\alpha^\vee\biggr] \\*
&\hspace{237pt}\times\det(Y)^{2\lambda+|\bk|-2d_2-1}e^{(X|Y^\sharp)} \\
&=\biggl(P(\zeta X^\itinv)\zeta+\zeta Y^\sharp+\biggl(\lambda-\frac{d_2}{2}+1\biggr)\zeta X^\itinv-(X|Y^\sharp)\zeta X^\itinv-\zeta Y^\sharp \\*
&\eqspace{}-2(1+d_2)\zeta X^\itinv-2(2\lambda+|\bk|-2d_2-1)\zeta X^\itinv\biggr)
\det(Y)^{2\lambda+|\bk|-2d_2-1}e^{(X|Y^\sharp)} \\
&=\biggl[P(\zeta X^\itinv)\zeta+\biggl(-3\lambda-2|\bk|+\frac{3}{2}d_2+1\biggr)\zeta X^\itinv-\zeta X^\itinv\sum_\alpha(X|E_\alpha^\vee)\frac{\partial}{\partial X_\alpha}\biggr] \\*
&\hspace{237pt}\times\det(Y)^{2\lambda+|\bk|-2d_2-1}e^{(X|Y^\sharp)}, 
\end{align*}
and we get 
\begin{align*}
&(\cB_\lambda)_1\int_{\Omega_2}e^{-(Y|Z)}\det(Y)^{2\lambda+|\bk|-2d_2-1}\biggl(\int_{A+\sqrt{-1}\fn^+_2}e^{(X|Y^\sharp)}e^{\zeta X^\itinv{}^t\hspace{-1pt}\hat{\zeta}}F(X)\,dX\biggr)dY \\
&=\int_{\Omega_2}e^{-(Y|Z)}\biggl(\int_{A+\sqrt{-1}\fn^+_2}e^{\zeta X^\itinv{}^t\hspace{-1pt}\hat{\zeta}}F(X)
\biggl(\biggl[P(\zeta X^\itinv)\zeta+\biggl(\lambda-\frac{d_2}{2}+1\biggr)\zeta X^\itinv \\*
&\eqspace{}+\zeta Y^\sharp-\sum_\alpha \frac{\partial}{\partial Y_\alpha}P(\zeta X^\itinv,Y)E_\alpha^\vee\biggr]\det(Y)^{2\lambda+|\bk|-2d_2-1}e^{(X|Y^\sharp)}\biggr)dX\biggr)dY \\
&=\int_{\Omega_2}e^{-(Y|Z)}\biggl(\int_{A+\sqrt{-1}\fn^+_2}e^{\zeta X^\itinv{}^t\hspace{-1pt}\hat{\zeta}}F(X)
\biggl(\biggl[\biggl(-3\lambda-2|\bk|+\frac{3}{2}d_2+1\biggr)\zeta X^\itinv \\*
&\eqspace{}+P(\zeta X^\itinv)\zeta-\zeta X^\itinv\sum_\alpha(X|E_\alpha^\vee)\frac{\partial}{\partial X_\alpha}\biggr]
\det(Y)^{2\lambda+|\bk|-2d_2-1}e^{(X|Y^\sharp)}\biggr)dX\biggr)dY \\
&=\int_{\Omega_2}e^{-(Y|Z)}\det(Y)^{2\lambda+|\bk|-2d_2-1}\biggl(\int_{A+\sqrt{-1}\fn^+_2}
\biggl(\biggl[\biggl(-3\lambda-2|\bk|+\frac{3}{2}d_2+1\biggr)\zeta X^\itinv \\*
&\eqspace{}+P(\zeta X^\itinv)\zeta+\sum_\alpha\frac{\partial}{\partial X_\alpha}\zeta X^\itinv(X|E_\alpha^\vee)\biggr]
e^{\zeta X^\itinv{}^t\hspace{-1pt}\hat{\zeta}}F(X)\biggr)e^{(X|Y^\sharp)}\,dX\biggr)dY. 
\end{align*}
Then since we have 
\begin{align*}
&\sum_\alpha\frac{\partial}{\partial X_\alpha}\zeta X^\itinv(X|E_\alpha^\vee)e^{\zeta X^\itinv{}^t\hspace{-1pt}\hat{\zeta}}F(X) \\
&=-\sum_\alpha \zeta X^\itinv E_\alpha X^\itinv(X|E_\alpha^\vee)e^{\zeta X^\itinv{}^t\hspace{-1pt}\hat{\zeta}}F(X)
+\sum_\alpha\zeta X^\itinv(E_\alpha|E_\alpha^\vee)e^{\zeta X^\itinv{}^t\hspace{-1pt}\hat{\zeta}}F(X) \\*
&\eqspace{}-\sum_\alpha\zeta X^\itinv(X|E_\alpha^\vee)(\zeta X^\itinv E_\alpha X^\itinv {}^t\hspace{-1pt}\hat{\zeta})e^{\zeta X^\itinv{}^t\hspace{-1pt}\hat{\zeta}}F(X)
+\sum_\alpha\zeta X^\itinv(X|E_\alpha^\vee)e^{\zeta X^\itinv{}^t\hspace{-1pt}\hat{\zeta}}\frac{\partial F}{\partial X_\alpha}(X) \\
&=-\zeta X^\itinv XX^\itinv e^{\zeta X^\itinv{}^t\hspace{-1pt}\hat{\zeta}}F(X)+3(1+d_2)\zeta X^\itinv e^{\zeta X^\itinv{}^t\hspace{-1pt}\hat{\zeta}}F(X) \\*
&\eqspace{}-(\zeta X^\itinv XX^\itinv {}^t\hspace{-1pt}\hat{\zeta})\zeta X^\itinv e^{\zeta X^\itinv{}^t\hspace{-1pt}\hat{\zeta}}F(X)
+\zeta X^\itinv e^{\zeta X^\itinv{}^t\hspace{-1pt}\hat{\zeta}}\sum_\alpha(X|E_\alpha^\vee)\frac{\partial F}{\partial X_\alpha}(X) \\
&=e^{\zeta X^\itinv{}^t\hspace{-1pt}\hat{\zeta}}\biggl[(2+3d_2)\zeta X^\itinv-P(\zeta X^\itinv)\zeta 
+\zeta X^\itinv\sum_\alpha(X|E_\alpha^\vee)\frac{\partial}{\partial X_\alpha}\biggr]F(X), 
\end{align*}
we get 
\begin{align*}
&(\cB_\lambda)_1\int_{\Omega_2}e^{-(Y|Z)}\det(Y)^{2\lambda+|\bk|-2d_2-1}\biggl(\int_{A+\sqrt{-1}\fn^+_2}e^{(X|Y^\sharp)}e^{\zeta X^\itinv{}^t\hspace{-1pt}\hat{\zeta}}F(X)\,dX\biggr)dY \\
&=\int_{\Omega_2}e^{-(Y|Z)}\det(Y)^{2\lambda+|\bk|-2d_2-1}\biggl(\int_{A+\sqrt{-1}\fn^+_2}
\biggl(\biggl[\biggl(-3\lambda-2|\bk|+\frac{3}{2}d_2+1\biggr)\zeta X^\itinv \\*
&\eqspace{}+P(\zeta X^\itinv)\zeta+\sum_\alpha\frac{\partial}{\partial X_\alpha}\zeta X^\itinv(X|E_\alpha^\vee)\biggr]
e^{\zeta X^\itinv{}^t\hspace{-1pt}\hat{\zeta}}F(X)\biggr)e^{(X|Y^\sharp)}\,dX\biggr)dY \\
&=\int_{\Omega_2}e^{-(Y|Z)}\det(Y)^{2\lambda+|\bk|-2d_2-1}\biggl(\int_{A+\sqrt{-1}\fn^+_2}e^{(X|Y^\sharp)}e^{\zeta X^\itinv{}^t\hspace{-1pt}\hat{\zeta}} \\*
&\eqspace{}\times\biggl[\biggl(-3\lambda-2|\bk|+\frac{9}{2}d_2+3\biggr)\zeta X^\itinv +\zeta X^\itinv\sum_\alpha(X|E_\alpha^\vee)\frac{\partial}{\partial X_\alpha}\biggr]F(X)\,dX\biggr)=0, 
\end{align*}
where the last equality holds since $F(X)=\det(X)^{\lambda+|\bk|-\frac{3}{2}d_2-1}f(X^\itinv)$ is homogeneous of degree $3\lambda+2|\bk|-\frac{9}{2}d_2-3$. 
Hence we get the desired formula. 
\end{proof}

Especially, when $\bk\in\BZ_{++}^3$ we have (\ref{formula_Fmethod_rank3}), and this proves Theorem \ref{thm_rank3}\,(1). 
\begin{proof}[Proof of Theorem \ref{thm_rank3}\,(2)]
By Theorem \ref{thm_rank3}\,(1) we have 
\begin{align*}
&(\lambda)_{(k_1+k_2,\min\{k_1,k_2+k_3\},k_3),d}\Bigl\langle f(x_2),e^{(x|\overline{z})_{\fp^+}}\Bigr\rangle_{\lambda,x} \\
&=\biggl(\lambda-\frac{d_2}{2}+(k_1+k_3,\max\{k_1,k_2+k_3\},k_2)\biggr)_{(k_2-k_3,\min\{k_3,k_1-k_2\},k_3),d_2}F^\downarrow(\lambda,\bk;f;z) \\
&=\sum_{\substack{0\le l_1 \\ 0\le l_2\le k_1-k_2 \\ 0\le l_3\le k_2-k_3 \\ l_1+l_2\le k_3}}
\biggl(\lambda-\frac{d_2}{2}+(k_1+k_3,\max\{k_1,k_2+k_3\},k_2)\biggr)_{(k_2-k_3-l_3,\min\{k_3,k_1-k_2\}-l_2,k_3-l_1),d_2} \\*
&\hspace{145pt}\times (-1)^{|\bl|}((-k_3,-k_2,-k_1))_{(l_1+l_2,l_1+l_3,l_2+l_3),d_2} F_{\bk,\bl}^\downarrow[f](z_1,z_2), 
\end{align*}
and this is analytically continued for all $\lambda\in\BC$. 
Moreover, all $F_{\bk,\bl}^\downarrow[f](z_1,z_2)$ are non-zero by Proposition \ref{prop_estimate_Fkl_rank3}\,(2) and linearly independent, 
and the coefficients of $F_{\bk,(k_3,0,k_2-k_3)}^\downarrow[f](z_1,\allowbreak z_2)$ and $F_{\bk,(\max\{0,k_3-k_1+k_2\},\min\{k_3,k_1-k_2\},k_2-k_3)}^\downarrow[f]\allowbreak(z_1,z_2)$, 
\begin{align*}
&\biggl(\lambda-\frac{d_2}{2}+(k_1+k_3,\max\{k_1,k_2+k_3\},k_2)\biggr)_{(0,\min\{k_3,k_1-k_2\},0),d_2} \\
&=(\lambda-d_2+\max\{k_1,k_2+k_3\})_{\min\{k_3,k_1-k_2\}}, \\
&\biggl(\lambda-\frac{d_2}{2}+(k_1+k_3,\max\{k_1,k_2+k_3\},k_2)\biggr)_{(0,0,\min\{k_3,k_1-k_2\}),d_2} \\
&=\biggl(\lambda-\frac{3}{2}d_2+k_2\biggr)_{\min\{k_3,k_1-k_2\}}
\end{align*}
do not vanish simultaneously. Hence $(\lambda)_{(k_1+k_2,\min\{k_1,k_2+k_3\},k_3),d}\bigl\langle f(x_2),e^{(x|\overline{z})_{\fp^+}}\bigr\rangle_{\lambda,x}$ is non-zero. 
\end{proof}

\begin{proof}[Proof of Theorem \ref{thm_rank3}\,(3)]
By combining Theorem \ref{thm_rank3}\,(1) and (\ref{formula_diff_expr_rank3}), for $\Re\lambda\allowbreak>2d+1$, $\bk\in\BZ_{++}^3$, we have 
\begin{align*}
C^{d,d_2}(\lambda,\bk)F^\downarrow(\lambda,\bk;f;z)
&=\frac{C^{d,d_2}(\lambda+2k_3,\bk-\underline{k_3}_3)}{(\lambda)_{\underline{2k_3}_3,d}}
\det_{\fn^+}(z)^{-\lambda+d+1}\det_{\fn^+_2}\biggl(\frac{\partial}{\partial z_2}\biggr)^{k_3} \\*
&\eqspace{}\times \det_{\fn^+}(z)^{\lambda+2k_3-d-1}F^\downarrow\Bigl(\lambda+2k_3,\bk-\underline{k_3}_3; f\det_{\fn^+_2}^{-k_3};z\Bigr), 
\end{align*}
that is, 
\begin{align*}
F^\downarrow(\lambda,\bk;f;z)&=\frac{1}{(\lambda-d_2+\bk)_{\underline{k_3}_3,d_2}}
\det_{\fn^+}(z)^{-\lambda+d+1}\det_{\fn^+_2}\biggl(\frac{\partial}{\partial z_2}\biggr)^{k_3} \\*
&\eqspace{}\times \det_{\fn^+}(z)^{\lambda+2k_3-d-1}F^\downarrow\Bigl(\lambda+2k_3,\bk-\underline{k_3}_3; f\det_{\fn^+_2}^{-k_3};z\Bigr). 
\end{align*}
Then this is meromorphically continued for all $\lambda\in\BC$, and if $\lambda=d+1-a$, $a=1,2,\ldots,k_3$, then the right hand side of 
\begin{align*}
\det_{\fn^+}(z)^{-a}F^\downarrow(d+1-a,\bk;f;z)&=\frac{1}{(d_2+1-a+\bk)_{\underline{k_3}_3,d_2}}
\det_{\fn^+_2}\biggl(\frac{\partial}{\partial z_2}\biggr)^{k_3}\det_{\fn^+}(z)^{2k_3-a} \\*
&\eqspace{}\times F^\downarrow\Bigl(d+1-a+2k_3,\bk-\underline{k_3}_3; f\det_{\fn^+_2}^{-k_3};z\Bigr) 
\end{align*}
becomes a polynomial. Hence $F^\downarrow(d+1-a,\bk;f;z)$ is divisible by $\det_{\fn^+}(z)^a$. Also, by $\cB_{\frac{n}{r}-a}\det_{\fn^+}(z)^a=\det_{\fn^+}(z)^a\cB_{\frac{n}{r}+a}$ 
(see the proof of \cite[Proposition XV.2.4]{FK}), we have 
\begin{align*}
&(\cB_{d+1+a})_1\det_{\fn^+}(z)^{-a}F^\downarrow(d+1-a,\bk;f;z) \\
&=\det_{\fn^+}(z)^{-a}(\cB_{d+1-a})_1F^\downarrow(d+1-a,\bk;f;z)=0, 
\end{align*}
and hence 
\begin{align*}
&\Bigl( f\det_{\fn^+_2}^{-a}\mapsto \det_{\fn^+}(z)^{-a}F^\downarrow(d+1-a,\bk;f;z)\Bigr) \\
&\in\Hom_{K_1^\BC}\bigl(\cP_{\bk-\underline{a}_3}(\fp^+_2),\Sol_{\cP(\fp^+)}((\cB_{d+1+a})_1)\bigr) \\
&\eqspace{}\simeq \BC\Bigl( f\det_{\fn^+_2}^{-a}\mapsto F^\downarrow\Bigl(d+1+a,\bk-\underline{a}_3;f\det_{\fn^+_2}^{-a};z\Bigr)\Bigr)
\end{align*}
holds. Then comparing the values at $z_1=0$, for $a=1,2,\ldots,k_3$ we get 
\begin{align*}
F^\downarrow(d+1-a,\bk;f;z)
=\det_{\fn^+}(z)^a F^\downarrow\Bigl(d+1+a,\bk-\underline{a}_3;f\det_{\fn^+_2}^{-a};z\Bigr) 
\end{align*}
(see \cite[Theorem 6.6]{N3} for more general case). 
We put $(k_1',k_2'):=(k_1-k_3,\allowbreak k_2-k_3)\in\BZ_{++}^2$, let $f'(x_2):=\det_{\fn^+_2}(x_2)^{-k_3}f(x_2)\in\cP_{(k_1',k_2',0)}(\fp^+_2)$, and rewrite the above as 
\begin{align}
&F^\downarrow\Bigl(d+1-a,(k_1',k_2',0)+\underline{k_3}_3;f'\det_{\fn^+_2}^{k_3};z\Bigr) \notag \\
&=\det_{\fn^+}(z)^a F^\downarrow\Bigl(d+1+a,(k_1',k_2',0)+\underline{k_3-a}_3;f'\det_{\fn^+_2}^{k_3-a};z\Bigr) \label{formula_factorize_rank3} \\
&\hspace{180pt} ((a,k_3)\in(\BZ_{\ge 0})^2,\; a\le k_3). \notag
\end{align}
Now we want to prove this for general $(a,k_3)\in\BC^2$. 
We recall that the left hand side of (\ref{formula_factorize_rank3}) is given by 
\begin{align}
&F^\downarrow\Bigl(d+1-a,(k_1',k_2',0)+\underline{k_3}_3;f'\det_{\fn^+_2}^{k_3};z\Bigr) \notag \\
&=\sum_{\substack{0\le l_1 \\ 0\le l_2\le k_1'-k_2' \\ 0\le l_3\le k_2'}}
\frac{(-k_3-(0,k_2',k_1'))_{(l_1+l_2,l_1+l_3,l_2+l_3),d_2}}{\bigl(a-2k_3-\frac{d_2}{2}-(k_2',k_1',k_1'+k_2')\bigr)_{(l_1,l_2,l_3),d_2}}
\det_{\fn^+_2}(z_2)^{k_3}F_{(k_1',k_2',0),\bl}^\downarrow[f'](z_1,z_2). \label{formula_factorize_rank3_left}
\end{align}
Similarly, the right hand side is computed as 
\begin{align*}
&\det_{\fn^+}(z)^a F^\downarrow\Bigl(d+1+a,(k_1',k_2',0)+\underline{k_3-a}_3;f'\det_{\fn^+_2}^{k_3-a};z\Bigr) \\
&=\sum_{\substack{0\le l_1' \\ 0\le l_2'\le k_1'-k_2' \\ 0\le l_3'\le k_2'}}
\frac{(a-k_3-(0,k_2',k_1'))_{(l_1'+l_2',l_1'+l_3',l_2'+l_3'),d_2}}{\bigl(a-2k_3-\frac{d_2}{2}-(k_2',k_1',k_1'+k_2')\bigr)_{(l_1',l_2',l_3'),d_2}} \\*
&\eqspace{}\times \det_{\fn^+_2}(z_2)^{k_3-a}F_{(k_1',k_2',0),\bl'}^\downarrow[f'](z_1,z_2)\det_{\fn^+_2}(z_2)^a\biggl(1+\frac{((z_1)^\sharp|z_2)}{\det_{\fn^+_2}(z_2)}\biggr)^a \\
&=\sum_{\substack{0\le l_1' \\ 0\le l_2'\le k_1'-k_2' \\ 0\le l_3'\le k_2'}}
\frac{(a-k_3-(0,k_2',k_1'))_{(l_1'+l_2',l_1'+l_3',l_2'+l_3'),d_2}}{\bigl(a-2k_3-\frac{d_2}{2}-(k_2',k_1',k_1'+k_2')\bigr)_{(l_1',l_2',l_3'),d_2}} \\*
&\eqspace{}\times \det_{\fn^+_2}(z_2)^{k_3}F_{(k_1',k_2',0),\bl'}^\downarrow[f'](z_1,z_2)\sum_{m=0}^\infty \frac{(-a)_m}{m!}\biggl(-\frac{((z_1)^\sharp|z_2)}{\det_{\fn^+_2}(z_2)}\biggr)^m. 
\end{align*}
Then we have 
\begin{align*}
&\biggl( f'\mapsto \frac{1}{m!}F_{(k_1',k_2',0),\bl'}^\downarrow[f'](z_1,z_2)\biggl(-\frac{((z_1)^\sharp|z_2)}{\det_{\fn^+_2}(z_2)}\biggr)^m \biggr) \\
&\in\Hom_{K_1^\BC}\bigl(\cP_{(k_1',k_2',0)}(\fp^+_2),\cP_{2|\bl'|+2m}(\fp^+_1)\otimes\cP_{(k_1'-l_2'-l_3',k_2'-l_1'-l_3',-l_2'-l_3')}(\fp^+_2)\otimes\cP_{(0,-m,-m)}(\fp^+_2)\bigr) \\
&\simeq\bigoplus_{\substack{\bl''\in(\BZ_{\ge 0})^3,\\ |\bl''|=m}} 
\Hom_{K_1^\BC}\bigl(\cP_{(k_1',k_2',0)}(\fp^+_2),\cP_{2|\bl'|+2m}(\fp^+_1)
\otimes\cP_{\substack{(k_1'-l_2'-l_3'-l_2''-l_3'',k_2'-l_1'-l_3'-l_1''-l_3'',\\ \hspace{77pt}-l_2'-l_3'-l_2''-l_3'')}}(\fp^+_2)\bigr) \\
&=\bigoplus_{\substack{\bl\in(\BZ_{\ge 0})^3,\,|\bl|=|\bl'|+m \\ l_j\ge l_j'}} 
\Hom_{K_1^\BC}\bigl(\cP_{(k_1',k_2',0)}(\fp^+_2),\cP_{2|\bl|}(\fp^+_1)\otimes\cP_{(k_1'-l_2-l_3,k_2'-l_1-l_3,-l_2-l_3)}(\fp^+_2)\bigr), 
\end{align*}
and by the Pieri rule, 
\begin{align*}
&\Hom_{K_1^\BC}\bigl(\cP_{(k_1',k_2',0)}(\fp^+_2),\cP_{2|\bl|}(\fp^+_1)\otimes\cP_{(k_1'-l_2-l_3,k_2'-l_1-l_3,-l_2-l_3)}(\fp^+_2)\bigr) \\
&\simeq \Hom_{K_1^\BC}\bigl(\cP_{(k_1',k_2',0)}(\fp^+_2)\otimes\cP_{2|\bl|}(\fp^+_1)^\vee,\cP_{(k_1'-l_2-l_3,k_2'-l_1-l_3,-l_2-l_3)}(\fp^+_2)\bigr)\ne\{0\}
\end{align*}
holds only if $0\le l_2\le k_1'-k_2'$, $0\le l_3\le k_2'$. Therefore, for $|\bl|=|\bl'|+m$ we set 
\begin{align*}
&F_{(k_1',k_2',0),\bl',\bl}[f'](z_1,z_2) \\ &:=\Proj_{(k_1'-l_2-l_3,k_2'-l_1-l_3,-l_1-l_2),z_2}^{\fp^+_2}
\biggl(\frac{1}{m!}F_{(k_1',k_2',0),\bl'}^\downarrow[f'](z_1,z_2)\biggl(-\frac{((z_1)^\sharp|z_2)}{\det_{\fn^+_2}(z_2)}\biggr)^m\biggm) \\
&\hspace{5pt}\in\cP_{2|\bl|}(\fp^+_1)\otimes\cP_{(k_1'-l_2-l_3,k_2'-l_1-l_3,-l_1-l_2)}(\fp^+_2), 
\end{align*}
so that 
\begin{gather*}
\frac{1}{m!}F_{(k_1',k_2',0),\bl'}^\downarrow[f'](z_1,z_2)\biggl(-\frac{((z_1)^\sharp|z_2)}{\det_{\fn^+_2}(z_2)}\biggr)^m
=\sum_{\substack{|\bl|=|\bl'|+m \\ l_1'\le l_1 \\ l_2'\le l_2\le k_1'-k_2' \\ l_3'\le l_3\le k_2'}}F_{(k_1',k_2',0),\bl',\bl}[f'](z_1,z_2). 
\end{gather*}
Then the right hand side of (\ref{formula_factorize_rank3}) becomes 
\begin{align}
&\det_{\fn^+}(z)^a F^\downarrow\Bigl(d+1+a,(k_1',k_2',0)+\underline{k_3-a}_3;f'\det_{\fn^+_2}^{k_3-a};z\Bigr) \notag \\
&=\sum_{\substack{0\le l_1 \\ 0\le l_2\le k_1'-k_2' \\ 0\le l_3\le k_2'}}\sum_{\substack{0\le l_1'\le l_1 \\ 0\le l_2'\le l_2 \\ 0\le l_3'\le l_3}}
\frac{(a-k_3-(0,k_2',k_1'))_{(l_1'+l_2',l_1'+l_3',l_2'+l_3'),d_2}(-a)_{|\bl|-|\bl'|}}{\bigl(a-2k_3-\frac{d_2}{2}-(k_2',k_1',k_1'+k_2')\bigr)_{(l_1',l_2',l_3'),d_2}} \notag \\*
&\eqspace{}\times \det_{\fn^+_2}(z_2)^{k_3}F_{(k_1',k_2',0),\bl',\bl}[f'](z_1,z_2). \label{formula_factorize_rank3_right}
\end{align}
If $a,k_3\in\BZ_{\ge 0}$, $a\le k_3$, then in (\ref{formula_factorize_rank3_left}) and (\ref{formula_factorize_rank3_right}), only $l_1+l_2\le k_3$ terms remain, 
and all denominators in the remaining terms are non-zero. Hence for $l_1+l_2\le k_3$, by comparing both formulas, we get 
\begin{align*}
&\frac{(-k_3-(0,k_2',k_1'))_{(l_1+l_2,l_1+l_3,l_2+l_3),d_2}}{\bigl(a-2k_3-\frac{d_2}{2}-(k_2',k_1',k_1'+k_2')\bigr)_{(l_1,l_2,l_3),d_2}}F_{(k_1',k_2',0),\bl}^\downarrow[f'](z_1,z_2) \\
&=\sum_{0\le l_j'\le l_j}\frac{(a-k_3-(0,k_2',k_1'))_{(l_1'+l_2',l_1'+l_3',l_2'+l_3'),d_2}(-a)_{|\bl|-|\bl'|}}{\bigl(a-2k_3-\frac{d_2}{2}-(k_2',k_1',k_1'+k_2')\bigr)_{(l_1',l_2',l_3'),d_2}}
F_{(k_1',k_2',0),\bl',\bl}[f'](z_1,z_2). 
\end{align*}
Then, since $\{(a,k_3)\in(\BZ_{\ge 0})^2\mid k_3\ge l_1+l_2,\,k_3\ge a\}\subset\BC^2$ is Zariski dense, this holds for all $(a,k_3)\in\BC^2$ without the assumption $l_1+l_2\le k_3$, 
and hence (\ref{formula_factorize_rank3}) holds for all $(a,k_3)\in\BC^2$. This completes the proof of Theorem \ref{thm_rank3}\,(3). 
\end{proof}

\begin{example}
Let $\bk=(k,k,k)$ with $k\in\BC$ and $f(x_2)=\det_{\fn^+_2}(x_2)^k$. Then by Proposition \ref{prop_example_Fkl_rank3}\,(2), 
for $x=(x_1,x_2)=(\xi,X)\in\fp^+$ we have 
\begin{align*}
F^\downarrow\Bigl(\lambda,\underline{k}_3;\det_{\fn^+_2}^k;x\Bigr)
&=\sum_{l=0}^\infty \frac{(-k)_{(l,l,0),d_2}}{\bigl(-\lambda-2k+\frac{3}{2}d_2+1\bigr)_{(l,0,0),d_2}} F_{(k,k,k),(l,0,0)}^\downarrow\bigl[\det_{\fn^+_2}^k\bigr](x_1,x_2) \\
&=\sum_{l=0}^\infty \frac{(-k)_l\bigl(-k-\frac{d_2}{2}\bigr)_l}{\bigl(-\lambda-2k+\frac{3}{2}d_2+1\bigr)_l l!} 
\biggl(-\frac{((x_1)^\sharp|x_2)}{\det_{\fn^+_2}(x_2)}\biggr)^l\det_{\fn^+_2}(x_2)^k \\
&={}_2F_1\biggl( \begin{matrix} -k,\,-k-\frac{d_2}{2} \\ -\lambda-2k+\frac{3}{2}d_2+1 \end{matrix};\frac{\xi X{}^t\hspace{-1pt}\hat{\xi}}{\det(X)}\biggr)\det(X)^k. 
\end{align*}
Also, for $k\in\BZ_{\ge 0}$ we have 
\[ C^{d,d_2}(\lambda,\underline{k}_3)=\frac{\bigl(\lambda-\frac{d_2}{2}+(2k,2k,k)\bigr)_{(0,0,k),d_2}}{(\lambda)_{(2k,k,k),d}}
=\frac{\bigl(\lambda+k-\frac{3}{2}d_2\bigr)_k}{(\lambda)_{(2k,k,k),d}}. \]
Combining these, we get 
\begin{align*}
&\Bigl\langle \det_{\fn^+_2}(x_2)^k,e^{(x|\overline{z})_{\fp^+}}\Bigr\rangle_{\lambda,x} \\
&=\frac{\bigl(\lambda+k-\frac{3}{2}d_2\bigr)_k}{(\lambda)_{(2k,k,k),d}}
{}_2F_1\biggl( \begin{matrix} -k,\,-k-\frac{d_2}{2} \\ -\lambda-2k+\frac{3}{2}d_2+1 \end{matrix};-\frac{((z_1)^\sharp|z_2)}{\det_{\fn^+_2}(z_2)}\biggr)\det_{\fn^+_2}(z_2)^k. 
\end{align*}
In addition, Theorem \ref{thm_rank3}\,(3) is equivalent to 
\begin{align*}
&{}_2F_1\biggl( \begin{matrix} -k,\,-k-\frac{1}{4}d \\ -\lambda-2k+\frac{3}{4}d+1 \end{matrix};-\frac{((x_1)^\sharp|x_2)}{\det_{\fn^+_2}(x_2)}\biggr) \\
&=\biggl(1+\frac{((x_1)^\sharp|x_2)}{\det_{\fn^+_2}(x_2)}\biggr)^{-\lambda+d+1}
{}_2F_1\biggl( \begin{matrix} -\lambda-k+d+1,\,-\lambda-k+\frac{3}{4}d+1 \\ -\lambda-2k+\frac{3}{4}d+1 \end{matrix};-\frac{((x_1)^\sharp|x_2)}{\det_{\fn^+_2}(x_2)}\biggr). 
\end{align*}
\end{example}

\begin{example}
Let $\bk=(k_1,k_2,0)\in\BZ_{++}^3$, let $e'\in\fn^+_2$ be a primitive idempotent, and let $\fp^+(e')_0=:\fp^{+\prime\prime}$, $\fp^+_2(e')_0=:\fp^{+\prime\prime}_2$ be as in (\ref{Peirce}). 
We take $f(x_2)=f(x_2'')\in\cP_{(k_1,k_2)}(\fp^{+\prime\prime}_2)$. Then by Proposition \ref{prop_example_Fkl_rank3}\,(3), for $\xi=(\xi,X)\in\fp^+$ we have 
\begin{align*}
&F^\downarrow(\lambda,(k_1,k_2,0);f;x) \\
&=\sum_{l=0}^{k_2}\frac{((0,-k_2,-k_1))_{(0,l,l),d_2}}{\bigl(-\lambda+\frac{3}{2}d_2+1-(k_2,k_1,k_1+k_2)\bigr)_{(0,0,l),d_2}} F_{(k_1,k_2,0),(0,0,l)}^\downarrow[f](x_1,x_2) \\
&=\sum_{l=0}^{k_2}\frac{\bigl(-k_2-\frac{d_2}{2})_l(-k_1-d_2)_l}{\bigl(-\lambda-k_1-k_2+\frac{d_2}{2}+1\bigr)_l}
\frac{(-k_2)_l\bigl(-k_1-\frac{d_2}{2}\bigr)_l}{\bigl(-k_2-\frac{d_2}{2})_l(-k_1-d_2)_l l!}\biggl(-\frac{\det_{\fn^{+\prime\prime}}(x_1'')}{\det_{\fn^{+\prime\prime}}(x_2'')}\biggr)^lf(x_2'') \\
&=\sum_{l=0}^{k_2}\frac{(-k_2)_l\bigl(-k_1-\frac{d_2}{2}\bigr)_l}{\bigl(-\lambda-k_1-k_2+\frac{d_2}{2}+1\bigr)_l l!}
\biggl(-\frac{\det_{\fn^{+\prime\prime}}(x_1'')}{\det_{\fn^{+\prime\prime}}(x_2'')}\biggr)^lf(x_2'') \\
&={}_2F_1\biggl( \begin{matrix} -k_2,\,-k_1-\frac{d_2}{2} \\ -\lambda-k_1-k_2+\frac{d_2}{2}+1 \end{matrix};\frac{({}^t\hspace{-1pt}\hat{\xi}\xi|e')}{\det_{\fn^{+\prime\prime}_2}(X'')}\biggr)^lf(X''). 
\end{align*}
We also have 
\[ C^{d,d_2}(\lambda,(k_1,k_2,0))=\frac{\bigl(\lambda-\frac{d_2}{2}+(k_1,k_1,k_2)\bigr)_{(k_2,0,0),d_2}}{(\lambda)_{(k_1+k_2,k_2,0),d}}
=\frac{\bigl(\lambda+k_1-\frac{d_2}{2}\bigr)_{k_2}}{(\lambda)_{(k_1+k_2,k_2),d}}. \]
Combining these, we get 
\begin{align*}
&\Bigl\langle f(x_2''),e^{(x|\overline{z})_{\fp^+}}\Bigr\rangle_{\lambda,x} \\
&=\frac{\bigl(\lambda+k_1-\frac{d_2}{2}\bigr)_{k_2}}{(\lambda)_{(k_1+k_2,k_2),d}}
{}_2F_1\biggl( \begin{matrix} -k_2,\,-k_1-\frac{d_2}{2} \\ -\lambda-k_1-k_2+\frac{d_2}{2}+1 \end{matrix};-\frac{\det_{\fn^{+\prime\prime}}(z_1'')}{\det_{\fn^{+\prime\prime}}(z_2'')}\biggr)^lf(z_2''). 
\end{align*}
\end{example}

\begin{example}
Let $\bk=(k_1,k_2,k_2)$ with $k_1-k_2\in\BZ_{\ge 0}$, $k_2\in\BC$, let $w_2\in\fp^+_2$ be of rank 1, and let $f(x_2):=(x_2|w_2)_{\fn^+_2}^{k_1-k_2}\det_{\fn^+_2}(x_2)^{k_2}$. 
Then we have 
\begin{align*}
&F^\downarrow(\lambda,(k_1,k_2,k_2);f;x) \\
&=\sum_{\substack{0\le l_1 \\ 0\le l_2\le k_1-k_2}}
\frac{((-k_2,-k_2,-k_1))_{(l_1+l_2,l_1,l_2),d_2}F_{(k_1,k_2,k_2),(l_1,l_2,0)}^\downarrow[f](x_1,x_2)}{\bigl(-\lambda+\frac{3}{2}d_2+1-(2k_2,k_1+k_2,k_1+k_2)\bigr)_{(l_1,l_2,0),d_2}} \\
&=\sum_{\substack{0\le l_1 \\ 0\le l_2\le k_1-k_2}}
\frac{(-k_2)_{l_1+l_2}\bigl(-k_2-\frac{d_2}{2}\bigr)_{l_1}(-k_1-d_2)_{l_2}}{\bigl(-\lambda-2k_2+\frac{3}{2}d_2+1\bigr)_{l_1}(-\lambda-k_1-k_2+d_2+1)_{l_2}} \\*
&\hspace{80pt}\times F_{(k_1-k_2,0,0),(l_1,l_2,0)}^\downarrow\bigl[(\cdot|w_2)_{\fn^+_2}^{k_1-k_2}\bigr](x_1,x_2)\det_{\fn^+_2}(x_2)^{k_2},
\end{align*}
and $F_{(k_1-k_2,0,0),(l_1,l_2,0)}^\downarrow\bigl[(\cdot|w_2)_{\fn^+_2}^{k_1-k_2}\bigr](x_1,x_2)$ is given in Proposition \ref{prop_example_Fkl_rank3}\,(4). 
Also, when $k_2\in\BZ_{\ge 0}$ we have 
\begin{align*}
C^{d,d_2}(\lambda,(k_1,k_2,k_2))
&=\frac{\bigl(\lambda-\frac{d_2}{2}+(k_1+k_2,\max\{k_1,2k_2\},k_2)\bigr)_{(0,\min\{k_2,k_1-k_2\},k_2),d_2}}{(\lambda)_{(k_1+k_2,\min\{k_1,2k_2\},k_2),d}} \\
&=\frac{(\lambda+\max\{k_1,2k_2\}-d_2)_{\min\{k_2,k_1-k_2\}}\bigl(\lambda+k_2-\frac{3}{2}d_2\bigr)_{k_2}}{(\lambda)_{(k_1+k_2,\min\{k_1,2k_2\},k_2),d}}, 
\end{align*}
and $\bigl\langle (x_2|w_2)_{\fn^+_2}^{k_1-k_2}\det_{\fn^+_2}(x_2)^{k_2},e^{(x|\overline{z})_{\fp^+}}\bigr\rangle_{\lambda,x}$ is given by the product of the above two formulas. 
\end{example}

\begin{example}
Let $\bk=(k_2,k_2,k_3)$ with $k_2-k_3\in\BZ_{\ge 0}$, $k_2\in\BC$, let $w_2\in\fp^+_2$ be of rank 1, and let $f(x_2):=(x_2^\itinv|w_2)_{\fn^+_2}^{k_2-k_3}\det_{\fn^+_2}(x_2)^{k_2}$. 
Then we have 
\begin{align*}
&F^\downarrow(\lambda,(k_2,k_2,k_3);f;x) \\
&=\sum_{\substack{0\le l_1 \\ 0\le l_3\le k_2-k_3}}
\frac{((-k_3,-k_2,-k_2))_{(l_1,l_1+l_3,l_3),d_2}F_{(k_2,k_2,k_3),(l_1,0,l_3)}^\downarrow[f](x_1,x_2)}{\bigl(-\lambda+\frac{3}{2}d_2+1-(k_2+k_3,k_2+k_3,2k_2)\bigr)_{(l_1,0,l_3),d_2}} \\
&=\sum_{\substack{0\le l_1 \\ 0\le l_3\le k_2-k_3}}
\frac{(-k_3)_{l_1}\bigl(-k_2-\frac{d_2}{2}\bigr)_{l_1+l_3}(-k_2-d_2)_{l_3}}{\bigl(-\lambda-k_2-k_3+\frac{3}{2}d_2+1\bigr)_{l_1}\bigl(-\lambda-2k_2+\frac{d_2}{2}+1\bigr)_{l_3}} \\*
&\hspace{70pt}\times F_{(0,0,k_3-k_2),(l_1,0,l_3)}^\downarrow\bigl[((\cdot)^\itinv|w_2)_{\fn^+_2}^{k_2-k_3}\bigr](x_1,x_2)\det_{\fn^+_2}(x_2)^{k_2}, 
\end{align*}
and $F_{(0,0,k_3-k_2),(l_1,0,l_3)}^\downarrow\bigl[((\cdot)^\itinv|w_2)_{\fn^+_2}^{k_2-k_3}\bigr](x_1,x_2)$ is given in Proposition \ref{prop_example_Fkl_rank3}\,(5). 
Also, when $k_3\in\BZ_{\ge 0}$ we have 
\begin{align*}
C^{d,d_2}(\lambda,(k_2,k_2,k_3))&=\frac{\bigl(\lambda-\frac{d_2}{2}+(k_2+k_3,k_2+k_3,k_2)\bigr)_{(k_2-k_3,0,k_3),d_2}}{(\lambda)_{(2k_2,k_2,k_3),d}} \\
&=\frac{\bigl(\lambda+k_2+k_3-\frac{d_2}{2}\bigr)_{k_2-k_3}\bigl(\lambda+k_2-\frac{3}{2}d_2\bigr)_{k_3}}{(\lambda)_{(2k_2,k_2,k_3),d}}, 
\end{align*}
and $\bigl\langle (x_2^\itinv|w_2)_{\fn^+_2}^{k_2-k_3}\det_{\fn^+_2}(x_2)^{k_2},e^{(x|\overline{z})_{\fp^+}}\bigr\rangle_{\lambda,x}$ is given by the product of the above two formulas. 
\end{example}

\subsection{Results on restriction of $\cH_\lambda(D)$ to subgroups}

In this subsection, we consider the decomposition of $\cH_\lambda(D)$ under $\widetilde{G}_1$. Let $V_\bk$ be an irreducible $K_1$-module isomorphic to $\cP_\bk(\fp^+_2)$. 
Then $\cH_\lambda(D)$ is decomposed under $\widetilde{G}_1$ as 
\[ \cH_\lambda(D)|_{\widetilde{G}_1}\simeq\hsum_{\bk\in\BZ_{++}^3}\cH_{\varepsilon_1\lambda}(D_1,V_\bk). \]
For each case, $\cH_\lambda(D)$ has a minimal $\widetilde{K}$-type 
\[ \chi^{-\lambda}\simeq \begin{cases} 
V_{\underline{\lambda/2}_3}^{(3)\vee} \boxtimes V_{\underline{\lambda/2}_3}^{(3)}\simeq V_{\underline{\lambda}_3}^{(3)\vee} \boxtimes V_{\underline{0}_3}^{(3)}
\simeq V_{\underline{0}_3}^{(3)\vee} \boxtimes V_{\underline{\lambda}_3}^{(3)} & (G=SU(3,3)), \\
V_{\underline{\lambda/2}_6}^{(6)\vee} & (G=SO^*(12)), \\
\chi^{-\lambda} & (G=E_{7(-25)}), \end{cases} \]
and $\cH_{\varepsilon_1\lambda}(D_1,V_\bk)$ has a minimal $\widetilde{K}_1$-type 
\[ \chi_1^{-\varepsilon_1\lambda}\otimes V_\bk\simeq \begin{cases} V_{\underline{\lambda}_3+(2k_1,2k_2,2k_3)}^{(3)\vee} & (G_1=SO^*(6)), \\
V_{\underline{\lambda/2}_3+(k_1,k_2,k_3)}^{(3)\vee} \boxtimes V_{\underline{\lambda/2}_3+(k_1,k_2,k_3)}^{(3)\vee} & (G_1=SO^*(6)\times SO^*(6)), \\
V_{\underline{\frac{3}{4}\lambda+\frac{1}{2}|\bk|}_2}^{(2)\vee} \boxtimes V_{\underline{\frac{1}{4}\lambda+\frac{1}{2}|\bk|}_6-(k_3,k_3,k_2,k_2,k_1,k_1)}^{(6)} \\
\simeq V_{\underline{0}_2}^{(2)\vee} \boxtimes V_{\substack{\underline{\lambda}_6+(k_1+k_2,k_1+k_2,k_1+k_3,\\ \hspace{20pt} k_1+k_3,k_2+k_3,k_2+k_3)}}^{(6)} & (G_1=SU(2,6)). \end{cases} \]
To describe the intertwining operators, for $\lambda\in\BC$, $\bk\in\BZ_{++}^3+\BC\underline{1}_3$, $f(x_2)\in\cP_\bk(\fp^+_2)$, 
we define a function $F^\uparrow(\lambda,\bk;f;y_1,x_2)$ on $\fp^-_1\oplus\fp^+_2$ by 
\begin{align}
&F^\uparrow(\lambda,\bk;f;y_1,x_2) \notag\\
:\hspace{-3pt}&=\sum_{\substack{\bl\in\BZ^3 \\ 0\le l_1\le k_1-k_2 \\ 0\le l_2\le k_2-k_3 \\ 0\le l_3}}
\frac{1}{\bigl(\lambda-\frac{d_2-2}{2}+(k_1+k_2,k_1+k_3,k_2+k_3)\bigr)_{(l_3,l_2,l_1),d_2}}F_{\bk,\bl}^\uparrow[f](y_1,x_2) \notag\\
&=\sum_{\substack{\bl\in\BZ^3 \\ 0\le l_1\le k_1-k_2 \\ 0\le l_2\le k_2-k_3 \\ 0\le l_3}}
\frac{1}{\bigl(\lambda+k_2+k_3-\frac{3}{4}d+1\bigr)_{l_1}}\frac{1}{\bigl(\lambda+k_1+k_3-\frac{1}{2}d+1\bigr)_{l_2}} \notag\\*
&\hspace{110pt}\times\frac{1}{\bigl(\lambda+k_1+k_2-\frac{1}{4}d+1\bigr)_{l_3}}F_{\bk,\bl}^\uparrow[f](y_1,x_2). \label{def_Fup_rank3}
\end{align}

\begin{theorem}
Let $\Re\lambda>1+2d$, $\bk\in\BZ_{++}^3$. 
\begin{enumerate}
\item (\cite[Theorems 5.7\,(5), 5.12]{N1}). Let $\rK_\bk(x_2)\in(\cP_\bk(\fp^+_2)\otimes \Hom(V_\bk,\BC))^{K_1}$. Then the linear map 
\begin{gather*}
\cF_{\lambda,\bk}^\uparrow\colon \cH_{\varepsilon_1\lambda}(D_1,V_\bk)_{\widetilde{K}_1}\longrightarrow\cH_\lambda(D)_{\widetilde{K}}|_{(\fg_1,\widetilde{K}_1)}, \\
(\cF_{\lambda,\bk}^\uparrow f)(x):=F^\uparrow\biggl(\lambda,\bk;\rK_\bk;\frac{1}{\varepsilon_1}\frac{\partial}{\partial x_1},x_2\biggr)f(x_1)
\end{gather*}
intertwines the $(\fg_1,\widetilde{K}_1)$-action, where $F^\uparrow$ is as in (\ref{def_Fup_rank3}). 
\item Let $\rK_\bk^\vee(y_2)\in(\cP_\bk(\fp^-_2)\otimes V_\bk)^{K_1}$. Then the linear map 
\begin{gather*}
\cF_{\lambda,\bk}^\downarrow\colon \cH_\lambda(D)|_{\widetilde{G}_1}\longrightarrow\cH_{\varepsilon_1\lambda}(D_1,V_\bk), \\
(\cF_{\lambda,\bk}^\downarrow f)(x_1):=F^\downarrow\biggl(\lambda,\bk;\rK_\bk^\vee;\frac{\partial}{\partial x}\biggr)f(x)\biggr|_{x_2=0}
\end{gather*}
intertwines the $\widetilde{G}_1$-action, where $F^\downarrow$ is as in (\ref{def_Fdown_rank3}). 
\item (\cite[Corollary 6.7]{N3}). Suppose $c_\bk,c_\bk^\vee\in\BR_{>0}$ satisfy, for $v\in V_\bk$, 
\[ \Vert \rK_\bk(x_2)v\Vert_{F,x_2}^2=c_\bk|v|_{V_\bk}^2, \qquad \Vert (v,\rK_\bk^\vee(\overline{x_2}))\Vert_{F,x_2}^2=c_\bk^\vee|v|_{V_\bk}^2. \]
Then the operator norms of $\cF_{\lambda,\bk}^\uparrow$, $\cF_{\lambda,\bk}^\downarrow$ are given by 
\[ c_\bk^{-1}\Vert \cF_{\lambda,\bk}^\uparrow\Vert_{\mathrm{op}}^2=c_\bk^\vee\Vert \cF_{\lambda,\bk}^\downarrow\Vert_{\mathrm{op}}^{-2}=C^{d,d_2}(\lambda,\bk), \]
where $C^{d,d_2}(\lambda,\bk)$ is as in (\ref{const_rank3}). 
\end{enumerate}
\end{theorem}

\begin{remark}
In \cite[Theorem 5.12]{N1}, the subscripts $(m_3,m_2,m_1),2$ are wrong, and $(m_3,\allowbreak m_2,m_1),1$ (in (1)), $(m_3,m_2,m_1),4$ (in (2)) are correct. 
\end{remark}

The family of differential operators $\cF_{\lambda,\bk}^\downarrow$ is meromorphically continued for all $\lambda\in\BC$, and gives rise to $\widetilde{G}_1$-intertwining operators 
\[ \cF_{\lambda,\bk}^\downarrow\colon \cO_\lambda(D)|_{\widetilde{G}_1}\longrightarrow\cO_{\varepsilon_1\lambda}(D_1,V_\bk) \]
for all $\lambda\in\BC$ except at its poles. Then by Theorem \ref{thm_rank3}\,(3), for special $\lambda$, the following holds. 

\begin{theorem}[{\cite[Corollary 6.8\,(4)]{N3}}]
Let $\bk\in\BZ_{++}^3$. For $a=1,2,\ldots,k_3$, we fix a $[K_1,K_1]$-isomorphism $V_\bk\simeq V_{\bk-\underline{a}_3}$, and assume that 
\[ \rK_\bk^\vee(y_2)=\det_{\fn^-}(y_2)^a\rK_{\bk-\underline{a}_3}^\vee(y_2)\in (\cP_\bk(\fp^-_2)\otimes V_\bk)^{K_1} \]
holds. Then we have 
\begin{multline*}
\cF_{d+1-a,\bk}^\downarrow=\cF_{d+1+a,\bk-\underline{a}_3}^\downarrow\circ\det_{\fn^-}\biggl(\frac{\partial}{\partial x}\biggr)^a\colon \\
\cO_{d+1-a}(D)|_{\widetilde{G}_1}\longrightarrow \cO_{\varepsilon_1(d+1-a)}(D_1,V_\bk)\simeq\cO_{\varepsilon_1(d+1+a)}(D_1,V_{\bk-\underline{a}_3}). 
\end{multline*}
\end{theorem}

For $\lambda>d$ we have 
\[ \cH_{\varepsilon_1\lambda}(D_1,V_\bk)_{\widetilde{K}_1}\simeq d\tau_\lambda(\mathcal{U}(\fg_1))\cP_\bk(\fp^+_2)\subset \cH_\lambda(D)_{\widetilde{K}}, \]
but this does not hold for smaller $\lambda$ in general. For such $\lambda$, by Theorem \ref{thm_rank3}\,(2) the following holds. 
For $i=1,2,3$, $\lambda\in\frac{d}{2}(i-1)-\BZ_{\ge 0}$, let $M_i^\fg(\lambda)\subset\cO_\lambda(D)_{\widetilde{K}}$ be the $(\fg,\widetilde{K})$-submodule given in (\ref{submodule}). 

\begin{theorem}
Let $\bk\in\BZ_{++}^3$. For $i=1,2,3$, 
\[ d\tau_\lambda(\mathcal{U}(\fg_1))\cP_\bk(\fp^+_2)\subset M_i^\fg(\lambda) \]
holds if and only if 
\[ \lambda\in\begin{cases} -(k_1+k_2)-\BZ_{\ge 0} & (i=1), \\
\frac{d}{2}-\min\{k_1,k_2+k_3\}-\BZ_{\ge 0} & (i=2), \\
d-k_3-\BZ_{\ge 0} & (i=3). \end{cases} \]
\end{theorem}

Here, the ``if'' direction follows from the holomorphy in Theorem \ref{thm_rank3}\,(2) and this is already given in \cite[Corollary 6.8\,(1)]{N3}. 
On the other hand, the ``only if'' direction follows from the non-vanishing in Theorem \ref{thm_rank3}\,(2) and this is new, except for $k_1=k_2=k_3$ or $k_3=0$.

\section{General rank case}\label{section_general}

In this section, we treat $\fp^\pm=(\fp^\pm)^\sigma\oplus(\fp^\pm)^{-\sigma}=\fp^\pm_1\oplus\fp^\pm_2$ such that $\fp^\pm$, $\fp^\pm_2$ are both simple, classical, of tube type, 
and $\rank\fp^\pm=\rank\fp^\pm_2$, that is, we treat 
\begin{align*}
(\fp^\pm,\fp^\pm_1,\fp^\pm_2)&=\begin{cases} (\BC^n,\BC^{n'},\BC^{n''}) \quad (n=n'+n'',\; n''\ge 3) & (\text{Case }1), \\
(M(r,\BC),\Alt(r,\BC),\Sym(r,\BC)) & (\text{Case }2), \\
(\Alt(2r,\BC),\Alt(r,\BC)\oplus\Alt(r,\BC),M(r,\BC)) & (\text{Case }3). \end{cases}
\end{align*}
Then the corresponding symmetric pairs $(G,(G^\sigma)_0,(G^{\sigma\vartheta})_0)=(G,G_1,G_2)$ are given by 
\begin{align*}
(G,G_1,G_2)
&=\begin{cases} (SO_0(2,n),SO_0(2,n')\times SO(n''),SO_0(2,n'')\times SO(n')) & (\text{Case }1), \\ 
(SU(r,r),SO^*(2r),Sp(r,\BR)) & (\text{Case }2), \\ (SO^*(4r),SO^*(2r)\times SO^*(2r),U(r,r)) & (\text{Case }3) \end{cases}
\end{align*}
(up to covering). $(r,d,d_2,\varepsilon_1)$ are given as follows, and we define $\delta$ by the following. 
\begin{equation}\label{str_const_general}
(r,d,d_2,\varepsilon_1,\delta)=\begin{cases} (2,n-2,n''-2,1,n'-2) & (\text{Case }1,\; n'\ge 2), \\
(2,n-2,n''-2,2,-1) & (\text{Case }1,\; n'=1), \\
(r,2,1,2,-1) & (\text{Case }2), \\
(r,4,2,1,0) & (\text{Case }3). \end{cases} 
\end{equation}
Especially $d=2d_2$ holds if $r\ge 3$. 
We fix a maximal tripotent $e\in\fp^+_2\subset\fp^+$, regard $\fp^+_2\subset\fp^+$ as Jordan algebras with the unit element $e$, the Euclidean real forms $\fn^+_2\subset\fn^+$, 
the symmetric cones $\Omega_2\subset\Omega$, and identify $\fp^+\simeq\fp^-$, $\fp^+_2\simeq\fp^-_2$ via $Q(e),Q(\overline{e})$.

\subsection{Definition of functions $F^\pm_{\bm,\bl}[f](x_1,y_2)$, $F^\downarrow_{l,\pm}\bigl(\begin{smallmatrix}\nu\\\mu\end{smallmatrix};f;x\bigr)$}

In this subsection, we define some functions on $\fp^\pm_1\oplus\fp^\pm_2$. 
First, for $\bm\in\BZ_{++}^{\lfloor r/2\rfloor}$, let $\bm^2:=(m_1,m_1,m_2,m_2,\ldots,m_{\lfloor r/2\rfloor},m_{\lfloor r/2\rfloor}(,0))\in\BZ_{++}^r$, let 
\[ \cP_{\langle\bm\rangle}(\fp^\pm_1):=\begin{cases} \cP_{(m,m)}(\BC^{n'}) & (\text{Case }1,\; n'\ge 3), \\ \cP_m(\BC)\otimes\cP_m(\BC) & (\text{Case }1,\; n'=2), \\
\cP_{2m}(\BC) & (\text{Case }1,\; n'=1), \\ \cP_{2\bm}(\Alt(r,\BC)) & (\text{Case }2), \\ \cP_\bm(\Alt(r,\BC))\otimes\cP_\bm(\Alt(r,\BC)) & (\text{Case }3) \end{cases} \]
(where $\BC^2=\BC\oplus\BC$ corresponds to the local isomorphism $SO_0(2,2)\simeq SL(2,\BR)\times SL(2,\BR)$), 
and $\Phi_\bm^{\fp^\pm_1,\fp^\mp_2}(x_1,y_2)$, $\tilde{\Phi}_\bm^{\fp^\pm_1,\fp^\mp_2}(x_1,y_2)\in\cP_{\langle\bm\rangle}(\fp^\pm_1)_{x_1}\otimes\cP_{\bm^2}(\fp^\mp_2)_{y_2}$ 
be the polynomials satisfying 
\begin{align*}
e^{\frac{1}{2}(x_1|P(y_2)x_1)_{\fn^\pm}}&=\sum_{\bm\in\BZ_{++}^{\lfloor r/2\rfloor}}
\frac{d_\bm^{\tilde{\fp}}}{\bigl(\frac{d}{2}\bigl(\bigl\lfloor\frac{r}{2}\bigr\rfloor-1\bigr)+1\bigr)_{\bm,d}}\Phi_\bm^{\fp^\pm_1,\fp^\mp_2}(x_1,y_2)
=\sum_{\bm\in\BZ_{++}^{\lfloor r/2\rfloor}}\tilde{\Phi}_\bm^{\fp^\pm_1,\fp^\mp_2}(x_1,y_2), 
\end{align*}
where 
\begin{align*}
\tilde{\fp}&:=\begin{cases} \BC & (\text{Case }1), \\ M\bigl(\bigl\lfloor\frac{r}{2}\bigr\rfloor,\BC\bigr) & (\text{Case }2), \\
\Alt\bigl(2\bigl\lfloor\frac{r}{2}\bigr\rfloor,\BC\bigr) & (\text{Case }3), \end{cases} \quad
d_\bm^{\tilde{\fp}}:=\dim\cP_\bm(\tilde{\fp})=\begin{cases} 1 & (\text{Case }1), \\ 
\dim\bigl(V_\bm^{(\lfloor r/2\rfloor)\vee}\bigr)^{\otimes 2} & (\text{Case }2), \\ 
\dim V_{\bm^2}^{(2\lfloor r/2\rfloor)\vee} & (\text{Case }3). \end{cases} 
\end{align*}
Then for $\mu\in\BC$, we have 
{\setlength{\belowdisplayskip}{0pt}
\begin{align*}
h_{\fn^\pm}(x_1,P(y_2)x_1)^{-\mu/2}&=\sum_{\bm\in\BZ_{++}^{\lfloor r/2\rfloor}}
\frac{(\mu)_{\bm,d}d_\bm^{\tilde{\fp}}}{\bigl(\frac{d}{2}\bigl(\bigl\lfloor\frac{r}{2}\bigr\rfloor-1\bigr)+1\bigr)_{\bm,d}}\Phi_\bm^{\fp^\pm_1,\fp^\mp_2}(x_1,y_2) \\
&=\sum_{\bm\in\BZ_{++}^{\lfloor r/2\rfloor}}(\mu)_{\bm,d}\tilde{\Phi}_\bm^{\fp^\pm_1,\fp^\mp_2}(x_1,y_2), 
\end{align*}}
{\setlength{\abovedisplayskip}{0pt}
\begin{align}
\det_{\fn^\pm}(x_1+x_2)^{-\mu}&=\det_{\fn^\pm}(x_2)^{-\mu} h_{\fn^\pm}(x_1,P(x_2^\itinv)x_1)^{-\mu/2} \notag \\
&=\det_{\fn^\pm}(x_2)^{-\mu}\sum_{\bm\in\BZ_{++}^{\lfloor r/2\rfloor}}
\frac{(\mu)_{\bm,d}d_\bm^{\tilde{\fp}}}{\bigl(\frac{d}{2}\bigl(\bigl\lfloor\frac{r}{2}\bigr\rfloor-1\bigr)+1\bigr)_{\bm,d}}\Phi_\bm^{\fp^\pm_1,\fp^\mp_2}(x_1,x_2^\itinv) \notag \\
&=\det_{\fn^\pm}(x_2)^{-\mu}\sum_{\bm\in\BZ_{++}^{\lfloor r/2\rfloor}}(\mu)_{\bm,d}\tilde{\Phi}_\bm^{\fp^\pm_1,\fp^\mp_2}(x_1,x_2^\itinv) \label{binom_general}
\end{align}}
(for details, see \cite[Propositions 2.4, 6.2]{N2}). 

\begin{example}\label{example_Phim_general}
Let $\mu=-1$. Then for $x=x_1+x_2\in\fp^+$, we have 
\begin{align*}
\det_{\fn^+}(x_1+x_2)&=\det_{\fn^+_2}(x_2)h_{\fn^+}(x_1,P(x_2^\itinv)x_1)^{1/2} \\
&=\det_{\fn^+_2}(x_2)\sum_{j=0}^{\lfloor r/2\rfloor}(-1)^{j}\binom{\lfloor r/2\rfloor}{j}\Phi_{(\underline{1}_j,\underline{0}_{\lfloor r/2\rfloor-j})}^{\fp^+_1,\fp^-_2}(x_1,x_2^\itinv), 
\end{align*}
where $\underline{k}_j:=\smash{(\overbrace{k,\ldots,k}^j)}$, and comparing the homogeneous terms of degree $2\bigl\lfloor \frac{r}{2}\bigr\rfloor$ with respect to $x_1$, we have 
\[ (-1)^{\lfloor r/2\rfloor}\det_{\fn^+_2}(x_2)\Phi_{\underline{1}_{\lfloor r/2\rfloor}}^{\fp^+_1,\fp^-_2}(x_1,x_2^\itinv)=\begin{cases} \det_{\fn^+}(x_1) & (r\colon\text{even}), \\
((x_1)^\sharp|x_2)_{\fn^+} & (r\colon\text{odd}), \end{cases} \]
where $x\mapsto x^\sharp$ is the adjugate in $\fp^+=\fn^{+\BC}$. That is, for $x_1\in\fp^+_1$, $y_2\in\fp^-_2$, we have 
\[ \Phi_{\underline{1}_{\lfloor r/2\rfloor}}^{\fp^+_1,\fp^-_2}(x_1,y_2)=\begin{cases} (-1)^{r/2}\det_{\fn^+}(x_1)\det_{\fn^-_2}(y_2) & (r\colon\text{even}), \\
(-1)^{\lfloor r/2\rfloor}((x_1)^\sharp|(y_2)^\sharp)_{\fn^+} & (r\colon\text{odd}). \end{cases} \]
More generally, for $k\in\BZ_{\ge 0}$ we have 
\[ \Phi_{\underline{k}_{\lfloor r/2\rfloor}}^{\fp^+_1,\fp^-_2}(x_1,y_2)=\Phi_{\underline{1}_{\lfloor r/2\rfloor}}^{\fp^+_1,\fp^-_2}(x_1,y_2)^k
=\begin{cases} (-1)^{kr/2}\det_{\fn^+}(x_1)^k\det_{\fn^-_2}(y_2)^k & (r\colon\text{even}), \\
(-1)^{k\lfloor r/2\rfloor}((x_1)^\sharp|(y_2)^\sharp)_{\fn^+}^k & (r\colon\text{odd}). \end{cases} \]
\end{example}

Next, for $\bm\in\BZ_{++}^{\lfloor r/2\rfloor}$, $\bl\in(\BZ_{\ge 0})^{\lceil r/2\rceil}$, we define $\phi_\pm(\bm,\bl)\in\BZ^r$ by 
\begin{align*}
\phi_+(\bm,\bl) &:=
\begin{cases} (m_1+l_1,m_1,m_2+l_2,m_2,\ldots,m_{r/2}+l_{r/2},m_{r/2}) & (r\colon\text{even}), \\
(m_1+l_1,m_1,m_2+l_2,m_2,\ldots,m_{\lfloor r/2\rfloor}+l_{\lfloor r/2\rfloor},m_{\lfloor r/2\rfloor},l_{\lceil r/2\rceil}) & (r\colon\text{odd}), \end{cases} \\
\phi_-(\bm,\bl) &:=
\begin{cases} (m_1,m_1-l_1,m_2,m_2-l_2,\ldots,m_{r/2},m_{r/2}-l_{r/2}) & (r\colon\text{even}), \\
(m_1,m_1-l_1,m_2,m_2-l_2,\ldots,m_{\lfloor r/2\rfloor},m_{\lfloor r/2\rfloor}-l_{\lfloor r/2\rfloor},-l_{\lceil r/2\rceil}) & (r\colon\text{odd}). \end{cases}
\end{align*}
Then for $l\in\BZ_{\ge 0}$, $f(y_2)\in\cP_{(l,\underline{0}_{r-1})}(\fp^\mp_2)$, we have 
\begin{align*}
\tilde{\Phi}_\bm^{\fp^\pm_1,\fp^\mp_2}(x_1,y_2)f(y_2)
&\in\bigoplus_{\substack{\bl\in(\BZ_{\ge 0})^{\lceil r/2\rceil}\\ |\bl|=l}} \cP_{\langle\bm\rangle}(\fp^\pm_1)_{x_1}\otimes\cP_{\phi_+(\bm,\bl)}(\fp^\mp_2)_{y_2}, \\
\tilde{\Phi}_\bm^{\fp^\pm_1,\fp^\mp_2}(x_1,y_2)f(y_2^\itinv)
&\in\bigoplus_{\substack{\bl\in(\BZ_{\ge 0})^{\lceil r/2\rceil}\\ |\bl|=l}} \cP_{\langle\bm\rangle}(\fp^\pm_1)_{x_1}\otimes\cP_{\phi_-(\bm,\bl)}(\fp^\mp_2)_{y_2}, 
\end{align*}
where $\cP_{\phi_+(\bm,\bl)}(\fp^\mp_2)$, $\cP_{\phi_-(\bm,\bl)}(\fp^\mp_2)=\{0\}$ if $\phi_+(\bm,\bl)$, $\phi_-(\bm,\bl)\notin\BZ_{+}^r:=\{\bk\in\BZ^r\mid k_1\ge\cdots\ge k_r\}$. 
According to these decompositions, for $\bl\in(\BZ_{\ge 0})^{\lceil r/2\rceil}$ with $|\bl|=l$, we define the polynomials $F_{\bm,\bl}^+[f](x_1,y_2)$ and $F_{\bm,\bl}^-[f](x_1,y_2)$ by 
\begin{alignat*}{2}
F_{\bm,\bl}^+[f](x_1,y_2) &:=\Proj^{\fp^\mp_2}_{\phi_+(\bm,\bl),y_2}\Bigl(\tilde{\Phi}_\bm^{\fp^\pm_1,\fp^\mp_2}(x_1,y_2)f(y_2)\Bigr)
&&\in\cP_{\langle\bm\rangle}(\fp^\pm_1)_{x_1}\otimes\cP_{\phi_+(\bm,\bl)}(\fp^\mp_2)_{y_2}, \\
F_{\bm,\bl}^-[f](x_1,y_2) &:=\Proj^{\fp^\mp_2}_{\phi_-(\bm,\bl),y_2}\Bigl(\tilde{\Phi}_\bm^{\fp^\pm_1,\fp^\mp_2}(x_1,y_2)f(y_2^\itinv)\Bigr)
&&\in\cP_{\langle\bm\rangle}(\fp^\pm_1)_{x_1}\otimes\cP_{\phi_-(\bm,\bl)}(\fp^\mp_2)_{y_2}. 
\end{alignat*}
Then $F_{\bm,\bl}^+[f](x_1,y_2)\ne 0$ holds only if 
\begin{align}
&0\le l_1, &&\textstyle 0\le l_j\le m_{j-1}-m_j\quad \bigl(1<j\le \frac{r}{2}\bigr) && && (r\colon\text{even}), \notag \\
&0\le l_1, &&\textstyle 0\le l_j\le m_{j-1}-m_j\quad \bigl(1<j\le \bigl\lfloor\frac{r}{2}\bigl\rfloor\bigr), &&0\le l_{\lceil r/2\rceil}\le m_{\lfloor r/2\rfloor}, && (r\colon\text{odd}), 
\label{formula_cond_ml_general-}
\end{align}
and $F_{\bm,\bl}^-[f](x_1,y_2)\ne 0$ holds only if 
\begin{align*}
&\textstyle 0\le l_j\le m_j-m_{j+1}\quad \bigl(1\le j< \frac{r}{2}\bigr), &&0\le l_{r/2} && && (r\colon\text{even}), \\
&\textstyle 0\le l_j\le m_j-m_{j+1}\quad \bigl(1\le j< \bigl\lfloor\frac{r}{2}\bigl\rfloor\bigr), &&0\le l_{\lfloor r/2\rfloor}\le m_{\lfloor r/2\rfloor}, &&0\le l_{\lceil r/2\rceil}
&& (r\colon\text{odd}). 
\end{align*}
If $r=3$, then these polynomials and $F^\uparrow_{\bk,\bl}[f]$, $F^\downarrow_{\bk,\bl}[f]$ in the previous section are related as 
\begin{align*}
F_{m,(l_1,l_2)}^+[f](x_1,y_2)&=F^\uparrow_{\substack{(l_1+l_2,0,0),\;\; \\ \;\;(l_2,0,m-l_2)}}[f](x_1,y_2)
=F^\downarrow_{\substack{(0,0,-l_1-l_2),\;\; \\ \;\;(m-l_2,0,l_2)}}[f((\cdot)^\itinv)](x_1,y_2^\itinv) \\*
&\hspace{118pt}\in\cP_{\langle m\rangle}(\fp^\pm_1)_{x_1}\otimes\cP_{(m+l_1,m,l_2)}(\fp^\mp_2)_{y_2}, \\
F_{m,(l_1,l_2)}^-[f](x_1,y_2)&=F^\uparrow_{\substack{(0,0,-l_1-l_2),\;\; \\ \;\;(0,l_1,m-l_1)}}[f((\cdot)^\itinv)](x_1,y_2)
=F^\downarrow_{\substack{(l_1+l_2,0,0),\;\; \\ \;\;(m-l_1,l_1,0)}}[f](x_1,y_2^\itinv) \\*
&\hspace{118pt}\in\cP_{\langle m\rangle}(\fp^\pm_1)_{x_1}\otimes\cP_{(m,m-l_1,-l_2)}(\fp^\mp_2)_{y_2}. 
\end{align*}
For general $r$, $F_{\bm,\bl}^\pm[f](x_1,y_2)$ satisfies the following. For $\bk\in\BZ_+^r$, let $d_\bk^{\fp^+_2}:=\dim\cP_\bk(\fp^+_2)$, 
and let $\Sigma_2\subset\fp^+_2$ be the Bergman--Shilov boundary of $D_2:=D\cap\fp^+_2$. 
For $x\in\fp^+$, let $|x|_\infty:=\inf\{t>0\mid t^{-1}x\in D\}$ be the spectral norm of $x$. 

\begin{proposition}\label{prop_estimate_Fml_general}
For $\bm\in\BZ_{++}^{\lfloor r/2\rfloor}$, $\bl\in(\BZ_{\ge 0})^{\lceil r/2\rceil}$, $|\bl|=l$, 
$f(y_2)\in\cP_{(l,\underline{0}_{r-1})}(\fp^-_2)$ 
and for $x_1\in\fp^+_1$, $y_2=ge\in\fp^-_2$ with $g\in K_1^\BC$, we have 
\begin{align*}
|F_{\bm,\bl}^+[f](x_1,y_2)|&\le \frac{\sqrt{d_{\bm^2}^{\fp^+_2}d_{(l,\underline{0}_{r-1})}^{\fp^+_2}}d_\bm^{\tilde{\fp}}}{\bigl(\frac{d}{2}\bigl(\bigl\lfloor\frac{r}{2}\bigr\rfloor-1\bigr)+1\bigr)_{\bm,d}}
(|{}^t\hspace{-1pt}g x_1|_\infty)^{2|\bm|}\Vert f(g(\cdot))\Vert_{L^2(\Sigma_2)}, \\
|F_{\bm,\bl}^-[f](x_1,y_2)|&\le \frac{\sqrt{d_{\bm^2}^{\fp^+_2}d_{(l,\underline{0}_{r-1})}^{\fp^+_2}}d_\bm^{\tilde{\fp}}}{\bigl(\frac{d}{2}\bigl(\bigl\lfloor\frac{r}{2}\bigr\rfloor-1\bigr)+1\bigr)_{\bm,d}}
(|{}^t\hspace{-1pt}g x_1|_\infty)^{2|\bm|}\Vert f({}^t\hspace{-1pt}g^{-1}(\cdot))\Vert_{L^2(\Sigma_2)}. 
\end{align*}
\end{proposition}
\begin{proof}
Since $f(x_2)\mapsto F_{\bm,\bl}^+[f](x_1,y_2)$, $f(x_2^\itinv)\mapsto F_{\bm,\bl}^-[f](x_1,y_2)$ are $K_1^\BC$-\hspace{0pt}equivariant, we have 
\begin{align*}
F_{\bm,\bl}^+[f](x_1,y_2)&=F_{\bm,\bl}^+[f](x_1,ge)=F_{\bm,\bl}^+[f(g(\cdot))]({}^t\hspace{-1pt}gx_1,e), \\
F_{\bm,\bl}^-[f](x_1,y_2)&=F_{\bm,\bl}^-[f](x_1,ge)=F_{\bm,\bl}^-[f({}^t\hspace{-1pt}g^{-1}(\cdot))]({}^t\hspace{-1pt}gx_1,e), 
\end{align*}
and hence it suffices to prove the inequalities when $y_2=e$. 
Next, since the reproducing kernel of $\cP_\bk(\fp^-_2)$ with respect to the $L^2(\Sigma_2)$-inner product is given by 
$K_\bk'(y_2,z_2):=d_\bk^{\fp^+_2}\Phi_\bk^{\fn^-_2}(P(\overline{z_2}^{\mathit{1/2}})y_2)=d_\bk^{\fp^+_2}\Phi_\bk^{\fn^-_2}(P(y_2^{\mathit{1/2}})\overline{z_2})$, 
where $\Phi_\bk^{\fn^-_2}(x_2):=\int_{K_1\cap K_L}\Delta_\bk^{\fn^-_2}(kx_2)\,dk$ (see \cite[Proposition XII.2.4]{FK}), we have 
\begin{align*}
&|F_{\bm,\bl}^\pm[f](x_1,e)|=\bigl|\langle F_{\bm,\bl}^\pm[f](x_1,\cdot),K_{\phi_\pm(\bm,\bl)}'(\cdot,e)\rangle_{L^2(\Sigma_2)}\bigr| \\
&\le \Vert F_{\bm,\bl}^\pm[f](x_1,\cdot)\Vert_{L^2(\Sigma_2)}\Vert K_{\phi_\pm(\bm,\bl)}'(\cdot,e)\Vert_{L^2(\Sigma_2)} \\
&=\Vert F_{\bm,\bl}^\pm[f](x_1,\cdot)\Vert_{L^2(\Sigma_2)}\sqrt{K_{\phi_\pm(\bm,\bl)}'(e,e)}
=\sqrt{d_{\phi_\pm(\bm,\bl)}^{\fp^+_2}}\Vert F_{\bm,\bl}^\pm[f](x_1,\cdot)\Vert_{L^2(\Sigma_2)}. 
\end{align*}
Now since $F_{\bm,\bl}^\pm[f](x_1,y_2)$ is defined by the orthogonal projection of $\tilde{\Phi}_\bm^{\fp^+_1,\fp^-_2}(x_1,y_2)f(y_2)$ 
or $\tilde{\Phi}_\bm^{\fp^+_1,\fp^-_2}(x_1,y_2)f(y_2^\itinv)$, we have 
\begin{align*}
\Vert F_{\bm,\bl}^+[f](x_1,\cdot)\bigr\Vert_{L^2(\Sigma_2)}
&\le \bigl\Vert\tilde{\Phi}_\bm^{\fp^+_1,\fp^-_2}(x_1,y_2)f(y_2)\bigr\Vert_{L^2(\Sigma_2),y_2} 
\le \bigl\Vert\tilde{\Phi}_\bm^{\fp^+_1,\fp^-_2}(x_1,\cdot)\bigr\Vert_{L^\infty(\Sigma_2)}\Vert f\Vert_{L^2(\Sigma_2)}, \\
\Vert F_{\bm,\bl}^-[f](x_1,\cdot)\bigr\Vert_{L^2(\Sigma_2)}
&\le \bigl\Vert\tilde{\Phi}_\bm^{\fp^+_1,\fp^-_2}\!(x_1,y_2)f(y_2^\itinv)\bigr\Vert_{L^2(\Sigma_2),x_2} 
\le \bigl\Vert\tilde{\Phi}_\bm^{\fp^+_1,\fp^-_2}\!(x_1,\cdot)\bigr\Vert_{L^\infty(\Sigma_2)}\Vert f\Vert_{L^2(\Sigma_2)}. 
\end{align*}
In addition, let $\Sigma_1\subset\fp^+_1$ be the Bergman--Shilov boundary of $D_1:=D\cap\fp^+_1$. Then we have  
\begin{align*}
\bigl\Vert\tilde{\Phi}_\bm^{\fp^+_1,\fp^-_2}(x_1,\cdot)\bigr\Vert_{L^\infty(\Sigma_2)}
&\le \sup_{\substack{|x_1'|_\infty\le |x_1|_\infty \\ y_2\in\Sigma_2}}\bigl|\tilde{\Phi}_\bm^{\fp^+_1,\fp^-_2}(x_1',y_2)\bigr|
=(|x_1|_\infty)^{2|\bm|}\sup_{\substack{x_1'\in\Sigma_1 \\ y_2\in\Sigma_2}}\bigl|\tilde{\Phi}_\bm^{\fp^+_1,\fp^-_2}(x_1',y_2)\bigr|. 
\end{align*}
Also, since $\tilde{\Phi}_\bm^{\fp^+_1,\fp^-_2}(x_1',y_2)$ is defined by using the roots of $t^{\lfloor r/2\rfloor}h_{\fn^+}(t^{-1}x_1',\allowbreak P(y_2)x_1')^{1/2}$, 
and these roots have absolute values 1 when $x_1'\in\Sigma_1$, $y_2\in\Sigma_2$, by \cite[Theorem XII.1.1]{FK} we have  
\begin{align*}
\bigl|\tilde{\Phi}_\bm^{\fp^+_1,\fp^-_2}(x_1',y_2)\bigr|
=\frac{d_\bm^{\tilde{\fp}}\bigl|\Phi_\bm^{\fp^+_1,\fp^-_2}(x_1',y_2)\bigr|}{\bigl(\frac{d}{2}\bigl(\bigl\lfloor\frac{r}{2}\bigr\rfloor-1\bigr)+1\bigr)_{\bm,d}}
\le \frac{d_\bm^{\tilde{\fp}}}{\bigl(\frac{d}{2}\bigl(\bigl\lfloor\frac{r}{2}\bigr\rfloor-1\bigr)+1\bigr)_{\bm,d}} \qquad (x_1'\in\Sigma_1,\, y_2\in\Sigma_2). 
\end{align*}
We also have 
\[ d_{\phi_\pm(\bm,\bl)}^{\fp^+_2}=\dim\cP_{\phi_\pm(\bm,\bl)}(\fp^+_2)\le \dim\bigl(\cP_{\bm^2}(\fp^+_2)\otimes\cP_{(l,\underline{0}_{r-1})}(\fp^\pm_2)\bigr)
=d_{\bm^2}^{\fp^+_2}d_{(l,\underline{0}_{r-1})}^{\fp^+_2}. \]
Combining these, we get the desired formulas. 
\end{proof}

The following proposition gives some examples of $F_{\bm,\bl}^\pm[f](x_1,y_2)$. 
\begin{proposition}\label{prop_example_Fml_general}
Let $\bm\in\BZ_{++}^{\lfloor r/2\rfloor}$, $\bl\in(\BZ_{\ge 0})^{\lceil r/2\rceil}$, $|\bl|=l$. 
\begin{enumerate}
\item Suppose $r$ is even. For $k\in\BZ_{\ge 0}$, $f(y_2)\in\cP_{(l,\underline{0}_{r-1})}(\fp^-_2)$, we have
\begin{align*}
F_{\bm+\underline{k}_{r/2},\bl}^\pm[f](x_1,y_2)
&=\frac{1}{(-k-\bm^\vee)_{\underline{k}_{r/2},d}}F_{\bm,\bl}^\pm[f](x_1,y_2)\det_{\fn^+}(x_1)^k\det_{\fn^-_2}(y_2)^k, 
\end{align*}
where $\bm^\vee=(m_{\lfloor r/2\rfloor},\ldots,m_1)$. 
\item Suppose $r$ is odd, and let $w_2\in\fp^+_2$ be of rank 1. Then for $k,l,m\in\BZ_{\ge 0}$ with $m\le k,l$, we have 
\begin{align*}
&F_{\underline{k}_{\lfloor r/2\rfloor},(l-m,\underline{0}_{\lfloor r/2\rfloor -1},m)}^+\bigl[(\cdot|w_2)_{\fn^+_2}^l\bigr](x_1,y_2) \\
&=\frac{(-1)^m\det_{\fn^+_2}(y_2)^k}{(-k)_{\underline{k}_{\lfloor r/2\rfloor},d}\bigl(-k-l+m-\frac{d_2}{2}(r-1)\bigr)_mm!} \\*
&\eqspace{}\times\sum_{j=m}^{\min\{k,l\}}\frac{(-k)_j(-l)_j(y_2^\itinv|(x_1)^\sharp)_{\fn^+_2}^{k-j}(y_2|w_2)_{\fn^+_2}^{l-j}((x_1)^\sharp|w_2)^j_{\fn^+_2}}
{\bigl(-k-l+2m-\frac{d_2}{2}(r-1)+1\bigr)_{j-m}(j-m)!}. 
\end{align*}
\item Suppose $r$ is odd, and let $w_2\in\fp^+_2$ be of rank 1. Then for $k,l,m\in\BZ_{\ge 0}$ with $m\le k,l$, we have 
\begin{align*}
&F_{\underline{k}_{\lfloor r/2\rfloor},(\underline{0}_{\lfloor r/2\rfloor -1},m,l-m)}^-\bigl[(\cdot|w_2)_{\fn^+_2}^l\bigr](x_1,y_2) \\
&=\frac{(-1)^m\det_{\fn^+_2}(y_2)^k}{(-k)_{\underline{k}_{\lfloor r/2\rfloor},d}\bigl(-k-l+m-\frac{d_2}{2}\bigr)_mm!}
\sum_{j=m}^{\min\{k,l\}}\frac{(-k)_j(-l)_j}{\bigl(-k-l+2m-\frac{d_2}{2}+1\bigr)_{j-m}(j-m)!} \\*
&\eqspace{}\times(y_2^\itinv|(x_1)^\sharp)_{\fn^+_2}^{k-j}(y_2^\itinv|w_2)_{\fn^+_2}^{l-j}
\bigl((y_2^\itinv|(x_1)^\sharp)_{\fn^+_2}(y_2^\itinv|w_2)_{\fn^+_2}-(P(y_2)^{-1}(x_1)^\sharp|w_2)_{\fn^+_2}\bigr)^j. 
\end{align*}
\item Suppose $r\ge 3$. Let $\bm\in\BZ_{++}^{\lceil r/2\rceil-1}$, $l\in\BZ_{\ge 0}$, $m_1\le l$, and put $m_{r/2}:=0$ for even $r$. 
We define $\bl(\bm)\in(\BZ_{\ge 0})^{\lceil r/2\rceil}$ by 
\[ \bl(\bm):=(l-m_1,m_1-m_2,m_2-m_3,\ldots,m_{\lceil r/2\rceil-2}-m_{\lceil r/2\rceil-1},m_{\lceil r/2\rceil-1}), \]
so that 
\begin{align*}
\phi_+(\bm,\bl(\bm))&=\begin{cases} (l,m_1,m_1,m_2,m_2,\ldots,m_{r/2-1},m_{r/2-1},0) & (r\colon\text{even}), \\
(l,m_1,m_1,m_2,m_2,\ldots,m_{\lceil r/2\rceil-1},m_{\lceil r/2\rceil-1}) & (r\colon\text{odd}) \end{cases} \\
&=(l,\bm^2)\in\BZ_{++}^r. 
\end{align*}
Let $e'\in\fn^+_2$ be a primitive idempotent, let $\fp^+(e')_0=:\fp^{+\prime\prime}$ be as in (\ref{Peirce}), let $\fp^+_1\cap\fp^{+\prime\prime}\allowbreak=:\fp^{+\prime\prime}_1$, 
$\fp^+_2\cap\fp^{+\prime\prime}\allowbreak=:\fp^{+\prime\prime}_2$, 
and for $x\in\fp^+$ let $x''\in\fp^{+\prime\prime}$ be the orthogonal projection. Then we have 
\begin{align*}
\det_{\fn^+_2}\hspace{-1pt}(x_2)^l F_{\bm,\bl(\bm)}^+\hspace{-1pt}\bigl[(\cdot|e')_{\fn^+_2}^l\bigr]\hspace{-1pt}(x_1,x_2^\itinv) 
&=\frac{(-l)_{\bm,d}}{\bigl(-l\hspace{-1pt}-\hspace{-1pt}\frac{d}{2}\bigr)_{\bm,d}}\det_{\fn^{+\prime\prime}_2}\hspace{-1pt}(x_2'')^l
\tilde{\Phi}_\bm^{\fp^{+\prime\prime}_1,\fp^{-\prime\prime}_2}\hspace{-4pt}(x_1'',(x_2'')^\itinv). 
\end{align*}
\end{enumerate}
\end{proposition}

\begin{proof}
(1) Follows from 
\begin{align*}
\tilde{\Phi}_{\bm+\underline{k}_{r/2}}^{\fp^+_1,\fp^-_2}(x_1,y_2)
&=\frac{d_{\bm+\underline{k}_{r/2}}^{\tilde{\fp}}}{\bigl(\frac{d}{2}\bigl(\frac{r}{2}-1\bigr)+1\bigr)_{\bm+\underline{k}_{r/2},d}} 
\Phi_{\bm+\underline{k}_{r/2}}^{\fp^+_1,\fp^-_2}(x_1,y_2) \\
&=\frac{d_{\bm}^{\tilde{\fp}}}{\bigl(\frac{d}{2}\bigl(\frac{r}{2}-1\bigr)+1\bigr)_{\bm+\underline{k}_{r/2},d}} 
\Phi_{\bm}^{\fp^+_1,\fp^-_2}(x_1,y_2)\Phi_{\underline{1}_{r/2}}^{\fp^+_1,\fp^-_2}(x_1,y_2)^k \\
&=\frac{\bigl(\frac{d}{2}\bigl(\frac{r}{2}-1\bigr)+1\bigr)_{\bm,d}}
{\bigl(\frac{d}{2}\bigl(\frac{r}{2}-1\bigr)+1\bigr)_{\bm+\underline{k}_{r/2},d}} 
\tilde{\Phi}_{\bm}^{\fp^+_1,\fp^-_2}(x_1,y_2)\Phi_{\underline{1}_{r/2}}^{\fp^+_1,\fp^-_2}(x_1,y_2)^k \\
&=\frac{(-1)^{kr/2}}{(-k-\bm^\vee)_{\underline{k}_{r/2},d}} \tilde{\Phi}_{\bm}^{\fp^+_1,\fp^-_2}(x_1,y_2)\Phi_{\underline{1}_{r/2}}^{\fp^+_1,\fp^-_2}(x_1,y_2)^k \\
&=\frac{1}{(-k-\bm^\vee)_{\underline{k}_{r/2},d}} \tilde{\Phi}_{\bm}^{\fp^+_1,\fp^-_2}(x_1,y_2) \det_{\fn^+}(x_1)^k\det_{\fn^-_2}(y_2)^k. 
\end{align*}

(2) Since $r$ is odd, $x_1\in\fp^+_1$ is not invertible, and $(x_1)^\sharp\in\fp^+_2$ is at most of rank 1. By Example \ref{example_Phim_general} and Proposition \ref{prop_Proj}\,(1), we have 
\begin{align*}
&F_{\underline{k}_{\lfloor r/2\rfloor},(l-m,\underline{0}_{\lfloor r/2\rfloor -1},m)}^+\bigl[(\cdot|w_2)_{\fn^+_2}^l\bigr](x_1,y_2)
=\Proj_{(k+l-m,\underline{k}_{r-2},m),y_2}^{\fp^-_2}\bigl(\tilde{\Phi}_{\underline{k}_{\lfloor r/2\rfloor}}^{\fp^+_1,\fp^-_2}(x_1,y_2)(y_2|w_2)_{\fn^+_2}^l\bigr) \\
&=\frac{(-1)^{k\lfloor r/2\rfloor}\det_{\fn^+_2}(y_2)^k}{\bigl(\frac{d}{2}\bigl(\bigl\lfloor\frac{r}{2}\bigr\rfloor-1\bigr)+1\bigr)_{\underline{k}_{\lfloor r/2\rfloor},d}}
\Proj_{(l-m,\underline{0}_{r-2},m-k),y_2}^{\fp^-_2}\bigl((y_2^\itinv|(x_1)^\sharp)_{\fn^+_2}^k(y_2|w_2)_{\fn^+_2}^l\bigr) \\
&=\frac{\det_{\fn^+_2}(y_2)^k}{(-k)_{\underline{k}_{\lfloor r/2\rfloor},d}}
\frac{(-1)^m}{\bigl(-k-l+m-\frac{d_2}{2}(r-1)\bigr)_mm!} \\*
&\eqspace{}\times\sum_{j=m}^{\min\{k,l\}}\frac{(-k)_j(-l)_j(y_2^\itinv|(x_1)^\sharp)_{\fn^+_2}^{k-j}(y_2|w_2)_{\fn^+_2}^{l-j}((x_1)^\sharp|w_2)^j_{\fn^+_2}}
{\bigl(-k-l+2m-\frac{d_2}{2}(r-1)+1\bigr)_{j-m}(j-m)!}. 
\end{align*}

(3) By Example \ref{example_Phim_general} and Proposition \ref{prop_Proj}\,(2), we have 
\begin{align*}
&F_{\underline{k}_{\lfloor r/2\rfloor},(\underline{0}_{\lfloor r/2\rfloor -1},m,l-m)}^-\bigl[(\cdot|w_2)_{\fn^+_2}^l\bigr](x_1,y_2)
=\Proj_{(\underline{k}_{r-2},k-m,m-l),y_2}^{\fp^-_2}\bigl(\tilde{\Phi}_{\underline{k}_{\lfloor r/2\rfloor}}^{\fp^+_1,\fp^-_2}(x_1,y_2)(y_2^\itinv|w_2)_{\fn^+_2}^l\bigr) \\
&=\frac{(-1)^{k\lfloor r/2\rfloor}\det_{\fn^+_2}(y_2)^k}{\bigl(\frac{d}{2}\bigl(\bigl\lfloor\frac{r}{2}\bigr\rfloor-1\bigr)+1\bigr)_{\underline{k}_{\lfloor r/2\rfloor},d}}
\Proj_{(\underline{0}_{r-2},-m,m-k-l),y_2}^{\fp^-_2}\bigl((y_2^\itinv|(x_1)^\sharp)_{\fn^+_2}^k(y_2^\itinv|w_2)_{\fn^+_2}^l\bigr) \\
&=\frac{\det_{\fn^+_2}(y_2)^k}{(-k)_{\underline{k}_{\lfloor r/2\rfloor},d}}\frac{(-1)^m}{\bigl(-k-l+m-\frac{d_2}{2}\bigr)_mm!}
\sum_{j=m}^{\min\{k,l\}}\frac{(-k)_j(-l)_j}{\bigl(-k-l+2m-\frac{d_2}{2}+1\bigr)_{j-m}(j-m)!} \\*
&\eqspace{}\times(y_2^\itinv|(x_1)^\sharp)_{\fn^+_2}^{k-j}(y_2^\itinv|w_2)_{\fn^+_2}^{l-j}
\bigl((y^\itinv_2|(x_1)^\sharp)_{\fn^+_2}(y_2^\itinv|w_2)_{\fp^+_2}-(P(y_2)^{-1}(x_1)^\sharp|w_2)_{\fn^+_2}\bigr)^j. 
\end{align*}

(4) Let $\fp_2^+(e')_2=\BC e'=:\fp^{+\prime}_2$, and let $P_{K_1}\subset K_1^\BC$ be a maximal parabolic subgroup stabilizing 
$\cP(\fp_1^{+\prime\prime})\subset\cP(\fp_1^+)$ and $\cP(\fp_2^{+\prime\prime})\subset\cP(\fp_2^+)$. 
First, we have 
\begin{align*}
&\bigl(\det_{\fn^+_2}(x_2)^lf(x_2^\itinv)\mapsto \det_{\fn^+_2}(x_2)^lF_{\bm,\bl(\bm)}^+[f](x_1,x_2^\itinv)\bigr) \\
&\in\Hom_{K_1}\bigl(\cP_{(\underline{l}_{r-1},0)}(\fp^+_2),\cP_{\langle\bm\rangle}(\fp^+_1)\otimes\cP_{\underline{l}_r-\phi_+(\bm,\bl(\bm))^\vee}(\fp^+_2)\bigr) \\
&\simeq \Hom_{K_1}\bigl(\cP_{(l,\underline{0}_{r-1})}(\fp^-_2)\otimes\cP_{\langle\bm\rangle}(\fp^-_1),\cP_{(l,\bm^2)}(\fp^-_2)\bigr), 
\end{align*}
and this space is 1-dimensional by the Pieri rule for Cases 2, 3. Especially, since \linebreak 
$\det_{\fn^+_2}(x_2)^l(x_2^\itinv|e')_{\fn^+_2}^l=\det_{\fn^{+\prime\prime}_2}(x_2'')^l\in\cP_{(\underline{l}_{r-1},0)}(\fp^+_2)$ is relatively invariant under the para\-bolic subgroup $P_{K_1}$, 
$\det_{\fn^+_2}(x_2)^lF_{\bm,\bl(\bm)}^+\bigl[(\cdot|e')_{\fn^+_2}^l\bigr](x_1,x_2^\itinv)\in\cP_{\langle\bm\rangle}(\fp^+_1)\otimes\cP_{\underline{l}_r-(l,\bm^2)^\vee}(\fp^+_2)$ 
is also relatively invariant under $P_{K_1}$, with the weight same with $\det_{\fn^+_2}(x_2)^l(x_2^\itinv|e')_{\fn^+_2}^l=\det_{\fn^{+\prime\prime}_2}(x_2'')^l$, 
and such function is unique up to constant multiple. Hence 
\[ \det_{\fn^+_2}(x_2)^lF_{\bm,\bl(\bm)}^+\bigl[(\cdot|e')_{\fn^+_2}^l\bigr](x_1,x_2^\itinv)
=c\det_{\fn^{+\prime\prime}_2}(x_2'')^l\tilde{\Phi}_\bm^{\fp^{+\prime\prime}_1,\fp^{-\prime\prime}_2}(x_1'',(x_2'')^\itinv) \]
holds for some $c\in\BC$. To determine $c$, it is enough to consider the case $x_1=x_1''\in\fp^{+\prime\prime}$. 
Then since $\tilde{\Phi}_\bm^{\fp^+_1,\fp^-_2}(x_1'',y_2)=\tilde{\Phi}_\bm^{\fp^{+\prime\prime}_1,\fp^{-\prime\prime}_2}(x_1'',y_2'')$ holds, by Lemma \ref{lem_Proj_Pieri}\,(4) we have 
\begin{align*}
&F_{\bm,\bl(\bm)}^+\bigl[(\cdot|e')_{\fn^+_2}^l\bigr](x_1'',x_2^\itinv)\bigr|_{x_2=x_2'+x_2''\in\fp^{+\prime}\oplus\fp^{+\prime\prime}} \\
&=\Proj_{(l,\bm^2)}^{\fp^-_2}\bigl((y_2|e')_{\fn^+_2}^l\tilde{\Phi}_\bm^{\fp^{+\prime\prime}_1,\fp^{-\prime\prime}_2}(x_1'',y_2'')\bigr)
\bigr|_{y_2^\itinv=x_2'+x_2''\in\fp^{+\prime}\oplus\fp^{+\prime\prime}} \\
&=\frac{(-l)_{\bm^2,d_2}}{\bigl(-l-\frac{d_2}{2}\bigr)_{\bm^2,d_2}}((x_2')^\itinv|e')_{\fn^+_2}^l\tilde{\Phi}_\bm^{\fp^{+\prime\prime}_1,\fp^{-\prime\prime}_2}(x_1'',(x_2'')^\itinv), 
\end{align*}
and hence we get 
\[ c=\frac{(-l)_{\bm^2,d_2}}{\bigl(-l-\frac{d_2}{2}\bigr)_{\bm^2,d_2}}
=\frac{(-l)_{\bm,d}\bigl(-l-\frac{d_2}{2}\bigr)_{\bm,d}}{\bigl(-l-\frac{d_2}{2}\bigr)_{\bm,d}\bigl(-l-\frac{d}{2}\bigr)_{\bm,d}}
=\frac{(-l)_{\bm,d}}{\bigl(-l-\frac{d}{2}\bigr)_{\bm,d}}. \qedhere \] 
\end{proof}

Using $F_{\bm,\bl}^\pm[f](x_1,y_2)$, for $l\in\BZ_{\ge 0}$, $\mu,\nu\in\BC$ and for $f(x_2)\in\cP_{(l,0,\ldots,0)}(\fp^\pm_2)$, 
we define a function $F_{l,+}^\downarrow\bigl(\begin{smallmatrix} \nu \\ \mu \end{smallmatrix};f;x\bigr)$ by 
\begin{align}
F_{l,+}^\downarrow\biggl(\begin{matrix} \nu \\ \mu \end{matrix};f;x\biggr):\hspace{-3pt}&=\det_{\fn^\pm_2}(x_2)^{-\nu}
\sum_{\bm\in\BZ_{++}^{\lfloor r/2\rfloor}}\sum_{\substack{\bl\in(\BZ_{\ge 0})^{\lceil r/2\rceil} \\ |\bl|=l}} F_{\bm,\bl}^-[f](x_1,x_2^\itinv) \notag\\*
&\eqspace{}\times \begin{cases}
\ds \frac{(\nu)_{\bm,d}\bigl(\nu-\frac{d_2}{2}-(\underline{0}_{r/2-1},l)\bigr)_{\bm-\bl+(\underline{0}_{r/2-1},l),d}}
{\bigl(\mu-\frac{d_2}{2}-(\underline{0}_{r/2-1},l)\bigr)_{\bm-\bl+(\underline{0}_{r/2-1},l),d}} & (r\colon\text{even}), \\
\ds \frac{(\nu)_{\bm,d}\bigl(\nu-\frac{d_2}{2}\bigr)_{\bm-\bl',d}\bigl(\nu-\frac{d_2}{2}(r-1)-l\bigr)_{l-l_{\lceil r/2\rceil}}}
{\bigl(\mu-\frac{d_2}{2}\bigr)_{\bm-\bl',d}\bigl(\mu-\frac{d_2}{2}(r-1)-l\bigr)_{l-l_{\lceil r/2\rceil}}} & (r\colon\text{odd}), \end{cases} \label{def_Fdown_general+}
\end{align}
where, when $r$ is odd, let $\bl':=(l_1,\ldots,l_{\lfloor r/2\rfloor})$ so that $\bl=(\bl',l_{\lceil r/2\rceil})$ holds, 
and we define a function $F_{l,-}^\downarrow\bigl(\begin{smallmatrix} \nu \\ \mu \end{smallmatrix};f;x\bigr)$ by 
\begin{align}
F_{l,-}^\downarrow\biggl(\begin{matrix} \nu \\ \mu \end{matrix};f;x\biggr)
&:=\det_{\fn^\pm_2}(x_2)^{-\nu}\sum_{\bm\in\BZ_{++}^{\lfloor r/2\rfloor}}\sum_{\substack{\bl\in(\BZ_{\ge 0})^{\lceil r/2\rceil} \\ |\bl|=l}} F_{\bm,\bl}^+[f](x_1,x_2^\itinv) \notag\\*
&\eqspace{}\times \frac{(\nu+(l,\underline{0}_{\lceil r/2\rceil-1}))_{\bm+\bl-(l,\underline{0}_{\lceil r/2\rceil-1}),d}\bigl(\nu-\frac{d_2}{2}\bigr)_{\bm,d}}
{\bigl(\mu-\frac{d_2}{2}+(l,\underline{0}_{\lceil r/2\rceil-1})\bigr)_{\bm+\bl-(l,\underline{0}_{\lceil r/2\rceil-1}),d}}, \label{def_Fdown_general-}
\end{align}
where we identify $\bm\in\BZ_{++}^{\lfloor r/2\rfloor}$ and $(\bm,0)\in\BZ_{++}^{\lceil r/2\rceil}$ when $r$ is odd. 
Then as a function of $\mu$, $F_{l,+}^\downarrow\bigl(\begin{smallmatrix} \nu \\ \mu \end{smallmatrix};f;x\bigr)$ has poles at 
\begin{align*}
\mu&\in \bigcup_{i=1}^{r/2-1}\biggl\{ \frac{d_2}{2}+\frac{d}{2}(i-1)-j \biggm| \begin{matrix} j\in\BZ, \\ j\ge 0 \end{matrix}\biggr\}
\cup \biggl\{ \frac{d_2}{2}+\frac{d}{2}\biggl(\frac{r}{2}-1\biggr)-j \biggm| \begin{matrix} j\in\BZ, \\ j\ge -l \end{matrix}\biggr\} && (r\colon\text{even}), \\
\mu&\in \bigcup_{i=1}^{\lfloor r/2\rfloor}\biggl\{ \frac{d_2}{2}+\frac{d}{2}(i-1)-j \biggm| \begin{matrix} j\in\BZ, \\ j\ge 0 \end{matrix}\biggr\}
\cup \biggl\{ \frac{d_2}{2}+\frac{d}{2}\biggl(\frac{r}{2}-1\biggr)-j \biggm| \begin{matrix} j\in\BZ, \\ -l\le j<0 \end{matrix} \biggr\} 
 &&(r\colon\text{odd}) 
\end{align*}
(we note that $d=2d_2$ holds if $r$ is odd), and $F_{l,-}^\downarrow\bigl(\begin{smallmatrix} \nu \\ \mu \end{smallmatrix};f;x\bigr)$ has poles at 
\begin{align*}
\mu&\in \biggl\{ \frac{d_2}{2}-j \biggm| \begin{matrix} j\in\BZ, \\ j\ge l \end{matrix}\biggr\}
\cup{} \bigcup_{i=2}^{r/2}{}\biggl\{ \frac{d_2}{2}+\frac{d}{2}(i-1)-j \biggm| \begin{matrix} j\in\BZ, \\ j\ge 0 \end{matrix}\biggr\} && (r\colon\text{even}), \\
\mu&\in \biggl\{ \frac{d_2}{2}-j \biggm| \begin{matrix} j\in\BZ, \\ j\ge l \end{matrix}\biggr\}
\cup{} \bigcup_{i=2}^{\lfloor r/2\rfloor}\biggl\{ \frac{d_2}{2}+\frac{d}{2}(i-1)-j \biggm| \begin{matrix} j\in\BZ, \\ j\ge 0 \end{matrix}\biggr\} \\*
&\eqspace{}\cup \biggl\{ \frac{d_2}{2}+\frac{d}{2}\biggl(\biggl\lceil\frac{r}{2}\biggr\rceil-1\biggr)-j 
\biggm| \begin{matrix} j\in\BZ, \\ 0\le j<l \end{matrix} \biggr\} && (r\colon\text{odd}) 
\end{align*}
for generic $\nu\in\BC$. On the other hand, if $\nu=-k\in-\BZ_{\ge 0}$, then $F_{l,+}^\downarrow\bigl(\begin{smallmatrix} -k \\ \mu \end{smallmatrix};f;x\bigr)$ 
becomes a finite sum and is a polynomial in $x$, 
and if $\nu=-k\in-\BZ_{\ge l}$, then $F_{l,-}^\downarrow\bigl(\begin{smallmatrix} -k \\ \mu \end{smallmatrix};f;x\bigr)$ also becomes a polynomial. 
For these cases, all poles satisfying $\mu\le \frac{d_2}{2}-k$ become removable, 
and we use the same notations $F_{l,+}^\downarrow\bigl(\begin{smallmatrix} -k \\ \mu \end{smallmatrix};f;x\bigr)$, 
$F_{l,-}^\downarrow\bigl(\begin{smallmatrix} -k \\ \mu \end{smallmatrix};f;x\bigr)$ to express these polynomials also for $\mu$ at such removable singularities. 
If $r=3$, then since $d=2d_2$ holds, $F_{l,\pm}^\downarrow\bigl(\begin{smallmatrix} -k \\ \mu \end{smallmatrix};f;x\bigr)$ 
and $F^\downarrow(\lambda,\bk;f;x)$ in the previous section are related as 
\begin{align*}
F^\downarrow\Bigl(\lambda,(k_1,k_2,k_2);f\det_{\fn^+_2}^{k_2};x\Bigr)
&=F_{k_1-k_2,+}^\downarrow\biggl(\begin{matrix} -k_2 \\ -\lambda-2k_2+d+1 \end{matrix};f;x\biggr), \\
F^\downarrow\Bigl(\lambda,(k_2,k_2,k_3);f((\cdot)^\itinv)\det_{\fn^+_2}^{k_2};x\Bigr)
&=F_{k_2-k_3,-}^\downarrow\biggl(\begin{matrix} -k_2 \\ -\lambda-2k_2+d+1 \end{matrix};f;x\biggr). 
\end{align*}

\begin{proposition}
Suppose $\mu\in\BC$ is not a pole of $F_{l,\pm}^\downarrow\bigl(\begin{smallmatrix} \nu \\ \mu \end{smallmatrix};f;x\bigr)$. 
Then $F_{l,\pm}^\downarrow\bigl(\begin{smallmatrix} \nu \\ \mu \end{smallmatrix};f;x\bigr)$ converges if $x_2=ge$ for some $g\in K_1^\BC$ and $|g^{-1}x_1|_\infty<1$ holds. 
Especially, this converges for $x=(x_1,x_2)\in\Omega$. 
\end{proposition}
\begin{proof}
By Proposition \ref{prop_estimate_Fml_general} we have 
\begin{align*}
|F_{\bm,\bl}^-[f](x_1,x_2^\itinv)|&\le \frac{\sqrt{d_{\bm^2}^{\fp^+_2}d_{(l,\underline{0}_{r-1})}^{\fp^+_2}}d_\bm^{\tilde{\fp}}}
{\bigl(\frac{d}{2}\bigl(\bigl\lfloor\frac{r}{2}\bigr\rfloor-1\bigr)+1\bigr)_{\bm,d}}
(|g^{-1} x_1|_\infty)^{2|\bm|}\Vert f(g(\cdot))\Vert_{L^2(\Sigma_2)}, \\
|F_{\bm,\bl}^+[f](x_1,x_2^\itinv)|&\le \frac{\sqrt{d_{\bm^2}^{\fp^+_2}d_{(l,\underline{0}_{r-1})}^{\fp^+_2}}d_\bm^{\tilde{\fp}}}
{\bigl(\frac{d}{2}\bigl(\bigl\lfloor\frac{r}{2}\bigr\rfloor-1\bigr)+1\bigr)_{\bm,d}}
(|g^{-1} x_1|_\infty)^{2|\bm|}\Vert f({}^t\hspace{-1pt}g^{-1}(\cdot))\Vert_{L^2(\Sigma_2)}, 
\end{align*}
and we have 
\begin{align*}
d_{\bm^2}^{\fp^+_2}&=\dim\cP_{\bm^2}(\fp^+_2)\le \dim\cP_{2|\bm|}(\fp^+_2)=\binom{2|\bm|+n_2-1}{n_2-1}, \\
d_\bm^{\tilde{\fp}}&=\dim\cP_\bm(\tilde{\fp})\le \dim\cP_{|\bm|}(\tilde{\fp})=\binom{|\bm|+\tilde{n}-1}{\tilde{n}-1}, 
\end{align*}
where $n_2:=\dim\fp^+_2$, $\tilde{n}:=\dim\tilde{\fp}$. 
Then the convergence for $|g^{-1}x_1|_\infty<1$ follows from standard argument. Especially when $x=(x_1,x_2)\in\Omega$, we have 
\begin{align*}
0< \det_{\fn^+}(x)&=\det_{\fn^+_2}(x_2)h_{\fn^+}(x_1,P(x_2)^{-1}x_1) =\det_{\fn^+_2}(x_2)h_{\fn^+}\bigl(P(x_2^{-\mathit{1/2}})x_1,P(x_2^{-\mathit{1/2}})x_1\bigr), 
\end{align*}
and hence $P(x_2^{-\mathit{1/2}})x_1\in D\cap \Omega$, that is, $\bigl|P(x_2^{-\mathit{1/2}})x_1\bigr|_\infty<1$ holds. 
Thus $F_{l,\pm}^\downarrow\bigl(\begin{smallmatrix} \nu \\ \mu \end{smallmatrix};f;x\bigr)$ converges for $x\in\Omega$. 
\end{proof}

These $F_{l,\pm}^\downarrow\bigl(\begin{smallmatrix}\nu \\ \mu\end{smallmatrix};f;x\bigr)$ have the following property. 
\begin{proposition}\label{prop_shift_general}
For $\mu,\nu\in\BC$, $k,l\in\BZ_{\ge 0}$, $f(x_2)\in\cP_{(l,\underline{0}_{r-1})}(\fp^\pm_2)$, we have 
\begin{align*}
&\det_{\fn^+_2}\biggl(\frac{\partial}{\partial x_2}\biggr)^k F_{l,+}^\downarrow\biggl(\begin{matrix}\nu \\ \mu\end{matrix};f;x\biggr)
=(-1)^{kr}(\nu-(\underline{0}_{r-1},l))_{\underline{k}_r,d_2} F_{l,+}^\downarrow\biggl(\begin{matrix}\nu+k \\ \mu\end{matrix};f;x\biggr) \\
&=(-1)^{kr}(\nu)_{\underline{k}_{\lfloor r/2\rfloor},d} \biggl(\nu-\frac{d_2}{2}\biggr)_{\underline{k}_{\lceil r/2\rceil-1},d}
\biggl(\nu-\frac{d_2}{2}(r-1)-l\biggr)_k F_{l,+}^\downarrow\biggl(\begin{matrix}\nu+k \\ \mu\end{matrix};f;x\biggr), \\
&\det_{\fn^+_2}\biggl(\frac{\partial}{\partial x_2}\biggr)^k F_{l,-}^\downarrow\biggl(\begin{matrix}\nu \\ \mu\end{matrix};f;x\biggr)
=(-1)^{kr}(\nu+(l,\underline{0}_{r-1}))_{\underline{k}_r,d_2} F_{l,-}^\downarrow\biggl(\begin{matrix}\nu+k \\ \mu\end{matrix};f;x\biggr) \\
&=(-1)^{kr}(\nu+(l,\underline{0}_{\lceil r/2\rceil-1}))_{\underline{k}_{\lceil r/2\rceil},d} \biggl(\nu-\frac{d_2}{2}\biggr)_{\underline{k}_{\lfloor r/2\rfloor},d}
F_{l,-}^\downarrow\biggl(\begin{matrix}\nu+k \\ \mu\end{matrix};f;x\biggr). 
\end{align*}
\end{proposition}
\begin{proof}
First, the numerator of each term of $F_{l,\pm}^\downarrow\bigl(\begin{smallmatrix}\nu \\ \mu\end{smallmatrix};f;x\bigr)$ satisfies 
\begin{align*}
&(\nu-(\underline{0}_{r-1},l))_{\phi_-(\bm,\bl)+(\underline{0}_{r-1},l),d_2} \\
&=\begin{cases}
\ds (\nu)_{\bm,d}\biggl(\nu-\frac{d_2}{2}-(\underline{0}_{r/2-1},l)\biggr)_{\bm-\bl+(\underline{0}_{r/2-1},l),d} & (r\colon\text{even}), \\
\ds (\nu)_{\bm,d}\biggl(\nu-\frac{d_2}{2}\biggr)_{\bm-\bl',d}\biggl(\nu-\frac{d_2}{2}(r-1)-l\biggr)_{l-l_{\lceil r/2\rceil}} & (r\colon\text{odd}), \end{cases} \\
&(\nu+(l,\underline{0}_{r-1}))_{\phi_+(\bm,\bl)-(l,\underline{0}_{r-1}),d_2}
=(\nu+(l,\underline{0}_{\lceil r/2\rceil-1}))_{\bm+\bl-(l,\underline{0}_{\lceil r/2\rceil-1}),d}\biggl(\nu-\frac{d_2}{2}\biggr)_{\bm,d}. 
\end{align*}
Then since $F_{\bm,\bl}^\mp[f](x_1,y_2)\in\cP(\fp^+_1)\otimes\cP_{\phi_\mp(\bm,\bl)}(\fp^-_2)$ holds, by Lemma \ref{lem_diff} we have 
\begin{align*}
&\det_{\fn^+_2}\biggl(\frac{\partial}{\partial x_2}\biggr)^k(\nu-(\underline{0}_{r-1},l))_{\phi_-(\bm,\bl)+(\underline{0}_{r-1},l),d_2}\det_{\fn^+_2}(x_2)^{-\nu}F_{\bm,\bl}^-[f](x_1,x_2^\itinv) \\
&=(-1)^{kr}(\nu+\phi_-(\bm,\bl))_{\underline{k}_r,d_2}(\nu-(\underline{0}_{r-1},l))_{\phi_-(\bm,\bl)+(\underline{0}_{r-1},l),d_2} \\*
&\hspace{195pt}\times\det_{\fn^+_2}(x_2)^{-\nu-k}F_{\bm,\bl}^-[f](x_1,x_2^\itinv) \\
&=(-1)^{kr}(\nu-(\underline{0}_{r-1},l))_{\phi_-(\bm,\bl)+\underline{k}_r+(\underline{0}_{r-1},l),d_2}\det_{\fn^+_2}(x_2)^{-\nu-k}F_{\bm,\bl}^-[f](x_1,x_2^\itinv) \\
&=(-1)^{kr}(\nu-(\underline{0}_{r-1},l))_{\underline{k}_r,d_2}(\nu+k-(\underline{0}_{r-1},l))_{\phi_-(\bm,\bl)+(\underline{0}_{r-1},l),d_2} \\*
&\hspace{195pt}\times\det_{\fn^+_2}(x_2)^{-\nu-k}F_{\bm,\bl}^-[f](x_1,x_2^\itinv), \\
&\det_{\fn^+_2}\biggl(\frac{\partial}{\partial x_2}\biggr)^k(\nu+(l,\underline{0}_{r-1}))_{\phi_+(\bm,\bl)-(l,\underline{0}_{r-1}),d_2}\det_{\fn^+_2}(x_2)^{-\nu}F_{\bm,\bl}^+[f](x_1,x_2^\itinv) \\
&=(-1)^{kr}(\nu+\phi_+(\bm,\bl))_{\underline{k}_r,d_2}(\nu+(l,\underline{0}_{r-1}))_{\phi_+(\bm,\bl)-(l,\underline{0}_{r-1}),d_2} \\*
&\hspace{195pt}\times\det_{\fn^+_2}(x_2)^{-\nu-k}F_{\bm,\bl}^+[f](x_1,x_2^\itinv) \\
&=(-1)^{kr}(\nu+(l,\underline{0}_{r-1}))_{\phi_+(\bm,\bl)+\underline{k}_r-(l,\underline{0}_{r-1}),d_2}\det_{\fn^+_2}(x_2)^{-\nu-k}F_{\bm,\bl}^+[f](x_1,x_2^\itinv) \\
&=(-1)^{kr}(\nu+(l,\underline{0}_{r-1}))_{\underline{k}_r,d_2}(\nu+k+(l,\underline{0}_{r-1}))_{\phi_+(\bm,\bl)-(l,\underline{0}_{r-1}),d_2} \\*
&\hspace{195pt}\times\det_{\fn^+_2}(x_2)^{-\nu-k}F_{\bm,\bl}^+[f](x_1,x_2^\itinv). 
\end{align*}
Hence the 1st and 3rd equalities follow from the definition of $F_{l,\pm}^\downarrow\bigl(\begin{smallmatrix}\nu \\ \mu\end{smallmatrix};f;x\bigr)$. 
The 2nd and 4th equalities are clear. 
\end{proof}

\subsection{Results on weighted Bergman inner products}

For $\lambda\in\BC$, $k,l\in\BZ_{\ge 0}$, we set 
\begin{align}
&C_{r,+}^{d,d_2}(\lambda,k,l) \notag\\
&:=\begin{cases} \ds \frac{\bigl(\lambda+k-\frac{d}{4}r+\frac{d_2}{2}+(l,\underline{0}_{r/2-1})\bigr)_{\underline{k}_{r/2},d}}
{(\lambda)_{(2k+l,\underline{2k}_{r/2-1},\underline{k}_{r/2}),d}} & (r\colon\text{even}), \\
\ds \frac{\bigl(\lambda-\frac{d}{2}\bigl\lfloor \frac{r}{2}\bigr\rfloor+\max\{2k,k+l\}\bigr)_{\min\{k,l\}}
\bigl(\lambda+k-\frac{d}{2}\bigl\lceil \frac{r}{2}\bigr\rceil+\frac{d_2}{2}\bigr)_{\underline{k}_{\lfloor r/2\rfloor},d}}
{(\lambda)_{(2k+l,\underline{2k}_{\lfloor r/2\rfloor-1},\min\{2k,k+l\},\underline{k}_{\lfloor r/2\rfloor}),d}} & (r\colon\text{odd}), \end{cases} \label{const_general+}
\end{align}
and when $k\ge l$, we set 
\begin{align}
C_{r,-}^{d,d_2}(\lambda,k,l) 
&:=\begin{cases} 
\ds \frac{\bigl(\lambda+k-\frac{d}{4}r+\frac{d_2}{2}\bigr)_{(\underline{k}_{r/2-1},k-l),d}}{(\lambda)_{(\underline{2k}_{r/2-1},2k-l,\underline{k}_{r/2-1},k-l),d}} & (r\colon\text{even}), \\
\ds \frac{\bigl(\lambda+k-\frac{d}{2}\bigl\lfloor \frac{r}{2}\bigr\rfloor+\frac{d_2}{2}+(k-l,\underline{0}_{\lfloor r/2\rfloor})\bigr)_{(l,\underline{k}_{\lfloor r/2\rfloor-1},k-l),d}}
{(\lambda)_{(\underline{2k}_{\lfloor r/2\rfloor},\underline{k}_{\lfloor r/2\rfloor},k-l),d}} & (r\colon\text{odd}). \end{cases} \label{const_general-}
\end{align}
If $r=3$, then since $d=2d_2$ holds, $C_{3,\pm}^{d,d_2}(\lambda,k,l)$ and $C^{d,d_2}(\lambda,\bk)$ in (\ref{const_rank3}) are related as 
\begin{align*}
C^{d,d_2}(\lambda,(k_1,k_2,k_2))&=C_{3,+}^{d,d_2}(\lambda,k_2,k_1-k_2), \\ 
C^{d,d_2}(\lambda,(k_2,k_2,k_3))&=C_{3,-}^{d,d_2}(\lambda,k_2,k_2-k_3). 
\end{align*}
Also, for general $r$ we have 
\[ C_{r,+}^{d,d_2}(\lambda,k,0)=C_{r,-}^{d,d_2}(\lambda,k,0)=C_{r+1,-}^{d,d_2}(\lambda,k,k)
=\frac{\bigl(\lambda+k-\frac{d}{2}\bigl\lceil\frac{r}{2}\bigr\rceil+\frac{d_2}{2}\bigr)_{\underline{k}_{\lfloor r/2\rfloor},d}}
{(\lambda)_{(\underline{2k}_{\lfloor r/2\rfloor},\underline{k}_{\lceil r/2\rceil}),d}}. \]

Now, for $\det_{\fn^+_2}(x_2)^k f(x_2)\in\cP_{(k+l,\underline{k}_{r-1})}(\fp^+_2)$, the following holds. 

\begin{theorem}\label{thm_general+}
Let $k,l\in\BZ_{\ge 0}$, $f(x_2)\in\cP_{(l,\underline{0}_{r-1})}(\fp^+_2)$. 
\begin{enumerate}
\item For $\Re\lambda>\frac{2n}{r}-1$, we have 
\[ \Bigl\langle \det_{\fn^+_2}(x_2)^kf(x_2),e^{(x|\overline{z})_{\fp^+}}\Bigr\rangle_{\lambda,x}
=C_{r,+}^{d,d_2}(\lambda,k,l)F_{l,+}^\downarrow\biggl(\begin{matrix} -k \\ -\lambda-2k+\frac{n}{r} \end{matrix};f;z\biggr). \]
\item As functions of $\lambda$, 
\begin{align*}
&(\lambda)_{(2k+l,\underline{2k}_{r/2-1},\underline{k}_{r/2}),d}\Bigl\langle \det_{\fn^+_2}(x_2)^kf(x_2),e^{(x|\overline{z})_{\fp^+}}\Bigr\rangle_{\lambda,x} && (r\colon\text{even}), \\
&(\lambda)_{(2k+l,\underline{2k}_{\lfloor r/2\rfloor-1},\min\{2k,k+l\},\underline{k}_{\lfloor r/2\rfloor}),d}
\Bigl\langle \det_{\fn^+_2}(x_2)^kf(x_2),e^{(x|\overline{z})_{\fp^+}}\Bigr\rangle{\vphantom{\Bigr|}}_{\lambda,x}
&& (r\colon\text{odd})
\end{align*}
are holomorphically continued for all $\lambda\in\BC$, and give non-zero polynomials in $z\in\fp^+$ for all $\lambda\in\BC$ if $f(x_2)\ne 0$. 
\item For $\mu,\nu\in\BC$, we have 
\[ F_{l,+}^\downarrow\biggl(\begin{matrix}\nu \\ \mu\end{matrix};f;z\biggr)=\det_{\fn^+}(z)^{\mu-2\nu}F_{l,+}^\downarrow\biggl(\begin{matrix}\mu-\nu \\ \mu\end{matrix};f;z\biggr). \]
Especially, if $\mu=a-2k$ and $\nu=-k$ with $a,k\in\BZ_{\ge 0}$, $a\le k$, then $F_{l,+}^\downarrow\bigl(\begin{smallmatrix}\nu \\ \mu\end{smallmatrix};f;z\bigr)$ is factorized as 
\[ F_{l,+}^\downarrow\biggl(\begin{matrix}-k \\ a-2k\end{matrix};f;z\biggr)=\det_{\fn^+}(z)^a F_{l,+}^\downarrow\biggl(\begin{matrix}a-k \\ a-2k\end{matrix};f;z\biggr). \]
\end{enumerate}
\end{theorem}

Theorem \ref{thm_general+}\,(1) and (3) for $\nu=-k\in\BZ_{\ge 0}$ case are already proved in \cite[Theorem 6.3\,(1)]{N2}, 
and we can prove (3) for general $\nu\in\BC$ by using the Zariski density of $-\BZ_{\ge 0}\subset\BC$. 
For Theorem \ref{thm_general+}\,(2), the holomorphy for general $r,k,l$ and the non-vanishing for $r$ even or $kl=0$ are given in \cite[Corollary 6.6\,(1)]{N2}, 
and the only new result is the non-vanishing for odd $r$ with $k,l\ne 0$. Next, for $\det_{\fn^+_2}(x_2)^kf(x_2^\itinv)\in\cP_{(\underline{k}_{r-1},k-l)}(\fp^+_2)$, the following holds. 
\begin{theorem}\label{thm_general-}
Let $k,l\in\BZ_{\ge 0}$ with $k\ge l$, and $f(x_2)\in\cP_{(l,\underline{0}_{r-1})}(\fp^+_2)$. 
\begin{enumerate}
\item For $\Re\lambda>\frac{2n}{r}-1$, we have 
\[ \Bigl\langle \det_{\fn^+_2}(x_2)^kf(x_2^\itinv),e^{(x|\overline{z})_{\fp^+}}\Bigr\rangle_{\lambda,x}
=C_{r,-}^{d,d_2}(\lambda,k,l)F_{l,-}^\downarrow\biggl(\begin{matrix} -k \\ -\lambda-2k+\frac{n}{r} \end{matrix};f;z\biggr). \]
\item As functions of $\lambda$, 
\begin{align*}
&(\lambda)_{(\underline{2k}_{r/2-1},2k-l,\underline{k}_{r/2-1},k-l),d}
\Bigl\langle \det_{\fn^+_2}(x_2)^kf(x_2^\itinv),e^{(x|\overline{z})_{\fp^+}}\Bigr\rangle_{\lambda,x} && (r\colon\text{even}), \\
&(\lambda)_{(\underline{2k}_{\lfloor r/2\rfloor},\underline{k}_{\lfloor r/2\rfloor},k-l),d}
\Bigl\langle \det_{\fn^+_2}(x_2)^kf(x_2^\itinv),e^{(x|\overline{z})_{\fp^+}}\Bigr\rangle_{\lambda,x} && (r\colon\text{odd})
\end{align*}
are holomorphically continued for all $\lambda\in\BC$, and give non-zero polynomials in $z\in\fp^+$ for all $\lambda\in\BC$ if $f(z_2)\ne 0$. 
\item For $\mu,\nu\in\BC$, we have 
\[ F_{l,-}^\downarrow\biggl(\begin{matrix}\nu \\ \mu\end{matrix};f;z\biggr)=\det_{\fn^+}(z)^{\mu-2\nu}F_{l,-}^\downarrow\biggl(\begin{matrix}\mu-\nu \\ \mu\end{matrix};f;z\biggr). \]
Especially, if $\mu=a-2k$ and $\nu=-k$ with $a,k\in\BZ_{\ge 0}$, $a\le k-l$, then $F_{l,-}^\downarrow\bigl(\begin{smallmatrix}\nu \\ \mu\end{smallmatrix};f;z\bigr)$ is factorized as 
\[ F_{l,-}^\downarrow\biggl(\begin{matrix}-k \\ a-2k\end{matrix};f;z\biggr)=\det_{\fn^+}(z)^a F_{l,-}^\downarrow\biggl(\begin{matrix}a-k \\ a-2k\end{matrix};f;z\biggr). \]
\end{enumerate}
\end{theorem}

The holomorphy in Theorem \ref{thm_general-}\,(2) was already proved in \cite[Theorem 6.2]{N3}, but the other results are new except for $l=0$ (Example \ref{example_scalar_general}) 
or $r=2$ (Example \ref{example_rank2_general}). 

\begin{remark}
For $+$ case, we proved in \cite[Theorem 6.3\,(1)]{N2}, 
\begin{align*}
&\Bigl\langle \det_{\fn^+_2}(x_2)^kf(x_2),e^{(x|\overline{z})_{\fp^+}}\Bigr\rangle_{\lambda,x} \\
&=\frac{1}{(\lambda)_{(2k+l,\underline{2k}_{r-1}),d}}\det_{\fn^+}(z)^{-\lambda+\frac{n}{r}}\det_{\fn^+_2}\biggl(\frac{\partial}{\partial z_2}\biggr)^k\det_{\fn^+}(z)^{\lambda+2k-\frac{n}{r}}f(z_2) \\
&=\frac{1}{(\lambda)_{(2k+l,\underline{2k}_{r-1}),d}}\det_{\fn^+}(z)^{-\lambda+\frac{n}{r}}\det_{\fn^+_2}\biggl(\frac{\partial}{\partial z_2}\biggr)^k
F_{l,+}^\downarrow\biggl(\begin{matrix} -\lambda-2k+\frac{n}{r} \\ -\lambda-2k+\frac{n}{r} \end{matrix};f;z\biggr) \\
&=C_{r,+}^{d,d_2}(\lambda,k,l)\det_{\fn^+}(z)^{-\lambda+\frac{n}{r}}F_{l,+}^\downarrow\biggl(\begin{matrix} -\lambda-k+\frac{n}{r} \\ -\lambda-2k+\frac{n}{r} \end{matrix};f;z\biggr) \\
&=C_{r,+}^{d,d_2}(\lambda,k,l)F_{l,+}^\downarrow\biggl(\begin{matrix} -k \\ -\lambda-2k+\frac{n}{r} \end{matrix};f;z\biggr). 
\end{align*}
Similarly, for $-$ case, by \cite[Theorem 3.3\,(1)]{N3} we have 
\begin{align}
&\Bigl\langle \det_{\fn^+_2}(x_2)^kf(x_2^\itinv),e^{(x|\overline{z})_{\fp^+}}\Bigr\rangle_{\lambda,x} \notag \\
&=\frac{\det_{\fn^+}(z)^{-\lambda+\frac{n}{r}}}{(\lambda)_{\underline{2(k-l)}_r,d}}\det_{\fn^+_2}\biggl(\frac{\partial}{\partial z_2}\biggr)^{k-l}
\det_{\fn^+}(x)^{\lambda+2(k-l)-\frac{n}{r}} \notag\\*
&\hspace{70pt}\times\Bigl\langle \det_{\fn^+_2}(x_2)^lf(x_2^\itinv),e^{(x|\overline{z})_{\fp^+}}\Bigr\rangle_{\lambda+2(k-l),x}, 
\label{formula_diff_expr_general-}
\end{align}
and combining with Proposition \ref{prop_shift_general} and Theorem \ref{thm_general-}\,(1), (3), we have 
\begin{align*}
&\Bigl\langle \det_{\fn^+_2}(x_2)^kf(x_2^\itinv),e^{(x|\overline{z})_{\fp^+}}\Bigr\rangle_{\lambda,x} \\
&=\frac{C_{r,-}^{d,d_2}(\lambda+2(k-l),l,l)}{(\lambda)_{\underline{2(k-l)}_r,d}}\det_{\fn^+}(z)^{-\lambda+\frac{n}{r}}\det_{\fn^+_2}\biggl(\frac{\partial}{\partial z_2}\biggr)^{k-l} \\*
&\hspace{105pt}\times \det_{\fn^+}(z)^{\lambda+2(k-l)-\frac{n}{r}}F_{l,-}^\downarrow\biggl(\begin{matrix} -l \\ -\lambda-2k+\frac{n}{r} \end{matrix};f;z\biggr) \\
&=\frac{\bigl(\lambda+2k-l-\frac{d}{2}\bigl\lfloor\frac{r}{2}\bigr\rfloor+\frac{d_2}{2}\bigr)_{\underline{l}_{\lceil r/2\rceil-1},d}}
{(\lambda)_{(\underline{2k}_{\lceil r/2\rceil-1},\underline{2k-l}_{\lfloor r/2\rfloor},2k-2l),d}}\det_{\fn^+}(z)^{-\lambda+\frac{n}{r}}\det_{\fn^+_2}\biggl(\frac{\partial}{\partial z_2}\biggr)^{k-l} \\*
&\hspace{180pt}\times F_{l,-}^\downarrow\biggl(\begin{matrix} -\lambda-2k+l+\frac{n}{r} \\ -\lambda-2k+\frac{n}{r} \end{matrix};f;z\biggr) \\
&=C_{r,-}^{d,d_2}(\lambda,k,l)\det_{\fn^+}(z)^{-\lambda+\frac{n}{r}}F_{l,-}^\downarrow\biggl(\begin{matrix} -\lambda-k+\frac{n}{r} \\ -\lambda-2k+\frac{n}{r} \end{matrix};f;z\biggr) \\
&=C_{r,-}^{d,d_2}(\lambda,k,l)F_{l,-}^\downarrow\biggl(\begin{matrix} -k \\ -\lambda-2k+\frac{n}{r} \end{matrix};f;z\biggr). 
\end{align*}
\end{remark}

First we prove Theorem \ref{thm_general-}\,(1). By \cite[Theorem 6.1]{N3} with $\bk=(\underline{k}_{r-1},\allowbreak k-l)$, $k_{r+1}=0$, when $d=2d_2$ we have 
\begin{align*}
&\frac{\bigl\langle \det_{\fn^+_2}(x_2)^kf(x_2^\itinv),e^{(x|\overline{z})_{\fp^+}}\bigr\rangle_{\lambda,x}}{\det_{\fn^+_2}(z_2)^kf(z_2^\itinv)}
=\frac{\prod_{a=2}^{2r-2}\prod_{i=\max\{1,a+1-r\}}^{\lfloor a/2\rfloor}\bigl(\lambda-\frac{d}{4}(a-1)\bigr)_{k_i+k_{a+1-i}}}
{\prod_{a=1}^{2r-1}\prod_{i=\max\{1,a+1-r\}}^{\lceil a/2\rceil}\bigl(\lambda-\frac{d}{4}(a-1)\bigr)_{k_i+k_{a+2-i}}} \\
&=\prod_{a=1}^{r-2}\frac{\prod_{i=1}^{\lfloor a/2\rfloor}\bigl(\lambda-\frac{d}{4}(a-1)\bigr)_{2k}}{\prod_{i=1}^{\lceil a/2\rceil}\bigl(\lambda-\frac{d}{4}(a-1)\bigr)_{2k}}
\prod_{a=r-1}^{r-1}\frac{\prod_{i=1}^{\lfloor a/2\rfloor}\bigl(\lambda-\frac{d}{4}(a-1)\bigr)_{2k}}
{\bigl(\lambda-\frac{d}{4}(a-1)\bigr)_{2k-l}\prod_{i=2}^{\lceil a/2\rceil}\bigl(\lambda-\frac{d}{4}(a-1)\bigr)_{2k}} \\*
&\eqspace{}\times\prod_{a=r}^{2r-3}\frac{\bigl(\lambda-\frac{d}{4}(a-1)\bigr)_{2k-l}\prod_{i=a+2-r}^{\lfloor a/2\rfloor}\bigl(\lambda-\frac{d}{4}(a-1)\bigr)_{2k}}
{\bigl(\lambda-\frac{d}{4}(a-1)\bigr)_{k}\bigl(\lambda-\frac{d}{4}(a-1)\bigr)_{2k-l}\prod_{i=a+3-r}^{\lceil a/2\rceil}\bigl(\lambda-\frac{d}{4}(a-1)\bigr)_{2k}} \\*
&\eqspace{}\times\prod_{a=2r-2}^{2r-2}\frac{\bigl(\lambda-\frac{d}{4}(a-1)\bigr)_{2k-l}}{\bigl(\lambda-\frac{d}{4}(a-1)\bigr)_{k}}
\prod_{a=2r-1}^{2r-1}\frac{1}{\bigl(\lambda-\frac{d}{4}(a-1)\bigr)_{k-l}} \\
&=\frac{1}{\prod_{a=1,\text{odd}}^{r-2}\bigl(\lambda-\frac{d}{4}(a-1)\bigr)_{2k}}
\frac{\prod_{a=r-1,\text{even}}^{r-1}\bigl(\lambda-\frac{d}{4}(a-1)+2k-l\bigr)_{l}}{\prod_{a=r-1,\text{odd}}^{r-1}\bigl(\lambda-\frac{d}{4}(a-1)\bigr)_{2k-l}} \\*
&\eqspace{}\times\frac{\prod_{a=r,\text{even}}^{2r-3}\bigl(\lambda-\frac{d}{4}(a-1)+k\bigr)_{k}}{\prod_{a=r,\text{odd}}^{2r-3}\bigl(\lambda-\frac{d}{4}(a-1)\bigr)_{k}}
\frac{\prod_{a=2r-2}^{2r-2}\bigl(\lambda-\frac{d}{4}(a-1)+k\bigr)_{k-l}}{\prod_{a=2r-1}^{2r-1}\bigl(\lambda-\frac{d}{4}(a-1)\bigr)_{k-l}} \\
&=\begin{cases} \ds \frac{\bigl(\lambda-\frac{d}{4}(r-1)+k\bigr)_{(\underline{k}_{r/2-1},k-l),d}}{(\lambda)_{(\underline{2k}_{r/2-1},2k-l,\underline{k}_{r/2-1},k-l),d}}
=\frac{\bigl(\lambda+k-\frac{d}{4}r+\frac{d_2}{2}\bigr)_{(\underline{k}_{r/2-1},k-l),d}}{(\lambda)_{(\underline{2k}_{r/2-1},2k-l,\underline{k}_{r/2-1},k-l),d}} & (r\colon\text{even}), \\
\ds \frac{\bigl(\lambda-\frac{d}{4}(r-2)+(2k-l,\underline{k}_{\lfloor r/2\rfloor})\bigr)_{(l,\underline{k}_{\lfloor r/2\rfloor-1},k-l),d}}
{(\lambda)_{(\underline{2k}_{\lfloor r/2\rfloor},\underline{k}_{\lfloor r/2\rfloor},k-l),d}} \\
\ds =\frac{\bigl(\lambda+k-\frac{d}{2}\bigl\lfloor\frac{r}{2}\bigr\rfloor+\frac{d_2}{2}+(k-l,\underline{0}_{\lfloor r/2\rfloor})\bigr)_{(l,\underline{k}_{\lfloor r/2\rfloor-1},k-l),d}}
{(\lambda)_{(\underline{2k}_{\lfloor r/2\rfloor},\underline{k}_{\lfloor r/2\rfloor},k-l),d}}& (r\colon\text{odd}), \end{cases}
\end{align*}and again by \cite[Theorem 6.1]{N3}, this also holds for $r=2$, $d\ne 2d_2$ case. 
Therefore, to prove (1), it is enough to show that $\bigl\langle \det_{\fn^+_2}(x_2)^kf(x_2^\itinv),e^{(x|\overline{z})_{\fp^+}}\bigr\rangle_{\lambda,x}$ is proportional to 
$F_{l,-}^\downarrow\bigl(\begin{smallmatrix} -k \\ -\lambda-2k+\frac{n}{r} \end{smallmatrix};f;z\bigr)$, and by the F-method, enough to demonstrate 
\begin{align}
\Bigl(\det_{\fn^+_2}(x_2)^kf(x_2^\itinv)\mapsto F_{l,-}^\downarrow\bigl(\begin{smallmatrix} -k \\ -\lambda-2k+\frac{n}{r} \end{smallmatrix};f;z\bigr)\Bigr)
\in \Hom_{K_1}\bigl(\cP_{(\underline{k}_{r-1},k-l)}(\fp^+_2),\Sol_{\cP(\fp^+)}((\cB_\lambda)_1)\bigr)&, \label{formula_Fmethod_general-}
\end{align}
where $(\cB_\lambda)_1$ is the orthogonal projection of $\cB_\lambda=\cB_\lambda^{\fp^+}\colon\cP(\fp^+)\to\cP(\fp^+)\otimes\fp^+$ given in (\ref{formula_Bessel_diff}) 
(with the identification $\fp^-\simeq\fp^+=\fn^{+\BC}$) onto $\fp^+_1$. 

\begin{proposition}
Let $\mu,\nu\in\BC$, $l\in\BZ_{\ge 0}$, $f(x_2)\in\cP_{(l,\underline{0}_{r-1})}(\fp^+_2)$. 
\begin{enumerate}
\item For $z=z_1+z_2\in\Omega$, $a_2\in\Omega_2$ with $z_1+a_2\in\Omega$ and for $\Re\mu>l+\frac{2n_2}{r}-1$, $\Re\nu>l+\frac{n_2}{r}-1$, we have 
\begin{align*}
F_{l,+}^\downarrow\biggl(\begin{matrix} \nu \\ \mu \end{matrix};f;z\biggr)
&=\frac{1}{(2\pi\sqrt{-1})^{n_2}}\frac{\Gamma_r^{d_2}(\mu-(\underline{0}_{r-1},l))}{\Gamma_r^{d_2}(\nu-(\underline{0}_{r-1},l))}
\int_{\Omega_2}e^{-(y_2|z_2)_{\fn^+_2}}\det_{\fn^+_2}(y_2)^{\nu-\mu} \\*
&\eqspace{}\times\biggl(\int_{a_2+\sqrt{-1}\fn^+_2}e^{(x_2|y_2)_{\fn^+_2}}\det_{\fn^+}(z_1+x_2)^{-\mu}f(x_2)\,dx_2\biggr)dy_2. 
\end{align*}
Similarly, for $\Re\mu>\frac{d_2}{2}+\frac{2n_2}{r}-1$, $\Re\nu>\frac{n_2}{r}-1$, we have 
\begin{align*}
F_{l,-}^\downarrow\biggl(\begin{matrix} \nu \\ \mu \end{matrix};f;z\biggr)
=\frac{1}{(2\pi\sqrt{-1})^{n_2}}\frac{\Gamma_r^{d_2}\bigl(\mu-\frac{d_2}{2}+(l,\underline{0}_{r-1})\bigr)}{\Gamma_r^{d_2}(\nu+(l,\underline{0}_{r-1}))}
\int_{\Omega_2}e^{-(y_2|z_2)_{\fn^+_2}}\det_{\fn^+_2}(y_2)^{\nu-\mu+\frac{d_2}{2}} \\*
{}\times\biggl(\int_{a_2+\sqrt{-1}\fn^+_2}e^{(x_2|y_2)_{\fn^+_2}}\det_{\fn^+_2}(x_2)^{-\frac{d_2}{2}}\det_{\fn^+}(z_1+x_2)^{-\mu+d_2}
f(x_2^\itinv)\,dx_2\biggr)dy_2&. 
\end{align*}
\item Suppose $\mu\in\BC$ is not a pole of $F_{l,\pm}^\downarrow\bigl(\begin{smallmatrix} \nu \\ \mu \end{smallmatrix};f;z\bigr)$. Then for $g\in \widetilde{K}^\BC_1$, we have 
\begin{align*}
F_{l,+}^\downarrow\biggl(\begin{matrix} \nu \\ \mu \end{matrix};f(g(\cdot));z\biggr)&=\det_{\fn^+_2}(ge)^\nu F_{l,+}^\downarrow\biggl(\begin{matrix} \nu \\ \mu \end{matrix};f;gz\biggr), \\
F_{l,-}^\downarrow\biggl(\begin{matrix} \nu \\ \mu \end{matrix};f({}^t\hspace{-1pt}g^{-1}(\cdot));z\biggr)
&=\det_{\fn^+_2}(ge)^\nu F_{l,-}^\downarrow\biggl(\begin{matrix} \nu \\ \mu \end{matrix};f;gz\biggr). 
\end{align*}
\item $F_{l,\pm}^\downarrow\bigl(\begin{smallmatrix} \nu \\ \mu \end{smallmatrix};f;z\bigr)$ satisfies the differential equation 
\[ (\cB_{-\mu+2\nu+\frac{n}{r}})_1F_{l,\pm}^\downarrow\biggl(\begin{matrix} \nu \\ \mu \end{matrix};f;z\biggr)=0. \]
\end{enumerate}
\end{proposition}

For $F_{l,+}^\downarrow\bigl(\begin{smallmatrix} \nu \\ \mu \end{smallmatrix};f;z\bigr)$, this proposition is a repost of \cite[Proposition 6.4]{N2}, 
and we only prove this for $F_{l,-}^\downarrow\bigl(\begin{smallmatrix} \nu \\ \mu \end{smallmatrix};f;z\bigr)$. 

\begin{proof}
(1) By (\ref{binom_general}) and the definition of $F_{\bm,\bl}^+[f](x_1,y_2)$, we have 
\begin{align*}
&\det_{\fn^+}(z_1+x_2)^{-\mu+d_2}f(x_2^\itinv)=\det_{\fn^+_2}(x_2)^{-\mu+d_2}h_{\fn^+}(z_1,P(x_2^\itinv)z_1)^{(-\mu+d_2)/2}f(x_2^\itinv) \\
&=\det_{\fn^+_2}(x_2)^{-\mu+d_2}\sum_{\bm\in\BZ_{++}^{\lfloor r/2\rfloor}}(\mu-d_2)_{\bm,d}\tilde{\Phi}_\bm^{\fp^+_1,\fp^-_2}(z_1,x_2^\itinv)f(x_2^\itinv) \\
&=\det_{\fn^+_2}(x_2)^{-\mu+d_2}\sum_{\bm\in\BZ_{++}^{\lfloor r/2\rfloor}}\sum_{\substack{\bl\in(\BZ_{\ge 0})^{\lceil r/2\rceil} \\ |\bl|=l}}(\mu-d_2)_{\bm,d}F_{\bm,\bl}^+[f](z_1,x_2^\itinv), 
\end{align*}
and hence by Lemma \ref{lem_Laplace}, we get 
\begin{align*}
&\frac{1}{(2\pi\sqrt{-1})^{n_2}}\frac{\Gamma_r^{d_2}\bigl(\mu-\frac{d_2}{2}+(l,\underline{0}_{r-1})\bigr)}{\Gamma_r^{d_2}(\nu+(l,\underline{0}_{r-1}))}
\int_{\Omega_2}e^{-(y_2|z_2)_{\fn^+_2}}\det_{\fn^+_2}(y_2)^{\nu-\mu+\frac{d_2}{2}} \\*
&\eqspace{}\times\biggl(\int_{a_2+\sqrt{-1}\fn^+_2}e^{(x_2|y_2)_{\fn^+_2}}\det_{\fn^+_2}(x_2)^{-\frac{d_2}{2}}\det_{\fn^+}(z_1+x_2)^{-\mu+d_2}f(x_2^\itinv)\,dx_2\biggr)dy_2 \\
&=\frac{1}{(2\pi\sqrt{-1})^n}\frac{\Gamma_r^{d_2}\bigl(\mu-\frac{d_2}{2}+(l,\underline{0}_{r-1})\bigr)}{\Gamma_r^{d_2}(\nu+(l,\underline{0}_{r-1}))}
\int_{\Omega_2}e^{-(y_2|z_2)_{\fn^+_2}}\det_{\fn^+_2}(y_2)^{\nu-\mu+\frac{d_2}{2}} \\*
&\eqspace{}\times\biggl(\int_{a_2+\sqrt{-1}\fn^+_2}e^{(x_2|y_2)_{\fn^+_2}}\det_{\fn^+_2}(x_2)^{-\mu+\frac{d_2}{2}}
\sum_{\bm,\bl}(\mu-d_2)_{\bm,d}F_{\bm,\bl}^+[f](z_1,x_2^\itinv)\,dx_2\biggr)dy_2 \\
&=\frac{\Gamma_r^{d_2}\bigl(\mu-\frac{d_2}{2}+(l,\underline{0}_{r-1})\bigr)}{\Gamma_r^{d_2}(\nu+(l,\underline{0}_{r-1}))}
\int_{\Omega_2}e^{-(y_2|z_2)_{\fn^+_2}}\det_{\fn^+_2}(y_2)^{\nu-\frac{n_2}{r}} \sum_{\bm,\bl}
\frac{(\mu-d_2)_{\bm,d}F_{\bm,\bl}^+[f](z_1,y_2)}{\Gamma_r^{d_2}\bigl(\mu-\frac{d_2}{2}+\phi_+(\bm,\bl)\bigr)}\,dy_2 \\
&=\frac{1}{\Gamma_r^{d_2}(\nu+(l,\underline{0}_{r-1}))}
\int_{\Omega_2}e^{-(y_2|z_2)_{\fn^+_2}} \sum_{\bm,\bl}\frac{(\mu-d_2)_{\bm,d}\det_{\fn^+_2}(y_2)^{\nu-\frac{n_2}{r}}F_{\bm,\bl}^+[f](z_1,y_2)}
{\bigl(\mu-\frac{d_2}{2}+(l,\underline{0}_{r-1})\bigr)_{\phi_+(\bm,\bl)-(l,\underline{0}_{r-1}),d_2}}\,dy_2 \\
&=\frac{1}{\Gamma_r^{d_2}(\nu+(l,\underline{0}_{r-1}))}
\int_{\Omega_2}e^{-(y_2|z_2)_{\fn^+_2}} \sum_{\bm,\bl}\frac{\det_{\fn^+_2}(y_2)^{\nu-\frac{n_2}{r}}F_{\bm,\bl}^+[f](z_1,y_2)}
{\bigl(\mu-\frac{d_2}{2}+(l,\underline{0}_{\lceil r/2\rceil-1})\bigr)_{\bm+\bl-(l,\underline{0}_{\lceil r/2\rceil-1}),d}}\,dy_2 \\
&=\sum_{\bm,\bl}\frac{\Gamma_r^{d_2}(\nu+\phi_+(\bm,\bl))}{\Gamma_r^{d_2}(\nu+(l,\underline{0}_{r-1}))}
\frac{\det_{\fn^+_2}(z_2)^{-\nu}F_{\bm,\bl}^+[f](z_1,z_2^\itinv)}{\bigl(\mu-\frac{d_2}{2}+(l,\underline{0}_{\lceil r/2\rceil-1})\bigr)_{\bm+\bl-(l,\underline{0}_{\lceil r/2\rceil-1}),d}} \\
&=\sum_{\bm,\bl}\frac{(\nu+(l,\underline{0}_{r-1}))_{\phi_+(\bm,\bl)-(l,\underline{0}_{r-1}),d_2}}
{\bigl(\mu-\frac{d_2}{2}+(l,\underline{0}_{\lceil r/2\rceil-1})\bigr)_{\bm+\bl-(l,\underline{0}_{\lceil r/2\rceil-1}),d}}\det_{\fn^+_2}(z_2)^{-\nu} F_{\bm,\bl}^+[f](z_1,z_2^\itinv) \\
&=\sum_{\bm,\bl}\frac{(\nu+(l,\underline{0}_{\lceil r/2\rceil-1}))_{\bm+\bl-(l,\underline{0}_{\lceil r/2\rceil-1}),d}\bigl(\nu-\frac{d_2}{2}\bigr)_{\bm,d}}
{\bigl(\mu-\frac{d_2}{2}+(l,\underline{0}_{\lceil r/2\rceil-1})\bigr)_{\bm+\bl-(l,\underline{0}_{\lceil r/2\rceil-1}),d}}\det_{\fn^+_2}(z_2)^{-\nu} F_{\bm,\bl}^+[f](z_1,z_2^\itinv) \\
&=F_{l,-}^\downarrow\biggl(\begin{matrix} \nu \\ \mu \end{matrix};f;z\biggr). 
\end{align*}

(2) Since $F_{\bm,\bl}^+[f](z_1,z_2^\itinv)$ satisfies 
\[ F_{\bm,\bl}^+[f({}^t\hspace{-1pt}g^{-1}(\cdot))](z_1,z_2^\itinv)=F_{\bm,\bl}^+[f](gz_1,(gz_2)^\itinv), \]
we get the desired formula. 

(3) As in the proof of \cite[Proposition 6.4]{N2}, we can show that 
\begin{align*}
&(\cB_{-\mu+2\nu+\frac{n}{r}})_1\int_{\Omega_2}e^{-(y_2|z_2)_{\fn^+_2}}\det_{\fn^+_2}(y_2)^{\nu-\mu+\frac{d_2}{2}}\biggl(\int_{a_2+\sqrt{-1}\fn^+_2}e^{(x_2|y_2)_{\fn^+_2}} \\*
&\eqspace{}\times\det_{\fn^+}(z_1+x_2)^{-\mu+d_2}\det_{\fn^+_2}(x_2)^{-\frac{d_2}{2}}f(x_2^\itinv)\,dx_2\biggr)dy_2 \\
&=\int_{\Omega_2}e^{-(y_2|z_2)_{\fn^+_2}}\det_{\fn^+_2}(y_2)^{\nu-\mu+\frac{d_2}{2}}\biggl(\int_{a_2+\sqrt{-1}\fn^+_2}e^{(x_2|y_2)_{\fn^+_2}} \\*
&\eqspace{}\times\det_{\fn^+}(z_1+x_2)^{-\mu+d_2}(\cB_{-\mu+d_2+\frac{n}{r}})_1\det_{\fn^+_2}(x_2)^{-\frac{d_2}{2}}f(x_2^\itinv)\,dx_2\biggr)dy_2. 
\end{align*}
Hence it is enough to prove $(\cB_{\lambda})_1\det_{\fn^+_2}(x_2)^{-d_2/2}f(x_2^\itinv)=0$. 
In the following, we write $(z_1,x_2)=(\zeta,X)$, let $\{E_\alpha\}\subset \fp^+_2$ be a basis with the dual basis $\{E_\alpha^\vee\}\subset \fp^+_2$, 
and let $\frac{\partial}{\partial X_\alpha}$, $\frac{\partial}{\partial Y_\alpha^\vee}$ denote the directional derivatives along $E_\alpha$ and $E_\alpha^\vee$. Then we have 
\begin{align*}
&(\cB_\lambda)_1\det_{\fn^+_2}(X)^{-d_2/2}f(X^\itinv)
=\frac{1}{2}\sum_{\alpha\beta}P(E_\alpha^\vee,E_\beta^\vee)\zeta\frac{\partial^2}{\partial X_\alpha\partial X_\beta} \det_{\fn^+_2}(X)^{-d_2/2}f(X^\itinv) \\
&=\frac{1}{2}\sum_{\alpha\beta}P(E_\alpha^\vee,E_\beta^\vee)\zeta\det_{\fn^+_2}(X)^{-d_2/2} \\*
&\eqspace{}\times\biggl(\biggl(\frac{d_2^2}{4}(E_\alpha|X^\itinv)_{\fn^+_2}(E_\beta|X^\itinv)_{\fn^+_2}+\frac{d_2}{2}(E_\beta|P(X)^{-1}E_\alpha)_{\fn^+_2}\biggr)f(X^\itinv) \\*
&\hspace{40pt}-d_2(E_\beta|X^\itinv)_{\fn^+_2}\frac{\partial}{\partial X_\alpha}f(X^\itinv)+\frac{\partial^2}{\partial X_\alpha \partial X_\beta}f(X^\itinv)\biggr) \\
&=\det_{\fn^+_2}(X)^{-d_2/2}\biggl(\biggl(\frac{d_2^2}{4}P(X^\itinv)\zeta+\frac{d_2}{4}\sum_{\alpha}P(E_\alpha^\vee,P(X)^{-1}E_\alpha)\zeta\biggr)f(X^\itinv) \\*
&\eqspace{}+\frac{d_2}{2}\sum_\alpha P(E_\alpha^\vee,X^\itinv)\zeta\sum_\gamma (E_\gamma|P(X)^{-1}E_\alpha)_{\fn^+_2}\frac{\partial f}{\partial Y_\gamma^\vee}(X^\itinv) \\*
&\eqspace{}+\frac{1}{2}\sum_{\alpha\beta}P(E_\alpha^\vee,E_\beta^\vee)\zeta\biggl(\sum_{\gamma\delta}(E_\gamma|P(X^\itinv)E_\alpha)_{\fn^+_2}(E_\delta|P(X^\itinv)E_\beta)_{\fn^+_2}
\frac{\partial^2f}{\partial Y_\gamma^\vee \partial Y_\delta^\vee}(X^\itinv) \\*
&\eqspace{}+\sum_{\gamma}(E_\gamma|P(X^\itinv,P(X)^{-1}E_\alpha)E_\beta)_{\fn^+_2}\frac{\partial f}{\partial Y_\gamma^\vee}(X^\itinv)\biggr)\biggr) \\
&=\det_{\fn^+_2}(X)^{-d_2/2}\biggl(\biggl(\frac{d_2^2}{4}P(X)^{-1}\zeta+\frac{d_2}{4}\sum_{\alpha}P(E_\alpha^\vee,P(X)^{-1}E_\alpha)\zeta\biggr)f(X^\itinv) \\*
&\eqspace{}+\frac{d_2}{2}\sum_\gamma P(P(X)^{-1}E_\gamma,X^\itinv)\zeta \frac{\partial f}{\partial Y_\gamma^\vee}(X^\itinv) \\*
&\eqspace{}+\frac{1}{2}\sum_{\gamma\delta}P(P(X)^{-1}E_\gamma,P(X)^{-1}E_\delta)\zeta\frac{\partial^2f}{\partial Y_\gamma^\vee \partial Y_\delta^\vee}(X^\itinv) \\*
&\eqspace{}+\frac{1}{2}\sum_{\alpha\gamma}P(E_\alpha^\vee,P(X^\itinv,P(X)^{-1}E_\alpha)E_\gamma)\zeta\frac{\partial f}{\partial Y_\gamma^\vee}(X^\itinv)\biggr). 
\end{align*}
Then at the 5th term, since $\sum_{\alpha}P(E_\alpha^\vee,P(X^\itinv,P(X)^{-1}E_\alpha)E_\gamma)\zeta$ does not depend on the choice of $\{E_\alpha\}$ and $\{E_\alpha^\vee\}$, 
we may replace these with $\{P(X^{\mathit{1/2}})E_\alpha'\}$ and $\{P(X^{-\mathit{1/2}})E_\alpha'\}$ by using an orthonormal basis $\{E_\alpha'\}\subset\fn^+_2$. 
Then by using the equality $X^\itinv=P(X^{-\mathit{1/2}})E$, where $E\in\fn^+_2$ is the unit element, and $P(P(X)Y,P(X)Z)=P(X)P(Y,Z)P(X)$, we have 
\begin{align*}
&\sum_{\alpha}P(E_\alpha^\vee,P(X^\itinv,P(X)^{-1}E_\alpha)E_\gamma)\zeta \\
&=\sum_{\alpha}P(P(X^{-\mathit{1/2}})E_\alpha',P(P(X^{-\mathit{1/2}})E,P(X^{-\mathit{1/2}})E_\alpha')E_\gamma)\zeta \\
&=\sum_{\alpha}P(X^{-\mathit{1/2}})P(E_\alpha',P(E,E_\alpha')P(X^{-\mathit{1/2}})E_\gamma)P(X^{-\mathit{1/2}})\zeta \\
&=\hspace{-1pt}P(X^{-\mathit{1/2}})\hspace{-1pt}\sum_{\alpha}\hspace{-1pt}\bigl(\hspace{-1pt}P(E_\alpha')D_{\fn^+}\hspace{-1pt}(P(X^{-\mathit{1/2}})E_\gamma,E)\hspace{-1pt}
+\hspace{-1pt}D_{\fn^+}\hspace{-1pt}(E,P(X^{-\mathit{1/2}})E_\gamma)P(E_\alpha')\bigr) P(X^{-\mathit{1/2}})\zeta \\
&=P(X^{-\mathit{1/2}})\biggl(-\frac{d_2}{2}D_{\fn^+}(P(X^{-\mathit{1/2}})E_\gamma,E)-\frac{d_2}{2}D_{\fn^+}(E,P(X^{-\mathit{1/2}})E_\gamma)\biggr) P(X^{-\mathit{1/2}})\zeta \\
&=-d_2P(X^{-\mathit{1/2}})P(P(X^{-\mathit{1/2}})E_\gamma,E)P(X^{-1/2})\zeta=-d_2P(P(X)^{-1}E_\gamma,X^\itinv)\zeta, 
\end{align*}
where at the 3nd equality we have used \cite[Part V, Proposition I.2.1\,(J2.3)]{FKKLR}, at the 4th equality $\sum_{\alpha}P(E_\alpha')\eta=-\frac{d_2}{2}\eta$ follows from 
direct computation with an explicit basis, and at the 5th equality we have used $D_{\fn^+}(Y,E)=D_{\fn^+}(E,Y)=P(Y,E)$. Similarly, at the 2nd term we have 
\begin{align*}
&\sum_{\alpha}P(E_\alpha^\vee,P(X)^{-1}E_\alpha)\zeta=\sum_{\alpha}P(P(X^{-\mathit{1/2}})E_\alpha^\vee,P(X^{-\mathit{1/2}})E_\alpha)\zeta \\
&=P(X^{-\mathit{1/2}})\sum_{\alpha}P(E_\alpha',E_\alpha')P(X^{-\mathit{1/2}})\zeta=-d_2P(X^{-\mathit{1/2}})^2\zeta=-d_2P(X)^{-1}\zeta. 
\end{align*}
Therefore we get 
\begin{align*}
&(\cB_\lambda)_1\det_{\fn^+_2}(X)^{-d_2/2}f(X^\itinv) \\
&=\det_{\fn^+_2}(X)^{-d_2/2}\biggl(\biggl(\frac{d_2^2}{4}P(X)^{-1}\zeta-\frac{d_2^2}{4}P(X)^{-1}\zeta\biggr)f(X^\itinv) \\*
&\eqspace{}+\frac{d_2}{2}\sum_\gamma P(P(X)^{-1}E_\gamma,X^\itinv)\zeta \frac{\partial f}{\partial Y_\gamma^\vee}(X^\itinv) \\*
&\eqspace{}+\frac{1}{2}\sum_{\gamma\delta}P(P(X)^{-1}E_\gamma,P(X)^{-1}E_\delta)\zeta\frac{\partial^2f}{\partial Y_\gamma^\vee \partial Y_\delta^\vee}(X^\itinv) \\*
&\eqspace{}-\frac{d_2}{2}\sum_\gamma P(P(X)^{-1}E_\gamma,X^\itinv)\zeta \frac{\partial f}{\partial Y_\gamma^\vee}(X^\itinv)\biggr) \\ 
&=\frac{1}{2}\det_{\fn^+_2}(X)^{-d_2/2}P(X)^{-1}\sum_{\gamma\delta}P(E_\gamma,E_\delta)P(X)^{-1}\zeta
\frac{\partial^2f}{\partial Y_\gamma^\vee \partial Y_\delta^\vee}(X^\itinv). 
\end{align*}
Now, let $\fp^+=\bigoplus_{1\le i<j\le r}\fp^+_{ij}$ be the simultaneous Peirce decomposition given in Section \ref{subsection_simul_Peirce} 
with respect to a Jordan frame $\{e_1,\ldots,e_r\}\subset\fn^+_2\subset\fn^+$. 
If $f(Y)=(Y_{11})^l\in\cP_{(l,\underline{0}_{r-1})}(\fp^+_2)$, then since $\fp^+_{11}\cap\fp^+_1=\{0\}$, for $\eta\in\fp^+_1$ we have 
\begin{align*}
\sum_{\gamma\delta}P(E_\gamma,E_\delta)\eta\frac{\partial^2f}{\partial Y_\gamma^\vee \partial Y_\delta^\vee}(Y)=l(l-1)P(E_{11},E_{11})\eta (Y_{11})^{l-2}=0, 
\end{align*}
and by the $K_1$-equivariance, $\sum_{\gamma\delta}P(E_\gamma,E_\delta)\eta\frac{\partial^2f}{\partial Y_\gamma^\vee \partial Y_\delta^\vee}(Y)=0$ holds 
for all $f\in\cP_{(l,\underline{0}_{r-1})}(\fp^+_2)$. 
Hence we get $(\cB_{\lambda})_1\det_{\fn^+_2}(x_2)^{-d_2/2}f(x_2^\itinv)=0$ for all $f\in\cP_{(l,\underline{0}_{r-1})}(\fp^+_2)$, 
and also get $(\cB_{-\mu+2\nu+\frac{n}{r}})_1F_{l,-}^\downarrow\bigl(\begin{smallmatrix} \nu \\ \mu \end{smallmatrix};f;z\bigr)=0$ 
for $\Re\mu>\frac{d_2}{2}+\frac{2n_2}{r}-1$, $\Re\nu>\frac{n_2}{r}-1$. 
Then by analytic continuation, this holds for all $\mu,\nu\in\BC$ except for poles. 
\end{proof}

Especially, when $\mu=-\lambda-2k+\frac{n}{r}$, $\nu=-k$ with $k\in\BZ_{\ge 0}$, $k\ge l$, we have (\ref{formula_Fmethod_general-}), and this proves Theorem \ref{thm_general-}\,(1). 
Next we prove Theorems \ref{thm_general+}\,(2) and \ref{thm_general-}\,(2). 

\begin{proof}[Proof of Theorem \ref{thm_general+}\,(2)]
Let $\mu:=-\lambda-2k+\frac{n}{r}=-\lambda-2k+1+\frac{d}{2}(r-1)$. 
We recall from the proof of \cite[Corollary 6.6\,(1)]{N2} that, by Theorem \ref{thm_general+}\,(1), for $f(x_2)\in\cP_{(l,\underline{0}_{r-1})}(\fp^+_2)$, for even $r$ we have 
\begin{align}
&(\lambda)_{(2k+l,\underline{2k}_{r/2-1},\underline{k}_{r/2}),d}\Bigl\langle \det_{\fn^+_2}(x_2)^kf(x_2),e^{(x|\overline{z})_{\fp^+}}\Bigr\rangle_{\lambda,x} \notag \\
&=\biggl(\lambda+k-\frac{d}{4}r+\frac{d_2}{2}+(l,\underline{0}_{r/2-1})\biggr)_{\underline{k}_{r/2},d}F_{l,+}^\downarrow\biggl( \begin{matrix} -k \\ \mu \end{matrix};f;z\biggr) \notag \\
&=(-1)^{kr/2}\biggl(\mu-\frac{d_2}{2}-(\underline{0}_{r/2-1},l)\biggr)_{\underline{k}_{r/2},d}\det_{\fn^+_2}(z_2)^k \notag \\*
&\eqspace{}\times \sum_{\bm\in\BZ_{++}^{\lfloor r/2\rfloor}}\sum_{\substack{\bl\in(\BZ_{\ge 0})^{\lceil r/2\rceil} \\ |\bl|=l}}
\frac{(-k)_{\bm,d}\bigl(-k-\frac{d_2}{2}-(\underline{0}_{r/2-1},l)\bigr)_{\bm-\bl+(\underline{0}_{r/2-1},l),d}}
{\bigl(\mu-\frac{d_2}{2}-(\underline{0}_{r/2-1},l)\bigr)_{\bm-\bl+(\underline{0}_{r/2-1},l),d}} F_{\bm,\bl}^-[f](z_1,z_2^\itinv), \label{formula_holo_general+even}
\end{align}
and for odd $r$ we have 
\begin{align}
&(\lambda)_{(2k+l,\underline{2k}_{\lfloor r/2\rfloor-1},\min\{2k,k+l\},\underline{k}_{\lfloor r/2\rfloor}),d}
\Bigl\langle \det_{\fn^+_2}(x_2)^kf(x_2),e^{(x|\overline{z})_{\fp^+}}\Bigr\rangle_{\lambda,x} \notag \\
&=\biggl(\lambda-\frac{d}{2}\biggl\lfloor \frac{r}{2}\biggr\rfloor+\max\{2k,k+l\}\biggr)_{\min\{k,l\}}
\biggl(\lambda+k-\frac{d}{2}\biggl\lceil \frac{r}{2}\biggr\rceil+\frac{d_2}{2}\biggr)_{\underline{k}_{\lfloor r/2\rfloor},d} 
F_{l,+}^\downarrow\biggl( \begin{matrix} -k \\ \mu \end{matrix};f;z\biggr) \notag \\
&=(-1)^{k\lfloor r/2\rfloor+\min\{k,l\}}\biggl(\mu-\frac{d_2}{2}\biggr)_{\underline{k}_{\lfloor r/2\rfloor},d}\biggl(\mu-\frac{d_2}{2}(r-1)-l\biggr)_{\min\{k,l\}}
\det_{\fn^+_2}(z_2)^k \notag \\*
&\eqspace{}\times \sum_{\bm\in\BZ_{++}^{\lfloor r/2\rfloor}}\sum_{\substack{\bl\in(\BZ_{\ge 0})^{\lceil r/2\rceil} \\ |\bl|=l}}
\frac{(-k)_{\bm,d}\bigl(-k-\frac{d_2}{2}\bigr)_{\bm-\bl',d}\bigl(-k-\frac{d_2}{2}(r-1)-l\bigr)_{l-l_{\lceil r/2\rceil}}}
{\bigl(\mu-\frac{d_2}{2}\bigr)_{\bm-\bl',d}\bigl(\mu-\frac{d_2}{2}(r-1)-l\bigr)_{l-l_{\lceil r/2\rceil}}} \notag\\*
&\hspace{280pt}\times F_{\bm,\bl}^-[f](z_1,z_2^\itinv). \label{formula_holo_general+odd}
\end{align}
These are holomorphically continued for all $\mu\in\BC$. Moreover, all non-zero $F_{\bm,\bl}^-[f](z_1,z_2^\itinv)$ are linearly independent, and when $r$ is even, 
the coefficient of (\ref{formula_holo_general+even}) for 
\[ F_{\underline{k}_{r/2},(\underline{0}_{r/2-1},l)}^-[f](z_1,z_2^\itinv)=\tilde{\Phi}_{\underline{k}_{r/2}}^{\fp^+_1,\fp^-_2}(z_1,z_2^\itinv)f(z_2) \]
is independent of $\mu$ and non-zero. Hence (\ref{formula_holo_general+even}) is non-zero for all $\mu\in\BC$ when $f(x_2)\ne 0$. Similarly, when $r$ is odd, 
the coefficients of (\ref{formula_holo_general+odd}) for $(\bm,\bl)=(\underline{k}_{\lfloor r/2\rfloor},(\underline{0}_{\lfloor r/2\rfloor},l))$ 
and $(\bm,\bl)=(\underline{k}_{\lfloor r/2\rfloor},(\underline{0}_{\lfloor r/2\rfloor-1},\min\{k,l\},\max\{l-k,0\}))$, 
\begin{align*}
&\frac{\bigl(\mu-\frac{d_2}{2}\bigr)_{\underline{k}_{\lfloor r/2\rfloor},d}\bigl(\mu-\frac{d_2}{2}(r-1)-l\bigr)_{\min\{k,l\}}}
{\bigl(\mu-\frac{d_2}{2}\bigr)_{\underline{k}_{\lfloor r/2\rfloor}-\underline{0}_{\lfloor r/2\rfloor},d}\bigl(\mu-\frac{d_2}{2}(r-1)-l\bigr)_{l-l}}
=\biggl(\mu-\frac{d_2}{2}(r-1)-l\biggr)_{\min\{k,l\}}, \\
&\frac{\bigl(\mu-\frac{d_2}{2}\bigr)_{\underline{k}_{\lfloor r/2\rfloor},d}\bigl(\mu-\frac{d_2}{2}(r-1)-l\bigr)_{\min\{k,l\}}}
{\bigl(\mu-\frac{d_2}{2}\bigr)_{\underline{k}_{\lfloor r/2\rfloor}-(\underline{0}_{\lfloor r/2\rfloor-1},\min\{k,l\}),d}\bigl(\mu-\frac{d_2}{2}(r-1)-l\bigr)_{l-\max\{l-k,0\}}} \\*
&=\biggl(\mu-\frac{d_2}{2}(r-2)+\max\{0,k-l\}\biggr)_{\min\{k,l\}}
\end{align*}
do not vanish simultaneously, and if $f(x_2)=(x_2|w_2)_{\fn^+_2}^l$ for some $w_2\in\fp^+_2$ of rank 1, then by Proposition \ref{prop_example_Fml_general}\,(3), we have 
\[ F_{\underline{k}_{\lfloor r/2\rfloor},(\underline{0}_{\lfloor r/2\rfloor},l)}^-[f](z_1,z_2^\itinv)\ne 0, \qquad
F_{\underline{k}_{\lfloor r/2\rfloor},(\underline{0}_{\lfloor r/2\rfloor-1},\min\{k,l\},\max\{l-k,0\})}^-[f](z_1,z_2^\itinv)\ne 0. \]
By the $K_1^\BC$-equivariance, these hold for all non-zero $f(x_2)\in\cP_{(l,\underline{0}_{r-1})}(\fp^+_2)$. 
Hence (\ref{formula_holo_general+odd}) is non-zero for all $\mu\in\BC$ if $f(x_2)\ne 0$. 
\end{proof}

\begin{proof}[Proof of Theorem \ref{thm_general-}\,(2)]
Again let $\mu:=-\lambda-2k+\frac{n}{r}=-\lambda-2k+1+\frac{d}{2}(r-1)$. By Theorem \ref{thm_general-}\,(1), for even $r$ we have 
\begin{align}
&(\lambda)_{(\underline{2k}_{r/2-1},2k-l,\underline{k}_{r/2-1},k-l),d}\Bigl\langle \det_{\fn^+_2}(x_2)^kf(x_2^\itinv),e^{(x|\overline{z})_{\fp^+}}\Bigr\rangle_{\lambda,x} \notag \\
&=\biggl(\lambda+k-\frac{d}{4}r+\frac{d_2}{2}\biggr)_{(\underline{k}_{r/2-1},k-l),d}F_{l,-}^\downarrow\biggl( \begin{matrix} -k \\ \mu \end{matrix};f;z\biggr) \notag \\
&=(-1)^{kr/2-l}\biggl(\mu-\frac{d_2}{2}+(l,\underline{0}_{r/2-1})\biggr)_{(k-l,\underline{k}_{r/2-1}),d}\det_{\fn^+_2}(z_2)^k \notag \\*
&\eqspace{}\times \sum_{\bm\in\BZ_{++}^{\lfloor r/2\rfloor}}\sum_{\substack{\bl\in(\BZ_{\ge 0})^{\lceil r/2\rceil} \\ |\bl|=l}}
\frac{(-k+(l,\underline{0}_{\lceil r/2\rceil-1}))_{\bm+\bl-(l,\underline{0}_{\lceil r/2\rceil-1}),d}\bigl(-k-\frac{d_2}{2}\bigr)_{\bm,d}}
{\bigl(\mu-\frac{d_2}{2}+(l,\underline{0}_{\lceil r/2\rceil-1})\bigr)_{\bm+\bl-(l,\underline{0}_{\lceil r/2\rceil-1}),d}} \notag \\*
&\hspace{265pt}\times F_{\bm,\bl}^+[f](z_1,z_2^\itinv), \label{formula_holo_general-even}
\end{align}
and for odd $r$ we have 
\begin{align}
&(\lambda)_{(\underline{2k}_{\lfloor r/2\rfloor},\underline{k}_{\lfloor r/2\rfloor},k-l),d}
\Bigl\langle \det_{\fn^+_2}(x_2)^kf(x_2^\itinv),e^{(x|\overline{z})_{\fp^+}}\Bigr\rangle_{\lambda,x} \notag \\
&=\biggl(\lambda+k-\frac{d}{2}\biggl\lfloor \frac{r}{2}\biggr\rfloor+\frac{d_2}{2}+(k-l,\underline{0}_{\lfloor r/2\rfloor})\biggr)_{(l,\underline{k}_{\lfloor r/2\rfloor-1},k-l),d}
F_{l,-}^\downarrow\biggl( \begin{matrix} -k \\ \mu \end{matrix};f;z\biggr) \notag \\
&=(-1)^{k\lfloor r/2\rfloor}\biggl(\mu-\frac{d_2}{2}+(l,\underline{0}_{\lfloor r/2\rfloor})\biggr)_{(k-l,\underline{k}_{\lfloor r/2\rfloor-1},l),d}\det_{\fn^+_2}(z_2)^k \notag \\*
&\eqspace{}\times \sum_{\bm\in\BZ_{++}^{\lfloor r/2\rfloor}}\sum_{\substack{\bl\in(\BZ_{\ge 0})^{\lceil r/2\rceil} \\ |\bl|=l}}
\frac{(-k+(l,\underline{0}_{\lceil r/2\rceil-1}))_{\bm+\bl-(l,\underline{0}_{\lceil r/2\rceil-1}),d}\bigl(-k-\frac{d_2}{2}\bigr)_{\bm,d}}
{\bigl(\mu-\frac{d_2}{2}+(l,\underline{0}_{\lceil r/2\rceil-1})\bigr)_{\bm+\bl-(l,\underline{0}_{\lceil r/2\rceil-1}),d}} \notag \\*
&\hspace{265pt}\times F_{\bm,\bl}^+[f](z_1,z_2^\itinv). \label{formula_holo_general-odd}
\end{align}
These are holomorphically continued for all $\mu\in\BC$. Moreover, all non-zero $F_{\bm,\bl}^+[f](z_1,z_2^\itinv)$ are linearly independent. 
When $r$ is even, the coefficient of (\ref{formula_holo_general-even}) for $(\bm,\bl)=((\underline{k}_{r/2-1},\allowbreak k-l),(\underline{0}_{r/2-1},l))$ is independent of $\mu$ and non-zero, 
and when $r$ is odd, the coefficient of (\ref{formula_holo_general-odd}) for $(\bm,\bl)=(\underline{k}_{\lfloor r/2\rfloor},(\underline{0}_{\lfloor r/2\rfloor},l))$ is independent of $\mu$ and non-zero. 
Moreover, if $f(x_2)=(x_2|e')_{\fn^+_2}^l$ for some primitive idempotent $e'\in\fn^+_2$, then for even $r$, by Proposition \ref{prop_example_Fml_general}\,(1), (4) 
and Example \ref{example_Phim_general} we have 
\begin{align*}
&\det_{\fn^+_2}(z_2)^kF_{(\underline{k}_{r/2-1},k-l),(\underline{0}_{r/2-1},l)}^+\bigl[(\cdot|e')_{\fn^+_2}^l\bigr](z_1,z_2^\itinv) \\
&=\frac{\det_{\fn^+}(z_1)^{k-l}}{(-(k-l,\underline{k}_{r/2-1}))_{\underline{k-l}_{r/2},d}} \det_{\fn^+_2}(z_2)^l
F_{(\underline{l}_{r/2-1},0),(\underline{0}_{r/2-1},l)}^+\bigl[(\cdot|e')_{\fn^+_2}^l\bigr](z_1,z_2^\itinv) \\
&=\frac{(-k)_l\det_{\fn^+}(z_1)^{k-l}}{(-k)_{(k,\underline{k-l}_{r/2-1}),d}}\frac{(-l)_{\underline{l}_{r/2-1},d}}{\bigl(-l-\frac{d}{2}\bigr)_{\underline{l}_{r/2-1},d}}
\det_{\fn^{+\prime\prime}_2}(z_2'')^l\tilde{\Phi}_{\underline{l}_{r/2-1}}^{\fp^{+\prime\prime}_1,\fp^{-\prime\prime}_2}(z_1'',(z_2'')^\itinv) \\
&=\frac{(-k)_l}{(-k)_{\underline{k}_{r/2},d}}\det_{\fn^+}(z_1)^{k-l}((z_1'')^\sharp|z_2'')_{\fn^{+\prime\prime}_2}^l\ne 0, 
\end{align*}
and for odd $r$, by Proposition \ref{prop_example_Fml_general}\,(2) we have 
\begin{align*}
&\det_{\fn^+_2}(z_2)^k F_{\underline{k}_{\lfloor r/2\rfloor},(\underline{0}_{\lfloor r/2\rfloor},l)}^+\bigl[(\cdot|e')_{\fn^+_2}^l\bigr](z_1,z_2^\itinv) \\
&=\frac{(-1)^l(-k)_l(-l)_l}{(-k)_{\underline{k}_{\lfloor r/2\rfloor},d}\bigl(-k-\frac{d_2}{2}(r-1)\bigr)_ll!}
(z_2|(z_1)^\sharp)_{\fn^+_2}^{k-l}((z_1)^\sharp|e')^l_{\fn^+_2} \\
&=\frac{(-k)_l}{(-k)_{(\underline{k}_{\lfloor r/2\rfloor},l),d}}(z_2|(z_1)^\sharp)_{\fn^+_2}^{k-l}((z_1)^\sharp|e')^l_{\fn^+_2}\ne 0,  
\end{align*}
and by the $K^\BC_1$-equivariance, 
\begin{align*}
F_{(\underline{k}_{r/2-1},k-l),(\underline{0}_{r/2-1},l)}^+[f](z_1,z_2^\itinv)&\ne 0 && (r\colon\text{even}), \\
F_{\underline{k}_{\lfloor r/2\rfloor},(\underline{0}_{\lfloor r/2\rfloor},l)}^+[f](z_1,z_2^\itinv)&\ne 0 && (r\colon\text{odd}) 
\end{align*}
hold for all non-zero $f(x_2)\in\cP_{(l,\underline{0}_{r-1})}(\fp^+_2)$. 
Hence (\ref{formula_holo_general-odd}) is non-zero for all $\mu\in\BC$ if $f(x_2)\ne 0$. 
\end{proof}

\begin{proof}[Proof of Theorem \ref{thm_general-}\,(3)]
Combining Theorem \ref{thm_general-}\,(1) and (\ref{formula_diff_expr_general-}), for $\Re\lambda>\frac{2n}{r}-1$, $k,l\in\BZ_{\ge 0}$, $k\ge l$, we have 
\begin{align*}
&C_{r,-}^{d,d_2}(\lambda,k,l)F_{l,-}^\downarrow\biggl(\begin{matrix} -k \\ -\lambda-2k+\frac{n}{r} \end{matrix};f;z\biggr) \\
&=\frac{C_{r,-}^{d,d_2}(\lambda+2(k-l),l,l)}{(\lambda)_{\underline{2(k-l)}_r,d}}\det_{\fn^+}(z)^{-\lambda+\frac{n}{r}}\det_{\fn^+_2}\biggl(\frac{\partial}{\partial z_2}\biggr)^{k-l} \\* 
&\eqspace{}\times\det_{\fn^+}(z)^{\lambda+2(k-l)-\frac{n}{r}}F_{l,-}^\downarrow\biggl(\begin{matrix} -l \\ -\lambda-2k+\frac{n}{r} \end{matrix};f;z\biggr), 
\end{align*}
that is, 
\begin{align*}
F_{l,-}^\downarrow\biggl(\begin{matrix} -k \\ -\lambda-2k+\frac{n}{r} \end{matrix};f;z\biggr)
&=\frac{\det_{\fn^+}(z)^{-\lambda+\frac{n}{r}}}{\bigl(\lambda-\frac{d}{4}r+\frac{d_2}{2}+(\underline{k}_{r-1},k-l)\bigr)_{\underline{k-l}_r,d_2}}
\det_{\fn^+_2}\biggl(\frac{\partial}{\partial z_2}\biggr)^{k-l} \\*
&\eqspace{}\times \det_{\fn^+}(z)^{\lambda+2(k-l)-\frac{n}{r}}F_{l,-}^\downarrow\biggl(\begin{matrix} -l \\ -\lambda-2k+\frac{n}{r} \end{matrix};f;z\biggr). 
\end{align*}
Then this is meromorphically continued for all $\lambda\in\BC$, and if $\lambda=\frac{n}{r}-a=\frac{d}{2}(r-1)+1-a$, $a=1,2,\ldots,k-l$, then the left hand side of 
\begin{align*}
\det_{\fn^+}(z)^{-a}F_{l,-}^\downarrow\biggl(\begin{matrix} -k \\ a-2k \end{matrix};f;z\biggr)
=\frac{1}{\bigl(\frac{d_2}{2}(r-1)+1-a+(\underline{k}_{r-1},k-l)\bigr)_{\underline{k-l}_r,d_2}}\hspace{20pt} \\*
{}\times \det_{\fn^+_2}\biggl(\frac{\partial}{\partial z_2}\biggr)^{k-l}\det_{\fn^+}(z)^{2(k-l)-a}F_{l,-}^\downarrow\biggl(\begin{matrix} -l \\ a-2k \end{matrix};f;z\biggr) 
\end{align*}
becomes a polynomial (note that $d_2=d/2$ or $r=2$ hold). Hence $F_{l,-}^\downarrow\bigl(\begin{smallmatrix} -k \\ a-2k \end{smallmatrix};\allowbreak f;z\bigr)$ is divisible by $\det_{\fn^+}(z)^a$. 
Also, by $\cB_{\frac{n}{r}-a}\det_{\fn^+}(z)^a=\det_{\fn^+}(z)^a\cB_{\frac{n}{r}+a}$ (see the proof of \cite[Proposition XV.2.4]{FK}), we have 
\begin{align*}
(\cB_{\frac{n}{r}+a})_1\det_{\fn^+}(z)^{-a}F_{l,-}^\downarrow\biggl(\begin{matrix} -k \\ a-2k \end{matrix};f;z\biggr)
&=\det_{\fn^+}(z)^{-a}(\cB_{\frac{n}{r}-a})_1F_{l,-}^\downarrow\biggl(\begin{matrix} -k \\ a-2k \end{matrix};f;z\biggr) =0, 
\end{align*}
and hence 
\begin{align*}
&\biggl( \det_{\fn^+_2}(x_2)^{k-a}f(x_2^\itinv)\mapsto \det_{\fn^+_2}(z)^{-a}F_{l,-}^\downarrow\biggl(\begin{matrix} -k \\ a-2k \end{matrix};f;z\biggr)\biggr) \\
&\in\Hom_{K_1^\BC}\bigl(\cP_{(\underline{k-a}_{r-1},k-l-a)}(\fp^+_2),\Sol_{\cP(\fp^+)}((\cB_{\frac{n}{r}+a})_1)\bigr) \\
&\eqspace{}\simeq \BC\biggl( \det_{\fn^+_2}(x_2)^{k-a}f(x_2^\itinv)\mapsto F_{l,-}^\downarrow\biggl(\begin{matrix}a-k \\ a-2k\end{matrix};f;z\biggr)\biggr)
\end{align*}
holds. Then comparing the values at $z_1=0$, for $a,k\in\BZ_{\ge 0}$ with $a\le k-l$, we get 
\begin{equation}\label{formula_factorize_general-}
F_{l,-}^\downarrow\biggl(\begin{matrix}-k \\ a-2k\end{matrix};f;z\biggr)=\det_{\fn^+}(z)^a F_{l,-}^\downarrow\biggl(\begin{matrix}a-k \\ a-2k\end{matrix};f;z\biggr)
\end{equation}
(see \cite[Theorem 6.6]{N3} for more general case). Now we want to prove this for general $(a,k)\in\BC^2$. 
We recall that the left hand side of (\ref{formula_factorize_general-}) is given by 
\begin{align}
F_{l,-}^\downarrow\biggl(\begin{matrix}-k \\ a-2k\end{matrix};f;z\biggr) 
&=\sum_{\bm\in\BZ_{++}^{\lfloor r/2\rfloor}}\sum_{\substack{\bl\in(\BZ_{\ge 0})^{\lceil r/2\rceil} \\ |\bl|=l}}
\frac{(-k+(l,\underline{0}_{\lceil r/2\rceil-1}))_{\bm+\bl-(l,\underline{0}_{\lceil r/2\rceil-1}),d}\bigl(-k-\frac{d_2}{2}\bigr)_{\bm,d}}
{\bigl(a-2k-\frac{d_2}{2}+(l,\underline{0}_{\lceil r/2\rceil-1})\bigr)_{\bm+\bl-(l,\underline{0}_{\lceil r/2\rceil-1}),d}} \notag \\*
&\eqspace{}\times \det_{\fn^+_2}(z_2)^kF_{\bm,\bl}^+[f](z_1,z_2^\itinv). \label{formula_factorize_general-_left}
\end{align}
Similarly, by (\ref{binom_general}), the right hand side is computed as 
\begin{align*}
&\det_{\fn^+}(z)^a F_{l,-}^\downarrow\biggl(\begin{matrix}a-k \\ a-2k\end{matrix};f;z\biggr) \\
&=\sum_{\bm'\in\BZ_{++}^{\lfloor r/2\rfloor}}\sum_{\substack{\bl'\in(\BZ_{\ge 0})^{\lceil r/2\rceil} \\ |\bl'|=l}}
\frac{(a-k+(l,\underline{0}_{\lceil r/2\rceil-1}))_{\bm'+\bl'-(l,\underline{0}_{\lceil r/2\rceil-1}),d}\bigl(a-k-\frac{d_2}{2}\bigr)_{\bm',d}}
{\bigl(a-2k-\frac{d_2}{2}+(l,\underline{0}_{\lceil r/2\rceil-1})\bigr)_{\bm'+\bl'-(l,\underline{0}_{\lceil r/2\rceil-1}),d}} \\*
&\eqspace{}\times \det_{\fn^+_2}(z_2)^{k-a}F_{\bm',\bl'}^+[f](z_1,z_2^\itinv)
\det_{\fn^+_2}(z_2)^a h_{\fn^+}(z_1,P(z_2^\itinv)z_1)^{a/2} \\
&=\sum_{\bm'\in\BZ_{++}^{\lfloor r/2\rfloor}}\sum_{\substack{\bl'\in(\BZ_{\ge 0})^{\lceil r/2\rceil} \\ |\bl'|=l}}
\frac{(a-k+(l,\underline{0}_{\lceil r/2\rceil-1}))_{\bm'+\bl'-(l,\underline{0}_{\lceil r/2\rceil-1}),d}\bigl(a-k-\frac{d_2}{2}\bigr)_{\bm',d}}
{\bigl(a-2k-\frac{d_2}{2}+(l,\underline{0}_{\lceil r/2\rceil-1})\bigr)_{\bm'+\bl'-(l,\underline{0}_{\lceil r/2\rceil-1}),d}} \\*
&\eqspace{}\times \det_{\fn^+_2}(z_2)^kF_{\bm',\bl'}^+[f](z_1,z_2^\itinv)
\sum_{\bm''\in\BZ_{++}^{\lfloor r/2\rfloor}}(-a)_{\bm'',d}\tilde{\Phi}_{\bm''}^{\fp^+_1,\fp^-_2}(z_1,z_2^\itinv). 
\end{align*}
Then we have 
\begin{align*}
&\bigl( f\mapsto F_{\bm',\bl'}^+[f](z_1,y_2)\tilde{\Phi}_{\bm''}^{\fp^+_1,\fp^-_2}(z_1,y_2)\bigr) \\
&\in\Hom_{K_1^\BC}\bigl(\cP_{(l,\underline{0}_{r-1})}(\fp^-_2),\cP_{\langle\bm'\rangle}(\fp^+_1)\otimes\cP_{\phi_+(\bm',\bl')}(\fp^-_2)
\otimes\cP_{\langle\bm''\rangle}(\fp^+_1)\otimes\cP_{\bm^{\prime\prime 2}}(\fp^-_2)\bigr) \\
&\subset \bigoplus_{\substack{\bk \\ |\bk|=|\langle\bm'\rangle|+|\langle\bm''\rangle| \\ k_j\ge \max\{\langle\bm'\rangle_j,\langle\bm''\rangle_j\}}}\,
\bigoplus_{\substack{\bk'\in\BZ_{++}^r \\ |\bk'|=|\phi_+(\bm',\bl')|+|\bm^{\prime\prime 2}| \\ k_j'\ge \max\{\phi_+(\bm',\bl')_j,(\bm^{\prime\prime 2})_j\}}}
\Hom_{K_1^\BC}\bigl(\cP_{(l,\underline{0}_{r-1})}(\fp^-_2),\cP_{\bk}(\fp^+_1)\otimes\cP_{\bk'}(\fp^-_2)\bigr), 
\end{align*}
where $\bk\in\BZ_{++}^{\lfloor r/\varepsilon_1\rfloor}$ (Case 1 ($n'\ne 2$), Case 2) or $\bk\in(\BZ_{++}^{\lfloor r/2\rfloor})^2$ (Case 1 ($n'=2$), Case 3). 
Now since $\Hom_{K_1^\BC}\bigl(\cP_{(l,\underline{0}_{r-1})}(\fp^-_2),\cP_{\bk}(\fp^+_1)\otimes\cP_{\bk'}(\fp^-_2)\bigr)\ne\{0\}$ implies 
$(\bk,\bk')=(\langle\bm\rangle,\phi_+(\bm,\bl))$ for some $(\bm,\bl)\in\BZ_{++}^{\lfloor r/2\rfloor}\times(\BZ_{\ge 0})^{\lceil r/2\rceil}$ 
satisfying $|\bl|=l$ and (\ref{formula_cond_ml_general-}), we set 
\begin{align*}
F_{\bm',\bl',\bm'',\bm,\bl}[f](z_1,y_2)&:=\Proj_{\phi_+(\bm,\bl),y_2}^{\fp^-_2}\bigl(F_{\bm',\bl'}^+[f](z_1,y_2)\tilde{\Phi}_{\bm''}^{\fp^+_1,\fp^-_2}(z_1,y_2)\bigr) \\
&\hspace{5pt}\in\cP_{\langle\bm\rangle}(\fp^+_1)_{z_1}\otimes\cP_{\phi_+(\bm,\bl)}(\fp^-_2)_{y_2}, 
\end{align*}
so that 
\[ F_{\bm',\bl'}^+[f](z_1,y_2)\tilde{\Phi}_{\bm''}^{\fp^+_1,\fp^-_2}(z_1,y_2)
=\sum_{\substack{\bm\in\BZ_{++}^{\lfloor r/2\rfloor} \\ |\bm|=|\bm'|+|\bm''| \\ m_j\ge \max\{m_j',m_j''\}}}\,
\sum_{\substack{\bl\in(\BZ_{\ge 0})^{\lceil r/2\rceil} \\ |\bl|=l \\ m_j+l_j\ge m_j'+l_j'}}F_{\bm',\bl',\bm'',\bm,\bl}[f](z_1,y_2). \]
Then the right hand side of (\ref{formula_factorize_general-}) becomes 
\begin{align}
&\det_{\fn^+}(z)^a F_{l,-}^\downarrow\biggl(\begin{matrix}a-k \\ a-2k\end{matrix};f;z\biggr) \notag \\
&=\det_{\fn^+_2}(z_2)^k\sum_{\bm\in\BZ_{++}^{\lfloor r/2\rfloor}}\sum_{\substack{\bl\in(\BZ_{\ge 0})^{\lceil r/2\rceil} \\ |\bl'|=l}}
\sum_{\substack{\bm',\bm''\in\BZ_{++}^{\lfloor r/2\rfloor} \\ |\bm'|+|\bm''|=|\bm| \\ m_j',m_j''\le m_j}}
\sum_{\substack{\bl'\in(\BZ_{\ge 0})^{\lceil r/2\rceil} \\ |\bl'|=l \\ m_j'+l_j'\le m_j+l_j}}(-a)_{\bm'',d} \notag \\*
&\eqspace{}\times \frac{(a-k+(l,\underline{0}_{\lceil r/2\rceil-1}))_{\bm'+\bl'-(l,\underline{0}_{\lceil r/2\rceil-1}),d}\bigl(a-k-\frac{d_2}{2}\bigr)_{\bm',d}}
{\bigl(a-2k-\frac{d_2}{2}+(l,\underline{0}_{\lceil r/2\rceil-1})\bigr)_{\bm'+\bl'-(l,\underline{0}_{\lceil r/2\rceil-1}),d}}
F_{\bm',\bl',\bm'',\bm,\bl}[f](z_1,z_2^\itinv). \label{formula_factorize_general-_right}
\end{align}
If $a,k\in\BZ_{\ge 0}$, $a\le k-l$, then in (\ref{formula_factorize_general-_left}), only $m_1+l_1\le k$ terms remain, 
and both in (\ref{formula_factorize_general-_left}) and (\ref{formula_factorize_general-_right}), all denominators in the terms with $m_1+l_1\le k$ are non-zero. 
Hence, for $m_1+l_1\le k$, by comparing the both formulas, we get 
\begin{align*}
&\frac{(-k+(l,\underline{0}_{\lceil r/2\rceil-1}))_{\bm+\bl-(l,\underline{0}_{\lceil r/2\rceil-1}),d}\bigl(-k-\frac{d_2}{2}\bigr)_{\bm,d}}
{\bigl(a-2k-\frac{d_2}{2}+(l,\underline{0}_{\lceil r/2\rceil-1})\bigr)_{\bm+\bl-(l,\underline{0}_{\lceil r/2\rceil-1}),d}}F_{\bm,\bl}^+[f](z_1,z_2^\itinv) \\
&=\sum_{\substack{\bm',\bm''\in\BZ_{++}^{\lfloor r/2\rfloor} \\ |\bm'|+|\bm''|=|\bm| \\ m_j',m_j''\le m_j}}
\sum_{\substack{\bl'\in(\BZ_{\ge 0})^{\lceil r/2\rceil} \\ |\bl'|=l \\ m_j'+l_j'\le m_j+l_j}}(-a)_{\bm'',d}F_{\bm',\bl',\bm'',\bm,\bl}[f](z_1,z_2^\itinv) \\*
&\eqspace{}\times \frac{(a-k+(l,\underline{0}_{\lceil r/2\rceil-1}))_{\bm'+\bl'-(l,\underline{0}_{\lceil r/2\rceil-1}),d}\bigl(a-k-\frac{d_2}{2}\bigr)_{\bm',d}}
{\bigl(a-2k-\frac{d_2}{2}+(l,\underline{0}_{\lceil r/2\rceil-1})\bigr)_{\bm'+\bl'-(l,\underline{0}_{\lceil r/2\rceil-1}),d}}. 
\end{align*}
Then since $\{(a,k)\in(\BZ_{\ge 0})^2\mid k\ge a+l, k\ge m_1+l_1\}\subset\BC^2$ is Zariski dense, this holds for all $(a,k)\in\BC^2$ without the assumption $m_1+l_1\le k$, 
and hence (\ref{formula_factorize_general-}) holds for all $(a,k)\in\BC^2$. This completes the proof of Theorem \ref{thm_general-}\,(3).  
\end{proof}

\begin{example}\label{example_scalar_general}
Suppose $l=0$, $f(x_2)=1\in\cP_{(0,\ldots,0)}(\fp^+_2)$. Then we have 
\begin{gather*}
F_{\bm,\underline{0}_{\lceil r/2\rceil}}^-[1](z_1,z_2^\itinv)=F_{\bm,\underline{0}_{\lceil r/2\rceil}}^+[1](z_1,z_2^\itinv)
=\tilde{\Phi}_\bm^{\fp^+_1,\fp^-_2}(z_1,z_2^\itinv), \\
F_{0,\pm}^\downarrow\biggl(\begin{matrix} \nu \\ \mu \end{matrix};1;z\biggr)
=\det_{\fn^+_2}(z_2)^{-\nu}\sum_{\bm\in\BZ_{++}^{\lfloor r/2\rfloor}}
\frac{(\nu)_{\bm,d}\bigl(\nu-\frac{d_2}{2}\bigr)_{\bm,d}}{\bigl(\mu-\frac{d_2}{2}\bigr)_{\bm,d}}\tilde{\Phi}_\bm^{\fp^+_1,\fp^-_2}(z_1,z_2^\itinv). 
\end{gather*}
Especially, for $k\in\BZ_{\ge 0}$, $\Re\lambda>\frac{2n}{r}-1$, we have 
\begin{align*}
&\Bigl\langle \det_{\fn^+_2}(x_2)^k,e^{(x|\overline{z})_{\fp^+}}\Bigr\rangle_{\lambda,x}
=C_{r,\pm}^{d,d_2}(\lambda,k,0)F_{0,\pm}^\downarrow\biggl(\begin{matrix} -k \\ -\lambda-2k+\frac{n}{r} \end{matrix};1;z\biggr) \\
&=\frac{\bigl(\lambda+k-\frac{d}{2}\bigl\lceil \frac{r}{2}\bigr\rceil+\frac{d_2}{2}\bigr)_{\underline{k}_{\lfloor r/2\rfloor},d}}
{(\lambda)_{(\underline{2k}_{\lfloor r/2\rfloor},\underline{k}_{\lceil r/2\rceil}),d}}\det_{\fn^+_2}(z_2)^k\sum_{\bm\in\BZ_{++}^{\lfloor r/2\rfloor}}
\frac{(-k)_{\bm,d}\bigl(-k-\frac{d_2}{2}\bigr)_{\bm,d}}{\bigl(-\lambda-2k+\frac{n}{r}-\frac{d_2}{2}\bigr)_{\bm,d}}\tilde{\Phi}_\bm^{\fp^+_1,\fp^-_2}(z_1,z_2^\itinv). 
\end{align*}
In addition, Theorems \ref{thm_general+}\,(3), \ref{thm_general-}\,(3) are equivalent to 
\begin{align*}
&\sum_{\bm\in\BZ_{++}^{\lfloor r/2\rfloor}}
\frac{(\nu)_{\bm,d}\bigl(\nu-\frac{d_2}{2}\bigr)_{\bm,d}}{\bigl(\mu-\frac{d_2}{2}\bigr)_{\bm,d}}\tilde{\Phi}_\bm^{\fp^+_1,\fp^-_2}(z_1,z_2^\itinv) \\
&=h_{\fn^+}(z_1,P(z_2)^{-1}z_1)^{(\mu-2\nu)/2}\sum_{\bm\in\BZ_{++}^{\lfloor r/2\rfloor}}
\frac{(\mu-\nu)_{\bm,d}\bigl(\mu-\nu-\frac{d_2}{2}\bigr)_{\bm,d}}{\bigl(\mu-\frac{d_2}{2}\bigr)_{\bm,d}}\tilde{\Phi}_\bm^{\fp^+_1,\fp^-_2}(z_1,z_2^\itinv). 
\end{align*}
\end{example}

\begin{example}\label{example_rank2_general}
Suppose $r=2$. We fix a Jordan frame $\{e',e''\}\subset\fn^+_2$, and consider $f(x_2)=(x_2|e')_{\fn^+_2}^l$ or $f(y_2)=(y_2|e'')_{\fn^+_2}^l$. Then we have 
\begin{gather*}
\begin{split}
F_{m,l}^-\bigl[(\cdot|e')_{\fn^+_2}^l\bigr](z_1,z_2^\itinv)&=\frac{1}{m!}\biggl(-\frac{\det_{\fn^+}(z_1)}{\det_{\fn^+}(z_2)}\biggr)^m(z_2|e')_{\fn^+_2}^l, \\
F_{m,l}^+\bigl[(\cdot|e'')_{\fn^+_2}^l\bigr](z_1,z_2^\itinv)&=\frac{1}{m!}\biggl(-\frac{\det_{\fn^+}(z_1)}{\det_{\fn^+}(z_2)}\biggr)^m(z_2^\itinv|e'')_{\fn^+_2}^l, 
\end{split}\\
\begin{split}
F_{l,+}^\downarrow\biggl(\begin{matrix} \nu \\ \mu \end{matrix};(\cdot|e')_{\fn^+_2}^l;z\biggr) 
&=\det_{\fn^+}(z_2)^{-\nu}\sum_{m=0}^\infty \frac{(\nu)_m\bigl(\nu-l-\frac{d_2}{2}\bigr)_m}{\bigl(\mu-l-\frac{d_2}{2}\bigr)_m m!}
\biggl(-\frac{\det_{\fn^+}(z_1)}{\det_{\fn^+}(z_2)}\biggr)^m(z_2|e')_{\fn^+_2}^l \\*
&=\det_{\fn^+}(z_2)^{-\nu}{}_2F_1\biggl(\begin{matrix} \nu,\,\nu-l-\frac{d_2}{2} \\ \mu-l-\frac{d_2}{2} \end{matrix};-\frac{\det_{\fn^+}(z_1)}{\det_{\fn^+}(z_2)}\biggr)(z_2|e')_{\fn^+_2}^l, \\
F_{l,-}^\downarrow\biggl(\begin{matrix} \nu \\ \mu \end{matrix};(\cdot|e'')_{\fn^+_2}^l;z\biggr) 
&=\det_{\fn^+}(z_2)^{-\nu}\sum_{m=0}^\infty \frac{(\nu+l)_m\bigl(\nu-\frac{d_2}{2}\bigr)_m}{\bigl(\mu+l-\frac{d_2}{2}\bigr)_m m!}
\biggl(-\frac{\det_{\fn^+}(z_1)}{\det_{\fn^+}(z_2)}\biggr)^m(z_2^\itinv|e'')_{\fn^+_2}^l \\*
&=\det_{\fn^+}(z_2)^{-\nu-l}{}_2F_1\biggl(\begin{matrix} \nu+l,\,\nu-\frac{d_2}{2} \\ \mu+l-\frac{d_2}{2} \end{matrix};-\frac{\det_{\fn^+}(z_1)}{\det_{\fn^+}(z_2)}\biggr)(z_2|e')_{\fn^+_2}^l. 
\end{split}
\end{gather*}
Especially, for $\bk\in\BZ_{++}^2$, $\Re\lambda>\frac{2n}{r}-1$, we have 
\begin{align*}
&\Bigl\langle \Delta_\bk^{\fn^+_2}(x_2),e^{(x|\overline{z})_{\fp^+}}\Bigr\rangle_{\lambda,x}
=\Bigl\langle (x_2|e')_{\fn^+_2}^{k_1-k_2}\det_{\fn^+_2}(x_2)^{k_2},e^{(x|\overline{z})_{\fp^+}}\Bigr\rangle_{\lambda,x} \\
&=\Bigl\langle (x_2^\itinv|e'')_{\fn^+_2}^{k_1-k_2}\det_{\fn^+_2}(x_2)^{k_1},e^{(x|\overline{z})_{\fp^+}}\Bigr\rangle_{\lambda,x} \\
&=C_{2,+}^{d,d_2}(\lambda,k_2,k_1-k_2)F_{k_1-k_2,+}^\downarrow\biggl(\begin{matrix} -k_2 \\ -\lambda-2k_2+\frac{n}{2} \end{matrix};(\cdot|e')_{\fn^+_2}^{k_1-k_2};z\biggr) \\
&=C_{2,-}^{d,d_2}(\lambda,k_1,k_1-k_2)F_{k_1-k_2,-}^\downarrow\biggl(\begin{matrix} -k_1 \\ -\lambda-2k_1+\frac{n}{2} \end{matrix};(\cdot|e'')_{\fn^+_2}^{k_1-k_2};z\biggr) \\
&=\frac{\bigl(\lambda+k_1-\frac{d-d_2}{2}\bigr)_{k_2}}{(\lambda)_{(k_1+k_2,k_2),d}}
{}_2F_1\biggl(\begin{matrix} -k_2,\,-k_1-\frac{d_2}{2} \\ -\lambda-k_1-k_2+\frac{d-d_2}{2}+1 \end{matrix};-\frac{\det_{\fn^+}(z_1)}{\det_{\fn^+}(z_2)}\biggr)\Delta_\bk^{\fn^+_2}(z_2). 
\end{align*}
\end{example}

\begin{example}
Suppose $r\ge 3$, $-\nu=k=l\in\BZ_{\ge 0}$, $f(y_2)\in\cP_{(l,\underline{0}_{r-1})}(\fp^+_2)$. Then since 
\begin{align*}
&(-l+(l,\underline{0}_{\lceil r/2\rceil-1}))_{\bm+\bl-(l,\underline{0}_{\lceil r/2\rceil-1}),d}F_{\bm,\bl}^+[f](z_1,z_2^\itinv) \\
&=((0,\underline{-l}_{\lceil r/2\rceil-1}))_{\bm+\bl-(l,\underline{0}_{\lceil r/2\rceil-1}),d}F_{\bm,\bl}^+[f](z_1,z_2^\itinv)\ne 0
\end{align*}
implies $m_1+l_1=l$ and $\phi_+(\bm,\bl)\in\BZ_+^r$, that is, 
\begin{gather*}
m_1\le l, \qquad m_{\lceil r/2\rceil}=0 \quad (\text{for even }r), \\
\bl=\bl(\bm):=(l-m_1,m_1-m_2,m_2-m_3,\ldots,m_{\lceil r/2\rceil-2}-m_{\lceil r/2\rceil-1},m_{\lceil r/2\rceil-1}), 
\end{gather*}
so that $\phi_+(\bm,\bl(\bm))=(l,\bm^2)$, $\bm+\bl(\bm)=(l,\bm)$, we have 
\begin{align*}
&F_{l,-}^\downarrow\biggl(\begin{matrix} -l \\ \mu \end{matrix};f;z\biggr) \\
&=\det_{\fn^+_2}(z_2)^l\!\sum_{\bm\in\BZ_{++}^{\lceil r/2\rceil-1}}
\frac{((0,\underline{-l}_{\lceil r/2\rceil-1}))_{\bm+\bl(\bm)-(l,\underline{0}_{\lceil r/2\rceil-1}),d}\bigl(-l-\frac{d_2}{2}\bigr)_{\bm,d}}
{\bigl(\mu-\frac{d_2}{2}+(l,\underline{0}_{\lceil r/2\rceil-1})\bigr)_{\bm+\bl(\bm)-(l,\underline{0}_{\lceil r/2\rceil-1}),d}} F_{\bm,\bl(\bm)}^+[f](z_1,z_2^\itinv) \\ 
&=\det_{\fn^+_2}(z_2)^l\sum_{\bm\in\BZ_{++}^{\lceil r/2\rceil-1}} 
\frac{\bigl(-l-\frac{d}{2}\bigr)_{\bm,d}\bigl(-l-\frac{d_2}{2}\bigr)_{\bm,d}}{\bigl(\mu-\frac{d+d_2}{2}\bigr)_{\bm,d}}F_{\bm,\bl(\bm)}^+[f](z_1,z_2^\itinv), 
\end{align*}
and hence, for $\Re\lambda>\frac{2n}{r}-1$ we have 
\begin{align*}
&\Bigl\langle \det_{\fn^+_2}(x_2)^lf(x_2^\itinv),e^{(x|\overline{z})_{\fp^+}}\Bigr\rangle_{\lambda,x}
=C_{r,-}^{d,d_2}(\lambda,l,l)F_{l,-}^\downarrow\biggl(\begin{matrix} -l \\ -\lambda-2l+\frac{n}{r} \end{matrix};f;z\biggr) \\
&=\frac{\bigl(\lambda+l-\frac{d}{2}\bigl\lfloor \frac{r}{2}\bigr\rfloor+\frac{d_2}{2}\bigr)_{\underline{l}_{\lceil r/2\rceil-1},d}}
{(\lambda)_{(\underline{2l}_{\lceil r/2\rceil-1},\underline{l}_{\lfloor r/2\rfloor}),d}} \\*
&\eqspace{}\times\det_{\fn^+_2}(z_2)^l\sum_{\bm\in\BZ_{++}^{\lceil r/2\rceil-1}} 
\frac{\bigl(-l-\frac{d}{2}\bigr)_{\bm,d}\bigl(-l-\frac{d_2}{2}\bigr)_{\bm,d}}
{\bigl(-\lambda-2l+\frac{d}{2}(r-2)+1-\frac{d_2}{2}\bigr)_{\bm,d}}F_{\bm,\bl(\bm)}^+[f](z_1,z_2^\itinv). 
\end{align*}
Especially, we fix a primitive idempotent $e'\in\fn^+_2$ and let $\fp^+_2(e')_0=:\fp^{+\prime\prime}_2$, $\fp^+(e')_0\cap\fp^+_1=:\allowbreak\fp^{+\prime\prime}_1$ be as in (\ref{Peirce}). 
If $f(y_2)=(y_2|e')_{\fn^+_2}^l$ so that $\det_{\fn^+_2}(x_2)^l\allowbreak f(x_2^\itinv)=\det_{\fn^{+\prime\prime}_2}(x_2'')^l$, 
then by Proposition \ref{prop_example_Fml_general}\,(4), we have 
\begin{align*}
&\Bigl\langle \det_{\fn^{+\prime\prime}_2}(x_2'')^l,e^{(x|\overline{z})_{\fp^+}}\Bigr\rangle_{\lambda,x}
=C_{r,-}^{d,d_2}(\lambda,l,l)F_{l,-}^\downarrow\biggl(\begin{matrix} -l \\ -\lambda-2l+\frac{n}{r} \end{matrix};(\cdot|e')_{\fn^+_2}^l;z\biggr) \\
&=\frac{\bigl(\lambda+l-\frac{d}{2}\bigl\lfloor \frac{r}{2}\bigr\rfloor+\frac{d_2}{2}\bigr)_{\underline{l}_{\lceil r/2\rceil-1},d}}
{(\lambda)_{(\underline{2l}_{\lceil r/2\rceil-1},\underline{l}_{\lfloor r/2\rfloor}),d}}\det_{\fn^{+\prime\prime}_2}(z_2'')^l \\*
&\eqspace{}\times\sum_{\bm\in\BZ_{++}^{\lceil r/2\rceil-1}} 
\frac{(-l)_{\bm,d}\bigl(-l-\frac{d_2}{2}\bigr)_{\bm,d}}
{\bigl(-\lambda-2l+\frac{d}{2}(r-2)+1-\frac{d_2}{2}\bigr)_{\bm,d}}\tilde{\Phi}_\bm^{\fp^{+\prime\prime}_1,\fp^{-\prime\prime}_2}(z_1'',(z_2'')^\itinv). 
\end{align*}
This fits with Example \ref{example_scalar_general} applied for $\fp^{+\prime\prime}_2$. 
\end{example}

\subsection{Results on restriction of $\cH_\lambda(D)$ to subgroups}

In this subsection, we consider the decomposition of $\cH_\lambda(D)$ under $\widetilde{G}_1$. Let $V_\bk$ be an irreducible $K_1$-module isomorphic to $\cP_\bk(\fp^+_2)$. 
Then $\cH_\lambda(D)$ is decomposed under $\widetilde{G}_1$ as 
\[ \cH_\lambda(D)|_{\widetilde{G}_1}\simeq\hsum_{\bk\in\BZ_{++}^r}\cH_{\varepsilon_1\lambda}(D_1,V_\bk). \]
For each case, $\cH_\lambda(D)$ has the minimal $\widetilde{K}$-type 
\[ \chi^{-\lambda}\simeq \begin{cases} \BC_{-\lambda}\boxtimes V_{\underline{0}_{\lfloor n/2\rfloor}}^{[n]\vee} & (G=SO_0(2,n)), \\
V_{\underline{\lambda/2}_r}^{(r)\vee} \boxtimes V_{\underline{\lambda/2}_r}^{(r)}\simeq V_{\underline{\lambda}_r}^{(r)\vee} \boxtimes V_{\underline{0}_r}^{(r)}
\simeq V_{\underline{0}_r}^{(r)\vee} \boxtimes V_{\underline{\lambda}_r}^{(r)} & (G=SU(r,r)), \\
V_{\underline{\lambda/2}_{2r}}^{(2r)\vee} & (G=SO^*(4r)), \end{cases} \]
and $\cH_{\varepsilon_1\lambda}(D_1,V_\bk)$ has the minimal $\widetilde{K}_1$-type 
\begin{align*}
&\chi_1^{-\varepsilon_1\lambda}\otimes V_\bk \\ &\simeq \begin{cases} 
\BC_{-(\lambda+k_1+k_2)}\boxtimes V_{\underline{0}_{\lfloor n'/2\rfloor}}^{[n']\vee}\boxtimes V_{(k_1-k_2,\underline{0}_{\lfloor n''/2\rfloor-1})}^{[n'']\vee} 
& (G_1=SO_0(2,n')\times SO(n''),\; n'\ge 3), \\[5pt]
\BC_{-(\lambda+k_1+k_2)}\hspace{-1pt}\boxtimes\hspace{-1pt} \BC_{-(\lambda+k_1+k_2)}\hspace{-1pt}\boxtimes\hspace{-1pt} V_{(k_1-k_2,\underline{0}_{\lfloor n''/2\rfloor-1})}^{[n'']\vee} 
& (G_1=SO_0(2,n')\times SO(n''),\; n'=2), \\
\BC_{-2(\lambda+k_1+k_2)}\boxtimes V_{(k_1-k_2,\underline{0}_{\lfloor n''/2\rfloor-1})}^{[n'']\vee} & (G_1=SO_0(2,n')\times SO(n''),\; n'=1), \\
V_{\underline{\lambda}_r+2\bk}^{(r)\vee} & (G_1=SO^*(2r)), \\
V_{\underline{\lambda/2}_r+\bk}^{(r)\vee} \boxtimes V_{\underline{\lambda/2}_r+\bk}^{(r)\vee} & (G_1=SO^*(2r)\times SO^*(2r)) \end{cases}
\end{align*}
(note that $SO_0(2,2)\simeq SL(2,\BR)\times SL(2,\BR)$, $SO_0(2,1)\simeq SL(2,\BR)$ hold up to covering). 
In this section, we only deal with the cases $\bk=(k+l,\underline{k}_{r-1})$ or $\bk=(\underline{k}_{r-1},k-l)$. Let $V_l:=V_{(l,\underline{0}_{r-1})}$, and let $V_l^\vee$ be its contragredient. 
Then we have 
\begin{align*}
\cH_{\varepsilon_1\lambda}(D_1,V_{(k+l,\underline{k}_{r-1})})&\simeq \cH_{\varepsilon_1(\lambda+2k)}(D_1,V_l), \\ 
\cH_{\varepsilon_1\lambda}(D_1,V_{(\underline{k}_{r-1},k-l)})&\simeq \cH_{\varepsilon_1(\lambda+2k)}(D_1,V_l^\vee). 
\end{align*}
To describe the intertwining operators, for $\mu\in\BC$, $l\in\BZ_{\ge 0}$ and for $f(x_2)\in\cP_{(l,\underline{0}_{r-1})}(\fp^+_2)$, 
we define a function $F_{l,+}^\uparrow\bigl(\begin{smallmatrix}-\\ \mu\end{smallmatrix};f;y_1,x_2\bigr)$ on $\fp^-_1\oplus\fp^+_2$ by 
\begin{align}
F_{l,+}^\uparrow\biggl(\begin{matrix}-\\ \mu\end{matrix};f;y_1,x_2\biggr)
:=\sum_{\bm\in\BZ_{++}^{\lfloor r/2\rfloor}}\sum_{\substack{\bl\in(\BZ_{\ge 0})^{\lceil r/2\rceil} \\ |\bl|=l}}
\frac{F_{\bm,\bl}^+[f](y_1,x_2)}{\bigl(\mu-\frac{\delta}{2}+(l,\underline{0}_{\lfloor r/2\rfloor-1})\bigr)_{\bm+\bl-(l,\underline{0}_{\lfloor r/2\rfloor-1}),d}}, \label{def_Fup_general+}
\end{align}
where $\delta$ is given in (\ref{str_const_general}), and define a function $F_{l,-}^\uparrow\bigl(\begin{smallmatrix}-\\ \mu\end{smallmatrix};f;y_1,x_2\bigr)$ on $\fp^-_1\oplus\fp^+_2$ by 
\begin{align}
&F_{l,-}^\uparrow\biggl(\begin{matrix}-\\ \mu\end{matrix};f;y_1,x_2\biggr) \notag\\
&:=\begin{cases}
\ds \sum_{\bm\in\BZ_{++}^{\lfloor r/2\rfloor}}\sum_{\substack{\bl\in(\BZ_{\ge 0})^{\lceil r/2\rceil} \\ |\bl|=l}}
\frac{F_{\bm,\bl}^-[f](y_1,x_2)}{\bigl(\mu-\frac{\delta}{2}-(\underline{0}_{\lfloor r/2\rfloor-1},l)\bigr)_{\bm-\bl+(\underline{0}_{\lfloor r/2\rfloor-1},l),d}}
 & (r\colon\text{even}), \\
\ds \sum_{\bm\in\BZ_{++}^{\lfloor r/2\rfloor}}\sum_{\substack{\bl\in(\BZ_{\ge 0})^{\lceil r/2\rceil} \\ |\bl|=l}}
\frac{F_{\bm,\bl}^-[f](y_1,x_2)}{\bigl(\mu-\frac{\delta}{2}\bigr)_{\bm-\bl',d}\bigl(\mu-l-\frac{\delta}{2}-\frac{d}{2}\bigl(\frac{r}{2}-1\bigr)\bigr)_{l-l_{\lceil r/2\rceil}}}
 & (r\colon\text{odd}), \end{cases} \label{def_Fup_general-}
\end{align}
where for odd $r$ let $\bl':=(l_1,\ldots,l_{\lfloor r/2\rfloor})$. 

\begin{theorem}
Let $\Re\lambda>\frac{2n}{r}-1$, $k,l\in\BZ_{\ge 0}$. 
\begin{enumerate}
\item (\cite[Theorems 5.7\,(3), 5.17\,(2), 5.19\,(1)]{N1}). Let $\rK_l(x_2)\in(\cP_{(l,\underline{0}_{r-1})}(\fp^+_2)\allowbreak\otimes \Hom(V_l,\allowbreak\BC))^{K_1}$. Then the linear map 
\begin{gather*}
\cF_{\lambda,k,l,+}^\uparrow\colon \cH_{\varepsilon_1(\lambda+2k)}(D_1,V_l)_{\widetilde{K}_1}\longrightarrow\cH_\lambda(D)_{\widetilde{K}}|_{(\fg_1,\widetilde{K}_1)}, \\
(\cF_{\lambda,k,l,+}^\uparrow f)(x):=\det_{\fn^+}(x_2)^kF_{l,+}^\uparrow\biggl(\begin{matrix}-\\ \lambda+2k\end{matrix};\rK_l;
\frac{1}{\varepsilon_1}\frac{\partial}{\partial x_1},x_2\biggr)f(x_1)
\end{gather*}
intertwines the $(\fg_1,\widetilde{K}_1)$-action, where $F_{l,+}^\uparrow$ is as in (\ref{def_Fup_general+}). 
\item (\cite[Theorem 8.5]{N2}). Let $\rK_l^\vee(y_2)\in(\cP_{(l,\underline{0}_{r-1})}(\fp^-_2)\otimes V_l)^{K_1}$. Then the linear map 
\begin{gather*}
\cF_{\lambda,k,l,+}^\downarrow\colon \cH_\lambda(D)|_{\widetilde{G}_1}\longrightarrow\cH_{\varepsilon_1(\lambda+2k)}(D_1,V_l), \\
(\cF_{\lambda,k,l,+}^\downarrow f)(x_1):=F_{l,+}^\downarrow\biggl(\begin{matrix}-k \\ -\lambda-2k+\frac{n}{r}\end{matrix};\rK_l^\vee;\frac{\partial}{\partial x}\biggr)f(x)\biggr|_{x_2=0}
\end{gather*}
intertwines the $\widetilde{G}_1$-action, where $F_{l,+}^\downarrow$ is as in (\ref{def_Fdown_general+}). 
\item (\cite[Theorem 8.5]{N2}). Suppose $c_{k,l,+}$, $c_{k,l,+}^\vee\in\BC$ satisfy, for $v\in V_l$, 
\begin{align*}
\bigl\Vert \det_{\fn^+_2}(x_2)^k\rK_l(x_2)v\bigr\Vert_{F,x_2}^2&=c_{k,l,+}|v|_{V_l}^2, \\
\bigl\Vert \det_{\fn^+_2}(x_2)^k(v,\rK_l^\vee(\overline{x_2}))\bigr\Vert_{F,x_2}^2&=c_{k,l,+}^\vee|v|_{V_l}^2. 
\end{align*}
Then the operator norms of $\cF_{\lambda,k,l,+}^\uparrow$, $\cF_{\lambda,k,l,+}^\downarrow$ are given by 
\[ c_{k,l,+}^{-1}\Vert \cF_{\lambda,k,l,+}^\uparrow\Vert_{\mathrm{op}}^2=c_{k,l,+}^\vee\Vert \cF_{\lambda,k,l,+}^\downarrow\Vert_{\mathrm{op}}^{-2}=C_{r,+}^{d,d_2}(\lambda,k,l), \]
where $C_{r,+}^{d,d_2}(\lambda,k,l)$ is as in (\ref{const_general+}). 
\end{enumerate}
\end{theorem}

\begin{theorem}
Let $\Re\lambda>\frac{2n}{r}-1$, $k,l\in\BZ_{\ge 0}$ with $k\ge l$. 
\begin{enumerate}
\item (\cite[Theorems 5.7\,(4), 5.19\,(1), 7.2]{N1}). Let $\rK_l^\vee(y_2)\in(\cP_{(l,\underline{0}_{r-1})}(\fp^-_2)\otimes \Hom(V_l^\vee,\BC))^{K_1}$. Then the linear map 
\begin{gather*}
\cF_{\lambda,k,l,-}^\uparrow\colon \cH_{\varepsilon_1(\lambda+2k)}(D_1,V_l^\vee)_{\widetilde{K}_1}\longrightarrow\cH_\lambda(D)_{\widetilde{K}}|_{(\fg_1,\widetilde{K}_1)}, \\
(\cF_{\lambda,k,l,-}^\uparrow f)(x):=\det_{\fn^+}(x_2)^kF_{l,-}^\uparrow\biggl(\begin{matrix}-\\ \lambda+2k\end{matrix};\rK_l^\vee;
\frac{1}{\varepsilon_1}\frac{\partial}{\partial x_1},x_2\biggr)f(x_1)
\end{gather*}
intertwines the $(\fg_1,\widetilde{K}_1)$-action, where $F_{l,-}^\uparrow$ is as in (\ref{def_Fup_general-}). 
\item Let $\rK_l(x_2)\in(\cP_{(l,\underline{0}_{r-1})}(\fp^+_2)\otimes V_l^\vee)^{K_1}$. Then the linear map 
\begin{gather*}
\cF_{\lambda,k,l,-}^\downarrow\colon \cH_\lambda(D)|_{\widetilde{G}_1}\longrightarrow\cH_{\varepsilon_1(\lambda+2k)}(D_1,V_l^\vee), \\
(\cF_{\lambda,k,l,-}^\downarrow f)(x_1):=F_{l,-}^\downarrow\biggl(\begin{matrix}-k \\ -\lambda-2k+\frac{n}{r}\end{matrix};\rK_l;\frac{\partial}{\partial x}\biggr)f(x)\biggr|_{x_2=0}
\end{gather*}
intertwines the $\widetilde{G}_1$-action, where $F_{l,-}^\downarrow$ is as in (\ref{def_Fdown_general-}). 
\item (\cite[Corollary 6.7]{N3}). Suppose $c_{k,l,-}$, $c_{k,l,-}^\vee\in\BR_{>0}$ satisfy, for $v\in V_l^\vee$, 
\begin{align*}
\bigl\Vert \det_{\fn^+_2}(x_2)^k\rK_l^\vee(x_2^\itinv)v\bigr\Vert_{F,x_2}^2&=c_{k,l,-}|v|_{V_l^\vee}^2, \\
\bigl\Vert \det_{\fn^+_2}(x_2)^k(v,\rK_l(\overline{x_2^\itinv}))\bigr\Vert_{F,x_2}^2&=c_{k,l,-}^\vee|v|_{V_l^\vee}^2. 
\end{align*}
Then the operator norms of $\cF_{\lambda,k,l,-}^\uparrow$, $\cF_{\lambda,k,l,-}^\downarrow$ are given by 
\[ c_{k,l,-}^{-1}\Vert \cF_{\lambda,k,l,-}^\uparrow\Vert_{\mathrm{op}}^2=c_{k,l,-}^\vee\Vert \cF_{\lambda,k,l,-}^\downarrow\Vert_{\mathrm{op}}^{-2}=C_{r,-}^{d,d_2}(\lambda,k,l), \]
where $C_{r,-}^{d,d_2}(\lambda,k,l)$ is as in (\ref{const_general-}). 
\end{enumerate}
\end{theorem}

The family of differential operators $\cF_{\lambda,k,l,\pm}^\downarrow$ is meromorphically continued for all $\lambda\in\BC$, and gives rise to $\widetilde{G}_1$-intertwining operators 
\begin{align*}
&\cF_{\lambda,k,l,+}^\downarrow\colon \cO_\lambda(D)|_{\widetilde{G}_1}\longrightarrow\cO_{\varepsilon_1(\lambda+2k)}(D_1,V_l), \\
&\cF_{\lambda,k,l,-}^\downarrow\colon \cO_\lambda(D)|_{\widetilde{G}_1}\longrightarrow\cO_{\varepsilon_1(\lambda+2k)}(D_1,V_l^\vee) 
\end{align*}
for all $\lambda\in\BC$ except at its poles. Then by Theorems \ref{thm_general+}\,(3), \ref{thm_general-}\,(3), for special $\lambda$ the following holds. 

\begin{theorem}[{\cite[Theorem 8.9\,(1)]{N2}, \cite[Corollary 6.8\,(4)]{N3}}]
Let $k,l\in\BZ_{\ge 0}$. For $a=1,2,\ldots,k$, we have 
\begin{align*}
\cF_{\frac{n}{r}-a,k,l,+}^\downarrow&=\cF_{\frac{n}{r}+a,k-a,l,+}^\downarrow\circ\det_{\fn^-}\biggl(\frac{\partial}{\partial x}\biggr)^a\colon 
\cO_{\frac{n}{r}-a}(D)|_{\widetilde{G}_1}\longrightarrow \cO_{\varepsilon_1\left(\frac{n}{r}-a+2k\right)}(D_1,V_l). 
\end{align*}
Similarly, when $k\ge l$, for $a=1,2,\ldots,k-l$, we have 
\begin{align*}
\cF_{\frac{n}{r}-a,k,l,-}^\downarrow&=\cF_{\frac{n}{r}+a,k-a,l,-}^\downarrow\circ\det_{\fn^-}\biggl(\frac{\partial}{\partial x}\biggr)^a\colon 
\cO_{\frac{n}{r}-a}(D)|_{\widetilde{G}_1}\longrightarrow \cO_{\varepsilon_1\left(\frac{n}{r}-a+2k\right)}(D_1,V_l^\vee). 
\end{align*}
\end{theorem}

For $\lambda>\frac{n}{r}-1$, we have 
\begin{align*}
\cH_{\varepsilon_1(\lambda+2k)}(D_1,V_l)_{\widetilde{K}_1}&\simeq d\tau_\lambda(\mathcal{U}(\fg_1))\cP_{(k+l,\underline{k}_{r-1})}(\fp^+_2)\subset \cH_\lambda(D)_{\widetilde{K}}, \\
\cH_{\varepsilon_1(\lambda+2k)}(D_1,V_l^\vee)_{\widetilde{K}_1}&\simeq d\tau_\lambda(\mathcal{U}(\fg_1))\cP_{(\underline{k}_{r-1},k-l)}(\fp^+_2)\subset \cH_\lambda(D)_{\widetilde{K}}, 
\end{align*}
but this does not hold for smaller $\lambda$ in general. For such $\lambda$, by Theorems \ref{thm_general+}\,(2), \ref{thm_general-}\,(2), the following holds. 
For $i=1,2,\ldots,r$, $\lambda\in\frac{d}{2}(i-1)-\BZ_{\ge 0}$, let $M_i^\fg(\lambda)$ be the $(\fg,\widetilde{K})$-submodule given in (\ref{submodule}). 

\begin{theorem}
\begin{enumerate}
\item Let $k,l\in\BZ_{\ge 0}$. For $i=1,2,\ldots,r$, 
\[ d\tau_\lambda(\mathcal{U}(\fg_1))\cP_{(k+l,\underline{k}_{r-1})}(\fp^+_2)\subset M_i^\fg(\lambda) \]
holds if and only if 
\[ \lambda\in\begin{cases} \frac{d}{2}(i-1)-(2k+l)-\BZ_{\ge 0} & (i=1), \\
\frac{d}{2}(i-1)-2k-\BZ_{\ge 0} & (2\le i\le \lfloor r/2\rfloor), \\
\frac{d}{2}(i-1)-\min\{2k,k+l\}-\BZ_{\ge 0} & (r\colon\text{odd},\,i=\lceil r/2\rceil), \\
\frac{d}{2}(i-1)-k-\BZ_{\ge 0} & (\lceil r/2\rceil+1\le i\le r). \end{cases} \]
\item Let $k,l\in\BZ_{\ge 0}$ with $k\ge l$. For $i=1,2,\ldots,r$, 
\[ d\tau_\lambda(\mathcal{U}(\fg_1))\cP_{(\underline{k}_{r-1},k-l)}(\fp^+_2)\subset M_i^\fg(\lambda) \]
holds if and only if 
\[ \lambda\in\begin{cases} \frac{d}{2}(i-1)-2k-\BZ_{\ge 0} & (1\le i\le \lceil r/2\rceil-1), \\
\frac{d}{2}(i-1)-(2k-l)-\BZ_{\ge 0} & (r\colon\text{even},\,i=r/2), \\
\frac{d}{2}(i-1)-k-\BZ_{\ge 0} & (\lfloor r/2\rfloor+1\le i\le r-1), \\
\frac{d}{2}(i-1)-(k-l)-\BZ_{\ge 0} & (i=r). \end{cases} \]
\end{enumerate}
\end{theorem}

Here, the ``if'' direction follows from the holomorphy in Theorems \ref{thm_general+}\,(2), \ref{thm_general-}\,(2), and these are already given in \cite[Section 7]{N2}, \cite[Corollary 6.8\,(1)]{N3}. 
On the other hand, the ``only if'' direction follows from the non-vanishing in Theorems \ref{thm_general+}\,(2), \ref{thm_general-}\,(2), and these are new except for $r$ even or $kl=0$ for $+$ case.

%








\end{document}